\documentclass[11pt]{amsart}
\usepackage{amsmath,latexsym,amsfonts,amssymb,amsthm}
\usepackage{geometry}
\usepackage{fullpage}
\usepackage{mathrsfs}
\usepackage{graphicx,color}
\usepackage{tikz-cd}
\usepackage[all,cmtip]{xy}
\usepackage{enumitem}
\usepackage{verbatim}
\usepackage{bbm}
\usepackage{bm}
\usepackage{mathtools}
\usepackage[colorlinks=true, linkcolor=red!80!black, urlcolor=purple, citecolor=blue!70!black]{hyperref}
\usepackage{tabu}

\numberwithin{equation}{section}
\newtheorem{theorem}{Theorem}[section]
\newtheorem{prop}[theorem]{Proposition}
\newtheorem{lemma}[theorem]{Lemma}

\newtheorem{thm}[theorem]{Theorem}
\newtheorem{lem}[theorem]{Lemma}
\newtheorem{cor}[theorem]{Corollary}
\newtheorem{conj}[theorem]{Conjecture}

\theoremstyle{definition}

\newtheorem{definition}[theorem]{Definition}
 \newtheorem{question}[theorem]{Question}
 \newtheorem{defn}[theorem]{Definition}
 \newtheorem{expl}[theorem]{Example}

\newtheorem{remark}[theorem]{Remark}

\newtheorem{notation}[theorem]{Notation}

\newtheorem{rem}[theorem]{Remark}

\newcommand{\bC}{\mathbb{C}}

\newcommand{\bQ}{\mathbb{Q}}
\newcommand{\bR}{\mathbb{R}}
\newcommand{\bG}{\mathbb{G}}
\newcommand{\calY}{\mathcal{Y}}
\newcommand{\calL}{\mathcal{L}}

\newcommand{\calF}{\mathcal{F}}
\newcommand{\calI}{\mathcal{I}}
\newcommand{\calX}{\mathcal{X}}
\newcommand{\calE}{\mathcal{E}}
\newcommand{\Hilb}{\mathrm{Hilb}}
\newcommand{\Sym}{\mathrm{Sym}}
\newcommand{\PGL}{\mathrm{PGL}}
\newcommand{\GL}{\mathrm{GL}}
\newcommand{\Hom}{\mathrm{Hom}}

\newcommand{\Proj}{\mathrm{Proj}}

\newcommand{\Pic}{\mathrm{Pic}}
\newcommand{\Aut}{\mathrm{Aut}}

\newcommand{\Spec}{\mathrm{Spec}\;}

\newcommand{\gquot}{/\!\!/}
\newcommand{\PP}{\mathbb{P}}
\newcommand{\bP}{\mathbb{P}}
\newcommand{\bA}{\mathbb{A}}

\newcommand{\SL}{\mathrm{SL}}
\newcommand{\calD}{\mathcal{D}}
\newcommand{\calO}{\mathcal{O}}
\newcommand{\cO}{\mathcal{O}}
\newcommand{\calP}{\mathcal{P}}

\newcommand{\calN}{\mathcal{N}}
\newcommand{\calU}{\mathcal{U}}

\newcommand{\calG}{\mathcal{G}}

\newcommand{\cA}{\mathcal{A}}
\newcommand{\calB}{\mathcal{B}}
\newcommand{\fa}{\mathfrak{a}}

\usepackage{graphicx}
\newcommand{\sslash}{\mathbin{/\mkern-6mu/}}

\newcommand{\Q}{\mathbb{Q}}
\newcommand{\bZ}{\mathbb{Z}}

\newcommand{\Chow}{\operatorname{Chow}}

\newcommand{\lc}{\mathrm{lc}}
\newcommand{\oM}{\overline M}

\newcommand{\cX}{\mathcal X}
\newcommand{\cD}{\mathcal D}
\newcommand{\cU}{\mathcal{U}}
\newcommand{\cC}{\mathcal{C}}
\newcommand{\cW}{\mathcal{W}}
\newcommand{\hcU}{\widehat{\cU}}
\newcommand{\hcW}{\widehat{\cW}}
\newcommand{\of}{\overline{f}}

\newcommand{\quotient}{/\hspace{-1.2mm}/}

\newcommand{\DF}{\mathrm{DF}_{\calD}(\calX, \calL)}

\newcommand{\Supp}{\textrm{Supp}}

\newcommand{\hvol}{\widehat{\mathrm{vol}}}

\newcommand{\ord}{\mathrm{ord}}
\newcommand{\Val}{\mathrm{Val}}

\newcommand{\id}{\mathrm{id}}

\newcommand{\vol}{\mathrm{vol}}

\newcommand{\ind}{\mathrm{ind}}
\newcommand{\GIT}{\mathrm{GIT}}

\newcommand{\CM}{\mathrm{CM}}
\newcommand{\CH}{\mathrm{CH}}
\newcommand{\klt}{\mathrm{klt}}
\newcommand{\kst}{\mathrm{kst}}
\newcommand{\oZ}{\overline{Z}}
\newcommand{\oC}{\overline{C}}
\newcommand{\oQ}{\overline{Q}}
\newcommand{\oW}{\overline{W}}

\newcommand{\lct}{\mathrm{lct}}

\newcommand{\cM}{\mathcal M}
\newcommand{\cY}{\mathcal Y}
\newcommand{\cN}{\mathcal N}
\newcommand{\cZ}{\mathcal Z}
\newcommand{\cL}{\mathcal L}
\newcommand{\cK}{\mathcal{K}}
\newcommand{\hS}{\widehat{S}}
\newcommand{\hE}{\widehat{E}}
\newcommand{\fD}{\mathfrak{D}}
\newcommand{\tM}{\widetilde{M}}
\newcommand{\tN}{\widetilde{N}}
\newcommand{\tfD}{\widetilde{\fD}}
\newcommand{\tsigma}{\tilde{\sigma}}
\newcommand{\tu}{\tilde{u}}
\newcommand{\Fut}{\mathrm{Fut}}
\newcommand{\fm}{\mathfrak{m}}
\newcommand{\bH}{\mathbb{H}}

\newcommand{\ocX}{\overline{\cX}}
\newcommand{\ocD}{\overline{\cD}}
\newcommand{\aut}{\mathfrak{aut}}
\newcommand{\U}{\mathrm{U}}
\newcommand{\tDelta}{\widetilde{\Delta}}
\newcommand{\calC}{\mathcal C}
\newcommand{\rmS}{\mathrm{S}}

\newcommand{\wt}{\mathrm{wt}}
\newcommand{\cE}{\mathcal{E}}
\newcommand{\Hodge}{\mathrm{Hodge}}
\newcommand{\Ric}{\mathrm{Ric}}

\newcommand{\K}{\mathrm{K}}
\newcommand{\oP}{\overline{P}}
\newcommand{\ocP}{\overline{\mathcal{P}}}
\newcommand{\cP}{\mathcal{P}}
\newcommand{\bfP}{\mathbf{P}}
\newcommand{\bfA}{\mathbf{A}}
\newcommand{\oSigma}{\overline{\Sigma}}
\newcommand{\bfV}{\mathbf{V}}
\newcommand{\hU}{\widehat{U}}
\newcommand{\cI}{\mathcal{I}}
\newcommand{\cQ}{\mathcal{Q}}
\newcommand{\cV}{\mathcal{V}}
\newcommand\numberthis{\addtocounter{equation}{1}\tag{\theequation}}
\newcommand{\cG}{\mathcal{G}}
\newcommand{\cF}{\mathcal{F}}

\begin{document}
\title{Wall crossing for K-moduli spaces of plane curves}
\author[Ascher]{Kenneth Ascher}
\address{Department of Mathematics, University of California, Irvine, CA, 92697, USA}
\email{kascher@uci.edu}
\author[DeVleming]{Kristin DeVleming}
\address{Department of Mathematics and Statistics,
University of Massachusetts, Amherst, MA 01003-9305, USA}
\email{kdevleming@umass.edu}
\author[Liu]{Yuchen Liu}
\address{Department of Mathematics, Northwestern University, Evanston, IL 60208, USA}
\email{yuchenl@northwestern.edu}
\date{\today}

\begin{abstract}

We construct proper good moduli spaces parametrizing K-polystable $\bQ$-Gorenstein smoothable log Fano pairs $(X, cD)$, where $X$ is a Fano variety and $D$ is a rational multiple of the anti-canonical divisor. We then establish a wall-crossing framework of these K-moduli spaces as $c$ varies. The main application in this paper is the case of plane curves of degree $d \geq 4$ as boundary divisors of $\bP^2$. In this case, we show that when the coefficient $c$ is small, the K-moduli space of these pairs is isomorphic to the GIT moduli space. We then show that the first wall crossing of these K-moduli spaces are weighted blow-ups of Kirwan type. 
 We also describe all wall crossings for degree $4,5,6$  and relate the final K-moduli spaces to Hacking's compactification and the moduli of K3 surfaces.
\end{abstract}

\pagenumbering{gobble}

\pagenumbering{arabic}

\maketitle{}

\section{Introduction}

Constructing compactifications of moduli spaces of varieties is a fundamental question in algebraic geometry. While it is well known that singular varieties will occur on the boundary of the moduli space, an interesting question to investigate is how different compactifications yield different singularities on the boundary. The classical approach of Deligne and Mumford for genus $\geq 2$ curves only allows stable curves (i.e. nodal singularities) in the boundary. Later on, it was realized that there are meaningful alternate compactifications allowing curves with worse singularities that can be obtained from running the minimal model program (MMP) on the Deligne-Mumford compactification, often going by the name the \emph{Hassett-Keel program}.

If we consider plane curves, there are two well-known compactifications of the moduli space. Classically, Mumford's geometric invariant theory (GIT) yields a projective variety
$\oP_d^{\rm GIT}$ parametrizing $\rmS$-equivalent classes of GIT semistable plane curves of degree $d$. Moreover,
in each $\rmS$-equivalence class, there exists a unique
closed orbit whose representative is a GIT polystable plane curve.
In the philosophy of Alper, this gives a \emph{good} moduli space morphism $\ocP_d^{\GIT}\to \oP_d^{\GIT}$ where $\ocP_d^{\GIT}$
is the GIT quotient stack.  An approach due to Hacking
(following ideas from Koll\'ar-Shepherd-Barron and Alexeev),
is via \emph{stable pairs} and the MMP \cite{Hac01, Hac04}. 
Roughly speaking, each smooth plane curve $C$
of degree $d\geq 4$ can be viewed as a boundary divisor
in $\bP^2$, and so one can study certain pairs of degenerations of the plane with a curve subject to some conditions on positivity and singularities of the pair coming from the MMP. Hacking's moduli stack $\ocP_d^{\rm H}$
is a proper Deligne-Mumford stack whose coarse moduli space
$\oP_d^{\rm H}$ is a projective variety.


One notable feature of these two different approaches is 
that although GIT semistable curves all lie in $\bP^2$,
they can be quite singular; on the other hand, although
the surfaces in the boundary of $\ocP_d^{\rm H}$
can be quite singular (possibly non-normal), the singularities
of the degenerate curve are usually mild. As these two moduli spaces are birational, it is a natural question to ask how to interpolate between them. 

For example, we illustrate the simplest non-trivial case: $\deg d =4$. The $6$-dimensional GIT quotient $\oP_4^{\GIT}$ generically parametrizes curves which are at worst cuspidal in $\bP^2$, has a curve parametrizing tacnodal curves (i.e. locally $y^2 + x^4=0$), and a point parametrizing the double conic. On the other hand, Hacking's space $\oP_4^{\rm H}$ (see also \cite{hassettquartic}) also generically parametrizes cuspidal curves in $\bP^2$, but has a divisor parametrizing curves on $\bP(1,1,4)$ (which are at worst cuspidal away from the singular point, and at worst nodal at the singularity), and has a codimension two locus parametrizing curves on the non-normal surface $\bP(1,1,2) \cup \bP(1,1,2)$. Here, the curves are snc at the double locus and at worst cuspidal elsewhere. In particular, one can see directly the trade-off between having very singular curves that still are in $\bP^2$, and having mildly singular curves on singular surfaces. It is natural to ask how to relate the two spaces in a modular way.

In this article, we investigate a new family of compactifications of the moduli space of smooth plane curves
using K-stability and (conical) K\"ahler-Einstein metrics.
For any smooth plane curve $C$ of degree $d\geq 4$,
the celebrated work \cite{CDS15} and \cite{Tia15} implies that
$(\bP^2,cC)$ admits a conical K\"ahler-Einstein metric
for any $0<c<\frac{3}{d}$ hence is K-stable. Thus it is natural to construct K-stability compactifications of these moduli spaces.
Recently Li, Wang, and Xu \cite{LWX14} showed that there exist proper good moduli spaces parametrizing K-polystable $\bQ$-Gorenstein smoothable Fano varieties (see also \cite{Oda15} and \cite{SSY16}).
Based on \cite{LWX14} and the very recent work by Tian and Wang \cite{TW19} on the solution of the log smooth Yau-Tian-Donaldson Conjecture, we construct K-moduli stacks and spaces for all  $\bQ$-Gorenstein smoothable log Fano pairs (see Definition \ref{defn:qgorsmoothable}). In particular, this implies that the K-stability compactification of the moduli space of log Fano pairs
$(\bP^2, cC)$ with $0<c<\frac{3}{d}$ a rational number exists as a proper good moduli space.

\begin{theorem}[=Theorem \ref{thm:lwxlog}]
 Let $\chi_0$ be the Hilbert polynomial of an anti-canonically polarized Fano manifold. Fix $r\in\bQ_{>0}$ and a rational number $c\in (0,\min\{1,r^{-1}\})$. Then there exists a reduced Artin stack  $\cK\cM_{\chi_0,r,c}$ of finite type over $\bC$ parametrizing all K-semistable $\bQ$-Gorenstein smoothable log Fano pairs $(X,cD)$ with Hilbert polynomial $\chi(X,\cO_X(-mK_X))=\chi_0(m)$ for sufficiently divisible $m$ and $D\sim_{\bQ}-rK_X$. Moreover, the Artin stack $\cK\cM_{\chi_0,r,c}$ admits a good moduli space $KM_{\chi_0,r,c}$ as a proper reduced scheme of finite type over $\bC$ whose closed points parametrize K-polystable pairs.\end{theorem}

It is thus natural to ask how the moduli spaces depend on the coefficient $c$. In this setting, we prove the following wall-crossing type result. 

\begin{theorem}[=Theorem \ref{thm:logFano-wallcrossing}]\label{thm:introwallcrossing}
 There exist rational numbers \[0=c_0<c_1<c_2<\cdots<c_k=\min\{1,r^{-1}\}\] such that $c$-K-(poly/semi)stability conditions do not change for $c\in (c_i,c_{i+1})$. For each $1\leq i\leq k-1$ and $0<\epsilon\ll 1$, we have open immersions
 \[
 \cK\cM_{\chi_0,r,c_i-\epsilon}\xhookrightarrow{\Phi_{i}^-}
 \cK\cM_{\chi_0,r,c_i}\xhookleftarrow{\Phi_{i}^+}\cK\cM_{\chi_0,r,c_i+\epsilon}
 \]
 which induce projective morphisms
 \[
 KM_{\chi_0,r,c_i-\epsilon}\xrightarrow{\phi_{i}^-}
 KM_{\chi_0,r,c_i}\xleftarrow{\phi_{i}^+}KM_{\chi_0,r,c_i+\epsilon}
 \]
 Moreover, all the above wall crossing morphisms have local VGIT presentations as in \cite[(1.2)]{AFS17}, and the CM $\bQ$-line bundles on $KM_{\chi_0, r,c_i\pm\epsilon}$ are $\phi_i^{\pm}$-ample.
\end{theorem}

While the two above results hold for any K-moduli stack and space of $\bQ$-Gorenstein smoothable log Fano pairs, the remainder of this paper will focus on $\ocP^\K_{d,c}$ and $\oP_{d,c}^{\K}$, that is, the K-moduli stack and space parametrizing K-semistable and K-polystable limits of $(\bP^2, cC)$, respectively, where $C$ is a smooth plane curve of degree $d$ and $c\in (0,\frac{3}{d})$ is a rational number.

Since K-semistable log Fano pairs
are always Kawamata log terminal (klt) by \cite{Oda13b}, as $c$ increases the surface $X$ in $\oP^\K_{d,c}$ must become
more singular, while the divisor $D$ becomes less singular; this is a general version of the phenomenon seen in the degree 4 example. In particular, it is reasonable to expect that these moduli spaces provide the proper framework for interpolating between $\ocP_d^{\GIT}$ and $\ocP_d^{\rm H}$. Our next results characterize
the behavior of the wall crossings for K-moduli spaces of plane curves. 

First, we give a complete understanding of the first wall crossing in all degrees. The K-moduli space corresponding to $0 < c \ll 1$ is isomorphic to the GIT moduli space, and the first wall crossing is a Kirwan type blowup of the GIT quotient. Note that since the K-moduli stacks and spaces for $d\leq 3$ are well-known (see Example \ref{expl:deg123}), we usually assume $d\geq 4$.

\begin{theorem}[First wall crossing]\label{mthm:firstwall}
Let $d \geq 4$ be an integer, $c \in (0, \frac{3}{d})$ be a rational number, let  $Q$ be a smooth conic in $\bP^2$, let $L$ be a line in $\bP^2$ transverse to $Q$, and let $x,y,z$ be coordinates of $\bP(1,1,4)$. Let \[
c_1 = \left\{ \begin{array}{lr} \frac{3}{2d} & d \textrm{ is even} \\ \frac{3}{2d-3} & d \textrm{ is odd }  \end{array}\right. \qquad
Q_d = \left\{ \begin{array}{lr} \frac{d}{2}Q & d \textrm{ is even} \\ \frac{d-1}{2}Q+L & d \textrm{ is odd }\end{array}\right. \quad
Q'_d = \left\{ \begin{array}{lr} z^{d/2} = 0 & d \textrm{ is even} \\ xyz^{(d-1)/2}=0  & d \textrm{ is odd }\end{array}\right. \qquad
\]
\begin{enumerate}
\item For any $0 < c < c_1$, a plane curve $C$ of degree $d$ is GIT (poly/semi)stable if and only if the log Fano pair $(\bP^2, cC)$ is K-(poly/semi)stable. Moreover, there is an isomorphism of Artin stacks $\ocP^{\K}_{d,c} \cong \ocP_d^{\GIT}$. 
\item There is an open immersion $\Phi^-: \ocP_d^{\GIT} = \ocP^\K_{d, c_1 - \epsilon} \hookrightarrow \ocP^\K_{d, c_1}$ which descends to an isomorphism of good moduli spaces. 
\item If $\Phi^+:\ocP_{d,c_1+\epsilon}^{\K}
\to \overline{\calP}^\K_{d,c_1}$ denotes the latter morphism in the first wall crossing, then there exists a stacky weighted blow up morphism $\rho:\ocP_{d,c_1+\epsilon}^{\K}\to \ocP_{d,c_1-\epsilon}^{\K}=\ocP_d^{\GIT}$ along $\{[Q_d]\}$ (see Definition \ref{def:weightedstackyblowup}) such that $\Phi^+=\Phi^-\circ\rho$.  
In particular, we have
\begin{enumerate}
\item The descent morphism  $\varrho=(\phi^{-})^{-1}\circ \phi^+ :\oP_{d,c_1+\epsilon}^{\K}
\to \oP^\K_{d,c_1-\epsilon}=\oP_d^{\GIT}$ of $\rho$ between good moduli spaces is a weighted blowup of the point $[Q_d]$.
\item If $d$ is even, then $\varrho$ is a partial desingularization of Kirwan type.
\end{enumerate}
\end{enumerate}
\end{theorem}

Before preceding, we note that Gallardo, Martinez-Garcia, and Spotti independently showed in \cite[Theorem 1.2]{GMGS} that a similar result to Theorem \ref{mthm:firstwall} (1) holds for all hypersurfaces in $\bP^n$ assuming the Gap Conjecture \cite[Conjecture 5.5]{SS17} which is true when $n\leq 3$ by \cite[Proposition 4.10]{LL16} and \cite[Theorem 1.3]{LX17b}. The following result removes this assumption.

\begin{theorem}[=Theorem \ref{thm:highdim}]\label{mthm:highdim}
Let $n$ and $d\geq 2$ be positive integers. Then there exists a positive rational number $c_1=c_1(n,d)$ such that for any fixed $0<c<c_1$, a hypersurface $S\subset\bP^n$ of degree $d$ is GIT (poly/semi)stable if and only if the log Fano pair $(\bP^n, cS)$ is K-(poly/semi)stable.
\end{theorem}

It is natural to ask what happens beyond the first wall crossing for plane curves. When $d$ is small ($d \leq 6$), we explicitly determine all K-moduli wall crossings. In fact, if $d$ is 4 or 6, we can relate our moduli spaces to Baily-Borel compactifications of moduli spaces of K3 surfaces (see Section \ref{sec:lowdegree}). Let $\oP^*_4$ denote the Baily-Borel compactification of the moduli space of ADE K3 surfaces of degree 4 with $\mathbb{Z}/4\mathbb{Z}$ symmetry constructed by Kond\={o} \cite{kondok3}, and let $\oP^*_6$ denote the Baily-Borel compactification of the moduli space of K3 surfaces of degree 2.

\begin{theorem}[$d = 4, 6$, see Theorem \ref{thm:quartsext} and Section \ref{sec:K3surface}] If $d = 4$ (resp. 6), there is only one wall crossing for K-moduli spaces given by the weighted blowup $\widehat{P}^\GIT_4 \to \oP^\GIT_4$ (resp. $\widehat{P}^\GIT_6 \to \oP^\GIT_6$) at the double conic (resp. triple conic).  Furthermore, the ample model of the Hodge line bundle on $\widehat{P}^\GIT_4$ (resp. $\widehat{P}^\GIT_6$) is $\oP^*_4$ (resp. $\oP^*_6)$. \end{theorem}

In fact, in the degree 4 case we can say more using Hyeon and Lee's results on the log minimal model program for moduli of genus three curves \cite{hyeon2010log}; see Section \ref{sec:quartics}.

\begin{theorem}[$d = 5$, see Theorems \ref{thm:firstwallps}, \ref{thm:secondwall} and Section \ref{sec:quintics}]
If $d=5$, then there are five wall crossings for K-moduli spaces
  of plane quintics. Among them, the first two are weighted blow-ups while the last three are flips. 
 \end{theorem} 
  
   In Table \ref{table:quintic}, we summarize the behavior of all wall crossings for plane quintics (see also Figure \ref{fig:quintics} on Page \pageref{fig:quintics}). Here we denote by $E_i^{\pm}$ the exceptional loci of $\phi_i^{\pm}:\oP_{5,c_i\pm\epsilon}^{\K}\to\oP_{5, c_i}^{\K}$ where general pairs parametrized by them are described in the table. The full description of $E_i^{\pm}$ will be presented in Theorems \ref{thm:k=gitp114}, \ref{thm:secondwall}, and Section \ref{sec:quintics}. We use $[x,y,z]$ for projective coordinates of weighted projective planes $\bP(1,1,4)$ and $\bP(1,4,25)$. The surface $X_{26}$ is the degree $26$ weighted hypersurface $(xw=y^{13}+z^2)$ in $\bP(1,2,13,25)$ with projective coordinates $[x,y,z,w]$.

\begin{table}[ht!]
\[
    \begin{tabu}{|c|c|c|c|}\hline
       i & c_i & E_i^- & E_i^+ \\\hline
       1 & \frac{3}{7} & (\bP^2, Q_5) & (\bP(1,1,4), (xyz^2+(ax^6+by^6)z+g(x,y)=0))\\
       2 & \frac{8}{15} & (\bP^2, A_{12}\textrm{-quintic}) & (X_{26}, (w=g(x,y)))\\
       3 & \frac{6}{11} & (\bP^2, A_{11}\textrm{-reducible quintics}) &       (\bP(1,1,4), (x^2 z^2 + y^6 z+ g(x,y)=0))\\
       4 & \frac{63}{115} & (\bP^2, A_{11}\textrm{-irreducible quintics}) &
        (\bP(1,4,25), (z^2+x^2 y^{12}+x^{10} g(x,y)=0))\\
       5 & \frac{54}{95} & (\bP^2, A_{10}\textrm{-quintics}) & (\bP(1,4,25), (z^2+x^6 y^{11}+x^{14} g(x,y)=0))\\
       \hline
    \end{tabu}
\]
    \caption{Wall crossings for K-moduli spaces of plane quintics}
    \label{table:quintic}
\end{table}

In general it is expected that K-moduli spaces are projective with ample CM line bundles (see \cite{CP18, BX18}). Using recent work of Codogni and Patakfalvi \cite{CP18} and Posva \cite{Pos19}, we show the following.

\begin{theorem}[=Theorem \ref{thm:projectivity}]
The K-moduli spaces $\oP_{d,c}^{\K}$ are projective when $d\in \{4,5,6\}$ with ample CM line bundles.
\end{theorem}

During the review process of this paper, we learned that the ampleness of CM line bundles on K-moduli spaces of log Fano pairs is proved in \cite{XZ19, LXZ21} using purely algebraic methods (see Remark \ref{rem:ps}). In particular,  the CM line bundle is ample on $\oP_{d,c}^{\K}$ for all degrees and all coefficients (see Theorem \ref{thm:cm-ample}). 

The remainder of our paper is devoted to some further questions which will serve as motivation for our future work. For example, in Section \ref{sec:higherdegree} we discuss the second weighted blow-up for $d \geq 7$ (see Theorem \ref{thm:firstwalls}). Another consequence of our work is that the birational maps $\oP^\K_{d, c'} \dashrightarrow \oP^\K_{d, c}$ are birational contractions for $0 < c < c' < \frac{3}{d}$  whenever $3 \mid d$ or $d < 13$ (see Theorem \ref{thm:contraction}).  If this is true for $3\nmid d$ and  $d \geq 13$, then together with the ampleness of the CM line bundle (Theorem \ref{thm:cm-ample}), this would imply that the wall crossing of K-moduli spaces exhibit similar behavior to the Hassett-Keel program for $\overline{M}_g$ (see Theorem \ref{thm:hassetkeel}). 

We show in Theorem \ref{thm:comparison} that the only difference between the K-moduli space $\oP^{\K}_{d,\frac{3}{d}-\epsilon}$ and $\oP^{\rm H}_d$ are the maximally lc pairs in the K-moduli space, and the non-normal pairs in the Hacking moduli space. We conjecture that there is a proper good moduli space of log Calabi-Yau pairs, which relates to the K-moduli and Hacking moduli spaces via the following wall crossing. 

\begin{conj}[Log Calabi-Yau wall crossings, see Conjecture \ref{conj:logCY}]
 There exists a proper good moduli space $\oP_{d}^{\rm CY}$
 which parametrizes $\rmS$-equivalence classes of
 \emph{semistable} log Calabi-Yau pairs $(X,\frac{3}{d}D)$
 where $X$ admits a $\bQ$-Gorenstein smoothing to $\bP^2$.
 Moreover, we have a log Calabi-Yau wall crossing diagram
\[
 \oP^{\K}_{d,\frac{3}{d}-\epsilon}\xrightarrow{~\phi^{-}_{\rm CY}~}
 \oP_{d}^{\rm CY}\xleftarrow{~\phi^{+}_{\rm CY}~}\oP_{d}^{\rm H}
\]
where $\oP_{d}^{\rm CY}$ is the common ample model of the Hodge line bundles on  $\oP^{\K}_{d,\frac{3}{d}-\epsilon}$ and $\oP_{d}^{\rm H}$.
\end{conj}

We partially verify this conjecture in degree $4,5,6$, and will investigate this space in forthcoming work.

\begin{remark}[Postscript]\label{rem:ps}
Since the first version of this article appeared on the arXiv, there has been much progress on the study of K-stability and K-moduli spaces of log Fano pairs and wall crossings. We list a few related works below. 

\begin{enumerate}
    \item The K-moduli spaces for log Fano pairs are shown to exist as a projective good moduli space where the CM line bundle is ample. This is a combination of many recent works \cite{Jia17, LWX18, CP18, BX18, ABHLX19, BLX19, Xu19, XZ19, XZ20, BHLLX20, LXZ21}. As a result, the construction of our moduli spaces $\cK\cM_{\chi_0, r,c}$ is generalized to all log Fano pairs using purely algebraic methods. 
    \item The Yau-Tian-Donaldson conjecture for all (possibly singular) log Fano pairs is solved as a combination of \cite{BBJ18, LTW19, Li19, LXZ21}.
    \item The wall crossing framework from section \ref{sec:construction}, except the local VGIT presentation from section \ref{sec:VGIT}, is generalized to all log Fano pairs $(X,D)$ satisfying $D\sim_{\bQ} -r K_X$ for some $r\in\bQ_{>0}$ in \cite{Zho21b} using purely algebraic methods.
    \item The wall crossing framework of this paper has been applied to the study of moduli of quartic K3 surfaces in \cite{ADL20, ADL21}. In particular, the paper \cite{ADL21} verifies Laza-O'Grady's conjecture on the Hassett-Keel-Looijenga program for quartic K3 surfaces (see \cite{LO19, LO18b, LO18a} for backgrounds).
    \item For moduli of stable pairs in terms of Koll\'ar-Shepherd-Barron-Alexeev, the wall crossing framework was established in \cite{ABIP21}. The main difference of these two wall crossing frameworks is that wall crossing maps in \cite{ABIP21} are always  morphisms (after normalization), while our wall crossing maps may be flips.
\end{enumerate}

\end{remark} 

\subsection*{Organization}
This paper is organized as follows. In Section \ref{sec:prelim} we collect preliminary materials on K-stability, normalized volumes, CM line bundles, good moduli spaces, and Hacking's moduli spaces. In Section \ref{sec:construction}, we give a detailed construction of K-moduli stacks and spaces of $\bQ$-Gorenstein smoothable log Fano pairs which is largely based on \cite{LWX14} with new inputs from \cite{Jia17, BX18, TW19}. We prove Theorem \ref{thm:lwxlog} which is a generalization of \cite[Theorem 1.3]{LWX14}. Our main new result is Theorem \ref{thm:logFano-wallcrossing} which characterizes fundamental behaviors of K-moduli wall crossings when varying the coefficient. Our construction heavily relies on the solution of  Yau-Tian-Donaldson Conjecture for log smooth log Fano pairs by Tian and Wang \cite{TW19} which is a generalization of \cite{CDS15, Tia15}. Hence our approach is a mixture of algebraic and analytic methods.

In Section \ref{sec:generaldegree}, we study the general properties of K-moduli stacks $\ocP_{d,c}^\K$ and spaces $\oP_{d,c}^{\K}$ of degree $d$ plane curves with coefficient $c$. We describe the well-known K-moduli stacks and spaces for degree at most $3$ in Example \ref{expl:deg123}. Using normalized volumes, we prove a result on bounding local Gorenstein indices of singular surfaces appearing in $\ocP_{d,c}^{\K}$ (see Theorem \ref{thm:localindex}). This is crucial in the detailed study of our K-moduli spaces.

Section \ref{sec:firstwall} is devoted to studying the first wall crossing in all degrees. We prove parts (1) and (2) of Theorem \ref{mthm:firstwall} in Section \ref{sec:firstwall1} by applying the index bound (Theorem \ref{thm:localindex}) and the Paul-Tian criterion (Theorem \ref{thm:paultian}). In Section  \ref{sec:firstwall2}, we show that the K-moduli stack $\ocP_{d,c_1+\epsilon}^{\K}$ is a weighted blow-up of the GIT moduli stack of Kirwan type, hence confirming part (3) of Theorem \ref{mthm:firstwall}. This is done by a careful analysis of GIT of curves on $\bP(1,1,4)$ (see Definition \ref{defn:gitp114} and Theorem \ref{thm:k=gitp114}) and the index bound (Theorem \ref{thm:localindex}). 

In Section \ref{sec:lowdegree}, we show that there is only one log Fano K-moduli wall crossing in degree $d=4$ or $6$ (see Theorem \ref{thm:quartsext}). This is proven by computing the log canonical thresholds of GIT polystable curves on $\bP^2$ and $\bP(1,1,4)$ (see Propositions \ref{prop:lctquartic} and \ref{prop:lctsextic}) and applying an interpolation result on K-stability (see Proposition \ref{prop:k-interpolation}). In Section \ref{sec:K3surface}, we relate the final K-moduli spaces $\oP_{d,\frac{3}{d}-\epsilon}^{\K}$ for $d=4$ or $6$ to the Baily-Borel compactification of moduli spaces of K3 surfaces as cyclic covers (see Theorems \ref{thm:logcy4} and \ref{thm:sextic-Hodge}). 

Sections \ref{sec:2ndwallquintics} and \ref{sec:quintics} are devoted to studying all wall crossings in degree $5$. In Section \ref{sec:2ndwallquintics}, we show that the second wall crossing of plane quintics precisely replaces the plane quintic with a unique $A_{12}$-singularity by curves on $X_{26}$ (see Theorem \ref{thm:secondwall}). Its proof involves a valuative criterion computation (see Proposition \ref{prop:a12onlyif}), an explicit construction of a special degeneration (see Proposition \ref{prop:a12deg}), and verifying the K-polystability of this degeneration using techniques of Ilten and S\"u{\ss} \cite{IS17} on $T$-varieties of complexity one (see Proposition \ref{prop:a12kps}).
In Section \ref{sec:quintics}, we use similar strategy to further study the rest wall crossings of plane quintics where the auxiliary computations are collected in Appendix \ref{sec:calculations}.
In Section \ref{sec:higherdegree}, we apply these results for quintics to get more information on the second weighted blow-up of K-moduli spaces in higher degrees (see Theorem \ref{thm:firstwalls}). 

In Section \ref{sec:questions}, we discuss further questions regarding our K-moduli spaces. In Section \ref{sec:projectivity}, we show that $\oP_{d,c}^{\K}$ is projective for degree $d\in \{4,5,6\}$ by proving the ampleness of CM line bundles using work of Codogni and Patakfalvi \cite{CP18} and Posva \cite{Pos19}. In Section \ref{sec:contraction}, we study the question of whether the birational map $\oP_{d,c'}^{\K}\dashrightarrow\oP_{d,c}^{\K}$ is a birational contraction when $0<c<c'<\frac{3}{d}$ (see Question \ref{question:contraction1}). We give affirmative answers when  $d\leq 13$ or $d$ is divisible by $3$ (see Theorem \ref{thm:contraction}). In Sections \ref{sec:logcy} and \ref{sec:logcyquintics} we provide evidence supporting Conjecture \ref{conj:logCY} on the log Calabi-Yau wall crossing when $d\in \{4,5,6\}$. In degree $4$, we relate our wall crossings to the log minimal model program for $\oM_3$ (see Section \ref{sec:quartics}). We give a set-theoretic description of the conjectural log Calabi-Yau moduli spaces of plane quintics in Section \ref{sec:logcyquintics}. Finally, we prove Theorem \ref{mthm:highdim} in Section \ref{sec:highdim} as an application of our machinery developed in Sections \ref{sec:prelim} and \ref{sec:construction}.


\subsection*{Acknowledgements}
This paper has benefited from many helpful discussions with Jarod Alper, Harold Blum, Brendan Hassett, J\'anos Koll\'ar, S\'andor Kov\'acs, Radu Laza, Chi Li, Zhiyuan Li, Zsolt Patakfalvi, Sam Payne, Chenyang Xu, Qizheng Yin, and Ziquan Zhuang. We wish to thank Valery Alexeev, Yongnam Lee, Yuji Odaka, and Hendrik S\"u{\ss} for many useful comments on a preprint. We are especially grateful to Xiaowei Wang for fruitful discussions regarding technical issues in Section \ref{sec:construction},  Feng Wang for kindly providing a proof of Theorem \ref{thm:lwx4.1}, and Quentin Posva for providing us a draft of \cite{Pos19}. 

We thank Patricio Gallardo for several helpful correspondences regarding GIT for plane quintics and Kirwan desingularizations of GIT quotients during the beginning stages of this project. We also note that Theorem \ref{mthm:firstwall} was proven independently of the related \cite[Theorem 1.2]{GMGS}.

Parts of this paper were completed while the authors were in residence at MSRI in Spring 2019. KA was supported in part by an NSF Postdoctoral Fellowship, and would like to thank the math department of University of Washington for providing wonderful visiting conditions. KD was partially supported by the Gamelin Endowed Postdoctoral Fellowship of the MSRI (NSF No. DMS-1440140).  YL was partially supported by the Della Pietra Endowed Postdoctoral Fellowship of the MSRI (NSF No. DMS-1440140). 

Finally, we thank S\'andor Kov\'acs and Karl Schwede for organizing a special session at the AMS Sectional at Portland State University in Apr. 2018 where this collaboration began.

\tableofcontents

\newcommand{\cT}{\mathcal{T}}
\newcommand{\fB}{\mathfrak{B}}

\newcommand{\dist}{\mathrm{dist}}
\newcommand{\lcm}{\mathrm{lcm}}
\newcommand{\red}{\mathrm{red}}
\newcommand{\tphi}{\tilde{\phi}}
\newcommand{\pr}{\mathrm{pr}}
\newcommand{\BO}{\overline{BO}}
\newcommand{\an}{\mathrm{an}}
\newcommand{\oT}{\overline{T}}
\newcommand{\ocL}{\overline{\cL}}
\newcommand{\sP}{\mathscr{P}}

\section{Preliminaries}\label{sec:prelim}

\subsection{K-stability of log Fano pairs}

In this section, we give a review of K-stability of log Fano pairs.

\begin{defn}
Let $X$ be a normal variety. Let $D$ be an effective $\bQ$-divisor on $X$.
Then $(X,D)$ is called a \emph{log pair}. If in addition $X$ is projective and  $-(K_X+D)$ is $\bQ$-Cartier ample, then we say that $(X,D)$ is a \emph{log Fano pair}. If a log Fano pair $(X,D)$ is klt, then we say that it is a \emph{klt log Fano pair}. We say that $X$ is a \emph{$\bQ$-Fano variety} if $(X,0)$ is a klt log Fano pair.
\end{defn}

We first recall the definition of a test configuration.

\begin{defn}[\cite{Tia97, Don02}]
Let $X$ be a projective variety. Let $L$ be an ample line bundle on $X$. 
\begin{enumerate}[label=(\alph*)]
\item A \emph{test configuration} $(\cX;\cL)/\bA^1$ of $(X;L)$ consists of the following data:
\begin{itemize}
 \item a variety $\cX$ together with a flat projective morphism $\pi:\cX\to \bA^1$;
 \item a $\pi$-ample line bundle $\cL$ on $\cX$;
 \item a $\bG_m$-action on $(\cX;\cL)$ such that $\pi$ is $\bG_m$-equivariant with respect to the standard action of $\bG_m$ on $\bA^1$ via multiplication;
 \item $(\cX\setminus\cX_0;\cL|_{\cX\setminus\cX_0})$
 is $\bG_m$-equivariantly isomorphic to $(X;L)\times(\bA^1\setminus\{0\})$.
\end{itemize}
\item Let $w_m$ be the weight of the $\bG_m$-action on the determinant line $\det H^0(X_0, L_0^{\otimes m})$, and $N_m:=h^0(X,L^{\otimes m})$. Then we have an asymptotic expansion 
\[
\frac{w_m}{mN_m}=F_0+m^{-1} F_1+m^{-2} F_2+\cdots
\]
with $F_i\in \bQ$. The \emph{generalized Futaki invariant} of $(\cX;\cL)/\bA^1$ is defined as $\Fut(\cX;\cL)=-2F_1$. More precisely, if we write
\[
N_m=a_0 m^n + a_1 m^{n-1} + O(m^{n-2}),\quad w_m=b_0 m^{n+1} + b_1 m^n + O(m^{n-1}),
\]
then $\Fut(\cX;\cL)=\frac{2(a_1 b_0-a_0 b_1)}{a_0^2}$.
\end{enumerate}
\end{defn}


\begin{defn}[\cite{Tia97, Don02, Li15, LX14, OS15}]
Let $(X,D=\sum_{i=1}^k c_i D_i)$ be a projective log pair. Let $L$ be an ample line bundle on $X$. 
\begin{enumerate}[label=(\alph*)]
\item A \emph{test configuration} $(\cX,\cD;\cL)/\bA^1$ of $(X,D;L)$ consists of the following data:
\begin{itemize}
\item a test configuration $(\cX;\cL)/\bA^1$ of $(X;L)$;
\item a formal sum $\cD=\sum_{i=1}^k c_i \cD_i$ of codimension one closed integral subschemes $\cD_i$ of $\cX$ such that $\cD_i$ is the Zariski closure of $D_i\times(\bA^1\setminus\{0\})$ under the identification between $\cX\setminus\cX_0$ and $X\times(\bA^1\setminus\{0\})$.
\end{itemize}
It is clear that $(\cD_i;\cL|_{\cD_i})/\bA^1$ is a test configuration of $(D_i; L|_{D_i})$.
\item  For each $1\leq i\leq k$, let $\tilde{w}_{i,m}$ be the weight of the $\bG_m$-action on the determinant line $\det H^0(D_{i,0}, L_{i,0}^{\otimes m})$, and $\tilde{N}_{i,m}:=h^0(D_{i},L_{i}^{\otimes m})$. Then we have an asymptotic expansion 
\[
 \tilde{N}_{i,m}=\tilde{a}_{i,0} m^{n-1}+O(m^{n-2}),\quad
 \tilde{w}_{i,m}=\tilde{b}_{i,0} m^{n}+O(m^{n-1}).
\]
We define $
\tilde{a}_0=\sum_{i=1}^k c_i \tilde{a}_{i,0}$ and $\tilde{b}_0=\sum_{i=1}^k c_i \tilde{b}_{i,0}$.
The \emph{relative Chow weight} of $(\cX,\cD;\cL)/\bA^1$ is defined as $\CH(\cX,\cD;\cL):=\frac{a_0\tilde{b}_0-b_0\tilde{a}_0}{a_0^2}$.
The \emph{generalized Futaki invariant} of $(\cX,\cD;\cL)/\bA^1$ is defined as $\Fut(\cX,\cD;\cL)=\Fut(\cX;\cL)+\CH(\cX,\cD;\cL)$.
\item A test configuration $(\cX,\cD;\cL)/\bA^1$ is called a \emph{normal} test configuration if $\cX$ is normal. 
A normal test configuration is called a \emph{product} test configuration if \[
(\cX,\cD;\cL)\cong(X\times\bA^1,D\times\bA^1;\pr_1^* L\otimes\cO_{\cX}(k\cX_0))
\] for some $k\in\bZ$. A product test configuration is called a \emph{trivial} test configuration if the above isomorphism is $\bG_m$-equivariant with respect to the trivial $\bG_m$-action on $X$ and the standard $\bG_m$-action on $\bA^1$ via multiplication. 

\item Let $(X,D)$ be a log Fano pair. Let $L$ be an ample line bundle on $X$ such that for some $l \in \bQ_{>0}$ we have $L\sim_{\bQ}-l(K_X+D)$. Then the log Fano pair $(X,D)$ is said to be:
\begin{enumerate}[label=(\roman*)]
    \item \emph{K-semistable} if $\Fut(\cX,\cD;\cL)\geq 0$ for any normal test configuration $(\cX,\cD;\cL)/\bA^1$ and any $l\in\bQ_{>0}$ such that $L$ is Cartier; 
    
    \item  \emph{K-stable} if it is K-semistable and $\Fut(\cX,\cD;\cL)=0$ for a normal test configuration $(\cX,\cD;\cL)/\bA^1$ if and only if it is a trivial test configuration; and
 
\item \emph{K-polystable} if it is K-semistable and $\Fut(\cX,\cD;\cL)=0$ for a normal test configuration $(\cX,\cD;\cL)/\bA^1$ if and only if it is a product test configuration.
\end{enumerate}
 
\item 
Let $(X,D)$ be a klt log Fano pair. Let $L$ be an ample line bundle on $X$ such that $L\sim_{\bQ}-l(K_X+D)$ for some $l\in\bQ_{>0}$. Then a normal test configuration $(\cX,\cD;\cL)/\bA^1$ is called a \emph{special test configuration} if $\cL\sim_{\bQ}-l(K_{\cX/\bA^1}+\cD)$ and $(\cX,\cD+\cX_0)$ is plt. In this case, we say that $(X,D)$ \emph{specially degenerates to} $(\cX_0,\cD_0)$ which is necessarily a klt log Fano pair.
\end{enumerate}
\end{defn}

\begin{rem} We give some useful remarks toward the above definition.
\begin{enumerate}
    \item We provide an intersection formula for the generalized Futaki invariant (cf. \cite{Wan12, Oda13a}). Let $(X,D)$ be a log Fano pair. Let $L$ be an ample line bundle on $X$ such that $L\sim_{\bQ}-l(K_X+D)$ for some $l\in\bQ_{>0}$. Assume $\pi:(\cX,\cD;\cL)\to \bA^1$ is a normal test configuration of $(X,D;L)$. Let $\bar{\pi}: (\bar{\cX},\bar{\cD};\bar{\cL})\to\bP^1$ be the natural $\bG_m$-equivariant compactification of $\pi$. Then we have the intersection formula
   \[
    \Fut(\cX,\cD;\cL):=\frac{1}{(-(K_X+D))^n}\left(\frac{n}{n+1}\cdot\frac{(\bar{\cL}^{n+1})}{l^{n+1}}+\frac{(\bar{\cL}^n\cdot (K_{\bar{\cX}/\bP^1}+\bar{\cD}))}{l^n}\right).
   \]
    \item By the work of Odaka \cite{Oda12}, any K-semistable log Fano pair is klt. By the work of Li and Xu \cite{LX14}, we know that to test K-(poly/semi)stability of a klt log Fano pair, it suffices to test only on special test configurations.
    \item A test configuration is called \emph{almost trivial} (resp. \emph{almost product}) if its normalization is trivial (resp. product). By \cite[Proposition 3.15]{BHJ17}, we know that the generalized Futaki invariant never increases under normalization. 
\end{enumerate}
\end{rem}



\begin{defn}
Let $X$ be a $\bQ$-Fano variety. Let $D\sim_{\bQ}-K_X$ be an effective $\bQ$-divisor on $X$. We say that $(X,D)$ is \emph{K-semistable} if $(X,D;L)$ is K-semistable for some Cartier divisor $L\sim_{\bQ} -lK_X$ and some $l\in\bZ_{>0}$. From \cite{Oda13b} this is equivalent to saying that $(X,D)$ is log canonical. 
\end{defn}

\subsection{Valuative criteria for K-stability}
In this section, we recall the \emph{valuative criteria} for K-stability due to \cite{Fuj16, Li17} with a slight improvement from \cite{BX18}. For this, we need to make a few definitions. 

\begin{definition} Let $X$ be a normal variety of dimension $n$. We say that $E$ is a \emph{prime divisor over} $X$ if $E$ is a divisor on a normal variety $Y$ where $f: Y \to X$ is a proper birational morphism.
Let $L$ be a $\bQ$-Cartier $\bQ$-divisor on $X$. Take $m \in \mathbb{Z}_{> 0}$ such that $mL$ is Cartier and let $x \in \bR_{\geq 0}$. If $X$ is projective, we define the \emph{volume} of $L-xE$ on $X$ as
\[\vol_X(L -xE) := \vol_Y(f^*L- xE) = \limsup_{\substack{m \to \infty \\ mL\textrm{ is Cartier}}} \dfrac{h^0(X, \cO_X(mL-\lceil mx \rceil E))}{m^n/n!}.\]\end{definition}

\begin{remark} By \cite[Definition 1.1, Remark 1.2]{Fuj16}, the above $\limsup$ is actually a limit, the function $\vol_X(L - xE)$ is a monotonically decreasing continuous function which vanishes for $x$ sufficiently large, and the definition does not depend on the choice of $f: Y \to X$. \end{remark}

\begin{definition} Let $(X,D)$ be a log pair such that $K_X+D$ is $\bQ$-Cartier. Let $E$ be a prime divisor over $X$. Assume $E$ is a divisor on $Y$ where $f: Y \to X$ is a proper birational morphism from a normal variety $Y$. We define the \emph{log discrepancy} of $E$ with respect to $(X,D)$ as
\[A_{(X,D)}(\ord_E) := 1 + \ord_E(K_Y - f^*(K_X + D)), \]  where $\ord_E$ is the divisorial valuation measuring order of vanishing along $E$. If $(X,D)$ is a log Fano pair, we also define the following functional
\[
S_{(X,D)}(\ord_E):=\frac{1}{\vol_X(-K_X-D)}\int_{0}^{\infty}  \vol_X(-K_X-D-tE)dt.
\]
Sometimes we also use the notation $A_{(X,D)}(E)$ and $S_{(X,D)}(E)$ for $A_{(X,D)}(\ord_E)$ and $S_{(X,D)}(\ord_E)$, respectively.
\end{definition}


The following theorem summarizes the valuative criteria of uniform K-stability \cite{Fuj16},  K-semistability \cite{Fuj16, Li17}, and K-stability \cite{BX18}. We will view part (2) of this theorem as an alternative definition of uniform K-stability of log Fano pairs. 

\begin{theorem}[\cite{Fuj16, Li17, BX18}]\label{thm:valuative}
Let $(X,D)$ be a log Fano pair. 
\begin{enumerate}
 \item $(X,D)$ is K-semistable (resp. K-stable) if and only if for
 any prime divisor $E$ over $X$,
 \[
    A_{(X,D)}(\ord_E)\geq~(\textrm{resp. }>)~S_{(X,D)}(\ord_E).
 \]
 \item $(X,D)$ is uniformly K-stable if and only if
 there exists $\epsilon>0$ such that 
 \[
  A_{(X,D)}(\ord_E)\geq (1+\epsilon) S_{(X,D)}(\ord_E)
 \]
 for any prime divisor $E$ over $X$.
\end{enumerate}

\end{theorem}

From Theorem \ref{thm:valuative} we see that uniform K-stability implies K-stability for log Fano pairs. Moreover, it follows from a recent result \cite[Theorem 1.6]{LXZ21} that these two stability notions are indeed equivalent to each other for log Fano pairs.

\begin{defn}[\cite{FO16, BJ17}]\label{def:delta}
The \emph{stability threshold} $\delta(X,D)$ of a klt log Fano pair $(X,D)$ is defined as \[
\delta(X,D):=\inf_{E}\frac{A_{(X,D)}(\ord_E)}{S_{(X,D)}(\ord_E)}
\]
where the infimum is taken over all prime divisors $E$ over $X$.
\end{defn}

\begin{thm}[\cite{FO16, BJ17}]
A klt log Fano pair $(X,D)$ is K-semistable (resp. uniformly K-stable) if and only if $\delta(X,D)\geq 1$ (resp. $>1$).
\end{thm}

\begin{definition}
 Let $X$ be a $\bQ$-Fano variety. Let $D\sim_{\bQ}-rK_X$ be an
effective $\bQ$-divisor. For a rational number $0<c<r^{-1}$, we say 
that $(X,D)$ is \emph{$c$-K-(poly/semi)stable (resp. uniformly
$c$-K-stable)} if $(X,cD)$ is K-(poly/semi)stable
(resp. uniformly K-stable).
\end{definition} 

Next we provide a useful result on interpolation of K-stability  (see e.g. \cite[Lemma 2.6]{Der16}). These kinds of interpolation results were known before in the smooth case via analytic arguments (see e.g. \cite{LS14}).

\begin{prop}\label{prop:k-interpolation}
Let $X$ be a $\bQ$-Fano variety. Let $D$ and $\Delta$ be effective $\bQ$-divisors on $X$ satisfying the following properties:
\begin{itemize}
    \item Both $D$ and $\Delta$ are proportional to $-K_X$ under $\bQ$-linear equivalence.
    \item $-K_X-D$ is ample, and $-K_X-\Delta$ is nef.
    \item The log pairs $(X,D)$ and $(X,\Delta)$ are K-(poly/semi)stable and K-semistable, respectively.
\end{itemize}
Then we have
\begin{enumerate}
    \item If $D\neq 0$, then $(X,tD+(1-t)\Delta)$ is K-(poly/semi)stable for any $t\in (0,1]$.
    \item If $D=0$, then $(X,(1-t)\Delta)$ is K-semistable
    for any $t\in (0,1]$.
    \item If $\Delta\sim_{\bQ}-K_X$ and $(X,\Delta)$ is klt, then $(X,tD+(1-t)\Delta)$ is uniformly K-stable for any $t\in (0,1)$.
\end{enumerate}
\end{prop}

\begin{proof}
Parts (1) and (2) follow directly from the linearity of generalized Futaki invariants in terms of the coefficient. For part (3), we use the valuative criterion (Theorem \ref{thm:valuative}). Simple computation shows that 
\[
S_{(X,tD+(1-t)\Delta)}(E)= tS_{(X,D)}(E),\quad
A_{(X,tD+(1-t)\Delta)}(E)= tA_{(X,D)}(E)+(1-t)A_{(X,\Delta)}(E).
\]
By \cite[Theorem A]{BJ17}, we know that $\delta_0:=\delta(X,\Delta;-K_X-D)=\inf_E A_{(X,\Delta)}(E)/S_{(X,D)}(E)$ is strictly positive. Since $(X,D)$ is K-semistable, Theorem \ref{thm:valuative} implies $A_{(X,D)}(E)\geq S_{(X,D)}(E)$. Hence 
\[
A_{(X,tD+(1-t)\Delta)}(E)\geq tS_{(X,D)}(E)+(1-t)\delta_0 S_{(X,D)}(E)= \left(1+\frac{(1-t)\delta_0}{t}\right)S_{(X,tD+(1-t)\Delta)}(E).
\]
This implies the uniform K-stability of $(X,tD+(1-t)\Delta)$ for any $t\in (0,1)$ by Theorem \ref{thm:valuative}.
\end{proof}

\subsection{Normalized volumes}
We give a brief review of normalized volume of valuations introduced by Chi Li \cite{Li18}. See \cite{LLX18} for a survey about recent developments on this subject. 

\begin{defn}
Let $(X,D)$ be a klt log pair of dimension $n$. Let $x\in X$ be a closed point. A \emph{valuation $v$ on $X$ centered at $x$} is a valuation of $\bC(X)$ such that $v\geq 0$ on $\cO_{X,x}$ and $v>0$ on $\fm_x$. The set of such valuations is denoted by $\Val_{X,x}$.
The \emph{volume} is a function $\vol_{X,x}:\Val_{X,x}\to \bR_{\geq 0}$ defined as 
\[
\vol_{X,x}(v):=\lim_{k\to\infty}\frac{\dim_{\bC}\cO_{X,x}/\{f\in\cO_{X,x}\mid v(f)\geq k\}}{k^n/n!}.
\]

The \emph{log discrepancy}  is a function $A_{(X,D)}:\Val_{X,x}\to \bR_{>0}\cup\{+\infty\}$ defined in \cite{JM12, BdFFU15}. Note that if $v=a\cdot\ord_E$ where $a\in\bR_{>0}$ and $E$ is a prime divisor over $X$ centered at $x$, then 
\[
A_{(X,D)}(v)=a(1+\ord_E(K_{Y}-\pi^*(K_X+D))),
\]
where $\pi:Y\to X$ is a birational model of $X$ containing $E$ as a divisor.

The \emph{normalized volume} is a function $\hvol_{(X,D),x}:\Val_{X,x}\to \bR_{>0}\cup\{+\infty\}$ defined as
\[
\hvol_{(X,D),x}(v):=\begin{cases}
A_{(X,D)}(v)^n\cdot\vol_{X,x}(v) & \textrm{ if } A_{(X,D)}(v)<+\infty\\
+\infty & \textrm{ if }A_{(X,D)}(v)=+\infty
\end{cases}
\]

The \emph{local volume} of a klt singularity $x\in (X,D)$ is defined as 
\[
\hvol(x,X):=\min_{v\in\Val_{X,x}}\hvol_{(X,D),x}(v).
\]
Note that the existence of a $\hvol$-minimizer is proven in \cite{Blu18}.
\end{defn}

The following theorem from \cite{LL16} generalizing \cite[Theorem 1.1]{Fuj15} and \cite[Theorem 1.2]{Liu18} is crucial in the study of explicit K-moduli spaces. 

\begin{thm}[{\cite[Proposition 4.6]{LL16}}]\label{thm:local-vol-global}
Let $(X,D)$ be a K-semistable log Fano pair of dimension $n$. Then for any closed point $x\in X$, we have 
\[
(-K_X-D)^n\leq \left(1+\frac{1}{n}\right)^n\hvol(x,X,D).
\]
\end{thm}

The following useful result is proved in \cite[Corollary 4]{BL18a} and independently in \cite[Proposition 2.36]{LX17} as an application of the lower semicontinuity of local volumes.

\begin{thm}\label{thm:Kss-spdeg}
Let $(X,D)$ be a klt log Fano pair. If $(X,D)$ specially degenerates to a K-semistable log Fano pair $(X_0,D_0)$, then $(X,D)$ is also K-semistable.
\end{thm}

\subsection{CM line bundles}
Let us start with the original definition of CM line bundles due to Paul and Tian \cite{PT06, PT09} (see also \cite{FR06}).

\begin{defn}\label{defn:absCM}
Let $f:\cX\to T$ be a proper flat morphism of schemes of finite type over $\bC$. Let $\cL$ be an $f$-ample line bundle on $\cX$.
We assume that the fibers $(\cX_t,\cL_t)$ of $f$ have constant pure dimension $n\geq 1$ and constant Hilbert polynomial $\chi$.
A result of Mumford-Knudsen \cite{KM76} said that there exists line bundles $\lambda_i=\lambda_{i}(\cX,\cL)$ on $T$ such that for all $k$,
\[
\det f_!(\cL^k)=\lambda_{n+1}^{\binom{k}{n+1}}\otimes\lambda_n^{\binom{k}{n}}\otimes\cdots\otimes\lambda_0.
\]
By flatness, the Hilbert polynomial $\chi(\cX_t,\cL_t^k)=a_0 k^n+a_1 k^{n-1}+ O(k^{n-2})$. We set $\mu=\mu(\cX,\cL):=\frac{2a_1}{a_0}$.
Then the \emph{CM line bundle} is defined as
\[
\lambda_{\CM,f,\cL}:=\lambda_{n+1}^{\mu+n(n+1)}\otimes\lambda_n^{-2(n+1)}.
\]
The \emph{Chow line bundle}\footnote{This can also be defined when $f$ is a well-defined family of $n$-dimensional cycles over a semi-normal base scheme $T$ (see \cite[Section I.3]{Kol96}).} is defined as
\[
\lambda_{\Chow,f,\cL}:=\lambda_{n+1}.
\]
\end{defn}

Next, we recall the definition of log CM line bundles. 

\begin{defn}[log CM line bundle]\label{defn:logCM1}
Assume $f:\cX\to T$ and $\cL$ satisfy the conditions in Definition \ref{defn:absCM}. For $i=1,\cdots,k$, let $\cD_i$ be a closed subscheme of $\cX$ such that $f|_{\cD_i}:\cD_i\to T$ is of pure dimension $n-1$, and either $f|_{\cD_i}$ is flat whose fibers have constant Hilbert polynomial, or $f|_{\cD_i}$ is a well-defined family of cycles over a semi-normal base scheme $T$. Let $c_i\in[0,1]$ be rational numbers. We define the log CM $\bQ$-line bundle of the data $(f:\cX\to T, \cL,\cD:=\sum_{i=1}^k c_i\cD_i)$ to be 
\[
 \lambda_{\CM,f,\cD,\cL}:=\lambda_{\CM,f,\cL}-\frac{n(\cL_t^{n-1}\cdot\cD_t)}{(\cL_t^n)}\lambda_{\Chow,f,\cL}+(n+1)\lambda_{\Chow,f|_{\cD},\cL|_{\cD}},
\]
where \[
(\cL_t^{n-1}\cdot\cD_t):=\sum_{i=1}^k c_i (\cL_t^{n-1}\cdot\cD_{i,t}),\quad
  \lambda_{\Chow,f|_{\cD},\cL|_{\cD}}:=\bigotimes_{i=1}^k\lambda_{\Chow,f|_{\cD_i},\cL|_{\cD_i}}^{\otimes c_i}.
  \]
\end{defn}

For any $\bG_m$-linearized line bundle over $\bA^1$ equipped with the standard $\bG_m$-action, we denote by $\wt(\cdot)$ the corresponding $\bG_m$-weight of the central fiber. The following result from \cite{PT06} is a fundamental property of (log) CM line bundles.

\begin{prop}\label{prop:Fut=CMwt}
 Let $(X,D=\sum_{i=1}^k c_i D_i)$ be an $n$-dimensional projective log pair. Let $L$ be an ample line bundle on $X$.  Let $\pi:(\cX,\cD;\cL)\to\bA^1$ be a test configuration of $(X,D;L)$. Then $\lambda_{\CM,\pi,\cL}$, $\lambda_{\Chow,\pi,\cL}$ and $\lambda_{\Chow,\pi|_{\cD_i},\cL|_{\cD_i}}$ are all $\bG_m$-linearized line bundles over $\bA^1$.
  Then we have
 \begin{align*}
 \Fut(\cX;\cL)& =\frac{1}{(n+1)(L^n)}\wt(\lambda_{\CM,\pi,\cL}),\\
 \CH(\cX,\cD;\cL)& =\frac{1}{(n+1)(L^n)}\left( -\frac{n(L^{n-1}\cdot D)}{(L^n)} \wt(\lambda_{\Chow,\pi,\cL})+(n+1)\wt(\lambda_{\Chow,\pi|_{\cD},\cL|_{\cD}})\right).
 \end{align*}
 In particular,
 \[
  \Fut(\cX,\cD;\cL)=\frac{1}{(n+1)(L^n)}\wt(\lambda_{\CM,\pi,\cD,\cL}).
 \]
\end{prop}

Next, we will introduce the concept of $\bQ$-Gorenstein flat families of log Fano pairs. In order to adapt this concept to our moduli problems, we need to include cases when the base is a normal Deligne-Mumford stack.

\begin{defn}\label{defn:qgorfamily}\leavevmode
\begin{enumerate}[label=(\alph*)]
\item Let $f:\cX\to T$ be a proper flat morphism between reduced schemes. Let $\cD = \sum_{i=1}^k c_i \cD_i$ be a finite $\bQ_{\geq 0}$-linear combination of reduced closed subschemes of $\cX$. We say $f:(\cX,\cD)\to T$ is a \emph{$\bQ$-Gorenstein flat family of log Fano pairs} if the following conditions hold:
\begin{itemize}
\item $f$ has normal, geometrically connected fibers of the same dimension $n$;
\item $\cD_i$ is a relative Mumford divisor on $\cX$ over $T$ for every $i$ (see \cite[Definition 1]{Kol19});
\item $-(K_{\cX/T}+\cD)$ is $\bQ$-Cartier and $f$-ample.
\end{itemize}
We define the \emph{CM $\bQ$-line bundle} of $f:(\cX,\cD)\to T$ as $\lambda_{\CM,f,\cD}:=l^{-n}\lambda_{\CM,f,\cD,\cL}$ where $\cL:=-l(K_{\cX/T}+\cD)$ is an $f$-ample Cartier divisor on $\cX$ for some $l\in\bZ_{>0}$. 

\item Let $\cX$ and $\cT$ be normal separated Deligne-Mumford stacks that are finite type over $\bC$. Let $\cD$ be an effective $\bQ$-divisor on $\cX$. 
We say $f:(\cX,\cD)\to \cT$ is a \emph{$\bQ$-Gorenstein flat family of log Fano pairs} if for some (or equivalently, any) \'etale cover $u:U\to \cT$ from a normal scheme $U$, the base change $f\times_{\cT} u: (\cX,\cD)\times_{\cT} U\to U$ is a $\bQ$-Gorenstein flat family of log Fano pairs.
\end{enumerate}
\end{defn}

In our moduli problems, we mainly consider the following class of log Fano pairs.

\begin{definition}\label{defn:qgorsmoothable}
 Let $c,r$ be positive rational numbers such that $cr<1$. 
 A log Fano pair $(X,cD)$ is \emph{$\bQ$-Gorenstein smoothable} if there exists a $\bQ$-Gorenstein flat family of log Fano pairs $\pi:(\cX,c\cD)\to B$ over a pointed smooth curve $(0\in B)$
 such that the following holds:
 \begin{itemize}
  \item Both $-K_{\cX/B}$ and $\cD$ are $\bQ$-Cartier, $\pi$-ample and  $\cD\sim_{\bQ,\pi}-rK_{\cX/B}$;
  \item Both $\pi$ and $\pi|_{\cD}$ are smooth morphisms
  over $B\setminus\{0\}$;
  \item $(\cX_0,c\cD_0)\cong (X,cD)$.
 \end{itemize} 
A $\bQ$-Gorenstein flat family of log Fano pairs $f:(\cX,c\cD)\to T$ is called a \emph{$\bQ$-Gorenstein smoothable log Fano family} if 
all fibers are $\bQ$-Gorenstein smoothable log Fano pairs and $\cD$ is $\bQ$-Cartier.
\end{definition}

For application purposes, it is always convenient to work with a smaller family rather than the whole Hilbert scheme. Thus the next criterion is important when checking K-stability in explicit families. It is a partial generalization of \cite[Theorem 1]{PT06} and \cite[Theorem 3.4]{OSS16}.

\begin{thm}\label{thm:paultian}
Let $f:(\cX,\cD)\to\cT$ be a $\bQ$-Gorenstein flat family of log Fano pairs over a normal proper Deligne-Mumford stack $\cT$ that is finite type over $\bC$. Denote by $T$ the coarse moduli space of $\cT$. Let $G$ be a reductive group acting on $\cX$ and $\cT$ such that $\cD$ is $G$-invariant and $f$ is $G$-equivariant. 
Assume in addition that 
\begin{enumerate}[label=(\alph*)]
\item if $\Aut(\cX_t,\cD_t)$ is finite for $t\in T$ then the stabilizer subgroup $G_t$ is also finite;
\item if $(\cX_t,\cD_t)\cong (\cX_{t'}, \cD_{t'})$ for $t,t'\in T$, then $t'\in G\cdot t$;
\item $\lambda_{\CM,f,\cD}$ descends to an ample $\bQ$-line bundle $\Lambda_{\CM,f,\cD}$ on $T$.
\end{enumerate}
Then $t\in T$ is GIT (poly/semi)stable with respect to the $G$-linearized $\bQ$-line bundle $\Lambda_{\CM,f,\cD}$ if $(\cX_t, \cD_t)$ is a K-(poly/semi)stable log Fano pair.
\end{thm}

\begin{proof}
We first show that K-semistability implies GIT-semistability. Denote by $\cL:=-l(K_{\cX/\cT}+\cD)$ a Cartier divisor on $\cX$ for $l\in\bZ_{>0}$. 
Let $t\in T$ be a closed point such that $(\cX_t,\cD_t)$ is K-semistable. Then $t$ induces a morphism $\tau:\Spec\bC\to\cT$ which is unique up to isomorphism. Let $\sigma:\bG_m\to G$ be a $1$-PS of $G$. Then we have a morphism $\rho:\bG_m\to \cT$ as the composition of $\tau\times\sigma:\bG_m=\bG_m\times\Spec\bC\to G\times \cT$ and the $G$-action morphism $G\times\cT\to\cT$. Since $\cT$ is proper, $\rho$ extends to a $\bG_m$-equivariant morphism $\bar{\rho}:\bA^1\to \cT$. Pulling back the morphism $f:(\cX,\cD)\to \cT$ and $\cL$ to $\bA^1$ under $\bar{\rho}$ yields a test configuration $(\cX_{t,\sigma},\cD_{t,\sigma};\cL_{t,\sigma})/\bA^1$ of $(\cX_t,\cD_t)$. By Proposition \ref{prop:Fut=CMwt}, we know that $\Fut(\cX_{t,\sigma},\cD_{t,\sigma};\cL_{t,\sigma})$ is a positive multiple of the GIT weight $\mu^{\Lambda_{\CM,f,\cD}}(t,\sigma)$ (see \cite[Definition 2.2]{MFK94}) which implies that both are non-negative by K-semistability of $(\cX_t,\cD_t)$. Thus the Hilbert-Mumford criterion implies that $t\in T$ is GIT semistable. 

Next, assume that $(\cX_t,\cD_t)$ is K-polystable. From the above discussion we see that $t$ is GIT semistable. Let $\sigma: \bG_m\to G$ be a $1$-PS such that $t_0:=\lim_{s\to 0}\sigma(s)\cdot t$ is GIT polystable. Then the GIT weight $\mu^{\Lambda_{\CM,f,\cD}}(t,\sigma)$ is zero, which implies that $\Fut(\cX_{t,\sigma},\cD_{t,\sigma};\cL_{t,\sigma})$ vanishes as well. Thus we have $(\cX_t,\cD_t)\cong (\cX_{t_0},\cD_{t_0})$ which implies that $t\in G\cdot t_0$ is GIT polystable by assumption (b). The stable part is a consequence of the polystable part and assumption (a).
\end{proof}



The following proposition provides an intersection formula for log CM line bundles. For the case without divisors this was proven by Paul and Tian \cite{PT06}. The current statement is a consequence of {\cite[Proposition 3.7]{CP18}}. We provide a proof here for readers' convenience.

\begin{prop}\label{prop:logCM2}
Let $f:(\cX,\cD)\to T$ be a $\bQ$-Gorenstein flat family of $n$-dimensional log Fano pairs over a normal proper variety $T$. Then
\begin{equation}\label{eq:CM-intersection}
 \mathrm{c}_1(\lambda_{\CM,f,\cD})=-f_*((-K_{\cX/T}-\cD)^{n+1}).
\end{equation}
\end{prop}

\begin{proof}
Since both sides of \eqref{eq:CM-intersection} are functorial under pull-backs,  by passing to a resolution we may assume that $T$ is smooth and projective. Then \cite[Lemma A.2]{CP18} implies that 
\[
\mathrm{c}_1(f_* \cL^{\otimes q})=\frac{f_*(\cL^{n+1})}{(n+1)!}q^{n+1} - \frac{f_*(K_{\cX/T}\cdot\cL^{n})}{2\cdot n!}q^n+O(q^{n-1}),
\]
where $\cL:=-l(K_{\cX/T}+\cD)$ is Cartier and $f$-ample.
Hence we have $ \mathrm{c}_1(\lambda_{\Chow,f,\cL})=f_*(\cL^{n+1})$ and $\mathrm{c}_1(\lambda_{n,f,\cL})=\frac{1}{2}f_*(n(\cL^{n+1})-(K_{\cX/T}\cdot\cL^{n}))$.
It is clear that $\mu(\cX,\cL)=-\frac{n(K_{\cX_t}\cdot \cL_t^{n-1})}{(\cL_t^n)}$. Hence
\[
 \mathrm{c}_1(\lambda_{\CM,f,\cL})=-\frac{n(K_{\cX_t}\cdot \cL_t^{n-1})}{(\cL_t^n)} f_*(\cL^{n+1})+(n+1)f_*(K_{\cX/T}\cdot\cL^n).
\]
We also know that $\mathrm{c}_1(\lambda_{\Chow,f|_{\cD},\cL|_{\cD}})
=f_*(\cL^n\cdot\cD)$. Thus
\begin{align*}
\mathrm{c}_1(\lambda_{\CM,f,\cD,\cL})&=\mathrm{c_1}(\lambda_{\CM,f,\cL})-\frac{n(K_{\cX_t}\cdot \cL_t^{n-1})}{(\cL_t^n)}\mathrm{c}_1(\lambda_{\Chow,f,\cL})+(n+1)\mathrm{c}_1(\lambda_{\Chow,f|_{\cD},\cL|_{\cD}})\\
&=\frac{n((-K_{\cX_t}-D_t)\cdot \cL_t^{n-1})}{(\cL_t^n)}f_*(\cL^{n+1})-(n+1)f_*((-K_{\cX/T}-\cD)\cdot\cL^{n})\\
& =-l^nf_*((-K_{\cX/T}-\cD)^{n+1}).\qedhere
\end{align*}
\end{proof}

Next we recall the definition of Hodge line bundles.

\begin{defn}\label{defn:hodge}
Let $f:\cX\to T$ be a $\bQ$-Gorenstein flat family of $\bQ$-Fano varieties over a normal base. Let $\cD$ be an effective $\bQ$-Cartier $\bQ$-divisor on $\cX$ not containing any fiber of $f$ such that $\cD\sim_{\bQ,f} -rK_{\cX/T}$ for some $r\in\bQ_{>0}$. The Hodge $\bQ$-line bundle $\lambda_{\Hodge,f,r^{-1}\cD}$ is defined as the $\bQ$-linear equivalence class of $\bQ$-Cartier $\bQ$-divisors on $T$ such that
\[
K_{\cX/T}+r^{-1}\cD\sim_{\bQ}f^*\lambda_{\Hodge,f,r^{-1}\cD}.
\]
\end{defn}

\begin{prop}\label{prop:cm-interpolation}
With the notation of Definition \ref{defn:hodge}, for any rational number $0\leq c<r^{-1}$ we have 
\begin{equation}\label{eq:cm-hodge}
(1-cr)^{-n}\lambda_{\CM,f,c\cD}=(1-cr)\lambda_{\CM,f}+ cr(n+1)(-K_{\cX_t})^n\lambda_{\Hodge,f,r^{-1}\cD}.
\end{equation}
\end{prop}

\begin{proof}
For simplicity we denote $\lambda_c:=\lambda_{\CM,f,c\cD}$ and $\lambda_{\Hodge}:=\lambda_{\Hodge,f,r^{-1}\cD}$. Let $\cL:=-lK_{\cX/T}$ be an ample Cartier divisor on $\cX$ for some $l\in \bZ_{>0}$. Since CM line bundles are invariant by twisting pull-back of a line bundle on the base, we know that $\lambda_c=l^{-n}(1-cr)^n \lambda_{\CM,f,c\cD,\cL}$. We also know that 
\[
\lambda_{\CM,f,c\cD,\cL}=\lambda_{\CM,f,\cL}-c\left(\frac{n(\cL_t^{n-1}\cdot\cD_t)}{(\cL_t^n)}\lambda_{\Chow,f,\cL}-(n+1)\lambda_{\Chow,f|_{\cD}, \cL|_{\cD}}\right).
\]
Hence to show \eqref{eq:cm-hodge} it suffices to show that $\lambda_{\CM,f,r^{-1}\cD,\cL}=(n+1)l^n(-K_{\cX_t})^n\lambda_{\Hodge}$. Since both sides are functorial under pull-backs, we may assume that $T$ is smooth and quasi-projective. By taking closure in the relative Hilbert scheme of $(\cX_t,\cD_t;\cL_t)$, passing to a resolution of the base, and taking normalization of the total family, we can find a smooth projective closure $\oT$ of $T$, an extension $\bar{f}:(\ocX,\ocD)\to \oT$ of $f$ and an $\bar{f}$-ample Cartier divisor $\ocL$ on $\ocX$ such that $\ocX$ is normal projective, $\ocD$ is an effective $\bQ$-Cartier $\bQ$-divisor on $\ocX$, $\bar{f}$ and $\bar{f}|_{\ocD}$ are pure dimensional, and $\ocL|_{\cX}=\cL$. Although $\bar{f}$ is not necessarily flat, the CM line bundle $\lambda_{\CM,\bar{f},c\ocD,\ocL}$ can still be defined by \cite[Lemma A.2]{CP18} such that its restriction to $T$ is $\lambda_{\CM,f,c\cD,\cL}$. By similar argument to the proof of Proposition \ref{prop:logCM2}, we have that 
\begin{align*}
\mathrm{c}_1(\lambda_{\CM,\bar{f},r^{-1}\ocD,\ocL})&=\frac{n((-K_{\cX_t}-r^{-1}D_t)\cdot \cL_t^{n-1})}{(\cL_t^n)}\bar{f}_*(\ocL^{n+1})-(n+1)f_*((-K_{\ocX/\oT}-r^{-1}\ocD)\cdot\ocL^{n})\\&=(n+1)f_*((K_{\ocX/\oT}+r^{-1}\ocD)\cdot\ocL^{n}).
\end{align*}
Since $K_{\cX/T}+r^{-1}\cD=f^*\lambda_{\Hodge}$, we know that 
\[
\lambda_{\CM,f,r^{-1}\cD,\cL}=(n+1)(\cL_t^n)\lambda_{\Hodge}=(n+1)l^n(-K_{\cX_t})^n\lambda_{\Hodge}.
\]
The proof is finished.
\end{proof}

\subsection{Good moduli spaces in the sense of Alper} We recall the definition and some notions regarding \emph{good moduli spaces} from \cite{alper}, as these objects naturally appear in the construction of moduli spaces in K-stability. 

\begin{definition}
A quasi-compact morphism $\phi: \mathcal{X} \to Y$ from an Artin stack to an algebraic space is a \emph{good moduli space} if 
\begin{enumerate}
    \item the push-forward functor on quasi-coherent sheaves is exact, and
    \item the induced morphism on sheaves $\calO_Y \to \phi_* \calO_{\mathcal{X}}$ is an isomorphism.
\end{enumerate}
\end{definition}

\begin{definition}
Let $\phi: \calX \to Y$ be a good moduli space. An open substack $\calU \subseteq \calX$ is \emph{saturated} for $\phi$ if $\phi^{-1}(\phi(\calU)) = \calU$.
\end{definition}

This is useful for the following reasons:

\begin{remark}\cite[Remark 6.2, Lemma 6.3]{alper}
\leavevmode
\begin{enumerate}
    \item If $\calU$ is saturated for $\phi$, then $\phi(\calU)$ is open and $\phi\mid_\calU: \calU \to \phi(\calU)$ is a good moduli space.
    \item If $\psi: \calX \to Z$ is a morphism to  a scheme $Z$ and $V \subseteq Z$ is an open subscheme, then $\psi^{-1}(V)$ is saturated for $\phi$. \end{enumerate} \end{remark}
    
    \subsection{Hacking's compact moduli of plane curves}
Hacking constructed a proper moduli stack $\ocP_d^{\rm H}$ of plane curves of degree $d\geq 4$ using tools from the MMP \cite{Hac01, Hac04}. It is a special case of the moduli theory of log canonically polarized pairs developed by Koll\'ar, Shepherd-Barron, and Alexeev. Roughly speaking the parametrized elements are demi-normal pairs $(X,D)$, where $X$ is a $\Q$-Gorenstein deformation of $\bP^2$ and $D$ is a degeneration of plane curves such that the pair satisfies some natural properties that will be reviewed below. First we recall the normal surfaces parametrized by $\ocP_d^{\rm H}$.

\begin{definition}\label{def:manetti}
A \emph{Manetti surface} is a klt projective surface that admits a $\Q$-Gorenstein smoothing to $\bP^2$. \end{definition}

\begin{prop}[{\cite[Theorems 8.2 \& 8.3]{Hac04}}]\label{prop:manetti}
A surface $X$ is a Manetti surface if and only if it is a $\Q$-Gorenstein deformation of the weighted projective plane $\bP(a^2, b^2, c^2)$ where $a^2 + b^2 + c^2 = 3abc$. Moreover, all such $X$ have unobstructed $\Q$-Gorenstein deformations. \end{prop}

We now give the definition of the surface pairs parametrized by $\ocP_d^{\rm H}$. 

\begin{definition} Let $X$ be a demi-normal surface and let $D$ be an effective $\Q$-Cartier divisor on $X$. Let $d \geq 4$ be an integer. The pair $(X,D)$ is a \emph{Hacking stable pair} of degree $d$ if:
\begin{enumerate}
\item The pair $(X,(\frac{3}{d} + \epsilon)D)$ is slc, and the divisor $K_X + (\frac{3}{d} + \epsilon)D$ is ample for any $0 < \epsilon \ll 1$,
\item $dK_X+3D\sim 0$, and
\item there is a $\Q$-Gorenstein deformation of $(X,D)$ to $(\bP^2,C_t)$ where $C_t$ is a family of plane curves of degree $d$. \end{enumerate} \end{definition}

We can now define the stack $\ocP_d^{\rm H}$. 

\begin{definition} Let $d \geq 4$ be an integer. We define the \emph{Hacking moduli stack} $\ocP_d^{\rm H}$ to be the reduced stack representing the following moduli pseudo-functor over reduced base $S$:
\[\ocP_d^{\rm H}(S) = \{ (\calX, \calD)/S \mid (\calX, \calD)/S \textrm{ is a $\Q$-Gorenstein family of Hacking stable pairs of degree }d\}. \]
\end{definition}

\begin{theorem}[{\cite[Theorem 4.4, 7.2]{Hac04} and \cite{Ale96}}] The stack $\ocP_d^{\rm H}$ is a reduced proper Deligne-Mumford stack of finite type over $\bC$. Its coarse moduli space $\oP_d^{\rm H}$ is a reduced projective variety which compactifies the moduli space of smooth plane curves
of degree $d$. Furthermore, if $3 \nmid d$ then 
\begin{enumerate}
\item the stack $\ocP_d^{\rm H}$ is smooth, and
\item the underlying surface of a Hacking stable pair of degree $d$ is either a Manetti surface, or the slt\footnote{Recall that a demi-normal pair $(X,D)$ is semi-log terminal (slt) if it is slc and its normalization is plt.} union of two normal surfaces glued along a smooth rational curve. \end{enumerate}
\end{theorem}

\section{Construction of 
K-moduli spaces of smoothable log Fano pairs}\label{sec:construction}

In this section we construct K-moduli stacks and spaces of $\bQ$-Gorenstein smoothable log Fano pairs (see Definition \ref{defn:qgorsmoothable}). Our construction is largely based on \cite{LWX14} with new input from \cite{Jia17, BX18, TW19}.
Results from this section can be applied to the study of many explicit K-moduli spaces, including the K-moduli spaces of plane curves as the main subject of this paper. We will investigate other applications in forthcoming work.

The following theorems are the main results of this section. The first theorem is a natural generalization of \cite[Theorem 1.3]{LWX14}.

\begin{thm}\label{thm:lwxlog}
 Let $\chi_0$ be the Hilbert polynomial of an anti-canonically polarized Fano manifold. Fix $r\in\bQ_{>0}$ and a rational number $c\in (0,\min\{1,r^{-1}\})$. Then there exists a reduced Artin stack  $\cK\cM_{\chi_0,r,c}$  of finite type over $\bC$ parametrizing all K-semistable $\bQ$-Gorenstein smoothable log Fano pairs $(X,cD)$ with Hilbert polynomial $\chi(X,\cO_X(-mK_X))=\chi_0(m)$ for sufficiently divisible $m$ and $D\sim_{\bQ}-rK_X$. Moreover, the Artin stack $\cK\cM_{\chi_0,r,c}$ admits a good moduli space $KM_{\chi_0,r,c}$ as a proper reduced scheme of finite type over $\bC$.
\end{thm}

Indeed, in Section \ref{sec:stackstab} we show that the K-moduli stack $\cK\cM_{\chi_0,r,c}$ represents the moduli pseudo-functor of $\bQ$-Gorenstein smoothable K-semistable log Fano families with certain numerical invariants over reduced base schemes. 

The second theorem provides a wall crossing principle for these K-moduli spaces when varying the coefficient $c$.

\begin{theorem}\label{thm:logFano-wallcrossing}
 There exist rational numbers $0=c_0<c_1<c_2<\cdots<c_k=\min\{1,r^{-1}\}$ such that $c$-K-(poly/semi)stability conditions do not change for $c\in (c_i,c_{i+1})$. For each $1\leq i\leq k-1$, we have open immersions
 \[
 \cK\cM_{\chi_0,r,c_i-\epsilon}\xhookrightarrow{\Phi_{i}^-}
 \cK\cM_{\chi_0,r,c_i}\xhookleftarrow{\Phi_{i}^+}\cK\cM_{\chi_0,r,c_i+\epsilon}
 \]
 which induce projective morphisms
 \[
 KM_{\chi_0,r,c_i-\epsilon}\xrightarrow{\phi_{i}^-}
 KM_{\chi_0,r,c_i}\xleftarrow{\phi_{i}^+}KM_{\chi_0,r,c_i+\epsilon}
 \]
 Moreover, all the above wall crossing morphisms have local VGIT presentations as in \cite[(1.2)]{AFS17}, and the CM $\bQ$-line bundles on $KM_{\chi_0, r,c_i\pm\epsilon}$ are $\phi_i^{\pm}$-ample (see Theorems \ref{thm:wallcrossingchart} and \ref{thm:VGIT-CM-ample} for the precise statements).
\end{theorem}

\begin{rem}
Recently there has been tremendous progress on constructing K-moduli spaces and stacks using purely algebraic methods. We point the reader to Remark \ref{rem:ps} for a further discussion and citations. Our construction is mostly based on the analytic works \cite{LWX14, Oda15, SSY16} and algebraic works \cite{Jia17, LWX18, BX18}. Since many of the works mentioned in Remark \ref{rem:ps}(1) appeared simultaneously to or after the preparation of this paper, our constructions do not rely on those results, though it is likely some of our arguments can be simplified using those results. 

Meanwhile, a suitable condition for families of log pairs over non-reduced bases was discovered in \cite{Kol19} as the K-flatness condition. Since  we only study families of $\bQ$-Gorenstein smoothable log Fano pairs, in this paper we restrict the bases of such families to be reduced.
\end{rem}

\subsection{Foundations}\label{sec:foundations}

We will fix an arbitrary rational number $\epsilon_0\in (0,1)$. For technical reasons, we will concentrate on constructing the K-moduli space of $\bQ$-Gorenstein smoothable log Fano pairs $(X,cD)$ where $D\sim_{\bQ} -rK_X$ and $c\in (0, \min\{1, (1-\epsilon_0)r^{-1}\})$ a rational number. As we will see in Theorem \ref{thm:almostCYstabilize}, the K-(poly/semi)stability conditions will not change for $c$ sufficiently close to $r^{-1}$. Besides, all numbers except $\beta$, $\beta_\bullet$, and $\fB$ are assumed to be rational.

The first boundedness result is a generalization of work in \cite{LWX14} which was a consequence of \cite{CDS15, Tia15}.

\begin{thm}\label{thm:bdd-bigangle}
Fix $n$ a positive integer, $r$ a positive rational number, and $\epsilon_0\in (0,1)$ a rational number. 
Then the following collection of $\bQ$-Gorenstein smoothable pairs
\[
\{(X,D)\mid \dim X=n, ~(X,cD)\textrm{ is K-semistable for some $c\in (0,\min\{1, (1-\epsilon_0)r^{-1}\})$}\}.
\]
is log bounded. In particular, there exists  $m_1=m_1(n,r,\epsilon_0)\in\bZ_{>0}$  such that $m_1 K_X$ is Cartier  whenever $(X,D)$ belongs to the above collection.
\end{thm}

\begin{proof}
Since $D\sim_{\bQ}-rK_X$, it suffices to bound the variety $X$. From the assumption that $(X,cD)$ is K-semistable for some $c<(1-\epsilon_0)r^{-1}$, by \cite[Theorem 7.2]{BL18b} we
conclude that $\delta(X)\geq \epsilon_0$. Since the volume of $X$ is a positive integer, $X$ belongs to a bounded family by \cite{Jia17}.
\end{proof}

\begin{prop}\label{prop:uK-Delta}
Fix $n,m\in\bZ_{>0}$, $r\in\bQ_{>0}$ and $\epsilon_0\in (0,1)$.
Let $X$ be a $\bQ$-Fano variety with $mK_X$ Cartier. 
Let $D\sim_{\bQ}-rK_X$ be a Weil divisor. Let $c<(1-\epsilon_0)r^{-1}$ be a rational number such that $(X,cD)$ is a log Fano pair. Then
\begin{enumerate}
\item  there exists a positive integer $q=q(n,r,\epsilon_0,m)$ and a Cartier divisor $\Delta\in |-qK_X|$ such that $(X,\lct(X;D)\cdot D+\Delta)$ is a log canonical pair.
\item  there exists $\gamma_0=\gamma_0(n,r,\epsilon_0,m)$ such that either $c> \lct(X;D)-\frac{\epsilon_0 r^{-1}}{n+1}$ which implies $\alpha(X,cD)<\frac{1}{n+1}$,  or $(X,cD+\frac{(1-\beta)(1-cr)}{q}\Delta)$ is uniformly K-stable for any $\beta\in (0, \gamma_0]$. 
\end{enumerate}
\end{prop}

\begin{proof}
(1) By \cite{HX15} such $X$ form a bounded family. This follows from the Bertini theorem for bounded families.

(2) Let us assume that $\alpha(X,cD)\geq \frac{1}{n+1}$. We know that $\alpha(X,cD;-K_X)=(1-cr)\alpha(X,cD)\geq \frac{\epsilon_0}{n+1}$. Thus $(X,cD+\frac{\epsilon_0 r^{-1}}{n+1}D)$ is log canonical which implies
$c\leq \lct(X;D)-\frac{\epsilon_0 r^{-1}}{n+1}$.
It is clear that $(X,(\lct(X;D)-\frac{\epsilon_0 r^{-1}}{n+1})D+\frac{1}{q}\Delta)$ form a bounded family of klt pairs. Hence \cite{BL18b} implies that there exists $\alpha_1=\alpha_1(n,r,\epsilon_0,m)>0$ such that 
\[
\alpha(X,cD+\tfrac{1-cr}{q}\Delta;-K_X)\geq
\alpha(X,(\lct(X;D)-\tfrac{\epsilon_0 r^{-1}}{n+1})D+\tfrac{1}{q}\Delta;-K_X)\geq \alpha_1.
\]
Hence 
\begin{align*}
\alpha(X,cD+\tfrac{(1-\beta)(1-cr)}{q}\Delta)& = \frac{1}{\beta(1-cr)}\alpha(X,cD+\tfrac{(1-\beta)(1-cr)}{q}\Delta;-K_X)\\
& \geq \frac{1}{\beta(1-cr)}\left(\beta\alpha(X,cD;-K_X)+(1-\beta)\alpha(X,cD+\tfrac{1-cr}{q}\Delta;-K_X)\right)\\
& \geq \alpha(X,cD;-K_X)+(\beta^{-1}-1)\alpha(X,cD+\tfrac{1-cr}{q}\Delta;-K_X)\\
& \geq \frac{\epsilon_0}{n+1}+(\beta^{-1}-1)\alpha_1.
\end{align*}
Let us take $\gamma_0:=\frac{\alpha_1}{1+\alpha_1}$, then for any $\beta\in (0,\gamma_0]$ we have
$\alpha(X,cD+\tfrac{(1-\beta)(1-cr)}{q}\Delta)\geq 1$.
Thus \cite{FO16} implies that $(X,cD+\frac{(1-\beta)(1-cr)}{q}\Delta)$ is uniformly K-stable.
\end{proof}

Next we use an important result obtained in the solution of Yau-Tian-Donaldson conjecture for log smooth log Fano pairs \cite{TW19}. It plays a crucial role in proving openness and properness of the K-moduli conjecture in our setting.  We are very grateful to Feng Wang for kindly providing a proof. For the case where $D_i=0$, i.e. the boundary is a single smooth pluricanonical divisor, see \cite{CDS15, Tia15, Ber16} or \cite[Theorem 4.1]{LWX14}. 

\begin{thm}[\cite{TW19}]\label{thm:lwx4.1}
Fix $n, q \in \mathbb{Z}_{>0}$,  $r \in \mathbb{Q}_{>0}$, and $\epsilon_0,\gamma_0\in(0,1)$. Let $X_i$ be a sequence of $n$-dimensional Fano manifolds with a fixed Hilbert polynomial $\chi_0$. Let $D_i\sim_{\bQ}-rK_{X_i}$ be smooth divisors on $X_i$. Let $\Delta_i$ be smooth divisors in $|-qK_{X_i}|$ that are transversal to $D_i$. Let $c_i$ and $\beta_i$ be a sequence converging respectively to $c_\infty$ and $\beta_\infty$ with $c_\infty<\min\{1,(1-\epsilon_0)r^{-1}\}$ and $0<\gamma_0\leq \beta_i\leq 1$. Suppose that each $X_i$ admits a conical K\"ahler-Einstein metric $\omega(\beta_i)$ solving:
\begin{equation}
\Ric(\omega(\beta_i))=\beta_i(1-c_i r)\omega(\beta_i) +c_i[D_i]+\frac{(1-\beta_i)(1-c_i r)}{q}[\Delta_i].
\end{equation}
Then the Gromov-Hausdorff limit of any subsequence of $\{(X_i,\omega(\beta_i))\}_i$ is homeomorphic to a $\bQ$-Fano variety $Y$. Furthermore, there are unique Weil divisors $E,\Gamma\subset Y$ such that
\begin{enumerate}
\item $(Y, c_\infty E+ \frac{(1-\beta_\infty)(1-c_\infty r)}{q}\Gamma)$ is a klt log Fano pair;
\item $Y$ admits a weak conical K\"ahler-Einstein metric $\omega(\beta_\infty)$ solving
\[
\Ric(\omega(\beta_\infty))=\beta_\infty(1-c_\infty r)\omega(\beta_\infty) +c_\infty [E]+\frac{(1-\beta_\infty)(1-c_\infty r)}{q}[\Gamma].
\]
In particular, $\Aut(Y,E+\Gamma)$ is reductive and the pair $(Y,c_\infty E+\frac{(1-\beta_\infty)(1-c_\infty r)}{q}\Gamma)$ is K-polystable;
\item there exists a positive integer $m_2=m_2(\chi_0,r,q,\epsilon_0,  \gamma_0)$, such that possibly after passing to a subsequence, there are Tian's embeddings $T_i:X_i\to\bP^N$ and $T_\infty: Y\to \bP^N$, defined by taking a suitable orthonormal basis of the complete linear system $|-mK_{X_i}|$ and $|-mK_{Y}|$ with respect to  $\omega(\beta_i)$ and $\omega(\beta_\infty)$ respectively, such that for any multiple $m$ of $m_2$ and $N+1=\chi(X_i, \cO_{X_i}(-mK_{X_i}))$, we have that $T_i(X_i)$ converge to $T_\infty(Y)$ as projective varieties, and $T_i(D_i)$ (respectively $T_i(\Delta_i)$) converge to $T_\infty(E)$ (respectively $T_\infty(\Gamma)$) as algebraic cycles.
\end{enumerate}
\end{thm}

\begin{proof}
It is a combination of \cite[Proposition 4.18, Corollary 4.19 and Lemma 5.4]{TW19}. By  taking subsequences, we can assume that $(X_i;\omega(\beta_i))$ converges to a metric space $(X; d)$ in the Gromov-Hausdorff topology. Since the divisors are ample, Cheeger-Colding-Tian's
theory applies. From \cite[Proposition 4.18]{TW19}, there exists a positive integer $m_2$ such that the partial $C^0$ holds for $(X_i;\omega(\beta_i); -m K_{X_i} )$ for any multiple $m$ of $m_2$. Then we get a sequence of Tian's embeddings $T_i: X_i\to \bP^N$ using an orthonormal basis of $H^0(X_i,\cO_{X_i}(-m K_{X_i}))$. By taking subsequences again, we can assume that $T_i(X_i)$ converges to $Y$ as cycles. Since $T_i$ are uniform Lipschitz, we get a map $T_\infty$ from $(X; d)$ to $Y$. By \cite[Corollary 4.19]{TW19}, $T_\infty$ is a homeomorphim and $Y$ is a normal projective variety. Moreover, if we define the
Gromov-Hausdorff limit of $D_i, \Delta_i$ as $D_\infty, \Delta_\infty$, then $E = T_\infty(D_\infty)$ and $\Gamma= T_\infty(\Delta_\infty)$ are divisors on $Y$ such that $(Y, c_\infty E+\frac{(1-\beta_\infty)(1-c_\infty r)}{q}\Gamma)$ is a klt log Fano pair, and it admits a weak conical K\"ahler-Einstein metric $\omega(\beta_\infty)$:
\[
\Ric(\omega(\beta_\infty))=\beta_\infty(1-c_\infty r)\omega(\beta_\infty) +c_\infty [E]+\frac{(1-\beta_\infty)(1-c_\infty r)}{q}[\Gamma].
\]
Since both $D_i$ and $\Delta_i$ are proportional to $-K_{X_i}$, we have that both $E$ and $\Gamma$ are also proportional to $-K_{Y}$ as cycle limits. Thus $Y$ is a $\bQ$-Fano variety, and both $E$ and $\Gamma$ are $\bQ$-Cartier divisors on $Y$.
By \cite[Lemma 5.4]{TW19}, $\Aut_0(Y, c_\infty E+\frac{(1-\beta_\infty)(1-c_\infty r)}{q}\Gamma)$ is reductive.
\end{proof}

We now introduce the relevant Hilbert schemes. 

\begin{defn}\label{def:hilbscheme}
 Let $\bH^{\chi;N}:=\Hilb_{\chi}(\bP^N)$ denote the Hilbert scheme of closed subschemes of $\bP^N$ with Hilbert polynomial $\chi$. 
Given a closed subscheme $X\subset\bP^N$ with Hilbert polynomial $\chi(X,\cO_{\bP^N}(k)|_X)=\chi(k)$, let $\Hilb(X)\in\bH^{\chi;N}$ denote its Hilbert point.
\end{defn}

Let $\chi_0$ be a Hilbert polynomial of an anti-canonically polarized
 Fano manifold. Let $m$ be a positive integer.
Denote $\chi(k):=\chi_0(mk)$, $\tilde{\chi}(k):=\chi_0(mk)-\chi_0(mk-r)$, and $N=\chi_0(m)-1$. 
 Let $\bm{\chi}=(\chi,\tilde{\chi})$ be the Hilbert polynomials of $(X,D)\hookrightarrow\bP^N$. Denote by $\bH^{\bm{\chi};N}=\bH^{\chi;N}\times\bH^{\tilde{\chi};N}$.
We define
\[
 Z:=\left\{\Hilb(X,D)\in\bH^{\bm{\chi};N}\left| \begin{array}{l}
                                           X \textrm{ is a Fano manifold, }D\sim_{\bQ}-rK_X\textrm{ is a smooth divisor,}\\
                                           \cO_{\bP^N}(1)|_{X}\cong\cO_X(-mK_X),
                                         \\  \textrm{and }H^0(\bP^N,\cO_{\bP^N}(1))\xrightarrow{\cong}H^0(X,\cO_X(-mK_X)).
                                                \end{array}
 \right.\right\}.
\]
Then $Z$ is a locally closed subscheme of $\bH^{\bm{\chi};N}$. Let $\overline{Z}$ be the Zariski closure of $Z$. We also define
\[
Z^{\klt}:=\left\{\Hilb(X,D)\in \overline{Z}\left| \begin{array}{l}
                                           X \textrm{ is a $\bQ$-Fano variety, }D\sim_{\bQ}-rK_X\textrm{ is an effective
                                           Weil divisor,}\\-m_1 K_X\textrm{ is Cartier, }
                                           \cO_{\bP^N}(1)|_{X}\cong\cO_X(-mK_X)
                                           ,\\\textrm{and }H^0(\bP^N,\cO_{\bP^N}(1))\xrightarrow{\cong}H^0(X,\cO_X(-mK_X)).
                                                \end{array}
 \right.\right\}.                                                
\]
and
\[
 Z^{\circ}_c:=\{\Hilb(X,D)\in Z^{\klt}\mid (X,cD)\textrm{ is K-semistable}\}.
\]
It is clear that $Z^{\klt}$ is a Zariski open subset of $\overline{Z}$.
We will see in Theorem \ref{thm:kstfinite}  that $Z_c^{\circ}$
is a Zariski open subset of $Z^{\klt}$. 
Let $Z^{\red}$ and $Z_c^{\red}$ be reduced schemes supported on $Z$ and $Z_c^\circ$, respectively. In the cases when we keep track of $m$, we use the notation
$Z_m$, $\overline{Z_m}$, $Z_m^{\klt}$, $Z_{c,m}^{\circ}$,
$Z_m^{\red}$, and $Z_{c,m}^{\red}$ instead of $Z$, $\overline{Z}$, $Z^{\klt}$, $Z_{c}^\circ$, $Z^{\red}$, and $Z_{c}^{\red}$, respectively.

We define the K-moduli stacks and spaces as follows.

\begin{defn}\label{defn:kmoduli}
 Let $\chi_0$ be a Hilbert polynomial of an anti-canonically polarized
 Fano manifold. Fix $r\in\bQ_{>0}$ and $0<\epsilon_0\ll 1$. Let $c\in (0,\min\{1,(1-\epsilon_0)r^{-1}\}) $ be a rational number. We denote by $\chi(k):=\chi_0(mk)$, $\tilde{\chi}(k):=\chi_0(mk)-\chi_0(mk-r)$, and $N_m:=\chi_0(m)-1$.
 As we will see in Theorem \ref{thm:stabilization}, the Artin stacks $[Z_{c,m}^{\red}/\PGL(N_m+1)]$ stabilize for $m$ sufficiently divisible which we simply denote by $\cK\cM_{\chi_0,r,c}$. Moreover, according to Theorem \ref{thm:lwxlog} the Artin stack $\cK\cM_{\chi_0,r,c}$ admits a proper reduced scheme $KM_{\chi_0,r,c}$ as its good moduli space.
 We define
 the \emph{K-moduli stack} (resp. \emph{K-moduli space}) with respect to the triple $(\chi_0,r,c)$ as the reduced Artin stack $\cK\cM_{\chi_0,r,c}$ (resp. reduced proper scheme $KM_{\chi_0,r,c}$).
\end{defn}

\subsection{Continuity method}\label{sec:continuity}

In this section, we will use Theorem \ref{thm:lwx4.1} and the continuity method \cite[Section 4.2 and 6]{LWX14} to prove the following theorem.

\begin{thm}\label{thm:bddtestK}
Fix $n\in\bZ_{>0}$ and $r \in \bQ_{>0}$. Fix $\epsilon_0\in (0,1)$. 
Fix a Hilbert polynomial $\chi_0$ of an
$n$-dimensional anti-canonically polarized Fano manifold.
Then there exists $m_3=m_3(\chi_0,r,\epsilon_0)$ such that for any multiple $m$ of $m_3$, any rational number $c\in (0,\min\{1,(1-\epsilon_0)r^{-1}\})$, and any log pair $(X,D)$ with $\Hilb(X,D)\in Z_m^{\klt}$, the following holds:
\begin{enumerate}
    \item If $(X,cD)$ is K-unstable, then either $c> \lct(X;D)-\frac{\epsilon_0 r^{-1}}{n+1}$ or there exists a $1$-PS in $\SL(N_m+1)$ that destabilizes $(X,cD)$.
    \item If $(X,cD)$ is K-semistable but not K-polystable, then there exists a $1$-PS in $\SL(N_m+1)$ that provides a special degeneration to a K-polystable log Fano pair $(X',cD')$. In addition, $\Hilb(X',D')\in Z_m^{\klt}$.
    \item If $(X,cD)$ is K-polystable, then it admits a weak conical K\"ahler-Einstein metric. In particular, $\Aut(X,D)$ is reductive.
    \item If $(X,cD)$ is K-stable, then it is uniformly K-stable.
\end{enumerate}    
\end{thm}

Before presenting the proof of Theorem \ref{thm:bddtestK}, we introduce the following notation and prove several preliminary results. 

\medskip

\begin{notation}\label{not:const} Let us fix $n\in\bZ_{>0}$, $r\in\bQ_{>0}$,  $\epsilon_0\in (0,1)$, and  $\chi_0$ a Hilbert polynomial of an $n$-dimensional Fano manifold. Let $m_1=m_1(n,r,\epsilon_0)\in\bZ_{>0}$ be chosen as in Theorem \ref{thm:bdd-bigangle}. Let $q=q(n,r,\epsilon_0,m_1)\in\bZ_{>0}$ and  $\gamma_0=\gamma_0(n,r,\epsilon_0,m_1)\in (0,1)$ be chosen as in Proposition \ref{prop:uK-Delta}. Let $m_2=m_2(\chi_0,r,q,\epsilon_0,\gamma_0)$ be chosen as in Theorem \ref{thm:lwx4.1}. Let us take $m_3:=\lcm(m_1,m_2)$. For $m\in\bZ_{>0}$ a multiple of $m_3$, let us pick an arbitrary pair $(X,D)$ with $\Hilb(X,D)\in Z_m^{\klt}$. We also fix $c\in (0, \min\{1, (1-\epsilon_0)r^{-1}\})$ such that $c\leq \lct(X;D)-\frac{\epsilon_0 r^{-1}}{n+1}$. To avoid bulky notation, let $a:=\frac{1-cr}{q}$.

 According to Proposition \ref{prop:uK-Delta}, let $\Delta\in |-qK_X|$ be chosen such that $(X,cD+(1-\beta)a\Delta)$ is uniformly K-stable for any $0<\beta\leq \gamma_0$. Let us choose a smoothing $(\cX,\cD+\tDelta)\to B$ over a pointed curve $0\in B$ such that all fibers over $B\setminus\{0\}$  are log smooth,  and $(\cX_0,\cD_0,\tDelta_0)\cong (X,D,\Delta)$. 
Denote by 
\[
\fB:=\sup\{\beta\in (0,1)\mid (X, cD+(1-\beta)a\Delta)\textrm{ is uniformly K-stable}\}.
\]
By \cite{Fuj17} we know that $\gamma_0<\fB\leq 1$. 
Since the pair $(X, cD+(1-\beta)a\Delta)$ is uniformly K-stable for any $\beta\in (0,\fB)$, by \cite{BL18b, TW19} we know that there exists a Zariski neighborhood $B_\beta$ of $0$ in $B$ such that $(\cX_b,c\cD_b+(1-\beta)a\tDelta_b)$ is uniformly K-stable hence admits conical K\"ahler-Einstein metrics for any $b\in B^{\circ}_\beta:=B_{\beta}\setminus\{0\}$. Consider the triple Hilbert scheme $\bH^{\bm{\chi},q;N}$ of $\bP^N$ with the same Hilbert polynomials as $(X,D,\Delta)$.
Let $\Hilb(\cX_b,c\cD_b+(1-\beta)a\tDelta_b)\in \bH^{\bm{\chi},q;N}$ be the Hilbert point of $(\cX_b,c\cD_b+(1-\beta)a\tDelta_b)$ via Tian's embedding. 
\end{notation}

\begin{prop}\label{prop:hilbconverge}
With Notation \ref{not:const}, the log Fano pair $(X, cD+(1-\beta)a\Delta)$ admits a weak conical K\"ahler-Einstein metric for any $\beta\in [\gamma_0,\fB)$. Moreover, for any sequence of points $b_i\to 0$ in $B_\beta^\circ$, there exists a sequence of matrices $g_i\in \U(N+1)$ such that \[
g_i\cdot \Hilb(\cX_{b_i}, c\cD_{b_i}+(1-\beta)a\tDelta_{b_i})\to \Hilb(X,cD+(1-\beta)a\Delta)\in\bH^{\bm{\chi},q;N}\quad \textrm{ as }i\to \infty.
\]
\end{prop}

\begin{proof}
 By Theorem \ref{thm:lwx4.1}, we know that after choosing suitable $g_i\in \U(N+1)$ the Hilbert points $g_i\cdot \Hilb(\cX_{b_i}, c\cD_{b_i}+(1-\beta)a\tDelta_{b_i})$ converge as $i\to\infty$ to $\Hilb(Y, cE+(1-\beta)a\Gamma)$ which is the Hilbert point of a log Fano pair $(Y, cE+(1-\beta)a\Gamma)$ via Tian's embedding of its weak conical K\"ahler-Einstein metric. Then \cite{Ber16} (see also \cite[Section 3.1]{LL16}) implies that $(Y, cE+(1-\beta)a\Gamma)$ is K-polystable. By Lemma \ref{lem:hilbconv}, after possibly replacing $(0\in B)$ by its quasi-finite cover $(0'\in B')$, the log Fano pair $(Y, cE+(1-\beta)a\Gamma)$ is a K-polystable fill-in of the family $(\cX,c\cD+(1-\beta)a\tDelta)\times_B (B'\setminus\{0'\})$. Since a K-polystable fill-in is always unique by \cite{BX18}, we know that 
 \[
 (Y, cE+(1-\beta)a\Gamma)\cong (X,cD+(1-\beta)a\Delta).
 \]
 The proof is finished. 
 \end{proof}

\begin{lem}\label{lem:hilbconv}
 Let $G$ be an algebraic group acting on $\bP^M$. Let $z:B\to \bP^M$ be a morphism from a smooth pointed curve $(0\in B)$. Denote by $B^\circ:=B\setminus\{0\}$. Suppose $z_0$ is a point in $\bP^M$ satisfying that there exists $g_i\in G$ and $B^\circ \ni b_i\to 0$ for $i\in\bZ_{>0}$ such that $g_i\cdot z(b_i)\to z_0$ as $i\to\infty$.
 Then there exists a quasi-finite morphism $\pi:(0'\in B')\to (0\in B)$ from a smooth pointed curve $(0'\in B')$ with $\{0'\}=\pi^{-1}(\{0\})$ and two morphisms $\tau: B'^\circ:=B'\setminus\{0\}\to G$ and $z':B'\to\bP^M$ such that $z'(b')=\tau(b')\cdot z(\pi(b'))$ for any $b'\in B'^\circ$ and $z'(0')=z_0$. 
\end{lem}

\begin{proof}
 Let $\phi: G\times B^\circ\to \bP^M\times B$ be the morphism defined as $\phi(g,b):=(g\cdot z(b),b)$. Let $\overline{G}$ be a normal proper variety that compactifies $G$. Let $\Gamma$ be the normalized graph of the rational map $\phi: \overline{G}\times B\to \bP^M\times B$. Hence we have a proper birational morphism $\psi:\Gamma\to \overline{G}\times B$ which is an isomorphism over $G\times B^{\circ}$, and a proper $B$-morphism $\tphi:\Gamma\to \bP^M\times B$. From the assumption we know that $(z_0,0)\in \overline{\phi(G\times B^\circ)}=\tphi(\Gamma)$. Let us take a point $\tilde{z}_0\in \tphi^{-1}(z_0,0)\subset\Gamma$. Then we may choose a smooth pointed curve $(0'\in B')$ together with a finite morphism $f:B'\to \Gamma$ such that $f(0')=\tilde{z}_0$ and $f(B')\cap (G\times B^\circ)\neq\emptyset$. After possibly shrinking $(0'\in B')$ we may assume that $f(B'^\circ)\subset G\times B^\circ$. Then by defining $\pi:=\pr_2\circ \psi\circ f$, $\tau:=\pr_1\circ f|_{B'^\circ}$, and $z':=\pr_1\circ \tphi\circ f$, it is easy to check that the conclusion is satisfied. 
\end{proof}

A priori $\fB$ might only be a real number. Nevertheless, the following result shows that $\fB$ has to be rational and we can find a destabilizing test configuration in $\bP^{N_m}$.

\begin{prop}\label{prop:continuity-Kus}
 With Notation \ref{not:const}, if $\fB<1$ then $(X,cD+(1-\fB)a\Delta)$ does not admit a weak conical K\"ahler-Einstein metric. There exists a $1$-PS $\lambda$ of $\SL(N_m+1)$ inducing a non-product special test configuration of $(X,cD+(1-\fB)a\Delta)$ such that the central fiber $(X', cD'+(1-\fB)a\Delta')$ admits a weak conical K\"ahler-Einstein metric.
Moreover, the generalized Futaki invariant of this special test configuration vanishes.
In particular, $\fB$ is a rational number.
\end{prop}

\begin{proof}
 We first show that $(X,cD+(1-\fB)a\Delta)$ does not admit a weak conical K\"ahler-Einstein metric.
 Since $(X,cD+(1-\gamma_0)a\Delta)$ is uniformly K-stable, we know that $(X,D,\Delta)$ has finite automorphism group by \cite[Corollary 3.5]{BX18}. Assume to the contrary that 
 $(X,cD+(1-\fB)a\Delta)$ admits a weak conical K\"ahler-Einstein metric. Then from \cite{BBEGZ, Dar17, DNG18} we know that the Mabuchi energy is proper. This indeed implies that $(X,cD+(1-\fB)a\Delta)$ is uniformly K-stable by \cite{BBJ18, BHJ16}. However, by \cite{Fuj17} we know that $(X,cD+(1-\fB-\epsilon)a\Delta)$ is also uniformly K-stable for $0<\epsilon\ll 1$ which contradicts our definition of $\fB$. 
 
 Next, let us choose $\gamma_0\leq \beta_i\nearrow \fB$ as $i\to \infty$. Then by Proposition \ref{prop:hilbconverge} we may choose 
 $B_{\beta_i}^\circ\ni b_i\to 0$ as $i\to \infty$ such that both $(\cX_{b_i},c\cD_{b_i}+(1-\beta_i)a\tDelta_{b_i})$ and $(X,cD+(1-\beta_i)a\Delta)$ admit a (weak) conical K\"ahler-Einstein metric for any $i$, and
 \[
 \lim_{i\to\infty}\dist_{\bH^{\bm{\chi},q;N}}\left(\Hilb(\cX_{b_i}, c\cD_{b_i}+(1-\beta_i)a\tDelta_{b_i}), \U(N+1)\cdot \Hilb(X,cD+(1-\beta_i)a\Delta)\right)=0.
 \]
 By the results of Theorem \ref{thm:lwx4.1}, there exists a sequence of matrices $g_i\in \U(N+1)$ and a log Fano pair $(X',cD'+(1-\fB)a\Delta')$ admitting a weak conical K\"ahler-Einstein metric such that 
 \[
  g_i\cdot\Hilb(\cX_{b_i},c\cD_{b_i}+(1-\beta_i)a\tDelta_{b_i})\to \Hilb(X',cD'+(1-\fB)a\Delta')\in\bH^{\bm{\chi},q;N}\quad\textrm{ as }i\to\infty.
 \]
 Since $\Hilb(X,cD+(1-\beta_i)a\Delta)\in \SL(N+1)\cdot\Hilb(X,D,\Delta)$, we know that 
 \[
 \Hilb(X',cD'+(1-\fB)a\Delta')\in\overline{ \SL(N+1)\cdot\Hilb(X,D,\Delta)}\subset\bH^{\bm{\chi},q;N}.
 \]
 On the other hand, we know that $(X',cD'+(1-\fB)a\Delta')$ is not isomorphic to $(X,cD+(1-\fB)a\Delta)$ since the latter does not admit a conical K\"ahler-Einstein metric. Since $(X', cD' + (1-\fB) a\Delta')$ admits a weak conincal K\"ahler-Einstein metric, its automorphism group is reductive by \cite[Theorem 5.2]{BBEGZ}. Hence by  \cite[Proposition 1]{Don12} there exists a $1$-PS $\lambda$ of $\SL(N+1)$ which induces a special test configuration $(\cX_\lambda, c\cD_\lambda+(1-\fB)a\Delta_\lambda)$ of $(X,cD+(1-\fB)a\Delta)$ with central fiber $(X',cD'+(1-\fB)a\Delta')$. Let $\cL_{\lambda}$ be a sufficiently divisible multiple of $-(K_{\cX_{\lambda}/\bA^1}+c\cD_\lambda)$.
 It is clear from the definition that $\beta\mapsto \Fut(\cX_\lambda, c\cD_\lambda+(1-\beta)a\Delta_\lambda;\cL_{\lambda})$ is a degree $1$ polynomial in $\beta$ with rational coefficients. Since $(X,cD+(1-\beta)a\Delta)$ is uniformly K-stable when $\beta<\fB$, we know that \[
 \Fut(\cX_\lambda, c\cD_\lambda+(1-\fB)a\Delta_\lambda;\cL_{\lambda})\geq 0.
 \]
 On the other hand, the $1$-PS $\lambda^{-1}$ of $\Aut(X',D',\Delta')\subset \SL(N+1)$ induces a product test configuration $(\cX_{\lambda^{-1}}',c\cD_{\lambda^{-1}}'+(1-\fB)a\Delta_{\lambda^{-1}}')$ of $(X',cD'+(1-\fB)a\Delta')$ such that 
 \[
 \Fut(\cX_{\lambda^{-1}}',c\cD_{\lambda^{-1}}'+(1-\fB)a\Delta_{\lambda^{-1}}';\cL_{\lambda^{-1}}')+\Fut(\cX_\lambda, c\cD_\lambda+(1-\fB)a\Delta_\lambda;\cL_{\lambda})=0.
 \]
 Again by \cite{Ber16} we know that $(X',cD'+(1-\fB)a\Delta')$ is K-polystable. Hence both generalized Futaki invariants in the above equation are zero. 
 Thus $\beta=\fB$ is the solution of  the equation $\Fut(\cX_\lambda, c\cD_\lambda+(1-\beta)a\Delta_\lambda;\cL_{\lambda})=0$ which implies $\fB$ is a rational number.
\end{proof}

Now we are ready to prove Theorem \ref{thm:bddtestK}.

\begin{proof}[Proof of Theorem \ref{thm:bddtestK}]
We follow Notation \ref{not:const}.

(1) Assume $(X,cD)$ is K-unstable and $c\leq \lct(X;D)-\frac{\epsilon_0 r^{-1}}{n+1}$. Thus we have $\fB<1$. Then by Proposition \ref{prop:continuity-Kus}, there exists a $1$-PS $\lambda$ of $\SL(N_m+1)$ which induces a test configuration $(\cX_{\lambda},c\cD_{\lambda}+(1-\beta)a\Delta_{\lambda})$ of $(X,cD+(1-\beta)a\Delta)$ such that 
\[
\Fut(\cX_{\lambda},c\cD_{\lambda}+(1-\beta)a\Delta_{\lambda};\cL_{\lambda})\geq 0 \quad\textrm{ if and only if }\beta \leq \fB.
\] 
Therefore, $\Fut(\cX_{\lambda},c\cD_\lambda;\cL_{\lambda})<0$ which implies that $(X,cD)$ is destabilized by $\lambda$.

(2) Assume that $(X,cD)$ is K-semistable but not K-polystable.
Let us choose $\gamma_0\leq \beta_i\nearrow 1$ as $i\to\infty$.
Following the proof of Proposition \ref{prop:continuity-Kus}, there exists a log Fano pair $(X',cD'+0\cdot\Delta')\subset\bP^N$ admitting a weak conical K\"ahler-Einstein metric such that
\[
\Hilb(X',cD'+0\cdot\Delta')\in \overline{\SL(N+1)\cdot\Hilb(X,D,\Delta)}\subset\bH^{\bm{\chi},q;N}.
\]
In particular, $(X',cD')$ is K-polystable by \cite{Ber16} hence is not isomorphic to $(X,cD)$. Then by \cite[Proposition 1]{Don12} we obtain a special degeneration from $(X,cD)$ to $(X',cD')$ induced by a $1$-PS in $\SL(N+1)$. In addition, since $\Hilb(\cX_{b_i}, c\cD_{b_i}+(1-\beta_i)a\tDelta_{b_i})$ converges to $\Hilb(X',cD'+0\cdot\Delta')$ in $\bH^{\bm{\chi},q;N}$ as $i\to \infty$, we know that $\Hilb(\cX_{b_i},\cD_{b_i})$ converges to $\Hilb(X',D')$ in $\bH^{\bm{\chi};N}$ for suitable embeddings. Hence $\Hilb(X',D')\in Z_m^{\klt}$.

(3) Assume $(X,cD)$ is K-polystable. Similar to part (2), there exists a log Fano pair $(X',c\cD'+0\cdot\Delta')$ in $\bP^N$ admitting a weak conical K\"ahler-Einstein metric such that
\[
\Hilb(X',cD'+0\cdot\Delta')\in \overline{\SL(N+1)\cdot\Hilb(X,D,\Delta)}\subset\bH^{\bm{\chi},q;N}.
\]
If $\Hilb(X',cD'+0\cdot\Delta')\in \SL(N+1)\cdot\Hilb(X,D,\Delta)$ then we are done. So we may assume $\Hilb(X',cD'+0\cdot\Delta')\not\in\SL(N+1)\cdot\Hilb(X,D,\Delta)$. Then again by \cite[Proposition 1]{Don12} we get a special degeneration from $(X,cD)$ to $(X',cD')$ induced by a $1$-PS in $\SL(N_m+1)$. By a similar argument as the proof of Proposition \ref{prop:continuity-Kus}, the generalized Futaki invariant of this special test configuration vanishes since both $(X,cD)$ and $(X',cD')$ are K-polystable. Hence they are isomorphic.

(4) Assume that $(X,cD)$ is K-stable. Then it admits a weak conical K\"ahler-Einstein metric and $\Aut(X,D)$ is reductive by (3). Hence K-stability of $(X,cD)$ implies that $\Aut(X,D)$ is a finite group, which implies uniform K-stability of $(X,cD)$ by \cite{BBEGZ, Dar17, DNG18, BBJ18, BHJ16}. The proof is finished. 
\end{proof}

\subsection{K-semistable thresholds are constructible}\label{sec:kstconst}

In this section, we show that the K-semistable thresholds are constructible functions satisfying certain semi-continuity conditions. In particular, this implies the openness of K-semistability in our setting. Our approach is based on \cite[Section 7 and A.1]{LWX14}.

\begin{defn}
Fix a rational number $\epsilon_0\in (0,1)$.
For any $(X,D)$ with $\Hilb(X,D)\in Z^{\klt}$, we define the  \emph{upper and lower K-semistable thresholds} as follows:
\begin{align*}
\kst_{+,\epsilon_0}(X,D) & :=\sup\{c\in (0,\min\{1,(1-\epsilon_0)r^{-1}\})\mid (X,cD)\textrm{ is K-semistable}\};\\
\kst_{-,\epsilon_0}(X,D) & :=\inf\{c\in (0,\min\{1,(1-\epsilon_0)r^{-1}\})\mid (X,cD)\textrm{ is K-semistable}\}.
\end{align*}
\end{defn}

Next, we will start the construction of the
K-moduli stack $\cK\cM_{\chi_0,r,c}$. In this section, we will focus on 
showing the openness of $c$-K-semistability
and constructibility of K-semistable thresholds.

\begin{thm}\label{thm:kstfinite}
The functions $\kst_{\pm,\epsilon_0}$ on $Z^{\klt}$ are constructible with rational values.
Moreover, $\kst_{+,\epsilon_0}$ (resp. $\kst_{-,\epsilon_0}$) is lower (resp. upper) semicontinuous
on $Z^{\klt}$. In particular, $Z_c^{\circ}$ are Zariski open subsets of $Z^{\klt}$ whenever $c\in (0,\min\{1,(1-\epsilon_0)r^{-1}\})$.
\end{thm}

Note that the semicontinuity  properties of these types
of functions (in relation to existence of conical KE metrics) were observed earlier in \cite{SSY16, LWX14}.

Before presenting the proof of Theorem \ref{thm:kstfinite}, we  recall some results from
\cite[Section A.1]{LWX14} (see also \cite[Chapter 2, Proposition 2.14]{MFK94} and \cite[Proof of Lemma 2.11]{Oda13c}).

\begin{lem}[{\cite[Lemma A.3]{LWX14}}]\label{lem:lwxA.3}
 Let $Z$ be a projective variety. Let $L,M$
be two $G$-linearized ample line bundles over $Z$.
Let $T\subset G$ be a maximal torus.
Then there is a finite set of linear functionals
$l_{1}^L,\cdots,l_{r_L}^L,l_{1}^M,\cdots,l_{r_M}^M$ which are rational on $\Hom_{\bQ}(\bG_m,T)$
with the following property:

For any $z\in Z$, there exist $I(z,L)\subset\{1,\cdots,r_L\}$, $I(z,M)\subset\{1,\cdots,r_M\}$ such that the $\lambda$-weight of $z\in Z$ with respect to the linearization of $G$ on $L\otimes M^{-1}$ is given by 
\[
\mu^L(z,\lambda)-\mu^M(z,\lambda)=\max\{l_{i}^{L}(\lambda)\mid 
i\in I(z,L)\}-\max\{l_{i}^{M}(\lambda)\mid i\in I(z,M)\}
\]
for all $1$-PS $\lambda$ of $T$. Moreover, the function
\begin{align*}
\psi^{L,M}:Z & \to 2^{\{1,\cdots,r_L\}}\times 2^{\{1,\cdots,r_M\}}\\
z&\mapsto (I(z,L), I(z,M))
\end{align*}
is constructible.

\end{lem}

We present the proof of Theorem \ref{thm:kstfinite} below which essentially follows from \cite[Proposition A.4]{LWX14}.

\begin{proof}[Proof of Theorem \ref{thm:kstfinite}]
Let us consider the $c$-K-stability on $Z^{\klt}$
for $c\in (0,\min\{1,(1-\epsilon_0)r^{-1}\})$.
Denote by $\pi:(\cX,\cD)\to \oZ$ the universal family with $\cL$ representing the pull back of the line bundle $\cO_{\bP^N}(1)$. According to Definition \ref{defn:logCM1}, denote by
\begin{equation}\label{eq:M_i}
M_1:=\lambda_{\CM,\pi,\cL}\quad\textrm{and}\quad 
M_2:=\frac{n(\cL_z^{n-1}\cdot\cD_z)}{(\cL_z^n)}\lambda_{\Chow,\pi,\cL}-(n+1)\lambda_{\Chow,\pi|_{\cD},\cL|_{\cD}}.
\end{equation}
Hence from Definition \ref{defn:logCM1} we know that 
$\lambda_{\CM,\pi,c\cD,\cL}=M_1-cM_2$.
Notice that by flatness of the universal family, this function $z\in\oZ\mapsto n(\cL_z^{n-1}\cdot\cD_z)/(\cL_z^n)$ does not depend on the choice of $z$.
Let $(X,D)$ be a log pair with $z:=\Hilb(X,D)\in Z^{\klt}$. For simplicity, denote by $G:=\SL(N+1)$. Then every $1$-PS $\lambda:\bG_m\to G$ naturally induces a test configuration $(\cX_\lambda,c\cD_\lambda;\cL_\lambda)$ of $(X,cD;L)$.
Moreover, Proposition \ref{prop:Fut=CMwt} implies that 
\[
\Fut(\cX_\lambda,\cD_\lambda;\cL_\lambda)=\frac{1}{(n+1)(L^n)}
\left(\mu^{M_1}(z,\lambda)-c\mu^{M_2}(z,\lambda)\right).
\]
Thus Theorem \ref{thm:bddtestK} implies that $(X,D)$ is $c$-K-semistable if and only if $c\leq  \lct(X;D)-\frac{\epsilon_0}{n+1}$ and
\[
\mu^{M_1}(z,\lambda)-c\mu^{M_2}(z,\lambda)\geq 0\quad\textrm{ for any $1$-PS $\lambda$ of $G$}.
\]
Pick a sufficiently divisible positive integer $k$ such that $M_1^{\otimes k}$ and $M_2^{\otimes k}$ are line bundles over $\oZ$. Let $M$ be a sufficiently ample $\SL(N+1)$-line bundle on $\oZ$ such that $L_1:=M\otimes M_1^{\otimes k}$ and $L_2:=M\otimes M_2^{\otimes k}$ are both ample line bundles. Then we have
\[
k\left(\mu^{M_1}(z,\lambda)-c\mu^{M_2}(z,\lambda)\right)=\left(\mu^{L_1}(z,\lambda)-\mu^{M}(z,\lambda)\right)-c\left(\mu^{L_2}(z,\lambda)-\mu^{M}(z,\lambda)\right).
\]
We fix a maximal torus $T\subset G$.
Hence using Lemma \ref{lem:lwxA.3}, we know that there exists a decomposition $Z^{\klt}=\sqcup_I S_I^T$ to constructible subsets $S_I^T$ where $I$ belongs to some finite index set, such that for any $z\in S_I^T$ the functions $\mu^{M_1}(z,\cdot)$ and $\mu^{M_2}(z,\cdot)$ are rational piecewise linear functions on $\Hom_{\bQ}(\bG_m, T)$ that are independent of the choice of $z$. We denote these two functions by $\mu_{1,I}(\cdot)$ and $\mu_{2,I}(\cdot)$, respectively. On the other hand, since any $1$-PS $\lambda$ of $G$ is conjugate via $g\in G$ to  a $1$-PS $g\lambda g^{-1}$ of $T$, and $\mu^{M_i}(z, \lambda)=\mu^{M_i}(g\cdot z, g\lambda g^{-1})$ for $i=1,2$, we know that $(X,D)$ is $c$-K-semistable if and only if $c\leq \lct(X;D)-\frac{\epsilon_0}{n+1}$ and 
\[
\mu_{1,I}(\lambda)\geq c\mu_{2,I}(\lambda)\quad
\textrm{for any $1$-PS $\lambda$ of $T$ and any $I$ with $z\in S_I^G$}.
\]
Here $S_I^G:=G\cdot S_I^T$ is a constructible subset of $Z^{\klt}$ by Chevalley's Lemma \cite[Exercise II.3.19]{Har77}. 
Since $\mu_{i,I}$ is a rational piecewise linear function on $\Hom_{\bQ}(\bG_m,T)$, the union $\cup_I S_I^G=Z^{\klt}$, and $z\in Z^{\klt}\mapsto\lct(X;D)$ is a constructible function with rational values (see e.g. \cite[Corollary 2.10]{Amb16}), we know that $\kst_{\pm,\epsilon_0}$ 
are constructible with rational values as well. Their semi-continuity follows from 
the very genericity of K-semistability \cite[Theorem 3]{BL18a}.
\end{proof}

\begin{cor}\label{cor:openness}
 Let $\pi:(\cX,c\cD)\to T$ be a $\bQ$-Gorenstein smoothable log Fano family over a normal base $T$ where $\cD\sim_{\pi}-rK_{\cX/T}$. Then for any $c\in (0,\min\{1,r^{-1}\})$, the set \[\{t\in T\mid (\cX_t,c\cD_t)\textrm{ is K-semistable}\}\] is a Zariski open subset of $T$.
\end{cor}

\begin{proof}
 For each $c\in (0,\min\{1,r^{-1}\})$ we may choose $0<\epsilon_0\ll 1$ such that $c<(1-\epsilon_0)r^{-1}$. Then the result follows from Theorem \ref{thm:bdd-bigangle} and  \ref{thm:kstfinite}.
\end{proof}

We finish this section with a useful result on K-polystability. See \cite[Theorem 1.1]{LS14} for a related result in the smooth case.

\begin{prop}\label{prop:openkps}
 Let $(X,c_0 D)$ be a $\bQ$-Gorenstein smoothable
 K-polystable log Fano pair. Then the set
 \[
  \{c\in(0,\min\{1,r^{-1}\})\colon (X,cD)\textrm{ is K-polystable}\}
 \]
 is either $\{c_0\}$ or an open interval containing $c_0$.
\end{prop}

\begin{proof}
By choosing $0<\epsilon_0\ll 1$ we may assume that $c_0<\min\{1,(1-\epsilon_0)r^{-1},\lct(X;D)-\frac{\epsilon_0}{n+1}\}$. 
By Proposition \ref{prop:k-interpolation}, we know that the set
 \[
 J:=\{c\in(0,\min\{1,r^{-1}\})\colon (X,cD)\textrm{ is K-polystable}\}
 \]
 is either $\{c_0\}$ or an interval containing $c_0$. Hence it suffices to show that $J$ contains an open neighborhood of $c_0$. Denote by $z_0=\Hilb(X,D)\in Z^{\klt}$.
 Let $G:=\SL(N+1)$ and $T$ be a maximal torus of $G$.
 
Assume to the contrary that $J$ does not contain any neighborhood of $c_0$ and $J\neq\{c_0\}$. Then there exists $0<|\epsilon|\ll 1$ such that 
$(X,(c_0+\epsilon)D)$ is K-polystable, and $(X,(c_0-\epsilon')D)$ is K-unstable whenever $0<\epsilon'/\epsilon\leq 1$. From the proof of Theorem \ref{thm:kstfinite}, we know that 
\[
\mu_{1,I}(\lambda)< (c_0-\epsilon')\mu_{2,I}(\lambda) \quad\textrm{for some $1$-PS $\lambda$ of $T$ and some $I$ with $z\in S_I^G$}.
\]
A priori $\lambda$ and $I$ may depend on the choice of $\epsilon'$. Nevertheless, since $I$ belongs to a finite index set, and $\mu_{i,I}$ is a rational piecewise linear function on $\Hom_{\bQ}(\bG_m,T)$ for $i=1,2$, there exist $\lambda$ and $I$ that are independent of the choice of $\epsilon'$ satisfying
\begin{equation}\label{eq:open-K-ps1}
    \mu_{1,I}(\lambda)<(c_0-\epsilon')\mu_{2,I}(\lambda) \quad\textrm{ whenever } 0<\epsilon'/\epsilon\leq 1.
\end{equation}
In particular, we know that $\mu_{1,I}(\lambda)\leq c_0\mu_{2,I}(\lambda)$. Since $z_0\in S_I^G=G\cdot S_I^T$, we  choose $g\in G$ such that $g \cdot z_0\in S_I^T$.
Then since $(X,c_0 D)$ is K-polystable, we have 
\begin{equation}\label{eq:open-K-ps2}
    \mu_{1,I}(\lambda)=\mu^{M_1}(z_0, g^{-1}\lambda g)\geq c_0\mu^{M_2}(z_0, g^{-1}\lambda g)= c_0\mu_{2,I}(\lambda). 
\end{equation}
Combining \eqref{eq:open-K-ps1} and \eqref{eq:open-K-ps2}, we have that $\mu^{M_1}(z_0, g^{-1}\lambda g)= c_0\mu^{M_2}(z_0, g^{-1}\lambda g)\neq 0$. This together with the K-polystability of $(X,c_0 D)$ implies that $g^{-1}\lambda g$ induces an almost product test configuration of $(X,c_0 D)$. Since $(X,(c_0+\epsilon)D)$ is also K-polystable, we have 
\[
\mu^{M_1}(z_0, g^{-1}\lambda g)= (c_0+\epsilon) \mu^{M_2}(z_0, g^{-1}\lambda g)
\]
which implies $\mu^{M_i}(z_0, g^{-1}\lambda g)=0$ for $i=1,2$. However, this contradicts to \eqref{eq:open-K-ps1}. Thus the proof is finished. 
\end{proof}

\subsection{Properness}

In this section, we prove the valuative criterion of properness of K-moduli spaces. Recall that Blum and Xu \cite{BX18} proved separatedness of K-moduli spaces (if they exist) for log Fano pairs. Hence we only need to show compactness of K-moduli spaces, i.e. the existence of a K-semistable filling for a K-semistable family over a punctured smooth curve.

\begin{thm}\label{thm:compactness}
 Let $0\in B$ be a smooth pointed curve. 
 Let $\pi^\circ:(\cX^\circ,c\cD^\circ)\to B^\circ$ be a $\bQ$-Gorenstein smoothable log Fano family over  $B^\circ:=B\setminus\{0\}$  where $\cD^{\circ}\sim_{\pi^\circ}-rK_{\cX^{\circ}/B^{\circ}}$ and $c\in (0,\min\{1,r^{-1}\})$. If all fibers of $\pi^\circ$ are K-semistable, then there exists a quasi-finite morphism $(0'\in B')\to (0\in B)$ from a smooth pointed curve $0'\in B'$ and a $\bQ$-Gorenstein smoothable log Fano family $\pi':(\cX',c\cD')\to B'$ such that
 $(\cX',\cD')\times_{B'}B'^\circ\cong (\cX^\circ,\cD^\circ)\times_{B^\circ}B'^\circ$ where $B'^\circ:=B'\setminus\{0'\}$ and $(\cX'_{0'},c\cD'_{0'})$ is K-semistable (even K-polystable). 
\end{thm}

\begin{proof}
 By choosing $0<\epsilon_0\ll 1$ we may assume that $c<\min\{1,(1-\epsilon_0)r^{-1}\}$. Let $b_i\to 0$ be a sequence of points in $B^{\circ}$. Denote by $(\cX_{b_i},\cD_{b_i})$ the fiber of $\pi^{\circ}$ over $b_i$. Let $\chi_0$ be the Hilbert polynomial of a smoothing of each fiber $\cX_{b_i}$ which certainly does not depend on the choice of $i$ or smoothing.
 Let us choose $m_3$, $m$, $\gamma_0$, $q$ and $a$ as in Notation \ref{not:const}. Since $(\cX_{b_i},c\cD_{b_i})$ is K-semistable, we have $c\leq \lct(\cX_{b_i};\cD_{b_i})$. Hence Proposition \ref{prop:uK-Delta} implies that there exists $\Delta_{b_i}\in |-qK_{\cX_{b_i}}|$ such that $(\cX_{b_i},c\cD_{b_i}+(1-\gamma_0)a\Delta_{b_i})$ is uniformly K-stable. Let us choose $\gamma_0\leq \beta_i\nearrow 1$. Then Proposition \ref{prop:k-interpolation} and \ref{prop:hilbconverge} implies that $(\cX_{b_i},c\cD_{b_i}+(1-\beta_i)a\Delta_{b_i})$ admits a weak conical K\"ahler-Einstein metric whose Hilbert point in $\bH^{\bm{\chi},q;N}$ is the limit of Hilbert points of conical K\"ahler-Einstein log smooth log Fano pairs. In particular, there exists a conical K\"ahler-Einstein log smooth log Fano pair $(Y_i, cE_i+(1-\beta_i)a\Gamma_i)$ as a smoothing of $(\cX_{b_i},c\cD_{b_i}+(1-\beta_i)a\Delta_{b_i})$ such that $\Hilb(Y_i,E_i)\in Z_m$, $\Gamma_i\in |-qK_{Y_i}|$, and 
\begin{equation}\label{eq:compact-dist}
 \lim_{i\to\infty} \dist_{\bH^{\bm{\chi},q;N}}\left(\Hilb(Y_i, cE_i+(1-\beta_i)a\Gamma_i),\Hilb(\cX_{b_i},c\cD_{b_i}+(1-\beta_i)a\Delta_{b_i})\right)=0.
\end{equation}
 By Theorem \ref{thm:lwx4.1}, there exists a sequence of matrices $g_i\in \U(N+1)$ and a log Fano pair $(Y,cE+0\cdot\Gamma)$ in $\bP^N$ admitting a weak conical K\"ahler-Einstein metric such that
 \[
 g_i\cdot\Hilb(Y_i, cE_i+(1-\beta_i)a\Gamma_i)\to\Hilb(Y,cE+0\cdot\Gamma)\in\bH^{\bm{\chi},q;N}\quad\textrm{ as }i\to\infty.
 \]
 This together with \eqref{eq:compact-dist} implies that
 \[
  g_i\cdot\Hilb(\cX_{b_i},c\cD_{b_i}+(1-\beta_i)a\Delta_{b_i})\to\Hilb(Y,cE+0\cdot\Gamma)\in\bH^{\bm{\chi},q;N}\quad\textrm{ as }i\to\infty.
 \]
 Thus there exists $g_i'\in\SL(N+1)$ such that $g_i'\cdot\Hilb(\cX_{b_i},\cD_{b_i})$ converges to $\Hilb(Y,E)$ in $\overline{Z}_m$. Thus by Lemma \ref{lem:hilbconv} after a quasi-finite base change of $\pi^\circ$ we may fill in $(Y,cE)$ as the K-polystable central fiber. The proof is finished. 
\end{proof}

\subsection{Almost log Calabi-Yau cases}

Notice that Definition \ref{defn:kmoduli} depends on the choice of $\epsilon_0$. Indeed, if $r<1$ then it suffices to choose $\epsilon_0=1-r$. When $r\geq 1$ we will show that there exists $\epsilon_0=\epsilon_0(n,r)\in (0,1)$ such that the K-moduli spaces/stacks are the same for any $c\in [(1-\epsilon_0)r^{-1},r^{-1})$.

\begin{thm}\label{thm:almostCYstabilize}
 For any $n\in\bZ_{>0}$ and any rational number $r\geq 1$, there exists $\epsilon_0=\epsilon_0(n,r)\in (0,1)$ such that for any $\bQ$-Gorenstein smoothable log Fano pair $(X,cD)$ with $c\in [(1-\epsilon_0)r^{-1},r^{-1})$, it is K-(poly/semi)stable if and only if $(X,(1-\epsilon_0)r^{-1}D)$ is K-(poly/semi)stable.
\end{thm}

\begin{proof}
Firstly, let us assume that $(X,(1-\epsilon_0)r^{-1}D)$ is K-(poly/semi)stable hence klt. 
By ACC of log canonical thresholds \cite{HMX14}, there exists $\epsilon_1=\epsilon_1(n,r)$ such that $(X,r^{-1}D)$ is log canonical whenever $(X,(1-\epsilon_1)r^{-1}D)$ is log canonical. This is guaranteed for any $\epsilon_0\in (0,\epsilon_1]$ since $(X,(1-\epsilon_0)r^{-1}D)$ is klt. Thus $(X,cD)$ is K-(poly/semi)stable for any $c\in [(1-\epsilon_0)r^{-1},r^{-1})$ provided $\epsilon_0\in (0, \epsilon_1]$.

Next, let $(X,cD)$ be a $\bQ$-Gorenstein smoothable log Fano pair  for some $c\in (0,1)$. We may choose a smoothing $\pi:(\cX,\cD)\to B$ over a smooth pointed curve $(0\in B)$ such that $\pi$ is smooth over $B\setminus\{0\}$ and $(\cX_0,\cD_0)\cong (X,D)$. 
By Lemma \ref{lem:almostlogCY} we may choose $\epsilon_2=\epsilon_2(n,r)\in (0,1)$ such that $(\cX_b,c'\cD_b)$ is K-polystable for any $c'\in [(1-\epsilon_2)r^{-1},r^{-1})$ and any $b\in B\setminus\{0\}$. For simplicity let us assume $\epsilon_2\leq \epsilon_1$. Then by Theorem \ref{thm:compactness} we know that there exists a K-polystable limit $(X',(1-\epsilon_2)r^{-1} D')$ of $(\cX_b,(1-\epsilon_2)r^{-1}\cD_b)$ after possibly passing to a finite cover of $B$. Since $\epsilon_2\leq \epsilon_1$, we know that $(X', r^{-1}D')$ is log canonical. Then by Proposition \ref{prop:k-interpolation} we know that $(X',c'D')$ is the K-polystable limit of $(\cX_b,c'\cD_b)$ whenever $c'\in [(1-\epsilon_2)r^{-1},r^{-1})$. 
Let us choose $\epsilon_0:=\frac{\epsilon_2}{2}$. Assume that $(X,cD)$ is K-(poly/semi)stable for some $c\in [(1-\epsilon_0)r^{-1},r^{-1})$. Then by \cite{BX18} we know that $(X,cD)$ specially degenerates to the K-polystable pair $(X',cD')$. Hence $(X,(1-\epsilon_2)r^{-1}D)$ specially degenerates to the K-polystable pair $(X,(1-\epsilon_2)r^{-1}D')$. In particular, $(X,(1-\epsilon_2)r^{-1}D)$ is K-semistable by Theorem \ref{thm:Kss-spdeg}. Again by Proposition \ref{prop:k-interpolation}, we know that $(X,(1-\epsilon_0)r^{-1}D)$ is K-(poly/semi)stable. The proof is finished.
\end{proof}

\begin{lem}\label{lem:almostlogCY}
 Let $n\in\bZ_{>0}$ and $r\geq 1$ be a rational number. Then there exists $\epsilon_2:=\epsilon_2(n,r)\in (0,1)$ such that for any pair $(X,D)$ where $X$ is an $n$-dimensional Fano manifold, $D$ is a smooth prime divisor on $X$ and $D\sim_{\bQ}-rK_X$, we have that $(X,cD)$ is K-polystable whenever $c\in [(1-\epsilon_2)r^{-1},r^{-1})$.
\end{lem}

\begin{proof}
 When $r>1$, this is a consequence of \cite[Theorem 5.2]{LWX14} based on ACC of log canonical thresholds \cite{HMX14}. When $r=1$, we know that $D\sim -K_X$ since $X$ is a Fano manifold. By boundedness of Fano manifolds, there exists a smooth proper morphism $\pi:(\cX,\cD)\to T$ over a (possibly disconnected) normal base scheme $T$ which parametrizes all pairs $(X,D)$ where $X$ is a Fano manifold and $D$ is a smooth anti-canonical divisor on $X$. For each $\epsilon\in (0,1)$, let us consider the subset
 \[
 T_{\epsilon}:=\{t\in T\mid (\cX_t,(1-\epsilon)\cD_t)\textrm{ is K-semistable}\}.
 \]
 By Corollary \ref{cor:openness} we know that $T_{\epsilon}$ is an open subset of $T$. Since $(\cX_t,\cD_t)$ is log canonical, by  Proposition \ref{prop:k-interpolation} we know that $T_{\epsilon}\subset T_{\epsilon'}$ whenever $0<\epsilon'<\epsilon<1$. Therefore, the noetherian property implies that $(T_\epsilon)$ stabilizes as $0<\epsilon\ll 1$. By \cite[Corollary 1]{JMR16}, we know that for each $t\in T$ there exists $\beta_t\in (0,1)$ such that $(\cX_t,(1-\beta_t)\cD_t)$ is K-polystable. In particular, we have $T_\epsilon=B$ for any  $0<\epsilon \ll 1$. Then again by interpolation we may choose $0<\epsilon_2\ll 1$ such that $(\cX_t,(1-\epsilon_2)\cD_t)$ is K-polystable for any $t\in T$. The proof is finished.
\end{proof}

The following result on boundedness is an easy consequence of Theorem \ref{thm:bdd-bigangle} and  \ref{thm:almostCYstabilize}.

\begin{cor}\label{cor:bdd}
Fix $r>1$ a positive rational number and $n$ a positive integer. Then the following collection of $\bQ$-Gorenstein smoothable pairs
\[
\{(X,D)\mid \dim X=n, ~(X,cD)\textrm{ is K-semistable for some $c\in [0,r^{-1})$}\}
\]
is log bounded.
\end{cor}

Next we prove finiteness of K-moduli walls.

\begin{prop}\label{prop:k-wall-finite}
 There exist rational numbers $0=c_0<c_1<c_2<\cdots<c_k=\min\{1,r^{-1}\}$ such that for any $c, c'\in (c_i, c_{i+1})$ and any $0\leq i\leq {k-1}$ we have $Z_c^{\red}=Z_{c'}^{\red}$. Moreover,  $Z_{c_i\pm \epsilon}^{\red}$ are Zariski open subsets of $Z_{c_i}^{\red}$ for each $1\leq i\leq k-1$. 
\end{prop}

\begin{proof}
The first statement follows from combining Theorems \ref{thm:kstfinite} and \ref{thm:almostCYstabilize}. The second statement follows from the continuity of generalized Futaki invariants with respect to coefficients.
\end{proof}

\subsection{Stabilization of quotient stacks}\label{sec:stackstab}
Next we study the stabilization problem for
the stacks $[Z_{c,m}^{\red}/\PGL(N_m+1)]$.

\begin{thm}\label{thm:stabilization}
 Assume $m$ is sufficiently divisible. Then for each $k\in \bZ_{>0}$ there exists a canonical isomorphism 
 \[
  \Theta_k: [Z_{c,m}^{\red}/\PGL(N_m+1)]
 \to [Z_{c,km}^{\red}/\PGL(N_{km}+1)]  
 \] between reduced Artin stacks.
\end{thm}

\begin{proof}
 We first construct $\Theta_k$ as a morphism.
 For simplicity, denote by $T:=Z_{c,m}^{\red}$, $T':=Z_{c,km}^{\red}$,
 $N:=N_m$, and $N':=N_{km}$.
 Let $(\cX,\cD)\subset\bP^{N}\times T$
 be the pull-back family of the universal family over the Hilbert
 scheme. Denote by $\pi: (\cX,\cD)\to T$ the projection
 morphism. Then $\pi_*\cO_{\cX}(k)$ is a rank $N'+1$ vector
 bundle over $T$. Let $p:\sP\to T$ be the $\PGL(N'+1)$-torsor
 corresponding to projectivized basis of the vector bundle
 $\pi_*\cO_{\cX}(k)$. Then we will define a $\PGL(N'+1)$-equivariant
 morphism $f:\sP\to T'$ as follows. Since $T'$ is a locally
 closed subscheme of the Hilbert scheme $\bH^{\bm{\chi}_k;N'}$, we will first construct 
 $f:\sP\to \bH^{\bm{\chi}_k;N'}$. Consider $\pi_{\sP}:(\cX_{\sP},\cD_{\sP})\to {\sP}$
 where $(\cX_{\sP},\cD_{\sP}):=(\cX,\cD)\times_T \sP$. Then we have a closed embedding
 $(\cX_{\sP},\cD_{\sP})\hookrightarrow\bP^{N'}\times \sP$ given by the projectivized basis
 information encoded in $\sP$. This gives a morphism
 $f:\sP\to \bH^{\bm{\chi}_k;N'}$. Since $T$ contains a Zariski dense
 open subset $T\cap Z_m$ parametrizing smooth log Fano pairs,
 we know that the restriction of $f$ on  $p^{-1}(T\cap Z_m)$ 
 factors through $Z_{km}$. Thus $f$ factors through
 the scheme theoretic closure $\overline{Z}_{km}$. It is clear that
 the image of $f$ lies inside $\Supp(T')$, so $f$ factors as
 $\sP\xrightarrow{f}T'\hookrightarrow \bH^{\bm{\chi}_k;N'}$
 where the latter map is a locally closed embedding. It is clear that 
 $f$ is $\PGL(N'+1)$-equivariant. Thus $f$ descends to  a morphism $g:T\to [T'/\PGL(N'+1)]$. On the other hand, we may lift the $\PGL(N+1)$-action on $T$ to $\sP$ via push forward sections. It is clear that $f$ is $\PGL(N+1)$-invariant which implies $g$ is also $\PGL(N+1)$-invariant. Thus we obtain $\Theta_k$ as the descent of $g$. 

From the above arguments it is clear that the actions of $\PGL(N+1)$ and $\PGL(N'+1)$ on $\sP$ commutes. Hence to show $\Theta_k$ is an isomorphism it suffices to show that $f:\sP\to T'$ is a $\PGL(N+1)$-torsor, from which $\Theta_k^{-1}$ can be constructed easily.
Let us consider the pull back of the  universal family $(\cX',\cD')\subset \bP^{N'}\times T'$ with $\pi':(\cX',\cD')\to T'$. Since all fibers of $\pi'$ are klt with the same volume, by \cite[Theorem 5.4]{Kol17} we know that $\cX'\to T'$ is a locally stable family, in particular $ K_{\cX'/T'}$ is $\bQ$-Cartier whose Cartier index is divisible by $km$. Since $-mK_{\cX'_{t'}}$ is Cartier for any fiber $\cX'_{t'}=\pi^{-1}(t')$, we know that $-mK_{\cX'/T'}$ is also Cartier. It is also clear from the construction that $t'\mapsto h^0(\cX'_{t'},\omega_{\cX'_{t'}}^{[m]})$ is a constant function on $T'$. Hence the coherent sheaf $\pi'_*\omega_{\cX'/T'}^{[m]}$ is a vector bundle of rank $N+1$ on $T'$. Thus we may cover $T'$ by Zariski open subsets $T_i'$ which trivialize $\pi'_*\omega_{\cX'/T'}^{[m]}$. Then over $T_i'$, a basis of sections of $\pi'_*\omega_{\cX'/T'}^{[m]}$ gives us a Zariski local section $T_i'\to \sP$ of $f$. These sections enable us to trivialize the map $f:\sP\to T'$ over $T_i'$.
\end{proof}

\begin{rem}\label{rem:pseudo-functor}
As a consequence of Theorem \ref{thm:stabilization}, we know that the K-moduli stack $\cK\cM_{\chi_0,r,c}$ represents the following moduli pseudo-functor over reduced base $S$:
\[
\cK\cM_{\chi_0,r,c}(S)=\left\{(\cX,\cD)/S\left| \begin{array}{l}(\cX,c\cD)/S\textrm{ is a $\bQ$-Gorenstein smoothable log Fano family,}\\ \cD\sim_{S,\bQ}-rK_{\cX/S},~\textrm{each fiber $(\cX_s,c\cD_s)$ is K-semistable,}\\ \textrm{and $\chi(\cX_s,\cO_{\cX_s}(-kK_{\cX_s}))=\chi_0(k)$ for $k$ sufficiently divisible.}\end{array}\right.\right\}.
\]
\end{rem}
\subsection{Existence of good moduli spaces and local VGIT}\label{sec:VGIT}

In this section, we will show that the K-moduli stack $\cK\cM_{\chi_0,r,c}$ admits a proper good moduli space $KM_{\chi_0,r,c}$ generalizing \cite[Section 8]{LWX14}. Moreover, there are finitely many wall crossings when $c$ varies in the interval $(0, \min\{1,r^{-1}\})$, and each wall crossing has a local VGIT presentation in the sense of \cite[(1.2)]{AFS17}. 

We follow Notation \ref{not:const}. 
Throughout this section, we will assume that $c\leq \min\{1, (1-\epsilon_0)r^{-1}\}$ thanks to Theorem \ref{thm:almostCYstabilize}.
 Let us fix two Pl\"ucker embeddings
$\bH^{\chi;N}\hookrightarrow\bP^M$ and $\bH^{\tilde{\chi};N}
\hookrightarrow\bP^{\tilde{M}}$.
Then we have an embedding $\bH^{\bm{\chi};N}\hookrightarrow
\bP^{\bm{M}}:=\bP^{M}\times\bP^{\tilde{M}}$.
Let $(X,cD)$ be a K-polystable log Fano pair parametrized
by a point in $Z_{c}^{\red}$. Then by Theorem \ref{thm:bddtestK}
it admits a weak conical K\"ahler-Einstein metric and $\Aut(X,D)\subset\SL(N+1)$ is reductive. (Note here that in order to obtain a natural linearization on $\cO_{\bP^{M}}(1,1)$, we always treat the automorphism group as a subgroup of $\SL(N+1)$.)
Let us pick a $\U(N+1)$-invariant metric on $\bP^{\bm{M}}$
coming from product of $\U(N+1)$-invariant Fubini-Study metrics on $\bP^{M}$
and $\bP^{\tilde{M}}$. Let $z_0=(z_{0,1}, z_{0,2}):=\Hilb(X,cD)\in \bP^{\bm{M}}$ be the Hilbert point of $(X,cD)$ via Tian's embedding with respect to the weak conical K\"ahler-Einstein metric.
Then we may decompose the tangent space as $\Aut(X,D)$-invariant
subspaces
\[
 T_{z_0}\bP^{\bm{M}}=W\oplus\aut(X,D)^\perp.
\]
Similarly, we have $\bP^M=\bP(W_1\oplus \bC\cdot z_{0,1}\oplus
\aut_1(X,D)^{\perp})$ and $\bP^{\tilde{M}}=\bP(W_2\oplus
\bC\cdot z_{0,2}\oplus\aut_2(X,D)^{\perp})$.
Let us take $z_0^*:=(z_{0,1}^*,z_{0,2}^*)\in (\bP^{M})^*
\times(\bP^{\tilde{M}})^*$ be the dual point of $z_0$.
Then the locus $(z_0^*\neq 0)$ gives an open immersion
$ \bA^{\bm{M}}:=\bA^{M}\times\bA^{\tilde{M}}\hookrightarrow
 \bP^{\bm{M}}$ which maps the origin to $z_0$.
Hence the image of $W$ under the
exponential map is a vector subspace of $\bA^{\bm{M}}$
which we also denote by $W$.
Let $\oW$ be the Zariski closure of $W$.

It is clear that $z_0^*$ is an $\Aut(X,D)$-fixed point.
Hence $\bA^{\bm{M}}$ and $W$ are both $\Aut(X,D)$-invariant.
We have two induced representations \[\rho_1:\Aut(X,D)\to\SL(W_1\oplus \bC\cdot z_{0,1}) \quad \textrm{ and }  
\rho_2:\Aut(X,D)\to\SL(W_2\oplus\bC\cdot z_{0,2}).\]
Let $\rho:=\rho_1\boxtimes\rho_2$ be the product
representation which induces a linearization of 
$\cO_{\oW}(1,1)$. 
Denote by $\rho_{(X,D)}:\Aut(X)\to\bG_m$ the character
corresponding to the $\Aut(X,D)$-linearization of 
$\cO_{\bP^{\bm{M}}}(1,1)|_{z_0}$.
Since the universal family $(\cX,\cD)\to\oZ$ is $\Aut(X,D)$-equivariant,
it induces an $\Aut(X,D)$-linearization on $M_2$
which we denote by $\rho_{M_2}$ (see \eqref{eq:M_i} for the definition of $M_2$).

\begin{defn}\label{defn:LWXLunaGIT} Let $z\in \oW\cap\oZ$ be a point.
\begin{enumerate}
 \item We say $z$ is $c$-GIT (poly/semi)stable if
 it is GIT (poly/semi)-stable with respect to 
 the $\Aut(X,D)$-action on $\oW$ with linearization
 $\rho\otimes\rho_{(X,D)}^{-1}$ of $\cO_{\oW}(1,1)$.
 \item We say $z$ is $(c+\epsilon)$-GIT (poly/semi)stable
 for some $0<|\epsilon|\ll 1$ if it is GIT (poly/semi)-stable with respect
 to the $\Aut(X,D)$-action on $\oW\cap\oZ$
 with linearization $\rho\otimes\rho_{X,D}^{-1}\otimes
 \rho_{M_2}^{\otimes-\epsilon}$ of $\cO_{\oW\cap\oZ}(1,1)
 \otimes M_2|_{\oW\cap\oZ}^{\otimes -\epsilon}$.
\end{enumerate}

\end{defn}

The next theorem on local GIT chart is a direct generalization of \cite[Theorem 8.8]{LWX14}. 

\begin{thm}\label{thm:localGIT}
 There is an $\Aut(X,D)$-invariant saturated
 affine Zariski open neighborhood $U_W$ of $z_0=\Hilb(X,cD)$
 in $\oW\cap \oZ$ such that every point in $U_W$ is $c$-GIT semistable whose corresponding log pair is $c$-K-semistable, and for any $\Hilb(Y,E)\in U_W$,
 $(Y,cE)$ is K-polystable if and only if $\Hilb(Y,E)$ is
 $c$-GIT polystable.
 
 Moreover, for all $c$-GIT polystable point $\Hilb(Y,E)\in U_W$, we have $\Aut(Y,E)<\Aut(X,D)$, i.e. the local GIT presentation $U_W\sslash \Aut(X,D)$ is stabilizer preserving in the sense of \cite[Definition 2.5]{AFS17}.
\end{thm}

\begin{proof}
By definition, $z_0$ is $c$-GIT polystable. Since saturated affine open neighborhoods of $z_0$ form a basis of all Zariski open neighborhoods of $z_0$, we just need to find one $U_W$ that is an $\Aut(X,D)$-invariant Zariski open neighborhood of $z_0$.
The semistable equivalence part of the statement follows from the openness of GIT semistability and openness of K-semistability in our setting (see Corollary \ref{cor:openness}). 
For the polystable equivalence part, the proof is the same as the proof of \cite[Theorem 8.8]{LWX14}, except that we replace \cite[Lemma 8.10]{LWX14} by Lemma \ref{lem:kpslimit}.

Next we prove the stabilizer preserving for polystable points. First, we recall the $\U(N+1)$-invariant slice $\Sigma$ constructed in \cite[Summary 8.6]{LWX14}. Let $\Sigma$ be the subset of $\bH^{\bm{\chi};N}$ consisting of Hilbert points of K-polystable log Fano pairs in $Z_{c}^{\red}$ via Tian's embedding with respect to their weak conical K\"ahler-Einstein metrics. By Lemma \ref{lem:kpslimit}, we know that $\Sigma$ satisfies \cite[Assumption A.9]{LWX14}. Hence we obtain stabilizer preserving for polystable points by \cite[Theorem A.10]{LWX14}. This finishes the proof.
\end{proof}

\begin{lem}\label{lem:kpslimit} Let $z_i=\Hilb(X_i,D_i)\in Z_c^{\red}$ be a sequence of Hilbert points of $c$-K-semistable log pairs converging to $z_0=\Hilb(X,cD)$. Then each $(X_i,cD_i)$ specially degenerates to a K-polystable log Fano pair $(Y_i, cE_i)\in Z_c^{\red}$, such that 
\[
\lim_{i\to\infty}\dist_{\bH^{\bm{\chi};N}} (\Hilb(Y_i, cE_i), \U(N+1)\cdot z_0)=0,
\]
where $\Hilb(Y_i, cE_i)$ is the Hilbert point corresponding to Tian's embedding of $(Y_i, cE_i)$ with respect to the weak K\"ahler-Einstein metric.
\end{lem}

\begin{proof}
 Assume to the contrary. After passing to a subsequence we know that $z_i':=\Hilb(Y_i,cE_i)$ converges to $z':=\Hilb(Y,cE)$ by taking Gromov-Hausdorff limits as Theorem \ref{thm:lwx4.1}, and $(Y,cE)$ is not isomorphic to $(X,cD)$. Since $z_i'\in \overline{\SL(N+1)\cdot z_i}$, there exists a sequence $g_i\in\SL(N+1)$ such that $g_i\cdot z_i\to z'$ as $i\to\infty$. Thus $z'\in\BO_{z_0}\cap Z_c^{\red}$ where $\BO_{z_0}$ is the broken orbit of $z_0$ with respect to the action of $\SL(N+1)$ on $\overline{Z}$ (see \cite[Section 3]{LWX14} for the definition). By \cite{BX18} on the uniqueness of K-polystable limit, we know that $(Y,cE)\cong (X,cD)$, a contradiction.
\end{proof}

The next result gives stabilizer preserving property of semistable points near $z_0$ as a straightforward consequence of \cite[Lemma 8.12]{LWX14}.

\begin{lem}\label{lem:stabpreserve} After possibly shrinking $U_W$, we have that $\Aut(Y,E)<\Aut(X,D)$ for any point  $\Hilb(Y,E)\in U_W$. 
\end{lem}

\begin{proof}
  We following the proof of \cite[Lemma 8.12]{LWX14}. By the proof of \cite[Theorem 3.1]{LWX14}, we know that the connected component $\Aut_0(Y,E)$ is a subgroup of $\Aut(X,D)$. Let us pick a finite subgroup $H<\Aut(Y,E)$ that meets every connected component of $\Aut(Y,E)$. Hence it suffices to show that $H<\Aut(X,D)$. We also know from \cite[Proof of Theorem 8.8]{LWX14} that the set 
 \[
 \{\Hilb(Y,E)\in \oW\cap\oZ\mid \Aut(Y,E)<\Aut(X,D)\}
 \]
 is constructible. Hence it suffices to show the statement for $z$ lies inside an analytic open neighborhood $U_W^{\an}$ of $z_0$ in $\oW\cap\oZ$. We may assume that $z$ degenerates via a $1$-PS $\lambda$ of $\Aut(X)$ to a $c$-GIT polystable point $z'=\Hilb(Y',E')$ in $U_W^{\an}$. Hence Theorem \ref{thm:localGIT} implies that $(Y',cE')$ is K-polystable. 
 
  Since $(Y,E)$ belongs to a bounded family, we may assume that $|H|$ is uniformly bounded from above.    Since $(Y,cE)$ is klt and belongs to a bounded family, there exists a positive integer $q_1=q_1(n,r,\epsilon_0)>1$ and an $H$-invariant divisor $\Gamma\in |-q_1 K_Y|$ such that $(Y,cE+\Gamma)$ is log canonical. Hence by Proposition \ref{prop:k-interpolation} we know that $(Y, cE+(1-\beta)a\Gamma)$ is uniformly K-stable where $a=\frac{1-c r}{q_1}$ and $\beta\in (0,1)$. Let us take $m_4:=\lcm(m_2(\chi_0,r, q_1,\epsilon_0,\gamma_0), m_3)$ where $m_2(\chi_0,r,q_1,\epsilon_0,\gamma_0)$ is chosen as in Theorem \ref{thm:lwx4.1}. From now on we assume that $m$ is a multiple of $m_4$. Then by Proposition \ref{prop:hilbconverge}, for each $\beta\in [\gamma_0,1)$ there exists an $H$-invariant weak conical K\"ahler-Einstein metric $\omega_{Y,cE}(\beta)$  together with a Tian embedding $\Hilb(Y,cE,\omega_{Y,cE}(\beta))\in\bH^{\bm{\chi};N} $ via $\omega_{Y,cE}(\beta)$. Moreover, from the uniqueness of K-polystable limits \cite{BX18} we have
  \[
  \Hilb(Y,cE,\omega_{Y,cE}(\beta))\xrightarrow{\beta\to 1} \U(N+1)\cdot \Hilb(Y',cE')\subset\bH^{\bm{\chi};N}.
  \]
  Hence the proof proceeds the same as \cite[Proof of Lemma 8.12]{LWX14}.
\end{proof}

The following Luna slice type result is also a straightforward consequence of \cite[Lemmas 8.15 and A.15]{LWX14}. We omit the proof here because it is identical to the proof therein on verifying the finite distance property.

\begin{lem}\label{lem:luna}
 After possibly shrinking $U_W$, the morphism
 \[
  \psi: \SL(N+1)\times_{\Aut(X,D)}U_W\to U  
 \]
 is a finite strongly \'etale $\SL(N+1)$-morphism onto a Zariski open subset $U$ of $Z_c^{\red}$.
\end{lem}

Now we are ready to prove Theorem \ref{thm:lwxlog}, the first main result of our construction of K-moduli stacks and spaces.

\begin{proof}[Proof of Theorem \ref{thm:lwxlog}]
We first show that the Artin stack $\cK\cM_{\chi_0,r,c}$ admits a good moduli space $KM_{\chi_0,r,c}$ as a proper reduced algebraic space. 
By \cite[Theorem 1.2]{AFS17}, this boils down to proving the following: for any closed point $[z_0]\in [Z_c^{\red}/\SL(N+1)]$, there is a saturated affine neighborhood $z_0:=\Hilb(X,cD)\in U_W\subset\oW\cap\oZ$ (as in Lemma \ref{lem:luna}) such that 
\begin{enumerate}
    \item The morphism $f:[U_W/\Aut(X,D)]\to [Z_c^{\red}/\SL(N+1)]$ is a local quotient presentation in the sense of \cite[Definition 2.1]{AFS17}. Moreover, $f$ is stabilizer preserving and sends closed point to closed point, and
    \item For any $\bC$-point $z\in Z_c^{\red}$ specializing to $z_0$ under the $\SL(N+1)$-action, the closure $\overline{\{[z]\}}$ admits a good moduli space.
\end{enumerate}

We have shown the $\Aut(X,D)$ is reductive, and $z_0$ is a $c$-GIT polystable point with stabilizer $\Aut(X,D)$. Since $\SL(N+1)\times_{\Aut(X,D)} U_W$ is the quotient of the affine scheme $\SL(N+1)\times U_W$ by the free action of the reductive group $\Aut(X,D)$, we know that $\SL(N+1)\times_{\Aut(X,D)} U_W$ is also affine. Hence Lemma \ref{lem:luna} implies that $f$ is \'etale and affine, in particular $U=\SL(N+1)\cdot U_W$ is affine. Thus $f$ is a local quotient presentation according to \cite[Definition 2.1]{AFS17}. In Lemma \ref{lem:stabpreserve} we showed that $f$ is stabilizer preserving. By Theorem \ref{thm:bddtestK}.2, we know that a closed point in $[Z_c^{\red}/\SL(N+1)]$ corresponds to a $c$-K-polystable pair. Thus $f$ sends closed point to closed point by Theorem \ref{thm:localGIT}. Hence  part (1) is proved. For part (2), notice that the closure $\overline{\SL(N+1)\cdot z}$ in $Z_c^{\red}$ is a closed subset of $U$ since $\overline{\SL(N+1)\cdot z}=\SL(N+1)\cdot \overline{\Aut(X,D)\cdot z}$ and $\psi$ is finite onto $U$. Since $U$ is affine, we know that $\overline{\SL(N+1)\cdot z}$ is also affine. Thus we finish the proof of part (2).

Indeed, we have shown that $U\sslash \SL(N+1)$ is affine. Hence the good moduli space $KM_{\chi_0,r,c}$ is a reduced scheme. Its properness follows from \cite{BX18} and Theorem \ref{thm:compactness}.
\end{proof}

The next theorem provides a local VGIT chart of K-moduli wall crossing in the slice $W$. Before stating the theorem, we recall a lemma we will need. 

\begin{lem}\label{lem:zerofut}\leavevmode
\begin{enumerate}
  \item (\cite{Kem78}) Let $G$ be a reductive group acting 
  on a polarized projective scheme $(Z,L)$.
  Let $z\in Z$ be a closed point. Let $\lambda:\bG_m\to G$
  be a 1-PS. Denote by $z'=\lim_{t\to 0}\lambda(t)\cdot z$.
  If $z$ is GIT semistable and $\mu^{L}(z,\lambda)=0$,
  then $z'$ is also GIT semistable.
  \item (\cite[Lemma 3.1]{LWX18})
  Let $(X,\Delta)$ be a log Fano pair. Let $(\cX,\tDelta;\cL)/\bA^1$
  be a normal test configuration of $(X,\Delta)$.
  If $(X,\Delta)$ is K-semistable and $\Fut(\cX,\tDelta;\cL)=0$,
  then $(\cX,\tDelta;\cL)/\bA^1$ is a special test configuration
  and $(X_0,\Delta_0)$ is also K-semistable.
 \end{enumerate}

\end{lem}

\begin{thm}\label{thm:localVGIT}
 There is an $\Aut(X,D)$-invariant saturated affine 
 Zariski open neighborhood $U_W$ (as in Lemma \ref{lem:luna}) of $z_0=\Hilb(X,cD)$
 such that for any $\Hilb(Y,E)\in U_W$ and
 any $|\epsilon|\ll 1$, the log Fano pair $(Y,(c+\epsilon)E)$
 is K-(poly/semi)stable if and only if $\Hilb(Y,E)$
 is $(c+\epsilon)$-GIT (poly/semi)stable.
\end{thm}

\begin{proof}
 Let $(Y,E)$ be a pair with $z:=\Hilb(Y,E)\in U_W$.
 Suppose $(Y,(c+\epsilon)E)$ is K-semistable. We will show that $\Hilb(Y,E)$ 
 is $(c+\epsilon)$-GIT semistable. Assume to the contrary
 that $\Hilb(Y,E)$ is $(c+\epsilon)$-GIT unstable,
 then there exists a 1-PS $\lambda:\bG_m\to \Aut(X,D)$ 
 such that 
 \[
 \mu^{\cO(1,1)}(z,\lambda)-\epsilon\mu^{M_2}(z,\lambda)<0.
 \]
 From the discussion in Section \ref{sec:kstconst}, we can choose $\lambda$
 so that $\mu^{\cO(1,1)}(z,\lambda)=0$
 and $\epsilon\mu^{M_2}(z,\lambda)>0$. 
 Since $z\in\oW\cap\oZ$ is $c$-GIT semistable,
 by Lemma \ref{lem:zerofut}.1
 we know that $z':=\lim_{t\to 0}\lambda(t)\cdot z$
 is also $c$-GIT semistable. Since $U_W$ is saturated,
 we know that $z'\in U_W$. Let $(Y',E')$ be the pair
 such that $\Hilb(Y',E')=z'$. Then Theorem \ref{thm:localGIT}
 implies that $(Y',cE')$ is also K-semistable.
 Hence $\mu^{M_1}(z,\lambda)-c\mu^{M_2}(z,\lambda)=0$
 since $\lambda$ induces a $c$-K-semistable family over $\bA^1$.
 Let $(\cY_\lambda,(c+\epsilon)\cE_\lambda;\cL_\lambda)$ be the test configuration of $(Y,(c+\epsilon)E)$ induced by $\lambda$.
 Thus 
 \[
  \Fut(\cY_\lambda,(c+\epsilon)\cE_\lambda;\cL_\lambda)=\mu^{M_1}(z,\lambda)-(c+\epsilon)\mu^{M_2}(z,\lambda)=-\epsilon\mu^{M_2}(z,\lambda)<0.
 \]
 This contradicts the assumption that $(Y,(c+\epsilon)E)$ is K-semistable.
 
 Next we want to show $(c+\epsilon)$-GIT semistability implies $(c+\epsilon)$-K-semistability in $U_W$.
 Suppose $z=\Hilb(Y,E)\in U_W$ is $(c+\epsilon)$-GIT semistable.
 Assume to the contrary that $(Y,(c+\epsilon)E)$ is K-unstable.
 From Theorem \ref{thm:bddtestK}, we know that to test
 K-(poly/semi)stablility of $(Y,(c+\epsilon)E)$
 it suffices to test all 1-PS in $\SL(N+1)$. 
 Hence there exists a 1-PS $\lambda:\bG_m\to \SL(N+1)$
 such that $\mu^{M_1}(z,\lambda)-(c+\epsilon)\mu^{M_2}(z,\lambda)<0$.
 Again from the discussion in Section \ref{sec:kstconst}, we may choose $\lambda$
 so that $\mu^{M_1}(z,\lambda)-c\mu^{M_2}(z,\lambda)=0$
 and $\epsilon\mu^{M_2}(z,\lambda)>0$.  Since $(Y,cE)$ is K-semistable, by \cite{LX14} we know that $\cY_\lambda$ is normal in codimension $1$ where $(\cY_\lambda,c\cE_\lambda;\cL_\lambda)$ is the test configuration of $(Y,cE)$ induced by $\lambda$. If $\cY_\lambda$ is not normal, then we may take its normalization and reembed it into $\bP^N$ without changing the generalized Futaki invariants. By doing so we may assume that $\lambda$ induces a normal test configuration of $(Y,cE)$.
 Denote by $z'=\Hilb(Y',E'):=\lim_{t\to 0}\lambda(t)\cdot z\in \oZ$.
 Since $(Y,cE)$ is K-semistable, Lemma \ref{lem:zerofut}.2 implies that
 $(Y',cE')$ is also a K-semistable log Fano pair.
 Let $z_1:=\Hilb(Y_1,E_1)$ be the $c$-GIT polystable degeneration of $z$
 in $U_W$. Hence $(Y_1, cE_1)$ is the K-polystable degeneration of $(Y,cE)$ by Theorem \ref{thm:localGIT}. From Theorem \ref{thm:bddtestK} we know that $(Y',cE')$ specially degenerates to a K-polystable log Fano pair $(Y_1',cE_1')$ in $\bP^N$ via a $1$-PS $\lambda'$ of $\SL(N+1)$. Since K-polystable degenerations of S-equivalent K-semistable log Fano pairs are isomorphic by \cite{LWX18}, we know that $(Y_1,cE_1)\cong (Y_1',cE_1')$. Hence there exists $g\in \SL(N+1)$ such that 
 \[
 \lim_{t\to 0}\lambda'(t)\cdot z'=g\cdot z_1=\Hilb(Y_1',cE_1').
 \]
 By Lemma \ref{lem:luna}, we know that $U=\SL(N+1)\cdot U_W$ is a Zariski open subset of $Z_c^{\red}$. Hence $z,z_1,g\cdot z_1,$ and $z'$ all belong to $U$. Hence 
 \[
  \epsilon\mu^{\psi^* M_2}((\id,z),\lambda)=
  \epsilon\mu^{M_2}(z,\lambda)>0.
 \]
 This implies that $(\id,z)$ is GIT unstable on $\SL(N+1)\times_{\Aut(X,D)}U_W$
 with respect to the $\SL(N+1)$-linearized $\bQ$-line bundle
 $\psi^*M_2^{\otimes -\epsilon}$. Since $\psi^*M_2$ is induced
 from the $\Aut(X,D)$-linearized $\bQ$-line bundle $M_2|_{U_W}$
 on $U_W$, we know that $z$ is $(c+\epsilon)$-GIT unstable in $U_W$.
 Thus we reach a contradiction. 
 
 Finally we show that the polystability conditions coincide. Denote by $U_{W,\epsilon}^{\mathrm{ss}}$ the $(c+\epsilon)$-GIT semistable locus in $U_W$, and $U_{\epsilon}^{\mathrm{ss}}$ the $(c+\epsilon)$-K-semistable locus in $U$. From Lemma \ref{lem:luna} and the discussion above, we know that $\psi: \SL(N+1)\times_{\Aut(X,D)}U_{W,\epsilon}^{\mathrm{ss}} \to U_{\epsilon}^{\mathrm{ss}}$ is a finite surjective strongly \'etale $\SL(N+1)$-morphism. We also see that $U$ is a saturated affine open subset of $Z_c^{\red}$ from the proof of Theorem \ref{thm:lwxlog}, hence $U_\epsilon^{\mathrm{ss}}$ is a saturated open subset of $Z_{c+\epsilon}^{\red}$. Thus a point $z=\Hilb(Y,E)\in U_{W,\epsilon}^{\mathrm{ss}}$ is $(c+\epsilon)$-GIT polystable if and only if its $\Aut(X,D)$-orbit is closed in $U_{W,\epsilon}^{\mathrm{ss}}$. This is equivalent to $\SL(N+1)\cdot z$ is closed in $U_{\epsilon}^{\mathrm{ss}}$ which is the same as saying $(Y,(c+\epsilon)E)$ is K-polystable by saturatedness of $U_\epsilon^{\mathrm{ss}}$. The proof is finished.
\end{proof}

\begin{thm}\label{thm:wallcrossingchart}
 Given any closed point $[z_0]$ in $\cK\cM_{\chi_0,r,c}$ and $0<\epsilon\ll 1$, after possibly shrinking $U_W$ there exist a local quotient representation $f:[U_W/G_{z_0}]\to \cK\cM_{\chi_0,r,c}$, a $G_{z_0}$-linearized line bundle $L_{z_0}$ on $U_W$, and a Cartesian diagram
 \[   
    \begin{tikzcd}
    {[U_W^-/G_{z_0}]}\arrow[hookrightarrow]{r}{}\arrow{d}{f^-} &
    {[U_W/G_{z_0}]}\arrow{d}{f}\arrow[hookleftarrow]{r}{} & {[U_W^+/G_{z_0}]}\arrow{d}{f^+}\\
    \cK\cM_{\chi_0,r,c-\epsilon}\arrow[hookrightarrow]{r}{\Phi^-}&
    \cK\cM_{\chi_0,r,c}\arrow[hookleftarrow]{r}{\Phi^+}& \cK\cM_{\chi_0,r,c+\epsilon}
    \end{tikzcd}
 \]
 such that the following are true:
 \begin{enumerate}
     \item The quotient stacks $[U_W^\pm/G_{z_0}]$ are the \emph{VGIT chambers} of $[U_W/G_{z_0}]$ with respect to $L_{z_0}$ (see \cite[Definition 2.4]{AFS17} for a definition). 
     \item All vertical arrows are finite strongly \'etale morphisms onto saturated open substacks of K-moduli stacks.
 \end{enumerate}
 In particular, we have a Cartesian diagram 
 \begin{equation} \label{eq:localVGIT}
    \begin{tikzcd}
    U_W^-\sslash G_{z_0}\arrow{r}{}\arrow{d}{f^-} &
    U_W\sslash G_{z_0}\arrow{d}{f} & \arrow{l}{} U_W^+\sslash G_{z_0}\arrow{d}{f^+}\\
    KM_{\chi_0,r,c-\epsilon}\arrow{r}{\phi^-}&
    KM_{\chi_0,r,c}& \arrow{l}[swap]{\phi^+} KM_{\chi_0,r,c+\epsilon}
    \end{tikzcd}
 \end{equation}
 where all vertical arrows are finite \'etale morphisms onto Zariski open subsets of K-moduli spaces, and all horizontal morphisms are projective.
\end{thm}

\begin{proof}
We first look at the K-moduli stack parts.
Let $U_W^{\pm}$ be the $(c\pm\epsilon)$-GIT semistable locus of $U_W$. By Lemma \ref{lem:luna}, we know that $f$ is a finite strongly \'etale morphism onto $[U/\SL(N+1)]$ which is a saturated open substack of $\cK\cM_{\chi_0,r,c}=[Z_c^{\red}/\SL(N+1)]$. The diagram is Cartesian by Theorem \ref{thm:localVGIT}. Hence $f^\pm$ are finite \'etale. From the proof of Theorem \ref{thm:localVGIT}, we know that $U^{\pm}:=\SL(N+1)\cdot U_W^{\pm}=Z_{c\pm\epsilon}^{\red}\cap U$. Hence their saturatedness follows from saturatedness of $U$.  The stabilizer preserving property of $f^\pm$ follows from \ref{lem:stabpreserve}. The morphisms $f^{\pm}$ also send closed point to closed point by Theorem \ref{thm:localVGIT}. Hence we finish proving part (2).
For part (1), we take $L_{z_0}:=\cO_{\oW}(1,1)|_{U_W}$ together with the $G_{z_0}$-linearization $\rho\otimes\rho_{(X,D)}^{-1}$ (see Definition \ref{defn:LWXLunaGIT}). After shrinking $U_W$ if necessary we may assume that $M_2|_{U_W}$ is trivial. Hence the $G_{z_0}$-representation $\rho_{M_2}^{-1}$ on $M_2^{-1}$ corresponds to a $G_{z_0}$-character $\chi_{L_{z_0}}:G_{z_0}\to \bG_m$. Then the VGIT chamber statement follows from Theorem \ref{thm:localVGIT}. For the K-moduli spaces statements, the Zariski open part is clear by definition of saturatedness, the finite part follows from the definition of good moduli spaces, and the strongly \'etale part follows from descent property (see \cite[Proposition 2.7]{AFS17}). 
\end{proof}

\begin{defn}
Let $\pi:(\cX,\cD)\to \oZ$ be the universal family over $\oZ$. Let $\pi_c:(\cX_c, \cD_c)\to Z_c^{\red}$ be the base change of $\pi$ to $Z_c^{\red}$. 
Let $0\leq c'< r^{-1}$ be another rational number.
Then the CM $\bQ$-line bundle $\lambda_{\CM,\pi_c,c'\cD_c}$ on $Z_c^{\red}$ descends to a $\bQ$-line bundle $\lambda_{c,c'}$ on $\cK\cM_{\chi_0,r,c}=[Z_c^{\red}/\PGL(N+1)]$. We call $\lambda_{c,c'}$ the \emph{CM $\bQ$-line bundle} on $\cK\cM_{\chi_0,r,c}$ with coefficient $c'$. We simply denote $\lambda_c:=\lambda_{c,c}$. We also denote $\lambda_{c,\Hodge}:=\lambda_{\Hodge, \pi_c, r^{-1} \cD}$. 
\end{defn}

\begin{prop}\label{prop:k-cm-interpolation}
With the above notation, the CM $\bQ$-line bundle $\lambda_c$ on  $\cK\cM_{\chi_0,r,c}$ descends to a $\bQ$-line bundle $\Lambda_{c}$ on the K-moduli space $KM_{\chi_0,r,c}$. If in addition $Z_{c}^{\red}=Z_{c\pm\epsilon}^{\red}$ for any $0<\epsilon\ll 1$, then the Hodge $\bQ$-line bundle $\lambda_{c,\Hodge}$ and the CM $\bQ$-line bundle $\lambda_{c,c'}$ on $\cK\cM_{\chi_0,r,c}$ descends to $\bQ$-line bundles $\Lambda_{c,\Hodge}$ and $\Lambda_{c,c'}$ on $KM_{\chi_0,r,c}$ for any $0\leq c'< r^{-1}$, respectively. Moreover, we have the following interpolation formula
\begin{equation}\label{eq:k-cm-interpolation}
(1-c'r)^{-n}\Lambda_{c,c'}=(1-c'r)\Lambda_{c,0}+c'r(n+1)(-K_X)^n \Lambda_{c,\Hodge}.
\end{equation}
\end{prop}

\begin{proof}
First we show the statements on descents.  Let $[z_0]=[\Hilb(X,D)]$ be any closed point of $\cK\cM_{\chi_0,r,c}$. By \cite[Theorem 10.3]{alper}, to show that $\lambda_{c,c'}$ on $\cK\cM_{\chi_0,r,c}$ can descend to $KM_{\chi_0,r,c}$ it suffices to show that the group of stabilizers $\Aut(X,D)$ acts trivially on $\lambda_{c,c'}^{\otimes k}$ for $k$ sufficiently divisible. Since $\Aut(X,D)$ is reductive, it suffices to show that  any $1$-PS $\sigma$ in $\Aut(X,D)$ acting on $\lambda_{c,c'}|_{[z_0]}$ has weight zero. By Proposition \ref{prop:Fut=CMwt} this weight is a non-zero multiple of the generalized Futaki invariant of the product test configuration $(\cX_{\sigma},c'\cD_{\sigma};\cL_{\sigma})$ of $(X,c'D)$ induced by $\sigma$. Since $(X,cD)$ is K-polystable, we know that the $\sigma$-weight of $\lambda_c$ at $[z_0]$ always vanishes. If in addition that $(X,(c\pm \epsilon)D)$ is K-semistable, then by the linearity of generalized Futaki invariants in terms of coefficients we know that $\Fut(\cX_{\sigma},c'\cD_{\sigma};\cL_{\sigma})=0$ for any $0\leq c'<r^{-1}$. By Proposition \ref{prop:cm-interpolation} we know that 
\[
(1-c'r)^{-n}\lambda_{c,c'}=(1-c'r)\lambda_{c,0}+c'r(n+1)(-K_X)^n \lambda_{c,\Hodge}.
\]
Hence $\lambda_{c,\Hodge}$ descends to $KM_{\chi_0,r,c}$ if $\lambda_{c,c'}$ descends for all $c'$. Thus the interpolation formula \eqref{eq:k-cm-interpolation} is just the descent of the above equation.
\end{proof}

\begin{thm}\label{thm:VGIT-CM-ample}
With the above notation, the $\bQ$-line bundles $\Lambda_{ c\pm \epsilon}$ on $KM_{\chi_0,r,c\pm \epsilon}$ are $\phi^{\pm}$-ample. In addition, we have $\lim_{\epsilon\to 0}\Lambda_{c\pm\epsilon}=\Lambda_{c\pm\epsilon,c}=(\phi^{\pm})^*\Lambda_c$.
\end{thm}

\begin{proof}
By Theorem \ref{thm:localVGIT}, it suffices to verify the statements over the local GIT chart $U_W\sslash G_{z_0}$ for each closed point $[z_0]\in \cK\cM_{\chi_0,r,c}$. 
Recall from \eqref{eq:M_i} that for the universal family $\pi:(\cX,\cD)\to \oZ$, we have a line bundle $\cL$ on $\cX$ as the pull back of $\cO_{\bP^N}(1)$ and $\bQ$-line bundles $M_{i}$ on $\oZ$ such that $\lambda_{\CM,\pi,c'\cD,\cL}=M_1-cM_2$. Let $\cL_c$ be the pull back of $\cL$ to $\cX_c$. Since 
\[
\cL_c\sim_{\bQ,\pi_c}-m K_{\cX_c/Z_c^{\red}}\sim_{\bQ,\pi_c} m(1-cr)^{-1} (-K_{\cX_c/Z_c^{\red}}-c\cD_c),
\]
we know that 
\[
m^{n}(1-cr)^{-n}\lambda_{\CM,\pi_c,c\cD_c}=\lambda_{\CM,\pi,c\cD,\cL}|_{Z_c^{\red}}=(M_1-cM_2)|_{Z_c^{\red}}.
\]
Thus $m^{n}(1-cr)^{-n}(f^\pm)^*\lambda_{c\pm\epsilon}$ on $[U_W^{\pm}/G_{z_0}]$ is the descent of $(M_1-(c\pm\epsilon)M_2)|_{U_{W}^{\pm}}$. By Definition \ref{defn:LWXLunaGIT} we know that the descent of $\pm M_2$ on $U_{W}^{\pm}\sslash G_{z_0}$ is anti-ample over $U_{W}\sslash G_{z_0}$, hence $\Lambda_{c\pm\epsilon}$ is $\phi^{\pm}$-ample. The second statement follows directly from the above computations.
\end{proof}

Finally, we are able to prove Theorem \ref{thm:logFano-wallcrossing} using the above results.

\begin{proof}[Proof of Theorem \ref{thm:logFano-wallcrossing}]
The statements follow from combining Proposition \ref{prop:k-wall-finite}, Theorems \ref{thm:lwxlog}, \ref{thm:localVGIT}, and \ref{thm:VGIT-CM-ample}.
\end{proof}

\section{General properties of K-moduli
of plane curves}\label{sec:generaldegree}

In this section we use results from Section \ref{sec:construction} to construct K-moduli stacks (resp. K-moduli spaces) of plane curves parametrizing K-semistable (resp. K-polystable) log Fano pairs $(X, cD)$ where $X$ is a $\bQ$-Gorenstein degeneration of $\bP^2$, the curve $D$ is a degeneration of smooth plane curves of degree $d$, and $0 < c < 3/d$. We also study properties of these K-moduli stacks and spaces.  We note that similar computations and comparisons of K-moduli to GIT were carried out in \cite{OSS16} in the case of low degree del Pezzo surfaces.

\subsection{Definition and properties}\label{sec:defmodulicurves}

We first recall the GIT moduli stacks and spaces of plane curves.

\begin{defn}
Let $d$ be a positive integer. Let $\bfP_d:=\bP\big(H^0(\bP^2,\cO_{\bP^2}(d))\big)$ be the projective space of dimension $\binom{d+2}{2}-1$ parametrizing all plane curves of degree $d$. It is clear that the natural $\PGL(3)$-action on $\bP^2$ lifts up to an action on $\bfP_d$.
\begin{enumerate}
    \item The line bundle $\cO_{\bfP_d}(1)$ has a unique $\SL(3)$-linearization. Let $\bfP_d^{\rm ss}$ be the GIT semistable locus of $\bfP_d$ with respect to the $\SL(3)$-linearized line bundle $\cO_{\bfP_d}(1)$. We define the \emph{GIT moduli stack} $\ocP_d^{\GIT}$ of  plane curves of degree $d$ to be $\ocP_d^{\GIT}:=[\bfP_d^{\rm ss}/\PGL(3)]$. We define the \emph{GIT moduli space} of  plane curves of degree $d$ to be $\oP_d^{\GIT}:=\bfP_d^{\rm ss}\sslash\SL(3)$.
    \item Let $\bfP_{d}^{\rm sm}$ be the Zariski open subset of $\bfP_d$ parametrizing smooth plane curves. Then the \emph{moduli stack} $\cP_d$ of smooth plane curves of degree $d$ is defined as $\cP_d:=[\bfP_{d}^{\rm sm}/\PGL(3)]$. When $d\geq 2$, it is clear from \cite[Chapter 4 \S 2]{MFK94} that $\bfP_d^{\rm sm}$ is a saturated Zariski open subset of $\bfP_d^{\rm ss}$. Hence $\cP_d$ admits a good moduli space $P_d$ as a Zariski open subset of $\oP_d^{\GIT}$. We call $P_d$ the \emph{moduli space} of smooth plane curves of degree $d$.
    Notice that when $d=1$, the GIT moduli space $\ocP_{d}^{\GIT}$ is empty, and $\cP_d$ does not admit a good moduli space due to the non-reductivity of stabilizers; when $d\geq 3$, the stack $\cP_d$ is Deligne-Mumford.
\end{enumerate}
\end{defn}

Now we begin with the definition of K-moduli stacks and spaces of plane curves. 
Recall that in Definition \ref{defn:kmoduli} we define K-moduli stacks and spaces of $\bQ$-Gorenstein smoothable log Fano pairs. In what follows, we adapt this definition to define K-moduli stacks and spaces of plane curves.

\begin{defn}
Let $d$ and $m$ be positive integers. Let $c\in (0, \min\{1,\frac{3}{d}\})$ be a rational number.  Denote by  $\chi(k):=\chi(\bP^2,\cO_{\bP^2}(3mk))$, $\tilde{\chi}(k)=\chi(\bP^2,\cO_{\bP^2}(3mk))-\chi(\bP^2, \cO_{\bP^2}(3mk-d))$, $\bm{\chi}:=(\chi,\tilde{\chi})$, and $N=h^0(\bP^2,\cO_{\bP^2}(3m))-1$. Let $\bH^{\bm{\chi};N}$ be the Hilbert schemes of pairs $(X,D)\hookrightarrow\bP^N$ of Hilbert polynomial $(\chi,\tilde{\chi})$. 

We define
\[
 Z:=\left\{\Hilb(X,D)\in\bH^{\bm{\chi};N}\left| \begin{array}{l} 
                           (X,D) \cong (\bP^2, C) \textrm{ where $C$ is a smooth plane curve of degree $d$,}\\     \cO_{\bP^N}(1)|_{X}\cong\cO_X(-mK_X),
                                         \\  \textrm{and }H^0(\bP^N,\cO_{\bP^N}(1))\xrightarrow{\cong}H^0(X,\cO_X(-mK_X)).
                                                \end{array}
 \right.\right\}.
\]
In other words, $Z$ parametrizes Hilbert points of $(3m)$-th Veronese embedding of $(\bP^2, C)$ into $\bP^N$. Then $Z$ is a locally closed subscheme of $\bH^{\bm{\chi};N}$. Let $\oZ$ be the Zariski closure of $Z$. We also define
\[
Z_c^{\circ}:=\left\{\Hilb(X,D)\in \overline{Z}\left| \begin{array}{l}
                                           X \textrm{ is a Manetti surface, }D\sim_{\bQ}-\frac{d}{3}K_X\textrm{ is an effective
                                           Weil divisor,}\\  (X,cD)\textrm{ is K-semistable},           \cO_{\bP^N}(1)|_{X}\cong\cO_X(-mK_X),
                            \\\textrm{and }H^0(\bP^N,\cO_{\bP^N}(1))\xrightarrow{\cong}H^0(X,\cO_X(-mK_X)).
                            \end{array}
 \right.\right\}.      
\]
By Corollary \ref{cor:openness}, we know that $Z^{\circ}_c$ is a Zariski open subset of $\oZ$. We denote by $Z_c^{\red}$ the reduced scheme supported on $Z_c^\circ$.

Assume $m$ is sufficiently divisible.
We define the \emph{K-moduli stack} $\ocP_{d,c}^{\K}$ of plane curves of degree $d$ with coefficient $c$ as the quotient stack
\[
\ocP_{d,c}^{\K}:=[Z_c^{\red}/\PGL(N+1)].
\]
By Theorem \ref{thm:stabilization}, we know that $\ocP_{d,c}^{\K}$ does not depend on the choice of sufficiently divisible $m$.
By Theorem \ref{thm:lwxlog}, we know that $\ocP_{d,c}^{\K}$ admits a good moduli space $\oP_{d,c}^{\K}$ as a reduced proper scheme of finite type over $\bC$. We call $\oP_{d,c}^{\K}$ the \emph{K-moduli space} of plane curves of degree $d$ with coefficient $c$.
\end{defn}

Indeed, if we denote by $\chi_0(k):=\chi(\bP^2,\cO(3k))$,  then $\ocP_{d,c}^{\K}=\cK\cM_{\chi_0, d/3, c}$ and $\oP_{d,c}^{\K}=KM_{\chi_0, d/3, c}$ as in Definition \ref{defn:kmoduli} since $\bP^2$ is the only smooth del Pezzo surface of degree $9$. From the definition, we also know that a pair $(X,cD)$ is parametrized by $\ocP_{d,c}^{\K}$ (resp. $\oP_{d,c}^{\K}$) if and only if $(X,cD)$ is K-semistable (resp. K-polystable), and it admits a $\bQ$-Gorenstein smoothing to $(\bP^2, cC_t)$ with $C_t$ a smooth plane curve of degree $d$.

The following useful proposition is an easy consequence of the Paul-Tian criterion Theorem \ref{thm:paultian}.

\begin{prop}\label{prop:P^2-paultian}
Let $C$ be a plane curve of degree $d$. Let $c\in (0, \min\{1,3/d\})$ be a rational number. 
If $(\bP^2, cC)$ is K-(poly/semi)stable, then $C$ is GIT (poly/semi)stable.
\end{prop}

\begin{proof}
Consider the universal family $\pi: (\bP^2\times\bfP_d, c\calC)\to \bfP_d$ of plane curves of degree $d$. From Theorem \ref{thm:paultian}, it suffices to show that the CM $\bQ$-line bundle $\lambda_{\CM, \pi, c\calC}$ is ample on $\bfP_d$. Here we use the intersection formula as in Proposition \ref{prop:logCM2}. It is clear that $-K_{\bP^2\times\bfP_d/\bfP_d}\sim \cO_{\bP^2\times\bfP_d}(3, 0)$ and $\calC\sim \cO_{\bP^2\times\bfP_d}(d, 1)$. Denote by $p:\bP^2\times\bfP_d\to\bP^2$ the projection to the first component. By computation,
\begin{align*}
    \lambda_{\CM, \pi, c\calC}& = -\pi_*\big((-K_{\bP^2\times\bfP_d/\bfP_d}-c\calC)^3\big)
    \qquad =-\pi_*\big(\cO_{\bP^2\times\bfP_d}(3-cd,-c)^3)\\
    & =-\pi_*\big( p^*\cO(3-cd)^3+ 3p^*\cO(3-cd)^2\cdot \pi^*\cO(-c)+ 3 p^*\cO(3-cd)\cdot \pi^*\cO(-c)^2+\pi^*\cO(-c)^3\big)\\
    & = 3(3-cd)^2 c\cO_{\bfP_d}(1).
\end{align*}
Hence $\lambda_{\CM, \pi, c\calC}$ is ample whenever $c\in (0, \frac{3}{d})$. The proof is finished.
\end{proof}

The following corollary was proved by Hacking \cite[Propositions 10.2 and 10.4]{Hac04} and Kim and Lee \cite[Theorem 2.3]{KL04}. We give a proof using K-stability and CM line bundles.

\begin{cor}
Let $C$ be a plane curve of degree $d$. If $\lct(\bP^2;C)\geq \frac{3}{d}$ (resp. $>\frac{3}{d}$), then $C$ is GIT semistable (resp. GIT stable).
\end{cor}

\begin{proof}
If $\lct(\bP^2;C)\geq\frac{3}{d}$, then the log Calabi-Yau pair $(\bP^2,\frac{3}{d}C)$ is K-semistable. Hence $(\bP^2, cC)$ is K-semistable for any $c\in (0,\frac{3}{d})$ by Proposition \ref{prop:k-interpolation}. Thus Proposition \ref{prop:P^2-paultian} implies that $C$ is GIT semistable. If $\lct(\bP^2;C)>\frac{3}{d}$, then again by Proposition \ref{prop:k-interpolation} we know that $(\bP^2, cC)$ is uniformly K-stable for any $c\in (0,\frac{3}{d})$. Hence $C$ is GIT stable by Proposition \ref{prop:P^2-paultian}.
\end{proof}

\begin{expl}\label{expl:deg123}
We summarize the description of K-moduli stacks and spaces for $d\leq 3$.
\begin{enumerate}
    \item $d=1$. In this case we know that $(\bP^2, cC)$ is K-unstable for $C$ a line and any $c\in (0, 1)$ by \cite[Example 3.16]{LS14}. Hence $Z\cap Z_c^{\circ}$ is empty. Since K-semistability is an open property by Corollary \ref{cor:openness}, we know $Z_c^{\circ}=\emptyset$. Hence both $\ocP_{1,c}^{\K}$ and  $\oP_{1,c}^{\K}$ are empty for any $c\in (0,1)$.
    \item $d=2$. Denote by $C$ a smooth plane conic curve.
    \begin{enumerate}
        \item If $c\in (0, \frac{3}{4})$, by \cite[Theorem 1.5]{LS14} we know that $(\bP^2,cC)$ is K-polystable. By Proposition \ref{prop:P^2-paultian}, we know that $(\bP^2,cC')$ is K-unstable for any singular plane conic curve $C'$. Thus the only K-semistable point in $\ocP_{2,c}^{\K}$ is $[(\bP^2, cC)]$ which is indeed K-polystable. Hence $\ocP_{2,c}^{\K}\cong\ocP_2^{\GIT}\cong [\Spec\bC/\PGL(2)]$ and $\oP_{2,c}^{\K}\cong\oP_2^{\GIT}\cong \Spec\bC$.
        \item If $c=\frac{3}{4}$, then by \cite[Proof of Theorem 1.5]{LS14} we know that $(\bP^2, \frac{3}{4}C)$ is K-semistable and admits a special degeneration to the K-polystable pair $(\bP(1,1,4), \frac{3}{4}C_0)$ where $C_0=(z=0)$ with $[x,y,z]$ the projective coordinates of $\bP(1,1,4)$. If $[(X,\frac{3}{4}D)]$ is a K-semistable point in $\ocP_{2,3/4}^{\K}$ with $X$ non-smooth, then by \cite{LWX18} it admits a special degeneration to $(\bP(1,1,4), \frac{3}{4}C_0)$. Hence the Gorenstein index of $X$ is $2$ and $D$ is smooth which implies that $(X,D)\cong (\bP(1,1,4),C_0)$ by \cite[Theorem 8.3]{Hac04}. Thus there are only two K-semistable points: $[(\bP^2, \frac{3}{4}C)]$ and $[(\bP(1,1,4),\frac{3}{4}C_0)]$ where the latter one is the only K-polystable point.
        \item If $c>\frac{3}{4}$, then by \cite[Example 3.16]{LS14} we know $(\bP^2, cC)$ is K-unstable. Hence similarly as in (1) both $\ocP_{2,c}^{\K}$ and  $\oP_{2,c}^{\K}$ are empty for any $c\in (\frac{3}{4},1)$. 
    \end{enumerate}
    \item $d=3$. We will show $\ocP_{3,c}^{\K}\cong\ocP_3^{\GIT}$ for any $c\in (0,1)$. From \cite[Page 80]{MFK94} we know that a plane cubic curve $C$ is GIT semistable (resp. stable) if and only if it has at worst nodal singularities (resp. smooth). Thus we know $(\bP^2, C)$ is a log canonical log Calabi-Yau pair whenever $C$ is GIT semistable. Then we know by Proposition \ref{prop:k-interpolation} that $(\bP^2, cC)$ is K-semistable for any $c\in (0,1)$. If $C$ is GIT stable i.e. smooth, then \cite{JMR16} implies $(\bP^2,(1-\epsilon)C)$ is K-polystable. Hence $(\bP^2, cC)$ is K-stable for any $c\in (0,1)$ by Proposition \ref{prop:k-interpolation}.
    It is well known that $(xyz=0)$ is the unique GIT polystable plane cubic curve up to a projective transformation. By \cite{Fuj17b}, we know that $(\bP^2, c(xyz=0))$ is K-polystable for any $c\in (0,1)$. 
    If $[(X,cD)]$ is a point in $\ocP_{3,c}^{\K}$, then it is a K-semistable limit of K-semistable log Fano pairs $(\bP^2, cC_t)$ where $\{C_t\}_{t\in T\setminus\{0\}}$ is a family of cubic curves over a punctured smooth curve $T\setminus\{0\}$. Since $C_t$ is GIT semistable by Proposition \ref{prop:P^2-paultian}, we know that (possibly after a finite base change of $T$) there exists an algebraic family $g_t\in \PGL(3)$ and a GIT polystable plane cubic curve $C_0$ such that $g_t\cdot C_t\to C_0$ in $\bfP_3$. Therefore, $(\bP^2, cC_0)$ is the K-polystable limit of $(\bP^2,cC_t)$. By \cite{BX18} we know that $(X,cD)$ admits a special degeneration to $(\bP^2, cC_0)$ which implies $X\cong \bP^2$ and $D$ is GIT semistable. Hence by similar arguments to the last paragraph in the proof of Theorem \ref{thm:firstwallbefore}, we have $\ocP_{3,c}^{\K}\cong\ocP_3^{\GIT}$ and hence $\oP_{3,c}^{\K}\cong\oP_3^{\GIT}$.
\end{enumerate}
\end{expl}

From now on, we will always assume $d\geq 4$. We now mention some basic properties satisfied by the loci $Z, Z_c^\circ$ as well as the K-moduli stacks and spaces we just defined.

\begin{prop}\label{prop:Zplanecurve}
With notation as above, the following properties hold for any $c\in (0, \frac{3}{d})$:
\begin{enumerate}
 \item $\Hilb(X,D)\in Z$ if and only if 
 $(X,D)$ is isomorphic to $(\bP^2,C)$ where $C$ is a smooth plane curve
 of degree $d$. Moreover, the locus $Z$ is a saturated open subset of $Z_c^{\circ}$.
 \item The locus $Z_c^{\circ}$ is smooth for any $c\in (0, \frac{3}{d})$. In particular, $Z_c^\circ= Z_c^\red$, and $\ocP_{d,c}^{\K}$ is a smooth Artin stack. 
 \item  The open immersion $\cP_{d}\hookrightarrow\ocP_{d,c}^{\K}$
 induces an open immersion between their good moduli spaces
 $P_{d}\hookrightarrow\oP_{d,c}^{\K}$.
 Furthermore $\oP^{\K}_{d,c}$ is a normal proper variety, and $P_{d}$ has only quotient singularities.
\end{enumerate} 
\end{prop}

\begin{proof} For part (1), it suffices to show that $(\bP^2,cC)$ is K-stable for any $c\in (0, 3/d)$ and any smooth plane curve $C$ of degree $d\geq 4$. This follows from Proposition \ref{prop:k-interpolation} since $(\bP^2, \frac{3}{d}C)$ is klt. Since $Z$ lies inside the K-stable locus of $Z_c^\circ$, we know that $Z$ is saturated in $Z_c^\circ$ by the uniqueness of K-polystable degeneration \cite{LWX18}.

For part (2), recall that for any point $\Hilb(X,D)\in Z_c^{\circ}$ the surface $X$ is a Manetti surface. Hence $X$ has unobstructed $\Q$-Gorenstein deformations by Proposition \ref{prop:manetti}.
Since K-semistability is an open condition by Corollary \ref{cor:openness}, it suffices to show that 
the $\Q$-Gorenstein deformations of the pair $(X,D)$ are also unobstructed.
Let $\pi:(\cX,\cD)\to T$ be a $\bQ$-Gorenstein smoothing of $(X,D)$ over a smooth pointed curve $0\in T$, i.e. $(\cX_0,\cD_0)\cong (X,D)$ and $\pi$ is smooth over $T^\circ:=T\setminus\{0\}$. Denote by $(\cX^\circ,\cD^\circ):=\pi^{-1}(T^\circ)$. Then it is clear that $dK_{\cX^\circ}+3\cD^{\circ}\sim_\pi 0$. By taking Zariski closure, we know that $dK_{\cX}+3\cD\sim_\pi 0$ which implies $dK_X+3D\sim 0$ by adjunction. In addition, if $3\mid d$ then we get $\frac{d}{3}K_X+D\sim 0$. Hence the statement of \cite[Lemma 3.13]{Hac04} holds for $(X,D)$. Since $X$ is klt and $D-K_X\sim_{\bQ}-\frac{d-3}{3}K_X$ is ample, Kawamata-Viehweg vanishing implies $H^1(X, \calO_X(D)) = 0$. Hence the statement of \cite[Lemma 3.14]{Hac04} also holds for $(X,D)$. Therefore, we may apply \cite[Theorem 3.12]{Hac04} to deduce that $(X,D)$ has unobstructed $\bQ$-Gorenstein deformations.

For part (3), the first statement follows from part (1) and \cite[Remark 6.2]{alper}. The normality of $\oP_{d,c}^{\K}$ follows from part (2) and a result of Alper \cite[Theorem 4.16 (viii)]{alper}. Since any smooth plane curve of degree $d\geq 4$ has finite automorphism group, we know that $\cP_d$ is a smooth Deligne-Mumford stack. Hence $P_d$ has only quotient singularities.
\end{proof}




There are certain open
subsets of $\oP^{\K}_{d,c}$ that remain
unchanged under subsequential wall crossings.
Let $P_{d,c}^{\klt}$ and $P_{d,c}^{\lc}$ be the subsets
of $\oP^{\K}_{d,c}$ parametrizing $c$-K-polystable curves
with $\lct>\frac{3}{d}$ and $\lct\geq \frac{3}{d}$, respectively. By the constructibility and lower semi-continuity of log canonical thresholds in bounded families, we know that both $P_{d,c}^{\klt}$ and $P_{d,c}^{\lc}$ are Zariski open subsets of $\oP_{d,c}^{\K}$. Denote by $\cP_{d,c}^{\klt}$ and $\cP_{d,c}^{\lc}$ the preimage of $P_{d,c}^{\klt}$ and $P_{d,c}^{\lc}$ under the quotient map $\ocP_{d,c}^{\K}\to \oP_{d,c}^{\K}$, respectively. 

\begin{prop}\label{prop:lclastwall}\leavevmode
\begin{enumerate}
\item There exist open immersions
 $\cP_{d,c}^{\klt}\hookrightarrow
 \cP_{d,c}^{\lc}\hookrightarrow \ocP^{\K}_{d,c'}$ which descend to open immersions
$P_{d,c}^{\klt}\hookrightarrow
 P_{d,c}^{\lc}\hookrightarrow \oP^{\K}_{d,c'}$ for any $0<c\leq c'<3/d$.
  Moreover, there exists an open immersion $P_{d,c}^{\klt}\hookrightarrow
 \oP_d^{\rm H}$ for any $c\in(0,3/d)$.
\item  Assume $c_0\in (0, 3/d)$ satisfies the following: for any K-polystable point $[(X,c_0 C)]\in
 \ocP^\K_{d,c_0}$  we have $\lct(X;C)\geq 3/d$ (or equivalently,
 $P_{d,c_0}^{\lc}=\oP^\K_{d,c_0}$). Then 
 $\ocP^\K_{d,c}\cong \ocP^\K_{d,c_0}$  for any $c_0<c<3/d$. In other
 words, there are no wall crossings among K-moduli spaces
 in the region $c\in (c_0-\epsilon,3/d)$ for $0<\epsilon\ll 1$.
 \end{enumerate}
\end{prop}

\begin{proof}
 For part (1), if $(X,cD)$ is log K-semistable and $\lct(X;D)\geq \frac{3}{d}$, then $(X,c'D)$ is K-semistable for any $c\leq c'<\frac{3}{d}$ by Proposition \ref{prop:k-interpolation}. Hence we have open immersions $\cP_{d,c}^{\klt}\hookrightarrow
 \cP_{d,c}^{\lc}\hookrightarrow \ocP^{\K}_{d,c'}$. To show that they descend to open immersions among the good moduli spaces, it suffices to show that the larger open substack $\cP_{d,c}^{\lc}$ is a saturated open substack of $\ocP_{d,c'}^{\K}$. Let $[(X,cD)]\in \cP_{d,c}^{\lc}$ be a point. Then $(X,cD)$ admits a K-polystable degeneration $(X_0,c D_0)$ in $\cP_{d,c}^{\lc}$ such that $\lct(X_0,D_0)\geq \frac{3}{d}$. Thus by Proposition \ref{prop:k-interpolation} we know that $(X_0,c'D_0)$ is K-polystable. Hence $\cP_{d,c}^{\lc}$ is saturated in $\ocP_{d,c'}^{\K}$. 
 
 By definition we know that $P_{d,c}^{\klt}$ admits an injective map to $\oP_{d}^{\rm H}$. To show that $P_{d,c}^{\klt}$ admits an open immersion to $\oP_d^{\rm H}$, by \cite{BL18b} it suffices to show that a Hacking stable pair $(X,D)$ belongs to $P_{d,c}^{\klt}$ if and only if it is uniformly $c$-K-stable. The ``if'' part is clear from the definition. For the ``only if'' part, if  $(X,D)$ is both Hacking stable and $c$-K-polystable, then by Theorem \ref{thm:bddtestK} it admits a weak conical K\"ahler-Einstein metric, and its automorphism group is finite. Hence $(X,cD)$ is uniformly K-stable by Theorem \ref{thm:bddtestK}(4). This finishes the proof of part (1).
 
 For part (2), notice from part (1) that $\cP_{d,c_0}^{\lc}$ is a saturated open substack of $\ocP_{d,c}^{\K}$ for $c\in (c_0, \frac{3}{d})$ which induces an open immersion $\varphi:\oP_{d,c_0}^{\K}=P_{d,c_0}^{\lc}\hookrightarrow\oP_{d,c}^{\K}$. Since the K-moduli spaces are normal proper varieties by Proposition \ref{prop:Zplanecurve}, we know that $\varphi$ is an isomorphism by \cite[Proposition 6.4]{alper}. Hence $\ocP_{d,c_0}^{\K}=\cP_{d,c_0}^{\lc}\cong \ocP_{d,c}^{\K}$ whenever $c\in (c_0,\frac{3}{d})$. The $c\in (c_0-\epsilon,c_0)$ part follows from Proposition \ref{prop:openkps}.
\end{proof}

\subsection{Index bounds}
In this section we prove the following theorem on bounding local Gorenstein indices of singular surfaces appearing in the boundary of K-moduli spaces. It is a K-stability analogue of Hacking's result \cite[Theorem 4.5]{Hac04} and \cite[Theorem 2.22]{Hac01}. As in Hacking's work, it is crucial in the study of singular objects in K-moduli spaces of plane curves.
\begin{theorem}\label{thm:localindex} Let $(X,cD)
 $ be a K-semistable log Fano pair that admits a $\bQ$-Gorenstein smoothing
 to $(\bP^2, cC_t)$ with $c\in (0,3/d)$ and $\deg C_t=d$.
 Let $x\in X$ be any singular point with Gorenstein index $\ind(x,K_X)$, then 
 \[
  \ind(x,K_X)\leq\begin{cases}
                  \min\{\lfloor\frac{3}{3-cd}\rfloor,d\} & \textrm{ if }3\nmid d,\\
                  \min\{\lfloor\frac{3}{3-cd}\rfloor,\frac{2d}{3}\} & \textrm{ if }3\mid d. 
                 \end{cases}
 \]
\end{theorem}

\begin{proof}
 Let $\beta:=1-cd/3\in (0,1)$.
 By \cite[Propositions 6.1, 6.2, \& Theorem 7.1]{Hac04} we know that a Gorenstein index $n$ point $x\in X$ is a  cyclic quotient singularity of type $\frac{1}{n^2}(1,an-1)$
 where $\gcd(a,n)=1$ and $3\nmid n$.
 
 We first show that $n \leq \lfloor\frac{3}{3-cd} \rfloor$. By Theorem \ref{thm:local-vol-global}, we know that 
 \[
 \hvol(x, X, cD)\geq \frac{4}{9} (-K_X-cD)^2 = 4 \beta^2.
 \]
 On the other hand, we have $\hvol(x,X, cD) \leq \hvol(x, X) = \frac{4}{n^2}$  by \cite[Proposition 4.10]{LL16}. Combining these two inequalities, we get $n\leq \lfloor \beta^{-1}\rfloor =   \lfloor\frac{3}{3-cd} \rfloor$.
 
 Next, we show the inequality $n\leq d$ or $\frac{2d}{3}$ depending on divisibility of $d$ by $3$.
 We know that $dK_X+3D\sim 0$, so if $x\not\in D$ then
 $n\mid d$ hence $n\leq d$ (in fact $n\leq d/3$ if
 $3\mid d$). From now on let us assume
 $x\in D$. Let $(\tilde{x}\in \tilde{X})$ be the 
 smooth cover of $(x\in X)$, with $\tilde{D}$ being the 
 preimage of $D$. Since the finite degree formula for local volumes is true in dimension $2$ by \cite[Theorem 2.7(3)]{LX17b} \cite[Theorem 4.15]{LLX18}, we have 
 \[
  \hvol(\tilde{x},\tilde{X},c\tilde{D})
  =n^2\cdot\hvol(x,X,c D).
  \]
  On the other hand, Theorem \ref{thm:local-vol-global} implies that
  $\hvol(x,X,c D)\geq 4\beta^2$,
  so we have
  \begin{equation}\label{eq_hvolestimate}
   n\leq \frac{\sqrt{  \hvol(\tilde{x},\tilde{X},c\tilde{D})}}{2\beta}
   \leq \frac{2-c\ord_{\tilde{x}}\tilde{D}}{2\beta}.
  \end{equation}
  In particular we have $n<\beta^{-1}$.
  We know that $\lct_{\tilde{x}}(\tilde{X};\tilde{D})>
  c$, and Skoda's estimate \cite{skoda} implies $\lct_{\tilde{x}}(\tilde{X};\tilde{D})\leq
  \frac{2}{\ord_{\tilde{x}}\tilde{D}}$, so we have
  $  \ord_{\tilde{x}}\tilde{D}<\frac{2}{c}$.
  Assume $\tilde{x}\in\tilde{X}$ has local coordinates
  $(u,v)$ where the cyclic group action is scaling on 
  each coordinate. Let $u^i v^j$ be a monomial appeared
  in the equation on $\tilde{D}$ with minimal $i+j=\ord_{\tilde{x}}\tilde{D}$.
  Then $dK_X+3D\sim 0$ implies $3(i+(na-1)j)\equiv dna\mod n^2$,
  in particular $i\equiv j\mod n$.
  
  \textbf{Case 1.} Assume $3\nmid d$. If $\beta\geq \frac{1}{d+1}$
  then $n<\beta^{-1}\leq d+1$. Thus we may assume
  $\beta<\frac{1}{d+1}$. Then 
  $ i+j=  \ord_{\tilde{x}}\tilde{D}<\frac{2}{c}<\frac{2(d+1)}{3}$.
  Assume to the contrary that $n\geq d+1$. Then $i\equiv j\mod
  n$ and $i+j<n$ implies that $i=j$. Hence $3(i+(na-1)j)\equiv dna\mod n^2$
  implies $3i\equiv d\mod n$. But since $i<\frac{d+1}{3}$,
  we know that $3i=d$ which is a contradiction.
  
  \textbf{Case 2.} Assume $3\mid d$. If $\beta\geq \frac{3}{2d+3}$
  then $n<\beta^{-1}\leq \frac{2d}{3}+1$. Thus we may
  assume $\beta<\frac{3}{2d+3}$. Then $      i+j=  \ord_{\tilde{x}}\tilde{D}<\frac{2}{c}
      <\frac{2d}{3}+1$.
 Assume to the contrary that $n\geq \frac{2d}{3}+1$.
Then $i\equiv j\mod n$ and $i+j<n$ implies $i=j$.
Hence $3(i+(na-1)j)\equiv dna\mod n^2$
  implies $3i\equiv d\mod n$. Hence $i=j=\frac{d}{3}$
  and $\ord_{\tilde{x}}\tilde{D}=\frac{2d}{3}$.
  Then \eqref{eq_hvolestimate} implies
  \[
   n\leq \frac{2-c\cdot\frac{2d}{3}}{2\beta}=\frac{2-2(1-\beta)}{2\beta}=1,
  \]
a contradiction!
\end{proof}

\section{The first wall crossing}\label{sec:firstwall}
 The goal of this section is to prove Theorem \ref{mthm:firstwall}, which completely describes the first wall crossing of K-moduli spaces of plane curves for all degrees. We show that K-moduli and GIT coincide for small weights (see Theorems \ref{thm:firstwallbefore} and \ref{thm:firstwallon}) and describe the explicit birational modification on the GIT moduli space occurring while crossing the first wall (see Theorem \ref{thm:firstwallafter}).

\subsection{Before the first wall}\label{sec:firstwall1}
 In this section,  we will show that the K-moduli space for small coefficient is isomorphic to the GIT moduli space. We prove two results, Theorems \ref{thm:firstwallbefore} and \ref{thm:firstwallon}, which correspond to parts (1) and (2) of Theorem \ref{mthm:firstwall}. 
 
 Before we start, let us fix some notation for the discussion of the first wall crossing.
 
\begin{notation}\label{not:firstwall} Let $d\geq 4$ be an integer. Let $c\in (0, \frac{3}{d})$ be a rational number. Let  $Q$ be a smooth conic in $\bP^2$, let $L$ be a line in $\bP^2$ transverse to $Q$, and let $x,y,z$ be coordinates of $\bP(1,1,4)$. Let \[
c_1 = \left\{ \begin{array}{lr} \frac{3}{2d} & d \textrm{ is even} \\ \frac{3}{2d-3} & d \textrm{ is odd }  \end{array}\right. \qquad
Q_d = \left\{ \begin{array}{lr} \frac{d}{2}Q & d \textrm{ is even} \\ \frac{d-1}{2}Q+L & d \textrm{ is odd }\end{array}\right. \quad
Q'_d = \left\{ \begin{array}{lr} z^{d/2} = 0 & d \textrm{ is even} \\ xyz^{(d-1)/2}=0  & d \textrm{ is odd }\end{array}\right. \qquad
\]
\end{notation}

We are ready to prove part (1) of Theorem \ref{mthm:firstwall}. 

\begin{theorem}[First wall crossing 1]\label{thm:firstwallbefore} We follow Notation \ref{not:firstwall}.
For any $0 < c < c_1$,  a plane curve $C$ of degree $d$ is GIT (poly/semi)stable if and only if the log Fano pair $(\bP^2, cC)$ is K-(poly/semi)stable. Moreover, there is an isomorphism of Artin stacks $\ocP^\K_{d,c} \cong \ocP_d^{\GIT}$. 
\end{theorem}

\begin{proof}
 We first show that if $(X,cD)$ is a K-semistable point in  $\ocP_{d,c}^{\K}$ for $0<c<c_1$ then $X\cong\bP^2$.  
 From Theorem \ref{thm:localindex}, we know that the local Gorenstein indices of $X$ are at most $\lfloor\beta^{-1}\rfloor$ where $\beta=1-\frac{cd}{3}$.
 If $c<\frac{3}{2d}$, then we have $\beta>\frac{1}{2}$. This implies 
 $X$ is a Gorenstein Manetti surface so $X\cong\bP^2$. Hence we may assume that $d\geq 5$ is odd and $\frac{3}{2d}\leq c<\frac{3}{2d-3}$. By the same argument as above,
 the local Gorenstein indices of $X$ are at most
 $\lfloor\beta^{-1}\rfloor<\frac{2d-3}{d-3}\leq \frac{7}{2}$.
 Hence $X$ has local Gorenstein index at most $3$, which implies
 $X\cong \bP^2$ or $\bP(1,1,4)$. We shall show that the $\bP(1,1,4)$
 case is impossible under the assumption $c<\frac{3}{2d-3}$.
 
 Assume to the contrary that $(X=\bP(1,1,4),cD)$ is a K-semistable
 point in $\ocP^\K_{d,c}$. Then $D$ is of degree $2d$
 in $\bP(1,1,4)$. Write $d=2l+1$, then the equation
 of $D$ is \[z^l f_2(x,y)+z^{l-1}f_6(x,y)+\cdots +f_{4l+2}(x,y)=0,\]
 where $f_i$ is a homogeneous polynomial of degree $i$ in $(x,y)$.
 Let $E$ be the $(-4)$-curve over the singular point
 $[0,0,1]$ of type $\frac{1}{4}(1,1)$. Then from the defining equation of $D$ we see that $\ord_E(D)\geq \frac{1}{2}$. Thus 
 \[
  A_{(X,cD)}(\ord_E)=A_X(\ord_E)-c~\ord_E(D)\leq\frac{1-c}{2}.
 \]
 On the other hand, $-K_X-cD\sim_{\bQ}\cO(6-2dc)$, 
 and $\vol_X(\cO(1)-tE)=\max\{\frac{1}{4}-4t^2,0\}$.
 Hence
 \[
  S_{(X,cD)}(\ord_E)=\frac{(6-2dc)}{\vol_X(\cO(1))}\int_0^{\infty}
  \vol_X(\cO(1)-tE)dt=1-\frac{dc}{3}.
 \]
 Since $c<\frac{3}{2d-3}$, we know that $A_{(X,cD)}(\ord_E)\leq \frac{1-c}{2}<1-\frac{dc}{3}=S_{(X,cD)}(\ord_E)$.  Hence  $(X,cD)$ is K-unstable by the valuative criterion (Theorem \ref{thm:valuative}).
 
 So far we have shown that any K-semistable point $(X,cD)$ in $\ocP_{d,c}^{\K}$ is isomorphic to $(\bP^2,cC)$ where $C$ is a plane curve of degree $d$. By Proposition \ref{prop:P^2-paultian} we know that K-(poly/semi)stability of $(\bP^2, cC)$ implies GIT (poly/semi)stability of $C$. Hence we just need to show the converse to deduce the equivalence between K-stability and GIT stability. Suppose $C$ is a GIT semistable plane curve. Take $\{C_t\}_{t\in T}$ a family of plane curves over a smooth pointed curve $(0\in T)$ such that $C_0=C$ and $C_t$ is smooth for $t\in T\setminus\{0\}$. Then by properness of K-moduli spaces (Theorem \ref{thm:compactness}) we have a K-polystable limit $(X,cD)$ of $(\bP^2, cC_t)$ as $t\to 0$ after a possible finite base change of $T$. Hence $(X,cD)\cong (\bP^2, C_0')$ where $C_0'$ is a GIT polystable plane curve. By the separatedness of GIT quotients, we know that $C$ specially degenerates to $g\cdot C_0$ for some $g\in\PGL(3)$. Thus $(\bP^2,cC)$ is K-semistable by Theorem \ref{thm:Kss-spdeg}. If in addition that $C$ is GIT polystable, then $(\bP^2,cC)$ has a K-polystable limit $(\bP^2, cC_0')$. In particular, by Proposition \ref{prop:P^2-paultian} we know that $C_0'$ is a GIT polystable point S-equivalent to $C$. Hence $C=g\cdot C_0$ for some $g\in\PGL(3)$ and $(\bP^2,cC)$ is K-polystable.
 
 From the equivalence between $c$-K-semistability and GIT semistability, we obtain a morphism of Artin stacks $\varphi:\ocP_d^{\GIT}\to \ocP_{d,c}^{\K}$. It suffices to show that $\varphi$ is an isomorphism of stacks. 
 Consider the morphism of Artin stacks $\psi:\ocP_{d,c}^{\K} \to B \PGL(3)$ sending $[(\cX, \cD) \to S]$ to  $[\cX\to S]$, where $B \PGL(3)$ is the classifying stack of $\bP^2$-bundles. Clearly, $\psi$ is representable as the group homomorphism $\Aut(X,D)\to \Aut(X)$ is injective for $[(X,D)]\in \ocP_{d,c}^{\K}$. We look at the base change of $\varphi$ under the natural quotient map $\Spec\bC \to B \PGL(3)$, in which we obtain a $\PGL(3)$-equivariant morphism of algebraic spaces $\tilde{\varphi}: \bfP_d^{\rm ss} \to Z$ where $Z = \ocP_{d,c}^{\K} \times_{B \PGL(3)} \Spec \bC$. Thus for any reduced scheme $S$ of finite type over $\bC$, the set $Z(S)$ is given by $\{(\cX,\cD;f)\}/\cong$ where $[(\cX, c\cD) \to S]\in \ocP_{d,c}^{\rm K}$ and $f:\cX\to \bP^2_S$ is an isomorphism. From the equivalence between K-semistability and GIT semistability, we know that $\tilde{\varphi}(S)$ is a bijection for every reduced $S$, which implies that $\varphi$ is an isomorphism between reduced algebraic spaces. Hence $\varphi$ is an isomorphism between Artin stacks.
 This finishes the proof.
\end{proof}

Next we discuss the K-moduli stack and space when $c=c_1$.

\begin{lem}\label{lem:Q_d'Kps}
 We follow Notation \ref{not:firstwall}. Then the log Fano pair $(\bP^2, c_1 Q_d)$ is K-semistable with K-polystable degeneration $(\bP(1,1,4),c_1 Q_d')$. Moreover, the only $c_1$-K-polystable curve on $\bP(1,1,4)$ of degree $2d$ is $Q_d'$. 
\end{lem}

\begin{proof}
 By taking the degeneration of $\bP^2$ to the normal cone of $Q$, the pair $(\bP^2, Q)$ specially degenerates to $(\bP(1,1,4), (z=0))$. Since $L$ intersects $Q$ transversally, it degenerates to the union of two distinct rulings of $\bP(1,1,4)$. Hence a suitable choice of projective coordinates of $\bP(1,1,4)$ yields that $(\bP^2,Q,L)$ specially degenerates to $(\bP(1,1,4), (z=0), (xy=0))$. 
 
 Next we show that $(\bP(1,1,4), c_1 Q_d')$ is K-polystable. When $d$ is even, we have $c_1 Q_d'=\frac{3}{4} (z=0)$. Hence by \cite[Proof of Theorem 1.5]{LS14} we know that $(\bP(1,1,4),\frac{3}{4} (z=0))$ is K-polystable. When $d$ is odd, we have $c_1 Q_d'=\frac{3}{2d-3}(xy=0)+\frac{3(d-1)}{2(2d-3)}(z=0)$.  Hence $(\bP(1,1,4), c_1 Q_d')$ is a projective cone over $(\bP^1, \frac{3}{2d-3}([0]+[\infty]))$ with polarization $\cO_{\bP^1}(4)\sim_{\bQ}\frac{2d-3}{d-3}(-K_{\bP^1}-\frac{3}{2d-3}([0]+[\infty]))$. Since $\frac{3(d-1)}{2(2d-3)}=1- \frac{1}{2}\cdot (\frac{2d-3}{d-3})^{-1}$ and $(\bP^1, \frac{3}{2d-3}([0]+[\infty]))$ admits a conical K\"ahler-Einstein metric, by \cite[Proposition 3.3]{LL16} we know that 
 $(\bP(1,1,4),c_1 Q_d')$ is conical K\"ahler-Einstein hence K-polystable. 
 
 Finally, we show that $Q_d'$ is the only $c_1$-K-polystable curve on $\bP(1,1,4)$ of degree $2d$. Suppose $(X:=\bP(1,1,4),c_1 D)$ is K-polystable with $\deg D=2d$. 
 Let $E$ be the $(-4)$-curve over the singular point $x:=[0,0,1]$ of type $\frac{1}{4}(1,1)$. Then by Theorem \ref{thm:local-vol-global} we have
 \begin{equation}\label{eq:p114Kps}
 4(K_{X}+c_1 D)^2\leq 9A_{(X,c_1D)}(\ord_E)^2\vol_{X,x}(\ord_E).
 \end{equation}
 By computation we know that 
 \[
 A_{(X,c_1D)}(\ord_E)=A_X(\ord_E)-c_1 \ord_E(D)=\tfrac{1}{2}-c_1 \ord_E(D), \quad \vol_{X,x}(\ord_E)=4.
 \]
 When $d$ is even, we know that $(K_{X}+c_1 D)^2=\frac{9}{4}$. Hence \eqref{eq:p114Kps} implies that $9\leq 9 (1-2c_1\ord_E (D))^2$, i.e. $D$ does not pass through $x$. Thus the equation of $D$ is given by 
 \[
 z^{d/2}+f_4(x,y)z^{(d-2)/2}+\cdots+f_{2d}(x,y)=0. 
 \]
 By taking the $1$-PS $\lambda:\bG_m\to\Aut(X)$ as $\lambda(t)([x,y,z])=[x,y,tz]$ for $t\in\bG_m$, we see that $\lim_{t\to 0}\lambda(t)\cdot D=Q_d'$. Thus $(X,c_1 D)\cong (\bP(1,1,4), c_1 Q_d')$ since they are both K-polystable. When $d$ is odd, we know that $(K_{X}+c_1 D)^2=\frac{9(d-3)^2}{(2d-3)^2}$. Hence \eqref{eq:p114Kps} implies that $\frac{36(d-3)^2}{(2d-3)^2}\leq 9(1-\frac{6}{2d-3}\ord_E(D))^2$, i.e. $\ord_E(D)\leq \frac{1}{2}$. Since the equation of $D$ is given by 
 \[
 z^l f_2(x,y)+z^{l-1}f_6(x,y)+\cdots +f_{4l+2}(x,y)=0
 \]
 where $l:=\frac{d-1}{2}$, we have $\ord_E(D)\geq \frac{1}{2}$ with equality holds if and only if $f_2\neq 0$. Hence $\ord_E(D)=\frac{1}{2}$, $f_2\neq 0$ and the equality of \eqref{eq:p114Kps} holds. Then by \cite[Lemma 33]{Liu18} we know that $E$ minimizes the normalized volume function at the singularity $x\in (X,c_1 D)$. So \cite[Theorem 1.2]{LX16} implies that $(E,\Delta_E)\cong (\bP_{[x,y]}^1, c_1(f_2(x,y)=0))$ is a K-semistable Koll\'ar component. Thus $f_2$ is a non-degenerate quadratic form in $(x,y)$, and after a suitable choice of projective coordinates of $\bP(1,1,4)$ we may assume that $f_2(x,y)=xy$. By taking the same $1$-PS $\lambda$ as before, we see $(X,c_1D)$ specially degenerates to $(\bP(1,1,4), c_1 Q_d')$. Thus $(X,c_1 D)\cong (\bP(1,1,4), c_1 Q_d')$ since they are both K-polystable. We finish the proof.
\end{proof}

\begin{cor}\label{cor:Q_dps}
The plane curve $Q_d$ is GIT polystable.
\end{cor}

\begin{proof}
If $d$ is even, we know that $(\bP^2,\epsilon Q_d=\frac{d\epsilon}{2} Q)$ is K-polystable by \cite[Theorem 1.5]{LS14}. Thus $Q_d$ is GIT polystable by Proposition \ref{prop:P^2-paultian}. If $d$ is odd, we know that $(\bP^2, c_1 Q_d)$ is K-semistable by Lemma \ref{lem:Q_d'Kps}. Hence $Q_d$ is GIT semistable by Proposition \ref{prop:P^2-paultian}. Assume to the contrary that $Q_d$ is not GIT polystable. Denote by $C_0$ the GIT polystable plane curve that is S-equivalent to $Q_d$.  If $\Supp(C_0)$ contains a smooth conic $Q$, then $C_0=\frac{d-1}{2}Q+L'$ where $L'$ is a tangent line of $Q$. But then $\Aut(\bP^2,C_0)\cong\Aut(\bP^1,[0])$ is non-reductive, a contradiction. Hence $\Supp(C_0)$ is a union of  lines. By Theorem \ref{thm:firstwallbefore}, we know that $(\bP^2,\epsilon C_0)$ is K-polystable. This implies that $(\bP^2,\frac{3}{d}C_0)$ is log canonical by \cite[Theorem 1.5]{Fuj17b}. But $\frac{3}{d}\cdot\frac{d-1}{2}>1$ since $d\geq 5$, so $(\bP^2,\frac{3}{d}C_0)$ cannot be log canonical, a contradiction. The proof is finished.
\end{proof}

We present the proof of part (2) of Theorem \ref{mthm:firstwall} as follows.

\begin{thm}[First wall crossing 2]\label{thm:firstwallon}
 We follow Notation \ref{not:firstwall}.
 A log Fano pair $(X,c_1 D)$ is a K-polystable point of $\oP_{d,c_1}^{\K}$ if and only if either $X\cong\bP^2$ and $D$ is a GIT polystable plane curve not projectively equivalent to $Q_d$, or $(X,D)\cong (\bP(1,1,4), Q_d')$. Moreover, there is an open immersion $\Phi^-:\ocP_d^{\GIT} = \ocP^\K_{d, c_1-\epsilon}\hookrightarrow \ocP_{d,c_1}^{\K}$ which descends to an isomorphism of good moduli spaces $\phi^-: \oP_d^{\GIT} = \oP^\K_{d, c_1-\epsilon} \xrightarrow{\cong} \oP^\K_{d,c_1}$.
\end{thm}

\begin{proof}
We first show that $\phi^-$ is an isomorphism. It is clear that $\phi^-$ is a birational morphism between normal proper varieties since $P_d$ is a common open subset of  $\oP_d^{\GIT}$ and $\oP_{d,c_1}^{\K}$ by Proposition \ref{prop:Zplanecurve}. Indeed, we will show the Picard number $\rho(\oP_d^{\GIT})$ is one. Since $\SL(3)$ has no non-trivial characters, there are injections 
\[
\Pic(\bfP_d^{\rm ss} \sslash \SL(3)) \hookrightarrow \Pic_{\SL(3)}(\bfP_d^{\rm ss})\hookrightarrow\Pic(\bfP_d^{\rm ss})
\]
by \cite[Proposition 4.2 and \S 2.1]{picGIT}). Since we have a surjection from $\Pic(\bfP_d)\cong\bZ$ to $\Pic(\bfP_d^{\rm ss})$, we know that $\rho(\oP_d^{\GIT})=1$. Thus $\phi^-$ is an isomorphism between good moduli spaces.

Let $C$ be a GIT polystable plane curve not projectively equivalent to $Q_d$. Let $(X,cD)$ be the K-polystable degeneration of $(\bP^2, cC)$. From the index estimate in the proof of Theorem \ref{thm:firstwallbefore}, we know that $X$ is isomorphic to either $\bP^2$ or $\bP(1,1,4)$.
If $X$ is isomorphic to $\bP(1,1,4)$, then $D$ has to be $Q_d'$ by Lemma \ref{lem:Q_d'Kps}. Thus $\phi^-([C])=\phi^-([Q_d])$ which contradicts to the injectivity of $\phi^-$. Thus $X\cong\bP^2$ hence $(X,cD)\cong (\bP^2,cC)$ by GIT polystability of $C$. The proof is finished by Lemma \ref{lem:Q_d'Kps}.
\end{proof}



\subsection{After the first wall}\label{sec:firstwall2}

In this section, we will show that the K-moduli stack $\ocP_{d,c_1+\epsilon}^{\K}$ is isomorphic to a  Kirwan type weighted blow up of the GIT moduli stack.

\begin{thm}[First wall crossing 3]\label{thm:firstwallafter}
Let $\Phi^+:\ocP_{d,c_1+\epsilon}^{\K}
\to \overline{\calP}^\K_{d,c_1}$ be the latter morphism in the first wall crossing. Then there exists a stacky weighted blow up morphism $\rho:\ocP_{d,c_1+\epsilon}^{\K}\to \ocP_{d,c_1-\epsilon}^{\K}=\ocP_d^{\GIT}$ along $\{[Q_d]\}$ (see Definition \ref{def:weightedstackyblowup}) such that $\Phi^+=\Phi^-\circ\rho$.  
In particular, we have
\begin{enumerate}
\item The descent morphism  $\varrho=(\phi^{-})^{-1}\circ \phi^+ :\oP_{d,c_1+\epsilon}^{\K}
\to \oP^\K_{d,c_1-\epsilon}=\oP_d^{\GIT}$ of $\rho$ between good moduli spaces is a weighted blowup of the point $[Q_d]$.
\item If $d$ is even, then $\varrho$ is a partial desingularization of Kirwan type.
\end{enumerate}
\end{thm}

The proof of this theorem will be split up into a few parts. Before we analyze the stack structure of $\ocP_{d,c_1+\epsilon}^{\K}$, we first give a complete description of its closed points. 

\begin{defn}\label{defn:gitp114} We define GIT stability for certain curves on $\bP(1,1,4)$.
\begin{enumerate}
\item Assume $d$ is even. Given a curve $D$ in $\bP(1,1,4)$ of degree $2d$ with equation
\begin{equation}\label{eq:even-p114}
z^{d/2}+f_8(x,y) z^{(d-4)/2}+f_{12}(x,y) z^{(d-6)/2}+\cdots+f_{2d}(x,y)=0,
\end{equation}
we identify $D$ to a point $(f_8, f_{12},\cdots, f_{2d})$ in the vector space $\bfA_d'\cong \oplus_{j=2}^{d/2} H^0(\bP^1, \cO_{\bP^1}(4j))$.
Consider the $\bG_m$-action $\sigma$ on $\bfA_d'$ with weight $j$ in each direct summand $H^0(\bP^1,\cO_{\bP^1}(4j))$. 
Let $\bfP_d'$ be the weighted projective space as the coarse moduli space of the quotient stack $[(\bfA_d'\setminus\{0\})/\bG_m]$. There is a natural $\SL(2)$-action on $\bfA_d'$ induced by the usual $\SL(2)$-action on $H^0(\bP^1,\cO_{\bP^1}(1))=\bC x\oplus\bC y$. Since this $\SL(2)$-action commutes with the previous $\bG_m$-action, it descends to an $\SL(2)$-action on $(\bfP_d',\cO_{\bfP_d'}(1))$. We say $[D]\in \bfP_d'$ is \emph{GIT (poly/semi)stable} if it is GIT (poly/semi)stable with respect to this $\SL(2)$-action on $(\bfP_d',\cO_{\bfP_d'}(1))$.
\item Assume $d$ is odd. Suppose $D$ is a curve in $\bP(1,1,4)$ of degree $2d$ with equation
\[
xyz^{(d-1)/2}+f_6(x,y)z^{(d-3)/2}+f_{10}(x,y)z^{(d-5)/2}+\cdots+f_{2d}(x,y)=0,
\]
where $f_6(x,y)$ contains no monomial divisible by $xy$. Then we identify $D$ to a point $(f_6,f_{10},\cdots, f_{2d})$ in the vector space \[\bfA_d':= V_1\oplus \oplus_{j=2}^{(d-1)/2}H^0(\bP^1,\cO_{\bP^1}(4j+2)),\] where $V_1:=\bC x^6\oplus\bC y^6$ is a sub vector space of $H^0(\bP^1, \cO_{\bP^1}(6))$. Consider the $\bG_m$-action $\sigma$ on $\bfA_d'$ with weight $1$ on $V_1$ and weight $j$ on each direct summand $H^0(\bP^1, \cO_{\bP^1}(4j+2))$. Let $\bfP_d'$ be the weighted projective space which is the coarse moduli space of the quotient stack $[(\bfA_d'\setminus\{0\})/\bG_m]$. Consider another $\bG_m$-action $\sigma'$ on $\bfA_d'$ induced by \[\sigma'(t)\cdot f_{4j+2}(x,y)=f_{4j+2}(tx,t^{-1}y)\] for $t\in\bG_m$ and $1\leq j\leq \frac{d-1}{2}$. Since  $\sigma'$ commutes with $\sigma$, it descends to a $\bG_m$-action on $(\bfP_d', \cO_{\bfP_d'}(1))$ which we also denote by $\sigma'$. We say $[D]\in \bfP_d'$ is \emph{GIT (poly/semi)stable} if it is GIT (poly/semi)stable with respect to the $\bG_m$-action $\sigma'$ on $(\bfP_d',\cO_{\bfP_d'}(1))$. 
\end{enumerate}
\end{defn}

    Care is taken in the above definition to define GIT stability for curves on $\PP(1,1,4)$ because the automorphism groups of weighted projective spaces are in general non-reductive.  Recently, the theory of non-reductive GIT and variation of non-reductive GIT has been developed (see e.g. \cite{DoranKirwan, BJK}) which may be useful in similar endeavors.

\begin{thm}\label{thm:k=gitp114}
Let $D$ be a curve on $\bP(1,1,4)$ of degree $2d$ such that $[D]\in\bfP_d'$. Then the pair $(\bP(1,1,4), (c_1+\epsilon)D)$ is K-(poly/semi)stable if and only if $[D]$ is GIT (poly/semi)stable in the sense of Definition \ref{defn:gitp114}.
\end{thm}

\begin{proof}
We first prove the ``only if'' part. Let $\pi:(\bP(1,1,4)\times \bfA_d',\cD)\to \bfA_d'$ be the universal family of pairs over $\bfA_d'$ where the fiber of $\pi$ over each point $D$ of $\bfA_d'$ is $(\bP(1,1,4),D)$. Then the $\bG_m$-action $\sigma$ on $\bfA_d'$ has a natural lifting to the universal family, which we also denote by $\sigma$, namely $\sigma (t)\cdot ([x,y,z],D)=([x,y,tz],\sigma(t)\cdot D)$. Hence by quotienting out $\sigma$ we obtain a $\bQ$-Gorenstein family of log Fano pairs over the Deligne-Mumford stack  $[(\bfA_d'\setminus\{0\})/\bG_m]$. The CM $\bQ$-line bundle $\lambda_{\CM,\pi, c\cD}$ on $\bfA_d'$ also descends to a $\bQ$-line bundle on $\bfP_d'$ which we denote by $\Lambda_{c}$. By Theorem \ref{thm:paultian}, it suffices to show that $\Lambda_{c_1+\epsilon}$ is ample. Since $\lambda_{\CM,\pi,c\cD}$ is a trivial $\bQ$-line bundle over $\bfA_d'$, the degree of $\Lambda_{c}$ is equal to the $\sigma$-weight of the central fiber $\lambda_{\CM,\pi,c\cD}\otimes\bC(0)$. By Proposition \ref{prop:Fut=CMwt}, we know that 
\[
\deg\Lambda_c=\Fut((\bP(1,1,4), cQ_d;\cO_{\bP(1,1,4)}(4))\times\bA^1)
\]
where the product test configuration $(\bP(1,1,4), cQ_d;\cO_{\bP(1,1,4)}(4))\times\bA^1$ is induced from the $\bG_m$-action $\sigma$. From the definition, easy computation, and K-polystability of $(\bP(1,1,4), c_1 Q_d)$ we know that $\Fut((\bP(1,1,4), cQ_d;\cO_{\bP(1,1,4)}(4))\times\bA^1)$ is linear in $c$, is negative when $c=0$, and zero when $c=c_1$. Hence it is positive when $c=c_1+\epsilon$. As a result, the CM $\bQ$-line bundle $\Lambda_{c_1+\epsilon}$ is ample on $\bfP_d'$. Note that when $d$ is odd, the action $\sigma'$ on $\bfA_d'$ has zero weight on the central fiber $\lambda_{\CM,\pi,c\cD}\otimes\bC(0)$ by a straightforward computation of the generalized Futaki invariant. Thus the $\bG_m$-linearization of a suitable positive power of  $\Lambda_{c_1+\epsilon}$ coincides with the $\bG_m$-linearization on $\cO_{\bfP_d'}(1)$ of $\sigma'$. This completes the proof of the ``only if'' part.

Next we prove the ``if'' part. Since each curve $[D]\in\bfP_d'$ admits a special degeneration to $Q_d$, we know $(\bP(1,1,4), c_1 D)$ is K-semistable by Lemma \ref{lem:Q_d'Kps} and Theorem \ref{thm:Kss-spdeg}. By Bertini's theorem, it is clear that a general curve $D$ in $\bfP_d'$ has at worst a nodal point at the unique singularity of $\bP(1,1,4)$ and smooth elsewhere. Thus for a general curve $D$ we know that $(\bP(1,1,4),\frac{3}{d}D)$ is klt. This implies that $(\bP(1,1,4),(c_1+\epsilon)D)$ is K-stable by Proposition \ref{prop:k-interpolation}. Let $[D_0]\in \bfP_d'$ be a GIT polystable point. From the above argument, we can find a family of curves $[D_t]\in \bfP_d'$ parametrized by a punctured smooth curve $t\in T\setminus\{0\}$ such that $(\bP(1,1,4),(c_1+\epsilon)D_t)$ is K-stable and $\lim_{t\to 0} [D_t]=[D_0]$. Then by properness of K-moduli spaces (i.e. Theorem \ref{thm:compactness}), after a possible finite base change of $T$ we obtain a K-polystable limit $(X,(c_1+\epsilon)D')$ of $(\bP(1,1,4),(c_1+\epsilon)D_t)$ as $t$ goes to $0$. By the index estimate Theorem \ref{thm:localindex}, we know that $X\cong\bP(1,1,4)$. From the continuity of generalized Futaki invariants in $c$, we know $(\bP(1,1,4), c_1 D')$ is K-semistable. Thus $D'$ specially degenerates to $Q_d$ by Lemma \ref{lem:Q_d'Kps}. After a suitable change of coordinates of $\bP(1,1,4)$, we may assume $[D']\in\bfP_{d}'$. Thus the ``only if'' part implies that $[D']$ is a GIT polystable limit of $g_t\cdot [D_t]$ for $g_t\in \SL(2)$ (for even $d$) or $\bG_m$ (for odd $d$). By separatedness of the GIT quotient, we know that $[D']$ and $[D_0]$ lie in the same orbit. Thus $(\bP(1,1,4),(c_1+\epsilon)D_0)$ is K-polystable. The proof regarding K-semistability and K-stability follows from similar arguments as in the proof of Theorem \ref{thm:firstwallbefore}.
\end{proof}

\begin{thm}\label{thm:firstwallps}
Let $[(X, (c_1+\epsilon)D)]$ be a K-polystable point in $\oP_{d,c_1+\epsilon}^{\K}$. Then 
 either $(X,D)\cong (\bP^2, C)$ where $C$ is a GIT polystable plane curve not projectively equivalent to $Q_d$,
 or $(X,D)\cong (\bP(1,1,4), D')$ where $[D']\in\bfP_d'$ is GIT polystable. Conversely, any such pair $(\bP^2, C)$ or $(\bP(1,1,4),D')$ is $(c_1+\epsilon)$-K-polystable.
\end{thm}

\begin{proof}
We first prove that K-polystability implies GIT polystability. Let $[(X,(c_1+\epsilon)D)]$ be a point in $\oP_{d,c_1+\epsilon}^{\K}$. Then from the index estimate in the proof of Theorem \ref{thm:firstwallbefore}, we know that $X$ is isomorphic to either $\bP^2$ or $\bP(1,1,4)$. If $(X,D)\cong(\bP^2,C)$, then $C$ is GIT polystable by Proposition \ref{prop:P^2-paultian}. It suffices to show that $(\bP^2, (c_1+\epsilon)Q_d)$ is K-unstable. In fact, if $(\bP^2, (c_1+\epsilon)Q_d)$ were K-semistable, then Proposition \ref{prop:k-interpolation} together with K-polystability of $(\bP^2,\epsilon Q_d)$ (see Theorem \ref{thm:firstwallbefore}) implies that $(\bP^2,c_1 Q_d)$ is K-polystable as well. But this contradicts Lemma \ref{lem:Q_d'Kps}, so $(\bP^2, (c_1+\epsilon)Q_d)$ is K-unstable. If $X\cong\bP(1,1,4)$, then the statement follows from Theorem \ref{thm:k=gitp114}. 

Next we prove that GIT polystability implies K-polystability. If $C\subset\bP^2$ is a GIT polystable plane curve not projectively equivalent to $Q_d$, then Theorem \ref{thm:firstwallon} implies that $(\bP^2, c C)$ is K-polystable for any $c\in (0,c_1]$. Hence $(\bP^2, (c_1+\epsilon)C)$ is also K-polystable by Proposition \ref{prop:openkps}. The $\bP(1,1,4)$ case follows from Theorem \ref{thm:k=gitp114}.
\end{proof}

So far we have shown that $\oP_{d}^{\GIT}\setminus\{[Q_d]\}\hookrightarrow
\oP^\K_{d,c_1+\epsilon}$ is an open immersion. Before proving Theorem \ref{thm:firstwallafter}, we set up some notation.

\begin{definition}\label{def:weightedstackyblowup}Let $Z \subset X$ be a smooth closed subvariety of a smooth quasi-projective variety $X$.  We define $\calB l_{\mathbf{w},Z}X$, the \emph{stacky weighted blowup} of $Z$ in $X$ with weight $\mathbf{w}$ (see also \cite[Section 3]{functorialresolution}).  The standard weighted blow up $Bl_{\mathbf{w},Z} X$  will be the coarse space of $\calB l_{\mathbf{w},Z}X$. 

Let $\cN_{Z/X}$ be the normal bundle of $Z \subset X$ and consider a group $G$ acting on $X$.  Suppose $\cN_{Z/X}$ has a decomposition with respect to representations of $G$ \[\cN_{Z/X} = \cN_1 \oplus \cN_2 \oplus \dots \oplus \cN_k .\]  Let $\mathbf{w} = ( w_1, w_2, \dots , w_k )$ be a weight vector with $w_i \in \mathbb{Z}_{>0}$.  This gives a monomial valuation $v$ of $\bC(X)$ centered at $Z$ of weights $( w_1, w_2, \dots , w_k )$ with respect to the decomposition of the normal bundle $\cN_{Z/X}$.  

Define $R : = \oplus_{m = 0}^\infty \fa_m(v) t^m$, where $\fa_m(v) = \{ s \in \cO_X \colon v(s) \ge m\}$.  The \emph{standard weighted blow up} of $Z$ in $X$ is defined by $Bl_{\mathbf{w},Z}X:=\Proj_X R$.  Define $Y := \mathrm{Spec}_X R$. Then we have the zero section $X \hookrightarrow Y$ whose defining ideal is given by $I = \oplus_{m = 1}^\infty \fa_m(v) t^m$.  We define the \emph{stacky weighted blowup} of $Z$ in $X$ to be \[ \calB l_{\mathbf{w},Z}X = [ (Y \setminus X ) / \mathbb{G}_m ] \]
where $\mathbb{G}_m$ acts as $t \mapsto \lambda t$, for $\lambda \in \mathbb{G}_m$.
\end{definition}

\begin{remark}The stack $\calB l_{\mathbf{w},Z}X$ is a smooth Deligne-Mumford stack, its coarse space is indeed $Bl_{\mathbf{w},Z} X$, and the exceptional divisor $E$ has weighted projective stacks as fibers over $Z$.  \end{remark}

We now proceed with our construction of the partial desingularization of Kirwan type. First, let us recall some representation theory of $\SL(2)$. Consider the standard action of $\SL(2)$ on $\bP^1=\bP(V^\vee)$. Then we have the dual $\SL(2)$-action on $V=H^0(\bP^1,\cO_{\bP^1}(1))$. We have a natural $\SL(2)$-action on $\bfV:=H^0(\bP^1, \cO_{\bP^1}(2))=\Sym^2 V$ so that the second Veronese embedding $\bP^1\hookrightarrow\bP^2=\bP(\bfV^\vee)$ is $\SL(2)$-equivariant. Denote by $Q=(q=0)$ the image of this Veronese embedding. Then we have 
\begin{equation}\label{eq:sl_2-1}
\Sym^d \bfV=\Sym^d(\Sym^2 V)\cong\begin{cases}
\oplus_{i=0}^{d/2}\Sym^{4i} V & \textrm{ if $d$ is even,}\\
\oplus_{i=0}^{(d-1)/2}\Sym^{4i+2} V &\textrm{ if $d$ is odd.}
\end{cases}
\end{equation}
Since $q$ is $\SL(2)$-invariant, we have an injection between $\SL(2)$-representations $\Sym^{d-2}\bfV\hookrightarrow\Sym^d\bfV$ by multiplying with $q$. Let $\bfV_d$ be the cokernel of this injection. Then from \eqref{eq:sl_2-1} we know that the restriction map $H^0(\bP^2,\cO_{\bP^2}(d))\to H^0(Q,\cO_Q(d))$ induces an isomorphism between $\bfV_d$ and $\Sym^{2d} V$. As a summary, we have
\begin{equation}\label{eq:sl_2-2}
 H^0(\bP^2, \cO_{\bP^2}(d))=\oplus_{j=0}^{\lfloor d/2\rfloor} q^{j}\cdot\bfV_{d-2j}\cong \oplus_{j=0}^{\lfloor d/2\rfloor} \Sym^{2d-4j}V.
\end{equation}

\begin{lemma}\label{lem:firstwallluna1} Let $d$ be even and let $Q_d$ denote the nonreduced curve defined by $(q^{d/2} = 0)$ where $(q=0)$ is a smooth plane conic. Then a Luna slice to $\SL(3)\cdot [Q_d] \subset \bfP^{\rm ss}_d$ at $[Q_d]$ is given by the locally closed subset
\[ W := \{ (q^{d/2} + f_4q^{d/2 -2} + f_6 q^{d/2-3} + \dots + f_d=0) \},\] where $f_{2i} \in \bfV_{2i}\subset H^0(\bP^2, \calO_{\bP^2}(2i))$ for $2\leq i\leq d/2$. \end{lemma}

\begin{proof}
Denote by $G:=\SL(3)$. Since $Q_d$ is GIT polystable by Corollary \ref{cor:Q_dps}, we know that the orbit $G[Q_d]$ is closed in $\bfP_d^{\rm ss}$. The tangent space of $\bfP_d^{\rm ss}$ at $[Q_d]$ is given by 
$\cT_{\bfP_d^{\rm ss}, [Q_d]}=(\Sym^d \bfV)/\bC q^{d/2}$. If $q_t$ is a small deformation of $q_0=q$ in $\Sym^2\bfV$, then we have 
$\frac{d}{dt}(q_t^{d/2})|_{t=0}=\frac{d}{2}q^{d/2 -1}\frac{dq_t}{dt}|_{t=0}$. Thus the tangent space of $G[Q_d]$ at $[Q_d]$ is given by $T_{G[Q_d], [Q_d]}=(q^{d/2-1}\cdot\Sym^2\bfV)/\bC q^{d/2}$. Hence by \eqref{eq:sl_2-2} the normal space of $G[Q_d]/\bfP_d^{\rm ss}$ at the point $[Q_d]$ satisfies
\begin{equation}\label{eq:normal-bundle-even} 
\calN_{G[Q_d]/\bfP_d^{\rm ss},[Q_d]} = \cT_{\bfP_d^{\rm ss}, [Q_d]}/\cT_{G[Q_d], [Q_d]} \cong \Sym^d\bfV/(q^{d/2-1}\cdot\Sym^2\bfV)\cong
\oplus_{i=2}^{d/2}\bfV_{2i}.
\end{equation}
Therefore, taking the exponential map of the normal space yields a Luna slice  \[ W := \{ (q^{d/2} + f_4q^{d/2 -2} + f_6 q^{d/2-3} + \dots + f_d=0) \},\] where  $f_{2i} \in \bfV_{2i}$ for $2\leq i\leq d/2$.
\end{proof}

\begin{lemma}\label{lem:firstwallluna2} Let $d$ be odd and let $Q_d$ denote the curve 
$(lq^{(d-1)/2} = 0)$, where $(q=0)$ is a smooth plane conic and $(l=0)$ is a line transverse to $q$. Then a Luna slice to $\SL(3)\cdot [Q_d] \in \bfP^{\rm ss}_d$ at $[Q_d]$ is given by the locally closed set
\[ W := \{ (lq^{(d-1)/2} + f_3 q^{(d-3)/2} + f_5 q^{(d-5)/2} + \dots + f_d =0)\},\] 
where $f_3\in\bfV_3^{\circ}$ and $f_{2i+1} \in \bfV_{2i+1}$ for any $2\leq i\leq (d-1)/2$. Here $\bfV_3^{\circ}$ is a subspace of $\bfV_3$ such that it pulls back to $\bC u^6+\bC v^6\subset H^0(\bP_{[u,v]}^1,\cO(6))$ under the isomorphism from $(\bP_{[u,v]}^1,(uv=0))$ to $((q=0),  (l=q=0))$. 
\end{lemma}

\begin{proof}
Again we consider the orbit $G[Q_d]$ under the action of $G:=\SL(3)$. It is clear that $\cT_{\bfP_d^{\rm ss},[Q_d]}=\Sym^d \bfV/\bC l q^{(d-1)/2}$. If $l_t$ and $q_t$ are small deformations of $l_0=l$ and $q_0=q$ in $\bfV$ and $\Sym^2\bfV$ respectively, then we have $\frac{d}{dt}(l_t q_t^{(d-1)/2})|_{t=0} =  \frac{d-1}{2}lq^{(d-3)/2} \frac{dq_t}{dt}|_{t=0} + q^{(d-1)/2} \frac{dl_t}{dt}|_{t=0}$. Thus the tangent space of $G[Q_d]$ at $[Q_d]$ is given by $T_{G[Q_d],[Q_d]}=(lq^{(d-3)/2}\cdot\Sym^2\bfV+ q^{(d-1)/2}\cdot \bfV)/\bC lq^{(d-1)/2}$. Hence the normal space of $G[Q_d]/\bfP_d^{\rm ss}$ at the point $[Q_d]$ satisfies
\[ \calN_{G[Q_d]/\bfP_d^{\rm ss},[Q_d]} = \cT_{\bfP_d^{\rm ss}, [Q_d]}/\cT_{G[Q_d], [Q_d]} \cong \Sym^d\bfV/(lq^{(d-3)/2}\cdot\Sym^2\bfV+ q^{(d-1)/2}\cdot \bfV).\]  
It is clear that $l\cdot\Sym^2\bfV$ and $q\cdot \bfV$ are both contained in $\Sym^3\bfV$. Let $G_1\cong \bG_m\rtimes \bZ/2$ be the  subgroup of $\SL(2)$ preserving both $l$ and $q$. Then both $l\cdot\Sym^2\bfV$ and $q\cdot \bfV$ are sub $G_1$-representation of $\Sym^3\bfV$. Since  $\Sym^3\bfV/(q\cdot\bfV)\cong\bfV_{3}$ by \eqref{eq:sl_2-2}, we denote by $\bfV_3^{\circ}$ the sub $G_1$-representation of $\bfV_{3}$  complementary to $(l\cdot\Sym^2\bfV+q\cdot \bfV)/(q\cdot\bfV)$. 

By \eqref{eq:sl_2-2} we have
\begin{align*}
\calN_{G[Q_d]/\bfP_d^{\rm ss},[Q_d]} & \cong  \Sym^3\bfV/(l\cdot\Sym^2\bfV + q\cdot\bfV) \oplus \Sym^d\bfV/(q^{(d-3)/2}\Sym^3\bfV)\\& \cong \bfV_{3}^\circ\oplus\oplus_{i=2}^{(d-1)/2}\bfV_{2i+1}.\label{eq:normal-bundle-odd}\numberthis
\end{align*}
Therefore, taking the exponential map of the normal space yields a Luna slice  
\[ W := \{ (lq^{(d-1)/2} + f_3 q^{(d-3)/2} + f_5 q^{(d-5)/2} + \dots + f_d=0) \},\] where  $f_3\in\bfV_3^{\circ}$ and $f_{2i+1} \in \bfV_{2i}$ for $2\leq i\leq (d-1)/2$.






Finally, we verify that $\bfV_3^\circ$ corresponds to $\bC u^6+\bC v^6$ under the pull back to $\bP^1$. Since the pull-back of $l$ on $\bP^1$ is $uv$, we know that $\bfV_3^\circ$ is complementary to $uv\cdot H^0(\bP^1, \cO(4))$ in $H^0(\bP^1, \cO(6))$. Since the identity component of $G_1$  acts on $H^0(\bP^1,\cO(1))=\bC u +\bC v$ diagonally, we know that $\bfV_3^\circ$ corresponds to $\bC u^6+\bC v^6$ under the pull back to $\bP^1$. This finishes the proof.
\end{proof}

We now proceed to the proof of Theorem \ref{thm:firstwallafter}.

\begin{thm}\label{thm:firstwall1}
Let $U = \bfP^{\rm ss}_d$, and consider the universal family $(\PP^2_U, \calC_U) \to U$, where each fiber is $c$-K-polystable for $c\in (0,c_1)$. After base change to a stacky weighted blowup $\widehat{\calU} \to U$ along the orbit $G[Q_d]$ with stacky exceptional divisor $\calE$ where $G=\SL(3)$, a blow-up along the conic component over $\cE$, and a divisorial contraction, there exists a  family $(\cX, \calC_{\cX}) \to \widehat{\calU}$ which: 
\begin{enumerate}
    \item is isomorphic to $(\PP^2_{U\setminus G[Q_d]}, \calC_{U\setminus G[Q_d]}) \to U\setminus G[Q_d]$ over the complement of $\calE$, 
    \item and whose fibers over $\calE$ are curves on $\bP(1,1,4)$. 
    \item Let $\cE_W$ be the exceptional divisor of the stacky weighted blow up $\hcW:=\hcU\times_U W\to W$ over the Luna slice $W$ as in Lemmas \ref{lem:firstwallluna1} and \ref{lem:firstwallluna2}. Then the  family $(\cY,\cC_{\cY})\times_{\hcU}\cE_W$ over $\cE_W$ is isomorphic to the universal family over $[(\bfA_d'\setminus\{0\})/\bG_m]$ as in Definition \ref{defn:gitp114}.
\end{enumerate}
\end{thm}

 \begin{proof}
  By \eqref{eq:normal-bundle-even} and \eqref{eq:normal-bundle-odd}, we have an $\Aut(Q_d)$-equivariant decomposition of the normal space \[
  \cN_{G[Q_d]/U,[Q_d]}\cong\begin{cases}
  \oplus_{i=2}^{d/2}\bfV_{2i} & \textrm{ if $d$ is even,}\\
  \bfV_3^\circ\oplus\oplus_{i=2}^{(d-1)/2}\bfV_{2i+1}& \textrm{ if $d$ is odd.}
  \end{cases}
  \]
  Thus by translating the above decomposition via the $G$-action, we obtain a $G$-equivariant decomposition of the normal bundle $\cN_{G[Q_d]/U}$ which we denote by $\oplus_{i=2}^{d/2}\cN_i$ and $\oplus_{i=1}^{(d-1)/2}\cN_i$ when $d$ is even and odd respectively.
  Let $\widehat{\calU} \to U$ be the weighted stacky blowup along the $G$-orbit of $Q_d$ as in Definition \ref{def:weightedstackyblowup} given by weight $i$ on $\cN_i$.  Let $\calE$ denote the corresponding stacky exceptional divisor, and let $\widehat{U}$ and $E$ denote the weighted projective blow up coarse space and corresponding exceptional divisor respectively. Note that $\cE$ is a smooth Cartier divisor in $\widehat{\calU}$ as we are doing a stacky weighted blowup.
  
  We note that this blowup is $\SL(3)$-equivariant. Let $(\PP^2_{\widehat{U}}, \calC_{\widehat{U}}) \to {\widehat{U}}$ denote the pullback of this universal family via the blowup, and denote by the pullback of the universal family to the stack by $(\bP^2_{\widehat{\calU}}, \calC_{\widehat{\calU}}) \to \widehat{\calU}$.  
  
  Consider the blowup $\phi: (\calY, \calC_\calY) \to (\bP^2_{\widehat{\calU}}, \calC_{\widehat{\calU}})$ of the universal family at the conic component $\cQ$ of $\cC_{\hcU}|_{\bP^2_{\cE}}$ over the exceptional divisor $\calE \subset \widehat{\calU}$, where $\cC_{\cY}$ is the strict transform of $\cC_{\hcU}$. Let $\calF$ denote the $\phi$-exceptional divisor. Since $\cQ$ is smooth of codimension $2$ in $\bP^2_{\hcU}$, and $\cE$ is a smooth Cartier divisor of $\hcU$, we know that $\pi=\pr_2\circ\phi:\cY \to \hcU$ is a projective equidimensional morphism between smooth Deligne-Mumford stacks, which implies that $\cY$ is flat over $\hcU$ by miracle flatness. 
  Furthermore, the codimension two locus $\cQ$ we blow up is $\SL(3)$-equivariant. In particular, we obtain the following diagram
 
 \[
 \begin{tikzcd}
 \calF \arrow[r, hook] \arrow[rrdd, bend right] & (\calY, \calC_\calY) \ar[r]\ar[rd] & (\bP^2_{\widehat{\calU}}, \calC_{\widehat{\calU}}) \ar[r] \ar[d] & (\bP^2_{\widehat{U}}, \calC_{\widehat{U}}) \ar[r] \ar[d] & (\bP^2_U, \calC_U) \ar[d]\\
 & & \widehat{\calU} \ar[r] & \widehat{U} \ar[r] & U\\
 & & \calE \ar[u, hook] \ar[r] & E \arrow[u, hook] \ar[r] & G[Q_d] \arrow[u, hook]\\
 \end{tikzcd}
 \]

 Note that fibers of $(\calY, \calC_\calY) \to \widehat{\calU}$ over $\widehat{\calU} \setminus \calE$ are unchanged, and the fibers over $\calE$ are simple normal crossing surfaces of the form $\bP^2 \cup \mathbb{F}_4$, where $\mathbb{F}_4$ denotes the fourth Hirzebruch surface, and they are glued along the conic component in $\bP^2$ and the negative section in $\mathbb{F}_4$.  Let $\calD$ denote the strict transform of $\bP^2_{\cE}$ under $\phi$. 
 
 We claim that there is a divisorial contraction $\psi:\cY \to \cX$ over ${\widehat{\calU}}$ that contracts $\cD$ to a section over $\cE$. Indeed, let $\cL:=\phi^* \cO_{\bP^2_{\hcU}}(1) + \frac{1}{2}\cD$ be a $\bQ$-Cartier $\bQ$-divisor on $\cY$. Then we know that $2\cL = \phi^* \cO_{\bP^2_{\hcU}}(2) + \cD$ is a Cartier divisor on $\cY$. First, we show that $2\cL$ is nef and big over $\hcU$. Clearly, $2\cL$ is ample over $\hcU\setminus\cE$ where $\phi$ is isomorphic. Now, let us restrict to a fiber $\cY_{e}$ over a geometric point $e\in |\cE|$. Since $\cF +\cD = \pi^* \cE$, we have $2\cL \sim_{\hcU} \phi^* \cO_{\bP^2_{\hcU}}(2) - \cF$. Denote by $\cY_e = \bP^2_e \cup \mathbb{F}_{4,e}$.
 Let $H_2$ be an element of $\calO_{\bP^2}(2)$. Then  $ \phi^*H_2 -F$ is 0 over $\bP^2_e$ and big and nef on $\mathbb{F}_{4,e}$, where $F$ is the restriction of $\calF$ to the fiber $\cY_e$. Indeed, if $l$ is a line in $\PP^2_e$, then $(\phi^*H_2\cdot l) = 2 = (F\cdot l)$, since $F$ came from blowing up a conic. Therefore, $((\phi^*H_2-F)\cdot l) = 0$ for any line in $\PP^2_e$.  An adjunction calculation shows that $(\phi^*H_2 - F)|_{\mathbb{F}_{4,e}} = 4f + s$, where $f$ is a fiber of the ruled surface $\mathbb{F}_{4,e}$ and $s$ is the negative section.  Therefore, $((\phi^*H_2 - F)\cdot f) = 1$ and $((\phi^*H_2 - F)\cdot s) = 0$ and  $\phi^*H_2 - F$ is big and nef on $\mathbb{F}_{4,e}$. 

In particular, we see that $2\calL$ is a big and nef line bundle over $\hcU$ that is 0 on $\calD$ over $\calE$. In order to show that $2\cL$ defines a divisorial contraction, we use cohomology and base change. In fact, since $\cY\to \hcU$ is a $\bQ$-Gorenstein flat family of slc Fano varieties (this can be checked fiberwise, as both $\bP^2$ and $\bP^2\cup \mathbb{F}_4$ are slc with ample anti-canonical divisors), by Fujino's vanishing theorem \cite[Theorem 1.7]{Fuj14} we know that $R^i \pi_* \cO_{\cY} (2m\cL) = 0$ and $\pi_* \cO_{\cY}(2m\cL)$ is locally free for every $i>0$ and $m>0$. Moreover, both sheaves commute with base change. Note that this argument can be applied over the smooth Deligne-Mumford stack $\hcU$ as we can first pass to an \'etale cover of $\hcU$ by smooth schemes, apply Fujino's vanishing theorem there, then descend to $\hcU$. See also \cite{Brochard} for cohomology and base change for Deligne-Mumford stacks. Therefore, since $2\cL$ is basepoint free on every geometric fiber, we know that $2\cL$ is relatively basepoint free, and we obtain a divisorial contraction $\psi: \cY \to \cX$ where 
\[
\cX: =\Proj_{\hcU} \oplus_{m=0}^\infty \pi_* \cO_{\cY}(2m\cL). 
\]
Let $\cC_{\cX}:=\psi_* \cC_{\cY}$. Since the line bundle $2\calL$ has a natural $\SL(3)$ linearization, this contraction is also $\SL(3)$-equivariant. Thus we obtain an $\SL(3)$-equivariant $\bQ$-Gorenstein flat family $(\cX, \cC_{\cX})\to \hcU$. 

Next, we verify conditions (1) and (2). Since $2\cL$ is ample over $\hcU\setminus \cE$, (1) follows directly from the construction. For (2), note that $2\cL|_{\cY_e}$ is trivial on $\bP^2_e$ and has class $4f+s$ on $\mathbb{F}_{4,e}$ for $e\in |\cE|$. Therefore, the linear system $|2\cL|_{\cY_e}|$ is basepoint free and contracts  $\bP^2_e\cup \mathbb{F}_{4,e}$ to $\bP(1,1,4)$. In particular, we know that the algebra $\oplus_{m=0}^\infty H^0(\cY_e, 2m\cL|_{\cY_e})$ is generated in degree $1$. Thus by cohomology and base change we know that the fiber $\cX_e$ is isomorphic to $\bP(1,1,4)$.

Finally, we discuss what happens to the family of curves over the Luna slice $W$. 
Let $\rho_W: \hcW \to W$ be the base change of the stacky weighted blow-up $\hcU\to U$. 
Denote by $(\cY_W, \cC_{\cY_W}):=(\cY, \cC_{\cY})\times_{\hcU} \hcW$ and $(\cX_W, \cC_{\cX_W}): = (\cX, \cC_{\cX})\times_{\hcU} \hcW$. Let $\phi_W: \cY_W \to \bP^2_{\hcW}$ be the pullback of $\phi$ under the base change $\hcW \to \hcU$. Let $\pi_W: \cY_W \to \hcW$ be the projection. Then we know that $\phi_W$ is the blow-up of the smooth conic component $\cQ_{W}\subset \bP^2_{\cE_{W}}$. Let $\cD_W:=(\phi_W)_*^{-1}\bP^2_{\cE_{W}}$. Denote by $\cL_W:=\phi_W^*\cO_{ \bP^2_{\hcW}}(1) + \frac{1}{2}\cD_W$. Then from the above discussion we know that $2\cL_W$ is relatively basepoint free over $\hcW$. We shall define a morphism from $\cY_W$ to a $\bP(1^3,2)$-bundle $\bP\cV$ over $\hcW$ that induces a closed embedding $\cX_W\hookrightarrow \bP\cV$. 

Since $\lfloor \cL_W\rfloor = \phi_W^*\cO_{ \bP^2_{\hcW}}(1)$, we know that $(\pi_W)_* \cO_{\cY_W} (\lfloor \cL_W\rfloor) = \bfV\otimes \cO_{\hcW} =: \cV_1$ where $\bfV = H^0(\bP^2, \cO(1))$. Moreover, we have 
\[
(\pi_W)_* \cO_{\cY_W} (2\cL_W) = (\pi_W)_* (\phi_W^* \cO_{ \bP^2_{\hcW}}(2) \otimes \cO_{\cY_W} ( \cD_W)).
\]
Recall that $\cF$ denotes the exceptional divisor of $\phi$. We know that $Q= (q=0)$ where $q\in H^0(\bP^2, \cO(2))$ defines a divisor $(\phi_W)_*^{-1}(Q\times \hcW) - \cF_{W}$ which corresponds to a section \[q_W\in H^0(\hcW, (\pi_W)_*(\phi_W^* \cO_{ \bP^2_{\hcW}}(2) \otimes \cO_{\cY_W} (- \cF_W))).\] Let $\cV_2:= \cO_{\hcW}(\cE_W)$. Then there is a surjection
\[
(\Sym^2 \cV_1 )\oplus \cV_2 \twoheadrightarrow (\pi_W)_* \cO_{\cY_W} (2\cL_W),
\]
whose first component is given by the usual embedding $\phi_W^* \cO_{ \bP^2_{\hcW}}(2) \hookrightarrow \phi_W^* \cO_{ \bP^2_{\hcW}}(2) \otimes \cO_{\cY_W} ( \cD_W)$, and whose second component is given by multiplying with $q_W$ as $\cD_W = \pi_W^*\cE_W  - \cF_W$. We define $\cV:= \cV_1\oplus \cV_2$ and define $\bP\cV := \Proj_{\hcW} \oplus_{m=0}^\infty \Sym^m \cV$ where $\cV_i$ has degree $i$. Since the ample model of $2\cL$ is defined by $|2\cL|$ on each fiber, by cohomology and base change from earlier we know that the above surjection gives a closed embedding $\cX_{W} \hookrightarrow \bP\cV$.

Next we work out the defining equations of $\cX_W$ and its family of curves $\cC_{\cX_W}$. We focus on the $d$ even case as the $d$ odd case is similar. Let $x_0, x_1, x_2$ be a basis of $\cV_1$ induced by the standard basis of $\bfV = H^0(\bP^2, \cO(1))$. Let $x_3$ be a non-zero constant section of $\cV_2 \otimes \cO_{\hcW}(-\cE_{W}) = \cO_{\hcW}$. Then the equation of $\cX_W$ is given by 
\[
\cX_W = (s  x_3 - q(x_0, x_1, x_2) = 0) \subset \bP\cV,
\]
where $s\in H^0(\hcW, \cO_{\hcW}(\cE_W))$ is some defining section of $\cE_W$. Since $d$ is even,  $W$ has affine coordinates $(f_4, f_6, \cdots, f_{d-2}, f_d)$ where each $f_{2i}\in \bfV_{2i} \subset H^0(\bP^2, \cO_{\bP^2}(2i))$. Since each $f_{2i}$ is the coordinate of $\cN_i$ with weight $i$ under the stacky weighted blow-up $\rho_W$, we  know that $\of_{2i} := s^{-i} \rho_W^* f_{2i}$ is a global section of  $\cO_{\hcW}(-i\cE_W)$. Therefore, we have 
\[
\cC_{\cX_W} = (x_3^{d/2} + \sum_{i=2}^ {d/2} \of_{2i}(x_0, x_1, x_2) x_3^{d/2-i} =0)|_{\cX_W}. 
\]
Next, we pullback $\cX_W\hookrightarrow \bP\cV$ under the quotient map $\bfA_d'\setminus \{0\}\to \cE_W = [(\bfA_d'\setminus \{0\})/\bG_m]$. Since $\cV_1$ is a trivial vector bundle and $\cV_2 = \cO_{\cW}(\cE_W)$, we know that the pull-back $\bP\cV\times_{\hcW} (\bfA_d'\setminus \{0\}$ is $\bG_m$-equivariantly isomorphic to the product $\bP(1^3,2)_{[x_0, x_1, x_2, x_3]}\times (\bfA_d'\setminus \{0\})$ where $\bG_m$ acts on $\bP(1^3,2)$ as $t\cdot [x_0, x_1, x_2, x_3] = [x_0, x_1, x_2, tx_3]$. Under an isomorphism $\bP^1 \to Q = (q=0)\subset \bP^2$, we may identify $\cX_W\times_{\hcW} (\bfA_d'\setminus \{0\})$ with $\bP(1,1,4)_{[x,y,z]}\times (\bfA_d'\setminus \{0\})$ where $x_0,x_1, x_2$ are a basis of quadratic forms in $x,y$ such that $q(x_0, x_1, x_2) = 0$, and $x_3 = z$. Clearly, the $\bG_m$-action on $\bP(1,1,4)\times (\bfA_d'\setminus \{0\})$ is $t\cdot [x,y,z] = [x,y,tz]$. Then the equation of $\cC_{\cX_W}\times_{\hcW} (\bfA_d'\setminus \{0\})$ becomes
\[
(z^{d/2} + \sum_{i=2}^ {d/2} g_{4i}(x,y) z^{d/2-i} =0)\subset \bP(1,1,4)\times (\bfA_d'\setminus \{0\}),
\]
where $g_{4i}(x,y): = \of_{2i}(x_0, x_1, x_2)\in \bfV_{2i} = H^0(\bP^1, \cO(4i))$.  Thus the pullback of the family $(\cX, \cC_{\cX})\to \hcU$ to $\cE_W$ is isomorphic to the universal family over $[(\bfA_d'\setminus \{0\})/\bG_m]$. This verifies (3).
\end{proof}

Let us choose an ideal sheaf $\cI\subset \cO_U$ such that $\hU\cong Bl_{\cI} U$. Let $\overline{\cI}\subset\cO_{\bfP_d}$ be an $\SL(3)$-equivariant extension ideal sheaf of $\cI$ that is cosupported on the Zariski closure of $G[Q_d]$ in $\bfP_d$. Let $\widehat{\bfP}_d$ be the normalization of $Bl_{\overline{\cI}}\bfP_d$ with $\pi_{\bfP_d}:\widehat{\bfP}_d\to \bfP_d$ the projection morphism. Let $\overline{E}$ be the $\pi_{\bfP_d}$-exceptional divisor on $\widehat{\bfP}_d$ such that $\cO_{\widehat{\bfP}_d}(-\overline{E})=\overline{\cI}\cdot\cO_{\widehat{\bfP}_d}$. Then for $k\gg 1$ the line bundle $L_k:=\pi_{\bfP_d}^*\cO_{\bfP_d}(k)\otimes \cO_{\widehat{\bfP}_d}(-\overline{E})$ is an $\SL(3)$-linearized ample line bundle on $\widehat{\bfP}_d$. By \cite{Kir85} we know that the GIT stability of $(\widehat{\bfP}_d, L_k)$ is independent of the choice of $k\gg 1$, and the GIT semistable locus $\widehat{\bfP}_d^{\rm ss}$ is contained in $\hU=\pi_{\bfP_d}^{-1}(\bfP_d^{\rm ss})$. Denote by  $\hU^{\rm ss}:=\widehat{\bfP}_d^{\rm ss}$ and  $\hcU^{\rm ss}:=\hcU\times_{\hU} \hU^{\rm ss}$.

\begin{thm}\label{thm:firstwall2}
There is an isomorphism $\psi: [\widehat{\calU}^{\rm ss}/\PGL(3)] \to \overline{\calP}^{\K}_{d,c_1+\epsilon}$.
\end{thm}

\begin{proof}
We start from the construction of a morphism $\psi:[\hcU^{\rm ss}/\PGL(3)]\to \overline{\calP}^{\K}_{d,c_1+\epsilon}$. Denote by $\hU^{\rm ps}$ and $U^{\rm ps}$ the GIT polystable locus in $\widehat{\bfP}_d$ and $\bfP_d$ respectively. Let $\hcU^{\rm ps}$  be the preimage of $\hU^{\rm ps}$ under the coarse moduli space morphism $\hcU\to \hU$. For simplicity denote $G:=\SL(3)$.  Then by \cite{Kir85} we know that 
\[
\hU^{\rm ps} = \pi_{\bfP_d}^{-1} (U^{\rm ps}\setminus G[Q_d]) \cup G\cdot E_W^{\rm ps}
\]
where $E_W^{\rm ps}$ is the GIT polystable locus in the exceptional divisor $E_W$ of the weighted blowup $\widehat{W}\to W$. Then by Theorems \ref{thm:k=gitp114}, \ref{thm:firstwallps}, and \ref{thm:firstwall1} we know that the fibers of $(\cY,(c_1+\epsilon)\cC_{\cY})\to \hcU$ over $\hcU^{\rm ps}$ are all K-polystable. Hence by Theorem \ref{thm:Kss-spdeg} we know that the fibers over $\hcU^{\rm ss}$ are all K-semistable. Then $\psi$ is constructed by the universality of the K-moduli stacks (see Section \ref{sec:stackstab}). 


Next we show that $\psi$ is an isomorphism between Artin stacks. By Theorem \ref{thm:firstwallbefore} it is clear that $\psi$ is birational. Since $\overline{\cP}^\K_{d, c_1+\epsilon}$ is normal, by Zariski's Main Theorem it suffices to show the above morphism is finite. To do so we use \cite[Proposition 6.4]{alper}. To use Alper's result, we must check that the morphism on the level of good moduli spaces is finite, and that the morphism $\psi$ is representable, separated, quasi-finite, and sends closed points to closed points. 

First, we note that the good moduli spaces are isomorphic since GIT-polystability and K-polystability coincide by the above argument and Theorem \ref{thm:firstwallps}. In particular, we also see that the morphism $\psi$ sends closed points to closed points, as the closed points are the same. 
Hence it suffices to show that $\psi$ is representable, separated, and quasi-finite. 

\begin{lem}\label{lem:k-to-git}
There exists a morphism $\varphi: \ocP_{d,c_1+\epsilon}^{\K}\to [U/\PGL(3)]$ such that the composition $\varphi\circ \psi: [\hcU^{\rm ss}/\PGL(3)]\to [U/\PGL(3)]$ is induced from the stacky weighted blow-up  $\hcU\to U$. 
\end{lem}

\begin{proof}[Proof of the lemma]
Recall from Section \ref{sec:foundations} that $Z_{m, c_1+\epsilon}^{\circ}$ is a locally closed subscheme of the relative Hilbert scheme $\bH^{\chi;N}\times \bH^{\tilde{\chi};N}$. For simplicity, denote by $T:=Z_{m, c_1+\epsilon}^{\circ}$. Let $\pi:(\cX,\cD)\to T$ be the universal family.  Let $T':=\pr_1(T)$ be the projection in the Hilbert scheme $\bH^{\chi;N}$. Let $\pi':\cX'\to T'$ be the universal family such that $\pi=\pi'\times_{T'} T$.
Let $H'\subset T'$ (resp. $H\subset T$) be the divisor parametrizing $\bP(1,1,4)$. By the proof of Proposition \ref{prop:Zplanecurve}(2) we know that $T$ and $T'$ are both smooth, and $\pr_1: T\to T'$ is a smooth morphism. Moreover, since $H'$ is a $\PGL(N+1)$-orbit in $\bH^{\chi;N}$, we know that $H'$ and $H$ are smooth prime divisors in $T'$ and $T$, respectively. 

Since $\pi'$ is a $\bP^2$-fibration over $T'\setminus H'$, there exists a dominant \'etale morphism $\widetilde{T}^\circ \to T'\setminus H'$ such that $\pi'\times_{T'} \widetilde{T}^{\circ}$ is a trivial $\bP^2$-bundle over $\widetilde{T}^{\circ}$. By Zariski's main theorem, there exists an open immersion $\widetilde{T}^{\circ}\hookrightarrow \widetilde{T}$ to a smooth variety $\widetilde{T}$ together with a quasi-finite morphism $\widetilde{T}\to T'$ \'etale away from $H'$ whose image contains the generic point of $H'$. In particular, $\widetilde{T}$ is flat over $T'$ by miracle flatness. Since both $T'\setminus H'$ and $H'$ are $\PGL(N+1)$-orbits, there exists $T_i'=g_i\cdot \widetilde{T}$ where $g_i\in\PGL(N+1)$ such that $\sqcup_{i} T_i'\to T'$ is a fppf covering. Denote by $H_i'$ the preimage of $H'$ in $T_i'$. Then from the above discussion we see that $\pi'\times_{T'} (T_i'\setminus H_i'):\cX'_{T_i'\setminus H_i'}\to T_i'\setminus H_i'$ is a trivial $\bP^2$-bundle. Let $\cL_i'$ be the Weil divisorial sheaf on $\cX'_{T_i'}$ as 
the Zariski closure of $\cO(1)$ on $\cX'_{T_i'\setminus H_i'}$. Since the families here are $\bQ$-Gorenstein with integral fibers, we know that $\cL_i'^{[-3]}$ is the same as $\omega_{\cX'_{T_i'}/T_i'}$ twisted by the pull-back of some line bundle on the base $T_i'$. After replacing $T_i'$ by its Zariski covering, we may assume that $\cL_i'^{[-3]}\cong \omega_{\cX'_{T_i'}/T_i'}$. By Kawamata-Viehweg vanishing, we know that $(\pi'_{T_i'})_* \cL_i'$ is a rank $3$ vector bundle over $T_i'$.

We claim that the $\PGL(3)$-torsors $\{\sP_i'/T_i'\}_i$, by taking a projectivized basis of $(\pi'_{T_i'})_* \cL_i'$, is a descent datum. Indeed, from the above construction we know that the cocycle condition of $\cL_i'$ is off by a third root of unity. Hence the cocycle condition of $(\pi'_{T_i'})_* \cL_i'$ is off by a scalar, which implies that a projectivized basis of $(\pi'_{T_i'})_* \cL_i'$  satisfies the cocycle condition. Hence by the fppf descent of $G$-torsors \cite[Tag 04U1]{stacksproject}, the $\PGL(3)$-torsors $\{\sP_i'/T_i'\}_i$ descend to a $\PGL(3)$-torsor $\sP'/T'$. Pulling back to $T$ we get a $\PGL(3)$-torsor $\sP/T$ where over $T\setminus H$ it is obtained by taking projectivized basis of $\pi\times_T (T\setminus H)$. It is clear that $\sP\to T$ is $\PGL(N+1)$-equivariant. Hence the morphism $\varphi$ is induced by the $\PGL(3)$-equivariant morphism $\sP\to U$ where $(t,[s_0,s_1,s_2])\mapsto [s_0,s_1,s_2](\cD_t)\subset \bP^2$. 
\end{proof}

From Lemma \ref{lem:k-to-git} and the separatedness of $\hcU^{\rm ss}\to U$, we know that to check $\psi$ is representable, separated, and quasi-finite, it suffices to show that the restriction of $\psi$ on $[\cE^{\rm ss}/\PGL(3)]$ maps isomorphically onto $[H/\PGL(N+1)]$ where $\cE^{\rm ss}:=\cE\cap \hcU^{\rm ss}$. To prove this, we will construct an inverse morphism $\psi^{-1}: [H/\PGL(N+1)]\to [\cE^{\rm ss}/\PGL(3)]$. We will focus on the case when $d$ is even, as the strategy for $d$ odd is similar. By Theorem \ref{thm:firstwall1}, we know that $\cE\cong \cE_W\times_{\PGL(2)}\PGL(3)$ where $\PGL(2)$ is identified with $\Aut(\bP^2, Q_d)$ as a subgroup of $\PGL(3)$. Hence we know that $[\cE^{\rm ss}/\PGL(3)]\cong [\cE_W^{\rm ss}/\PGL(2)]$. From Theorem \ref{thm:firstwall1} we know that $\cE_W\cong[(\bfA_d'\setminus\{0\})/\bG_m]$ is the weighted projective stack. 
Let us consider the action of $\GL(2)\times \bG_m$ on $\bfA_d'$ where $\GL(2)$ acts on $H^0(\bP^1,\cO_{\bP^1}(4j))$ as the symmetric power of the standard $\GL(2)$-action on $(\bP^1,\cO(1))$, and $\bG_m$ acts as $\sigma$. Consider a $1$-PS $\tau:\bG_m\to \GL(2)\times \bG_m$ defined as $\tau(t)=(\mathrm{diag}(t,t), t^4)$, then it is clear that $\tau$ acts as identity on $\bfA_d'$. Consider the group $\cG:=\GL(2)\times\bG_m/_{\tau}\bG_m$. It is clear that the quotient $\bar{\sigma}$ of $\sigma$ gives a $1$-PS in $\cG$ and $\cG/_{\bar{\sigma}}\bG_m\cong \PGL(2)$. Hence we know that $[\cE_W/\PGL(2)]\cong [(\bfA_d'\setminus\{0\})/\cG]$. 

Next, we will construct a morphism $[H/\PGL(N+1)]\to [(\bfA_d'\setminus\{0\})/\cG]$ which directly induces $\psi^{-1}$. Indeed, this reduces to construct a $\PGL(N+1)$-equivariant $\cG$-torsor $\sP_H/H$ and a $\PGL(N+1)$-invariant morphism $\sP_H\to \bfA_d'\setminus\{0\}$. Let $\pi_H:(\cX_H,\cD_H)\to H$ be the universal family where each fiber is isomorphic to $\bP(1,1,4)$ with a degree $2d$ curve. By similar argument to the proof of Lemma \ref{lem:k-to-git}, there exists an \'etale covering $\sqcup_i H_i\to H$ and a Weil divisorial sheaf $\cL_{H_i}$  on $\cX_{H_i}$ (as fiberwise $\cO(1)$ on $\bP(1,1,4)$) such that $\cL_{H_i}^{[-6]}\cong \omega_{\cX_{H_i}/H_i}$. Let $\cF_i:= (\pi_{H_i})_*\cL_i$ as a rank $2$ vector bundle on $H_i$. Define $\cF_i':= (\pi_{H_i})_*\cL_i^{[4]}/\Sym^4\cF_i$ as a line bundle on $H_i$. Then taking basis $(s_0,s_1)$ and $s_2$ of $\cF_i$ and $\cF_i'$ resp. gives a $\GL(2)\times\bG_m$-torsor $\widetilde{\sP}_{H_i}/H_i$. By projectivizing $(s_0,s_1,s_2)\mapsto (ts_0,ts_1,t^4 s_2)$ under $\tau$, we get a $\cG$-torsor $\sP_{H_i}/H_i$. Since the descent datum of $\{\cL_i\}$ is off by a sixth root of unity, it is easy to see that $\{\sP_{H_i}/H_i\}_i$ form an \'etale descent datum hence descend to a $\cG$-torsor $\sP_H/H$. As the $\sqcup_i H_i$ and $\cL_{H_i}$ can be chosen $\PGL(N+1)$-equivariantly, we know that $\sP_H/H$ is also $\PGL(N+1)$-equivariant. 
Notice that a  projectivized basis $[s_0,s_1,s_2]$ in $\sP_H$ is the same as an equivalent class of projective coordinates $[x,y,z]$ for $\bP(1,1,4)$ under the equivalence relation $z\sim z+g(x,y)$. For a point $t\in H$ and a projectivized basis $[s_0,s_1,s_2]$ lying over $t$, there is a unique projective coordinates $[x,y,z]$ in the equivalent class such that $\cD_t$ has the form \eqref{eq:even-p114}. This gives the $\PGL(N+1)$-invariant morphism $\sP_H\to \bfA_d'\setminus\{0\}$ whose image is contained in the GIT semistable locus by Theorem \ref{thm:k=gitp114}. Thus the proof is finished.
\end{proof}

\begin{proof}[Proof of Theorem \ref{thm:firstwallafter}]
The theorem follows from Theorems \ref{thm:k=gitp114}, \ref{thm:firstwall1}, and \ref{thm:firstwall2}.
\end{proof}

We have thus completed the proof of Theorem \ref{mthm:firstwall}.

\begin{proof}[Proof of Theorem \ref{mthm:firstwall}]
The proof follows from Theorems \ref{thm:firstwallbefore}, \ref{thm:firstwallon}, and \ref{thm:firstwallafter}.
\end{proof}

\section{K-moduli spaces of plane quartics and sextics and K3 surfaces}\label{sec:lowdegree}

In the first section we show that the wall crossing discussed in Section \ref{sec:firstwall} is the only wall crossing in the log Fano region for $d = 4$ and $6$. Using that, we relate the K-moduli spaces to certain moduli spaces of K3 surfaces.

\subsection{K-moduli wall crossings}
The goal of this section is to prove the following theorem.

\begin{theorem}\label{thm:quartsext} Assume $d=4$ or $6$. Then for any $\frac{3}{2d}<c<\frac{3}{d}$,
we have $\ocP^\K_{d,c}\cong \ocP^\K_{d,\frac{3}{2d}+\epsilon}$
as Artin stacks. In other words, there is only one wall crossing in the log Fano region.
\end{theorem}

\begin{rem}
 Note that the K-moduli space $\oP_{4,\frac{1}{2}}^{\K}$ was previously described by Odaka, Spotti, and Sun \cite{OSS16} in their study of K-moduli spaces of del Pezzo surfaces of degree $2$. 
\end{rem}

Recall in Theorem \ref{mthm:firstwall}, we constructed $\ocP^\K_{d, c_1 + \epsilon}$ as the partial Kirwan blow-up of $\ocP_d^{\GIT}$. Therefore, by Proposition \ref{prop:lclastwall},  to prove the above theorem it suffices to show that $\lct(X;C) \geq 3/d$ for any K-polystable point $[(X,(c_1+\epsilon)C)]\in \oP_{d,c_1+\epsilon}^{\K}$
for $d=4$ or $6$. Using an explicit description
of such GIT polystable points, we will verify the lct inequality as follows. Here we consider GIT of curves on $\bP(1,1,4)$ in the sense of Definition \ref{defn:gitp114}.

\begin{prop}[Degree $d = 4$]\label{prop:lctquartic}\leavevmode
 \begin{enumerate}
  \item Any GIT polystable plane curve of degree 4 that is not the double conic
  has $\lct \geq 3/4$.
  \item Any GIT polystable curve of degree $8$ in 
  $\bP(1,1,4)$ has $\lct \geq 3/4$.
 \end{enumerate}
\end{prop}

\begin{proof}
 (1) Recall that a plane quartic curve
 is GIT stable if and only if it has type $A_1$ or $A_2$ singularities.
 Moreover, the only reduced strictly GIT polystable quartic curves
 are the ``cat-eye'' and ``ox'' which have type singularities of type $A_3$ and possibly $A_1$  (see \cite[Table on Page 80]{MFK94} and \cite[Proposition 7]{hyeon2010log}).
 
 (2) Any GIT semistable curve $C$ of degree $8$ in $\bP(1,1,4)$
 is given by the equation $z^2=f(x,y)$ where $f\neq 0$ is a degree $8$ polynomial whose roots have multiplicities at most $4$.
 Therefore $C$ has at worst singularities of type $A_3$.
\end{proof}

\begin{prop}[Degree $d=6$]\label{prop:lctsextic} \leavevmode
\begin{enumerate}
   \item Any GIT polystable plane curve of degree 6 that is 
   not the triple conic  has $\lct \geq 1/2$.
  \item Any GIT polystable curve of degree $12$ in 
  $\bP(1,1,4)$ has $\lct \geq 1/2$.
  \end{enumerate}
  
\end{prop}

\begin{proof}
 (1) The GIT polystable plane sextics are classified in
 \cite[Theorem 2.4]{Sha80}. In Shah's terminology, 
 we only need to check that all curves in Group I, II, or
 III have $\lct\geq 1/2$. Group I are ADE singularities,
 so they have $\lct>1/2$. Group II  and III both have $\lct=1/2$.
 
 (2) This follows from \cite[Theorem 4.3]{Sha80}.
 In Shah's terminology, Case 1(i) corresponds to ADE singularities
 so $\lct>1/2$. Case 1(ii) has equations
 $(z^3+a(xy)^4 z+b(xy)^6=0)$ where $a,b$ are not simultaneously zero, and the $\lct=1/2$ in this case.
 Case 2 is similar.
\end{proof}

\begin{proof}[Proof of Theorem \ref{thm:quartsext}] As mentioned above, the proof follows from Theorem \ref{mthm:firstwall}, Proposition \ref{prop:lclastwall}, Theorem \ref{thm:firstwallps}, and the two above propositions explicitly calculating the lct of the GIT polystable curves. \end{proof}

\subsection{Relating degree 4 and 6 plane curves to K3 surfaces}\label{sec:K3surface} In this section, we describe a relation between K-moduli spaces of plane curves of degree $4$ and $6$ and certain Baily-Borel compactifications of moduli spaces of K3 surfaces which already appear in the literature. In the case of quartics, we use work of Hyeon and Lee \cite{hyeon2010log} and Kond\=o \cite{kondok3}. In the case of sextics, we use work of Shah \cite{Sha80} and Looijenga \cite{Loo03} (see also Laza \cite{laza2012ksba}). 

\subsubsection{Plane quartics}
We recall that the GIT quotient $\oP_4^{\GIT}$ generically parametrizes curves in $\bP^2$ with at worst cuspidal singularities. There is a curve parametrizing plane curves with a tacnode (locally $(x^2 + y^4=0)$), and there is a point on this curve parametrizing the double conic. 
In \cite{hyeon2010log}, the authors construct two GIT moduli spaces which do not coincide with the standard GIT quotient $\oP^\GIT_4$. In particular, they construct $\oM_3^{\rm hs} := \Hilb_{3,2} \quotient \SL(6)$ and $\oM_3^{\rm cs} := \Chow_{3,2} \quotient \SL(6)$, where $\Hilb_{3,2}$ (resp. $\Chow_{3,2}$) denotes the closure of the locus of bicanonical curves of genus three in the Hilbert scheme (resp. Chow scheme). They then show the existence of the following diagram (see \cite[Theorem 1 and Page 4]{hyeon2010log}):
\[
\oM_3^{\rm cs}\xleftarrow{~~\Psi^+~~}\oM_3^{\rm hs} \xrightarrow{~~\Theta~~} \oP^{\GIT}_4
\]
where $\Theta$ is a divisorial contraction corresponding to the blowup of the point parametrizing the double conic in $\oP^{\GIT}_4$ and $\Psi^+$ is a small contraction identifying all tacnodal curves in $\oM_3^{\rm hs}$ to the same point in $\oM_3^{\rm cs}$. By Theorem \ref{mthm:firstwall} and Theorem \ref{thm:quartsext}, we have $\overline{M}^{\rm hs}_3 \cong \oP^\K_{4,\frac{3}{4}-\epsilon}$.
By construction, the space $\oP^{\rm K}_{4, \frac{3}{4}-\epsilon}$ has a divisor parametrizing curves on $\bP(1,1,4)$ (the exceptional divisor of the weighted blowup of the double conic), along with a curve which still parametrizes the tacnodal curves. This is the curve which is contracted via $\Psi^+$ to a point in $\oM_3^{\rm cs}$. 

In \cite{kondok3}, Kond\=o constructs a moduli space of K3 surfaces by considering $\bZ/4\bZ$-cover of $\bP^2$ branched along a quartic curve. Kond\=o's moduli space $\oP^*_4$ is a Baily-Borel compactification of the moduli space $M$ of ADE K3 surfaces of degree $4$ with $\bZ/4\bZ$-symmetry. The boundary $\oP_4^*\setminus M$ is a single point. Hyeon and Lee prove (see \cite[Proposition 21]{hyeon2010log}) that $\oP^*_4 \cong \oM_3^{\rm cs}$ and identify the point in the boundary of $\oP^*_4$ with the locus of tacnodal curves (i.e. the image of the tacnodal curves in $\oM_3^{\rm hs}$ under the small contraction $\psi^+$). We now prove that the moduli space $\oP^*_4$ of K3 surfaces is the ample model of the Hodge line bundle (see Proposition \ref{prop:k-cm-interpolation}) on $\oP^\K_{4, 3/4 - \epsilon}$, thus relating our K-moduli to a moduli space of K3 surfaces. 
 
 \begin{theorem}\label{thm:logcy4}The moduli space $\oP^*_4$ is the ample model of the Hodge line bundle on $\oP_{4,3/4-\epsilon}^\K$. \end{theorem}

\begin{proof}
For simplicity, denote by $\lambda_{\Hodge}$ the Hodge line bundle on $\oP_{4,3/4-\epsilon}^\K$.
Let $M$ be the open subset of $\oP_4^*$ parametrizing ADE K3 surfaces. Then by taking $\bZ/4\bZ$-quotient it is clear that $M$ also parametrizes quartic curves on $\bP^2$ and degree $8$ curves on $\bP(1,1,4)$ with at worst $A_2$-singularities. Indeed, the moduli stack $\cM$ associated to $M$ is a $(\bZ/4\bZ)$-gerbe over $\cP_{4, \frac{3}{4} - \epsilon}^{\klt}$. Thus $M$ can be identified with the open subset $P_{4, \frac{3}{4} - \epsilon}^{\klt}$ of $\oP^{\K}_{4, \frac{3}{4} - \epsilon}$ whose complement has codimension $\geq 2$. By \cite[Section 6.2]{huybrechts}, $\lambda_{\Hodge}|_M$ is pulled back from the Hodge line bundle on the relevant period domain $\mathbb{D}/\Gamma$, which is the descent of $\calO(-1)$, and thus ample. Since $\oP_4^*$ is the Baily-Borel compactification of $M$, we know that $\lambda_{\Hodge}|_M$ extends to an ample $\bQ$-line bundle on $\oP_4^*$.
By \cite{CP18} we know that $\lambda_{\Hodge}$ is nef on $\oP^{\K}_{4,\frac{3}{4}-\epsilon}$. Since $M$ are big open subsets in both $\oP^{\K}_{4,\frac{3}{4}-\epsilon}$ and $\oP_4^*$, we know that $\lambda_{\Hodge}$ is big and semiample, and $\oP^*_4 \cong  \Proj \big( \bigoplus_{k=0}^\infty H^0(\oP^{\K}_{4,\frac{3}{4}-\epsilon}, \lambda_{\Hodge}^{\otimes k})\big)$ is the ample model of $\lambda_{\Hodge}$. This finishes the proof.

Finally, we remark here that an alternative proof can be obtained using \cite[Theorem 1.2]{fujino2003} and  functoriality of the Hodge line bundle. 
\end{proof}

We note that Hyeon-Lee's paper also studies the spaces $\oM_3^{\rm cs}, \oM_3^{\rm hs}, \oP^*_4,$ and $\oP^{\GIT}_4$ in the context of the log minimal model program on $\oM_3$ (i.e. the Deligne-Mumford compactification) as well as their relations with Hacking's $\oP^{\rm H}_4$. We postpone discussing this viewpoint until Section \ref{sec:logcy}.

\subsubsection{Plane sextic curves}
For sextic curves, we recall that Shah constructed a partial Kirwan desingularization of the GIT quotient of plane sextic curves \cite{Sha80}. In particular as above there is a divisorial contraction $\widehat{P}^{\GIT}_6 \to \oP^{\GIT}_6$ corresponding to a weighted blowup of the triple conic. Shah also constructed a set-theoretic morphism $\widehat{P}^{\GIT}_6 \to \oP^*_6$, where $\oP^*_6$ denotes the Baily-Borel compactification of the space of polarized K3 surfaces of degree two. This map was shown to be algebraic by Looijenga \cite{Loo86, Loo03} (see also \cite[Theorem 1.9]{laza2012ksba}). In particular, we have a similar diagram in the sextic case. 
\[
 \oP^*_6\longleftarrow{}\widehat{P}^{\GIT}_6\longrightarrow{}\oP^{\GIT}_6
\]
Again using Theorems \ref{mthm:firstwall} and \ref{thm:quartsext}, we can identify $\widehat{P}^{\GIT}_6 \cong \oP^{\K}_{6, \frac{1}{2}-\epsilon}$. The proof of Theorem \ref{thm:logcy4} gives the following, noting that the codimension of the klt locus inside $\oP^\K_{6, \frac{1}{2}-\epsilon}$ is $\geq 2$. 

\begin{theorem}\label{thm:sextic-Hodge}
The moduli space $\oP^*_6$ is the ample model of the Hodge line bundle on  $\oP_{6, \frac{1}{2}-\epsilon}^{\K}$. \end{theorem}

As in the case of quartics, we will discuss the relation of the above picture with Hacking's $\oP^{\rm H}_6$ and $\oP^{\ast}_6$ in Section \ref{sec:logcy}.

\section{The second wall crossing for plane quintics}\label{sec:2ndwallquintics}

 In this section we discuss the second wall crossing for K-moduli spaces of plane quintics. For simplicity, we abbreviate $\ocP_{5,c}^{\K}$ and $\oP_{5,c}^{\K}$ to $\ocP_{c}^{\K}$ and $\oP_{c}^{\K}$, respectively. The main result goes as follows.
 
\begin{thm}[Second wall crossing for plane quintics]\label{thm:secondwall}
 Let $C_0$ be a plane quintic curve with a singular point of type $A_{12}$. Denote by $X_{26}:=(xw-y^{13}-z^2=0)\subset\bP(1,2,13,25)$. Let $C_0'$ be the curve $(w=0)$ on $X_{26}$.
 \begin{enumerate}
     \item There is no wall crossing for K-moduli stacks $\ocP_{c}^{\K}$ when $c\in (\frac{3}{7},\frac{8}{15})$.
     \item There is an isomorphism of good moduli spaces $\phi_2^-:\oP_{\frac{8}{15}-\epsilon}^{\K}\to\oP_{\frac{8}{15}}^{\K}$ which only replaces $[(\bP^2, C_0)]$ by $[(X_{26}, C_0')]$.
     \item There is a weighted blow-up morphism $\phi_2^+:\oP_{\frac{8}{15}+\epsilon}^{\K}\to\oP_{\frac{8}{15}}^{\K}$ at the point $[(X_{26}, C_0')]$. The exceptional divisor of $\phi_2^+$ parametrizes curves on $X_{26}$ of the form $(w=g(x,y))$ where $g\neq 0$ and $g$ does not contain the term $x y^{12}$.
 \end{enumerate}
 In particular, the second wall for K-moduli spaces of plane quintics is $c_2=\frac{8}{15}$.
\end{thm}

We will split the proof of Theorem \ref{thm:secondwall} into several steps.  

\subsection{K-polystable replacement of $A_{12}$-quintic curve}

Let $C_0$ be a plane quintic curve with a singular point of type $A_{12}$. Then by \cite{Yoshihara, Wal96} we know that
up to a projective transformation, the equation of $C_0$ is 
\[
C_0 = \big((y^2-xz)^2\left(\tfrac{1}{4}x+y+z\right)-x^2(y^2-xz)(x+2y)+x^5=0\big).
\]
In the affine patch $[x,y,1]$, there
is a unique $6$-jet $x'=x-y^2+y^5-\frac{1}{2}y^6$ so that
the equation of $C$ in the coordinates $(x',y)$ becomes 
\[
x'^2=ay^{13}+\textrm{higher order terms, where }a\neq 0. 
\]
Here we assign weights $13$ and $2$ to $x'$ and $y$, respectively. Since $C_0$ has only one singularity which is a double point, we know that it is a GIT stable plane quintic curve by \cite[Table on Page 80]{MFK94}.

In this section, we will show that $\frac{8}{15}$ is the upper K-semistable threshold of $(\bP^2, C_0)$ by constructing its K-polystable degeneration. The goal is to prove the following. 

\begin{theorem}\label{thm:a12} The log Fano pair $(\bP^2, cC_0)$ is  K-semistable if and only if $0 < c \leq \frac8{15}$. Moreover, $(X_{26}, \frac{8}{15}C_0')$ is the K-polystable degeneration of $(\bP^2, \frac{8}{15}C_0)$.
\end{theorem} 

We prove this in steps. 

\begin{prop}\label{prop:a12onlyif}
If the log Fano pair $(\bP^2, cC_0)$ is K-semistable then $0 < c \leq \frac8{15}$. 
\end{prop}

\begin{proof}
Suppose $(\bP^2, cC_0)$ is K-semistable. Let us perform the $(13,2)$-weighted blow up in the coordinates $(x',y)$, and denote the resulting surface and exceptional divisor by $(X,E)$.
Let $\pi:X\to \bP^2$ be the weighted blow up morphism. 
Straightforward computation shows that 
\[
 A_{(\bP^2,cC_0)}(E)=15-26c, \quad -K_{\bP^2}-cC_0\sim_{\bQ}(3-5c)H,
\]
where $H\sim\cO(1)$ is the hyperplane divisor on $\bP^2$.
If $\oC_0:=\pi_*^{-1}C_0\subset X$, then we know that $\oC_0$ is a smooth
rational curve in the smooth locus of $X$.
It is easy to see that $(E^2)=-\frac{1}{26}$, the curve $\oC_0\sim 5\pi^*H-26E$,
and $(\oC_0^2)=-1$. Thus the Mori cone of $X$ is generated
by $E$ and $\oC_0$. Hence $\pi^*H - tE$ is ample if and only if $0<t<5$, and big if and only if $0\leq t<\frac{26}{5}$. 
Then by computations, 
\[
 \vol_X(\pi^*H -tE)=\begin{cases}
                      1-\frac{t^2}{26} & \textrm{ if }0\leq t\leq 5;\\
                      \frac{(26-5t)^2}{26} & \textrm{ if }5\leq t\leq \frac{26}{5}.
                     \end{cases}
\]
Hence $S_{(\bP^2, cC_0)}(E)=\int_{0}^\infty \vol_X(\pi^*H-tE)dt=\frac{17}{5}(3-5c)$.
Since $(\bP^2, cC_0)$ is K-semistable, by the valuative criterion (Theorem \ref{thm:valuative}) we know that 
\[
15-26c=A_{(\bP^2, cC_0)}(E)\geq S_{(\bP^2, cC_0)}(E)=\frac{17}{5}(3-5c). 
\]
This is equivalent to $c\leq \frac{8}{15}$. 
\end{proof}

Now we construct a special degeneration. 

\begin{prop}\label{prop:a12deg}
 The log Fano pair $(\bP^2,cC_0)$ admits
 a special degeneration to $(X_{26}, cC_0')$
 where $C_0'$ is given by the equation 
 $(w=0)$ with $[x,y,z,w]$ being the projective
 coordinates of $X_{26}$.
\end{prop}

\begin{proof}

We construct the special degeneration. Consider the family 
$(\bP^2,C_0)\times\bA^1$ and perform the following birational transformations,
\[
\begin{tikzcd}
 & \cX\arrow{ld}{\pi}\arrow{rd}{g}\arrow[rr,dashed,"f"]& & \cX^+\arrow{ld}{h}\arrow{rd}{\psi}\\
 \bP^2\times\bA^1& & \cY & & \cZ
 \end{tikzcd}
\]
where in the central fiber we have
\[
 \begin{tikzcd}
 & S\cup X\arrow{ld}{\pi}\arrow{rd}{g}\arrow[rr,dashed,"f"]& & \hS\cup X'\arrow{ld}{h}\arrow{rd}{\psi}\\
 \bP^2& & S\cup X' & & S'
 \end{tikzcd}
\]
\begin{enumerate}
\item $\pi$ is the $(13,2,1)$-weighted blow up of $\bP^2\times\bA^1$ in the coordinates $(x',y,t)$ where $t$ is the parameter of $\bA^1$,
\item the surface $S=\bP(1,2,13)$ is the exceptional divisor of $\pi$,  
\item the map $g$ is the 
contraction of $\oC_0$ in $X\subset\cX_0$ where $\oC_0$ is the strict transform of $C_0$ in $X$, 
\item the map $f$ is the Atiyah flop of
the curve $\oC_0$ in $\cX_0$ (by computation the normal bundle
$\cN_{\oC_0/\cX}\cong\cO_{\oC_0}(-1)\oplus\cO_{\oC_0}(-1)$), and
\item $\psi$ is the divisorial contraction that contracts $X'$ to a point.
\end{enumerate}

Let us analyze the geometry of these birational maps. Suppose
$S$ has projective coordinates $[x_1,x_2,x_3]$ of weights $(1,2,13)$ respectively.
Then $S\cap X=E=(x_1=0)$, and $\oC_0\cap E=\{p\}$ is a smooth point
of $S$ and $X$. So $h:\hS=Bl_{p}S\to S$, the map $g:X\to X'$ contracts
the $(-1)$-curve $\oC_0$, and $\psi: \hS\to S'$ contracts
$h_*^{-1}(E)$. A simple analysis of the singularity of $S'$ shows that $S'$ has only one singularity of type $\frac{1}{25}(1,4)$.

Let $F$ be the exceptional divisor of $h:\hS\to S=\bP(1,2,13)$. We may look
at the $\bQ$-divisor $D:=F+\frac{26}{25}h_*^{-1}E$. It is clear that the ample model of $D$ on $\hS$ is exactly $S'$. The projective coordinate ring $\oplus_{m\geq 0} H^0(\hS,\cO_{\hS}(\lfloor mD\rfloor))$ has four distinguished generators in degree $1$, $2$, $13$ and $25$,
corresponding to $x_1$, $x_2$, $x_3$ and $x_3^2+x_2^{13}$
on $S$. If we denote $x,y,z,w$ as these four generators,
then their relation  is $xw=z^2+y^{13}$ which shows $S'\cong X_{26}$. It is clear that $\pi_*^{-1} C_0\times\bA^1\cap S$ has the equation $x_2^{13}+x_3^2=0$. Hence the degeneration of $C_0$ on $X_{26}$ is the strict transform of the curve $(x_2^{13}+x_3^2=0)$ which is exactly $(w=0)$. This finishes the proof.
\end{proof}

 In the appendix we use techniques of Ilten and S\"u{\ss} \cite{IS17} to show that $(X_{26},\frac{8}{15}C_0')$ is indeed K-polystable (see Proposition \ref{prop:a12kps}). We now prove Theorem \ref{thm:a12}.

\begin{proof}[Proof of Theorem \ref{thm:a12}]
By Proposition \ref{prop:a12deg}, the log Fano pair $(\bP^2, \frac8{15}C_0)$ admits a special degeneration to $(X_{26}, \frac8{15} C^\prime_0)$, which is K-polystable by Proposition \ref{prop:a12kps}. Therefore, the pair $(\bP^2, \frac8{15}C_0)$ is K-semistable by Theorem \ref{thm:Kss-spdeg}. We then conclude that $(\bP^2, cC_0)$ is K-semistable for any $c\in (0,\frac{8}{15})$ using Propositions \ref{prop:k-interpolation} and \ref{prop:a12onlyif}.
\end{proof}

\subsection{Proof of second wall crossing}

In this section, we prove Theorem \ref{thm:secondwall}. Before presenting its proof, we provide several results that are needed. First we limit the surfaces that can appear.

\begin{lem}\label{lem:quinticindex}
 If $[(X,cD)]\in \ocP_{c}^{\K}$ for some $c\in (0,\frac{3}{5})$, then $X$ is isomorphic to one of the following surfaces: $\bP^2$, $\bP(1,1,4)$, $X_{26}$, or $\bP(1,4,25)$. 
\end{lem}

\begin{proof}
By the index estimate (Theorem \ref{thm:localindex}), we know that the Gorenstein index of any singular point on $X$ is at most $5$. Thus from the classification of Manetti surfaces, we know that $X$ is either isomorphic to $\bP^2$, $\bP(1,1,4)$, $\bP(1,4,25)$ or isomorphic to some of their partial smoothings. It is clear (e.g. from \cite{Hac04}) that $X_{26}$ is the only new surface appearing which is a partial smoothing.
\end{proof}

Next, we discuss K-stability of curves on $X_{26}$.

\begin{prop}\label{prop:valcrit-X_26}
 Let $C$ be a curve on $X_{26}$ of degree $25$. If $(X_{26},cC)$ is K-semistable, then $c\geq \frac{8}{15}$. If in addition that $C$ passes through the singular point of $X_{26}$, then $(X_{26},cC)$ is K-unstable for any $c\in (0, \frac{3}{5})$.
\end{prop}

\begin{proof}
 For simplicity we denote by $X:=X_{26}$.
 Let us consider the unique singular point $[0,0,0,1]$ on $X$. Denote by
 $[x,y,z,w]$ the projective coordinates where $X$
 is defined by $xw=y^{13}+z^2$. If we set $w=1$,
 then we have a cyclic quotient map $\pi:\bA_{(y,z)}^2\to X$
 defined by $\pi(y,z)=[y^{13}+z^2,y,z,1]$. Let $F$ be
 the exceptional divisor on $\bA_{(y,z)}^2$ given by the 
 $(2,13)$-weighted blow up. Let $E$ be the quotient of $F$
 over $X$. Then it is clear that $\ord_E=\pi_*\ord_F/25$,
 $(F^2)=-\frac{1}{26}$ and $(E^2)=-\frac{25}{26}$. Let $\Gamma$
 be the curve $x=0$ on $X$. Then $\ord_E(\Gamma)=\ord_F(\Gamma)/25=\frac{26}{25}$.
 Hence on the blow-up $\mu:Y\to X$ extracting $E$, the proper
 transform $\overline{\Gamma}$ of $\Gamma$ satisfies
 $\overline{\Gamma}=\mu^*\Gamma-\frac{26}{25}E$ and
 $(\overline{\Gamma}^2)=-1$. So the Mori cone of $Y$ 
 is generated by $E$ and $\overline{\Gamma}$. Computation
 shows
 \[
  \vol_X(\cO(1)-tE)=\begin{cases}
                     \frac{1}{25}-\frac{25}{26}t^2&\textrm{ if }0\leq t\leq \frac{1}{25}\\
                     \frac{1}{26}(\frac{26}{25}-t)^2&\textrm{ if }\frac{1}{25}\leq t\leq\frac{26}{25}
                    \end{cases}
 \]
 Thus 
 \[
 S_{(X,cC)}(\ord_E)=
 \frac{15-25c}{\vol_X(\cO(1))}\int_{0}^{\infty}\vol_X(\cO(1)-tE)dt  =\frac{9}{25}(15-25c).
 \]
 Since $A_X(\ord_E)=A_{\bA^2}(\ord_F)/25=15/25=3/5$ and $(X,cC)$ is K-semistable, the valuative criterion (Theorem \ref{thm:valuative}) yields 
 $\frac{3}{5}\geq A_{(X,cC)}(\ord_E)\geq S_{(X,cC)}(\ord_E)=\frac{9}{25}(15-25c)$
 which implies $c\geq \frac{8}{15}$.
 
 If $C$ passes through $[0,0,0,1]$, then its equation is given by $(f(x,y)z+g(x,y)=0)$ where $\deg f=12$ and $\deg g=25$. Let $\widetilde{C}$ be the preimage of $C$ under $\pi$. Then $\widetilde{C}$ has equation
 \[
 f(y^{13}+z^2, y)z+g(y^{13}+z^2, y)=0.
 \]
 Then by simple calculation we see that $\ord_E(C)=\ord_F(\widetilde{C})/25\geq 1$. Thus we have
 \[
 A_{(X,cC)}(\ord_E)\leq \frac{1}{5}(3-5c)<\frac{9}{25}(15-25c)=S_{(X,cC)}(\ord_E).
 \]
 This implies that $(X,cC)$ is always K-unstable for $c\in (0,\frac{3}{5})$ by the valuative criterion (Theorem \ref{thm:valuative}). The proof is finished.
\end{proof}

\begin{prop}\label{prop:onlykps-X_26}
 Let $C$ be a curve on $X_{26}$ of degree $25$. Then 
 $(X_{26}, \frac{8}{15}C)$ is K-semistable if and only if $C$ does not pass through the unique singular point of $X_{26}$. Moreover, $(X_{26}, \frac{8}{15}C)$ is K-polystable if and only if  $C\cong C_0'$ under an automorphism of $X_{26}$. 
\end{prop}

\begin{proof}
We first look at the K-semistable statement. The ``only if'' part holds by Proposition \ref{prop:valcrit-X_26}. For the ``if'' part, suppose $C$ does not pass through $[0,0,0,1]$. Hence the equation of $C$ is given by $w=f(x,y)z+g(x,y)$. Consider the $1$-PS in $\Aut(X_{26})$ defined by $[x,y,z,w]\mapsto [t^{26}x,t^{2}y, t^{13}z, w]$. It is clear that $C$ specially degenerates to $C_0'=(w=0)$ via this $1$-PS as $t\to\infty$. Hence $(X_{26},\frac{8}{15}C)$ is K-semistable by Theorem \ref{thm:Kss-spdeg} and the K-polystability of $(X_{26},\frac{8}{15}C_0')$ (see Proposition \ref{prop:a12kps}). The K-polystable statement follows by uniqueness of K-polystable degenerations \cite{LWX18}.
\end{proof}

Now we are ready to prove Theorem \ref{thm:secondwall}.

\begin{proof}[Proof of Theorem \ref{thm:secondwall}]
(1) Assume to the contrary that there are wall crossings within the interval $(\frac{3}{7}, \frac{8}{15})$. Let $c_2\in(\frac{3}{7}, \frac{8}{15})$ be the second wall. Then there exists a new K-polystable pair $(X, c_2 D)$ such that $(X,cD)$ is K-unstable for any $c\neq c_2$ by Proposition \ref{prop:openkps}. Thus  $X$ is not isomorphic to $X_{26}$ or $\bP(1,4,25)$ by Propositions \ref{prop:valcrit-X_26} and \ref{prop:valcrit-p1425} (the valuative criterion for curves on $\bP(1,4,25)$) since $c_2<\frac{8}{15}$. By Lemma \ref{lem:quinticindex}, we are left with two possibilities, i.e. $X$ is isomorphic to $\bP^2$ or $\bP(1,1,4)$. If $X\cong\bP^2$, then Proposition \ref{prop:k-interpolation} implies that $(X, \epsilon D)$ is also K-polystable which is a contradiction. Hence we may assume $(X,D)\cong (\bP(1,1,4),C)$.

Assume the equation of $C$ is given by $f(x,y)z^2+g(x,y)z+h(x,y)=0$. If $f(x,y)$ is a non-degenerate quadratic form, then it is clear that $C$ specially degenerates to $Q_5'$. Since $(\bP(1,1,4), \frac{3}{7} Q_5')$ is K-polystable by Lemma \ref{lem:Q_d'Kps}, we know that $(\bP(1,1,4),\frac{3}{7}C)$ is K-semistable by Theorem \ref{thm:Kss-spdeg}. But this is a contradiction since $(X,\frac{3}{7}D)$ is K-unstable. Thus $f(x,y)$ is a degenerate quadratic form. By Proposition \ref{prop:valcrit-p114} (the valuative criterion for curves on $\bP(1,1,4)$), we know that $c_2\geq \frac{6}{11}>\frac{8}{15}$, a contradiction. This finishes the proof of part (1).

(2) From Theorem \ref{thm:a12}, we know that $\phi_2^-$ replaces $[(\bP^2,C_0)]$ with $[(X_{26}, C_0')]$. By Proposition \ref{prop:onlykps-X_26}, we know that $(X_{26},\frac{8}{15}C_0')$ is the only new K-polystable pair appearing in $\oP_{\frac{8}{15}}^{\K}$. Hence to show $\phi_2^-$ is an isomorphism it suffices to show that the preimage of $[(X_{26}, C_0')]$ under $\phi_2^{-}$ is exactly $[(\bP^2,C_0)]$. Denote by $E_2^{\pm}$ the preimage of $[(X_{26}, C_0')]$ under the morphisms $\phi_2^{\pm}$.
Assume to the contrary that $E_2^-$ contains at least two points. Since all K-moduli spaces of plane curves are normal by Proposition \ref{prop:Zplanecurve}, we know that $E_2^-$ is connected hence has positive dimension. 
It is clear that $\Aut_0(X_{26},C_0')\cong\bG_m$. 
Let $U_W$ be the Luna slice at the point $z_0=\Hilb(X_{26},\frac{8}{15}C_0')\in Z_{\frac{8}{15}}^{\circ}$ satisfying Theorem \ref{thm:wallcrossingchart}. Hence we know that $z_0$ is the only $\bG_m$-invariant point in $U_W$. The smoothness of $Z_c^{\circ}$ (Proposition \ref{prop:Zplanecurve}) implies that $U_W$ is also smooth. Applying \cite[Theorem 0.2.5]{dolgachevhu} or \cite[Corollary 1.13]{thaddeus} to the local VGIT presentation  \eqref{eq:localVGIT} near $z_0$ implies that
\[
\dim(E_2^-)+\dim(E_2^+)+1=\dim(\oP_{\frac{8}{15}}^{\K})=12.
\]
In particular, we know that the locus $E_2^+$ has codimension at least two in the K-moduli space. However, we will show that this is not true.

Let $C$ be a curve on $X_{26}$ of the form $(w=g(x,y))$ where $g\neq 0$ and $g$ does not contain the term $x y^{12}$. Hence $C$ does not pass through the singular point $[0,0,0,1]$ of $X_{26}$. It is clear that $C$ is a smooth curve on $X_{26}$ for a general choice of $g$. Hence $(X_{26}, \frac{3}{5}C)$ is klt for a general choice of $g$. Since $(X_{26}, \frac{8}{15}C)$ is K-semistable by Proposition \ref{prop:onlykps-X_26}, we know that $(X_{26}, (\frac{8}{15}+\epsilon)C)$ is K-stable for a general $g$.
In the affine chart $x=1$, the equation of $C$ becomes
\begin{equation}\label{eq:hyperelliptic}
z^2=-y^{13}+a_{11}y^{11}+a_{10} y^{10}+\cdots+ a_0.
\end{equation}
Hence $C$ is a hyperelliptic curve of arithmetic genus six. Indeed, from the above discussion we see that any smooth hyperelliptic curve $C$ of genus six produces a K-stable pair $(X_{26},(\frac{8}{15}+\epsilon)C)$ in the K-moduli space $\oP_{\frac{8}{15}+\epsilon}^{\K}$. Since the moduli space of smooth hyperelliptic curves of genus six has dimension $11$, we know that $E_2^+$ has dimension at least $11$ which contradicts to the assumption that $\phi_2^-$ is not an isomorphism. This finishes the proof of part (2).

(3) From the local VGIT argument in part (2), we know that $\phi_2^+$ is a weighted blow up since $\phi_2^-$ is an isomorphism (see \cite[Theorem 0.2.5]{dolgachevhu} or \cite[Corollary 1.13]{thaddeus}). 
If $[(X,D)]$ is a point in $E_2^+$, then it admits a special degeneration to  $(X_{26},C_0')$. Thus $X$ is either $\bP^2$ or $X_{26}$. However, if $X\cong\bP^2$ then $(X,\frac{8}{15} D)$ is K-polystable as well by Proposition \ref{prop:k-interpolation}, a contradiction. Hence $X\cong X_{26}$ and $D$ does not pass through the singular point on $X$ by Proposition \ref{prop:valcrit-X_26}. Thus after a suitable change of coordinates we can put $D$ into the form in the statement. Note that $g\neq 0$ is because otherwise $(X,D)\cong (X_{26},C_0')$ is $(\frac{8}{15}+\epsilon)$-K-unstable.

To prove the rest of part (3), it suffices to show that  $(X_{26},(\frac{8}{15}+\epsilon)C)$ is K-polystable for any curve $C$ on $X_{26}$ described in the statement. Given such a curve $C$, we may find a family of smooth hyperelliptic curves $D_t$ on $X_{26}$ of the same form over a punctured smooth curve $T\setminus\{0\}$ such that $\lim_{t\to 0} D_t=C$. Let $(X,(\frac{8}{15}+\epsilon)D)$ be the K-polystable limit of $(X_{26}, (\frac{8}{15}+\epsilon)D_t)$ using Theorem \ref{thm:compactness}. Since the Gorenstein index of $X$ is a multiple of $X_{26}$ which is $5$, we know that $X$ can only be $X_{26}$ or $\bP(1,4,25)$. But $\bP(1,4,25)$ is impossible by Proposition \ref{prop:valcrit-p1425}. Hence $X\cong X_{26}$ and $D$ is a curve not passing through the singular point $[0,0,0,1]$. By an automorphism of $X_{26}$, we may assume that $D$ has the equation $(w=h(x,y))$ where $h$ is a homogeneous polynomial of degree $25$. Since $(X_{26},(\frac{8}{15}+\epsilon)C_0')$ is K-unstable by Proposition \ref{prop:openkps}, after a further change of coordinates we may assume that $h\neq 0$ and $h$ does not contain the term $xy^{12}$. Thus we conclude that $(X, D)\cong (X_{26}, C)$. The proof is finished.
\end{proof}

The following result follows easily from the proof of Theorem \ref{thm:secondwall} (see e.g. \eqref{eq:hyperelliptic}).

\begin{cor}\label{cor:hyperellipticgenus6}
The moduli stack of smooth hyperelliptic curves $C$ of genus six  with a marked Weierstrass point $p$ admits a locally closed embedding into the K-moduli stack $\ocP_{\frac{8}{15}+\epsilon}^{\K}$. Moreover, this embedding is stabilizer preserving and sends closed points to closed points. In particular, the coarse moduli space of such pairs $(C,p)$ admits a locally closed embedding into the K-moduli space $\oP_{\frac{8}{15}+\epsilon}^{\K}$ whose image closure is the exceptional divisor of $\phi_2^{+}$.
\end{cor}

Corollary \ref{cor:hyperellipticgenus6} is a strengthening of an earlier result of \cite{griffin}.  Although not explicitly stated, it is a consequence of \cite[Theorem 1, Theorem 1.A]{griffin} that a smooth hyperelliptic curve $C$ of genus six admits an embedding into $X_{26}$ coming from the marked Weierstrass point $p$.  Indeed, the author computes an embedding $C \hookrightarrow \PP(1,1,1,2,3,3)$ by 
\[ C \cong \Proj(R(C, \mathcal{O}(5p)) \cong \Proj~ \mathbb{C}[x_1,x_2,x_3,y,z_1,z_2]/I  \]
where $x_i$ has weight 1, $y$ weight 2, and $z_i$ weight 3 and $I$ depends on a uniquely determined degree $5$ polynomial $Q$.  One can show that $X_{26}$ admits an embedding into $\PP(1,1,1,2,3,3,5)$ such that, if $t$ is the variable of weight $5$, the curve $C \subset X_{26}$ is cut out by the equation $t = Q$.  

\subsection{Applications to higher degree}\label{sec:higherdegree}
It is natural to ask what can be said about wall crossing beyond the first wall in higher degree, and we address that now. The key observation is the following proposition (see Definition \ref{def:delta} for the definition of $\delta$).

\begin{prop}\label{prop:deltax26}
 We have $\delta(X_{26})=\frac{1}{9}$.
\end{prop}

\begin{proof}
 For simplicity denote by $X:=X_{26}$. We follow notation from Proposition \ref{prop:valcrit-X_26}. Consider the valuation $\ord_E$ centered over the unique singular point of $X$. Since $-K_X\sim_{\bQ}\cO_X(15)$, by Proposition \ref{prop:valcrit-X_26} we have
 \[
 A_X(\ord_E)=\frac{3}{5}, \quad S_X(\ord_E)= \frac{15}{\vol_X(\cO(1))}\int_{0}^{\infty}\vol_X(\cO(1)-tE)dt =\frac{27}{5}.
 \]
 Hence we have $ \delta(X)\leq A_X(\ord_E)/S(\ord_E)=\frac{1}{9}$.
 Assume to the contrary that $\delta(X)<\frac{1}{9}$. From Proposition \ref{prop:a12kps}, we know that $(X,\frac{8}{15}C_0')$ is K-polystable where $C_0'=(w=0)$.  On the other hand, \cite[Theorem 7.2]{BL18b} implies that $(X,\frac{8}{15}C_0')$ is K-unstable by taking $\beta=\frac{1}{9}$ and $D=\frac{3}{5}C_0'$. This is a contradiction. Therefore, $\delta(X)=\frac{1}{9}$.
\end{proof}

Using the above proposition, we can prove the following.

\begin{thm}\label{thm:firstwalls}
 Let $d\geq 4$ be an integer.
 \begin{enumerate}
  \item For any $c<\frac{8}{3d}$, the only surfaces
  appearing in the K-moduli stack
  $\ocP^\K_{d,c}$ are $\bP^2$ or $\bP(1,1,4)$.
  \item Suppose $5\mid d$, then we have the following wall crossing at $c=\frac{8}{3d}$: \[
      \oP^{\K}_{d,\frac{8}{3d}-\epsilon}\xrightarrow{\phi^-}
      \oP^{\K}_{d,\frac{8}{3d}}\xleftarrow{\phi^+}
      \oP^{\K}_{d,\frac{8}{3d}+\epsilon}
     \]
  where $\phi^-$ is an isomorphism near 
  $(\bP^2,\frac{d}{5}C)$ whose replacement is $(X_{26},\frac{d}{5}C_0')$,
  and $\phi^+$ is a weighted blow up at $(X_{26},\frac{d}{5}C_0')$.
  In particular, $\frac{8}{3d}$ is the second smallest wall extracting a divisor.
\end{enumerate}

\end{thm}

\begin{proof}
(1) Let $(X,cD)$ be a K-semistable pair appearing in $\ocP^\K_{d,c}$ for some
$c<\frac{8}{3d}$. By Theorem \ref{thm:localindex}, we know that any local Gorenstein index of $X$ is at most $9$. Then from the classification of Manetti surfaces we know that $X$ is isomorphic to one of $\bP^2$, $\bP(1,1,4)$, $\bP(1,4,25)$ or $X_{26}$.
If $X\cong\bP(1,4,25)$, then $\delta(X)=\frac{1}{10}$
by \cite[Section 7]{BJ17}. Since $1-\frac{cd}{3}>\frac{1}{10}$,  by \cite[Theorem 7.2]{BL18b} the pair $(X,cD)$ is K-unstable by taking 
$\beta=1-\frac{cd}{3}$ . This is a contradiction.
If $X\cong X_{26}$, then $\delta(X)=\frac{1}{9}$ by Proposition \ref{prop:deltax26}.
Hence the same argument implies that $(X,cD)$ is K-unstable whenever $c<\frac{8}{3d}$, again a contradiction. 

(2) The statement essentially follows from the proof of Theorem \ref{thm:secondwall}. Indeed, if $C$ is a general curve on $X_{26}$ of degree $5d$, then $C$ does not pass through the singular point $[0,0,0,1]$ on $X_{26}$. Thus the same argument as the proof of Proposition \ref{prop:onlykps-X_26} implies that $(X,\frac{8}{3d}C)$ is K-semistable. Since a general $C$ is smooth, by Proposition \ref{prop:k-interpolation} we know that $(X_{26},cC)$ is K-stable for any $c\in (\frac{8}{3d}, \frac{3}{d})$. It is clear that $\dim\Aut(X_{26})=\dim\Aut(\bP^2)-1$, hence such pairs $(X_{26},C)$ form a divisor in the K-moduli space $\oP_{d,\frac{8}{3d}+\epsilon}^{\K}$ (see Section \ref{sec:contraction}). Hence the proof is finished by \cite[Theorem 0.2.5]{dolgachevhu} or \cite[Corollary 1.13]{thaddeus}.
\end{proof}

\section{Log Fano wall crossings for K-moduli spaces of plane quintics}\label{sec:quintics}

In this section, we discuss all wall crossings of K-moduli spaces of plane quintics in the log Fano region $c\in (0,\frac{3}{5})$. For simplicity, we again abbreviate $\oP_5^{\GIT}$, $\ocP_{5,c}^{\K}$, and  $\oP_{5,c}^{\K}$ to $\oP^{\GIT}$, $\ocP_{c}^{\K}$, and $\oP_{c}^{\K}$, respectively.
Thanks to Sections \ref{sec:firstwall} and \ref{sec:2ndwallquintics}, we have detailed descriptions of the first two wall crossings. The main results of this section, namely Theorems \ref{thm:3rdwall}, \ref{thm:4thwall}, \ref{thm:5thwall}, and \ref{thm:allwalls} will show that there are three more wall crossings after the first two walls. We begin by a description of $\oP^{\GIT}$.

\subsection{GIT of plane quintics}\label{sec:GIT5}
By Theorem \ref{mthm:firstwall}, we know that the K-moduli space $\oP_{\epsilon}^{\K}$ is isomorphic to the GIT quotient for plane quintics, so we begin with a description of the (classical) GIT quotient $\oP^{\GIT}$ for plane quintics. This was calculated by Mumford \cite[Chapter 4, Section 5]{MFK94}. A detailed description also appears in \cite{laza2009deformations}. Under the identification of $\oP_{\epsilon}^{\K}$ and $\oP^{\GIT}$, Proposition \ref{prop:lclastwall} provides open embeddings 
\[
P_{\epsilon}^{\klt}\hookrightarrow P_{\epsilon}^{\lc}
\hookrightarrow \oP^{\GIT}
\]
where $P_{\epsilon}^{\klt}$ and $P_{\epsilon}^{\lc}$ denote the loci in $\oP^{\GIT}$ parametrizing GIT polystable plane quintics  with lct $>\frac{3}{5}$ and $\geq \frac{3}{5}$, respectively. 

\begin{lemma}\label{lem:quinticgit}
The boundary $\oP^{\GIT} \setminus P_{\epsilon}^{\klt}$ is a disjoint union of the following locally closed strata:

Zero-dimensional loci
\begin{itemize}
\item $\Sigma_1=\{[Q_5]\}$, and
\item $\Sigma_2$ parametrizing a plane quintic curve with an $A_{12}$ singularity.
\end{itemize}

One-dimensional loci
\begin{itemize}
\item $\Sigma_3$ parametrizing a reducible plane quintic curve with an $A_{11}$ singularity, 
\item $\Sigma_4$ parametrizing an irreducible plane quintic curve with an $A_{11}$ singularity, and
\item $\Sigma_6$  parametrizing the union of two conics tangent at two distinct points and a line through them (two $D_6$ singularities), i.e.
$$
\left( z \left( xy-z^2 \right)\left( xy- az^2 \right)=0 \right)\quad \textrm{ where }a\neq 1.
$$
\end{itemize}

Two dimensional locus
\begin{itemize}
\item $\Sigma_5$ parametrizing a plane quintic curve with an $A_{10}$ singularity.
\end{itemize}

Three dimensional locus
\begin{itemize}
\item $\Sigma_7$ parametrizing a plane quintic curve with an $A_{9}$ singularity.
\end{itemize}
Moreover, the incidence of such strata is as follows:
 
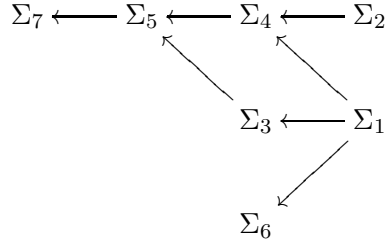
\begin{align*}
 \xymatrix{
\Sigma_7  & \Sigma_5 \ar@{->}[l]  & \Sigma_4   \ar@{->}[l]& \Sigma_2 \ar@{->}[l] & 
\\
 &                                                  &  \Sigma_3\ar@{->}[ul] & \Sigma_1  \ar@{->}[l]  \ar@{->}[ul]\ar@{->}[dl] &\\
 &&\Sigma_6
 }
 \end{align*}
 Here $\Sigma_i\to\Sigma_j$ means that $\Sigma_i$ is contained in the Zariski closure of $\Sigma_j$ in $\oP^{\GIT}$. The closure of the stratum $\oSigma_{6}=\Sigma_6\sqcup \Sigma_1$, is isomorphic to $\mathbb P^1$, and is the only strictly semistable stratum. The other strata are contained in the stable locus. In addition, $\oP^{\GIT}\setminus P_{\epsilon}^{\lc}=\bigsqcup_{i=1}^5 \Sigma_i$.
\end{lemma}
\begin{proof}
By the GIT analysis in \cite{laza2009deformations}, we know that
if $C$ is a  a plane curve of degree $5$, then
\begin{itemize}
\item $C$ is GIT stable if and only $C$ is either smooth or has singularities 
of type $A_k$ with $1 \leq k \leq 12$, $D_4$ and $D_5$.  
\item $C$ is GIT strictly semistable (i.e. semistable but not stable) if and only if it has a singularity of type $D_k$ with $6 \leq k \leq 12$ such that if $k=9$, then $C$ is not the union of a nodal quartic and line. 
\item $C$ is GIT strictly polystable (i.e. polystable but not stable) if and only $C$ is the union of double conic and a transverse line or the union of two tangent conics and a line passing through their tangent points.
\end{itemize}
Therefore, the GIT quotient $\oP^{\GIT}$ is the union of the GIT stable locus and a smooth rational curve parametrizing the GIT strictly
polystable plane quintics. The statement follows by considering the log canonical thresholds of these singularities and the jet computations (see Proposition \ref{prop:jetsummary}).
\end{proof}

\subsection{Explicit wall crossings}\label{sec:explicitwalls}
As we saw in Theorem \ref{mthm:firstwall}, the GIT quotient $\oP^{\GIT}$ of plane quintics can be identified with the K-moduli space $\ocP^\K_{c}$ where $0 < c < 3/7$. In this section we discuss the subsequent wall crossings among the K-moduli spaces of plane quintic curves. The following diagram gives an overview of the K-moduli spaces for plane quintics based on results from Sections \ref{sec:firstwall}, \ref{sec:2ndwallquintics} and this section (see Table \ref{table:quintic} for a summary).

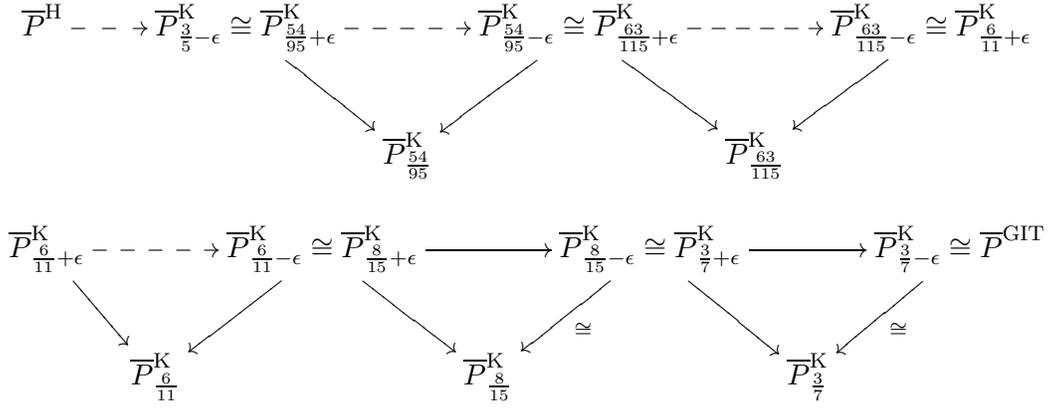
\begin{figure}[!htp]\caption{Log Fano wall crossings for K-moduli spaces of plane quintics}\label{fig:quintics}
\[\xymatrixcolsep{1em}
\xymatrix{ \oP^{\rm H}  \ar@{-->}[rr] & & \oP^{\K}_{\frac{3}{5}-\epsilon} \cong \oP^{\K}_{\frac{54}{95}+\epsilon}  \ar[rd] \ar@{-->}[rr] & & \oP^{\K}_{\frac{54}{95}-\epsilon} \cong \oP^{\K}_{\frac{63}{115} + \epsilon} \ar[ld] \ar@{-->}[rr] \ar[rd] & & \oP^{\K}_{\frac{63}{115} - \epsilon} \cong \oP^{\K}_{\frac{6}{11}+\epsilon} \ar[ld] \\ & & & \oP^{\K}_{\frac{54}{95}} & & \oP^{\K}_{\frac{63}{115}} }  \]

\[\xymatrixcolsep{1em}
\xymatrix{ \oP^{\K}_{\frac{6}{11}+\epsilon} \ar[rd] \ar@{-->}[rr] & & \oP^{\K}_{\frac{6}{11}-\epsilon} \cong \oP^{\K}_{\frac{8}{15}+\epsilon} \ar[ld] \ar[rd] \ar[rr] & & \oP^{\K}_{\frac{8}{15}-\epsilon} \cong \oP^{\K}_{\frac{3}{7} + \epsilon} \ar[ld]^\cong \ar[rr] \ar[rd] & & \oP^{\K}_{\frac{3}{7} - \epsilon} \cong \oP^{\textrm{GIT}} \ar[ld]^\cong \\ & \oP^{\K}_{\frac{6}{11}} & & \oP^{\K}_{\frac{8}{15}} & & \oP^{\K}_{\frac{3}{7}}}  \]\end{figure}




The following results describe the remaining three walls which occur after the first two walls for K-moduli spaces of plane quintics. Their proofs will be presented in Section \ref{sec:quintic345proof} with ingredients from calculations of Section \ref{sec:calculations}. If the birational map $\oP^{\GIT}\dashrightarrow \oP^{\K}_c$ is isomorphic
at the generic point of a locus $\Sigma_i$, then we denote by $\oSigma_{c,i}$ the Zariski closure of the proper transform of $\Sigma_i$ in $\oP^{\K}_c$.

\begin{theorem}[Third wall crossing]\label{thm:3rdwall} The third wall
is $c_3=\frac{6}{11}$.
\begin{enumerate}
 \item The birational morphism 
 $\phi_{3}^{-}:\oP^{\K}_{\frac{6}{11}-\epsilon}\to
 \oP^{\K}_{\frac{6}{11}}$ is an isomorphism 
 away from the locus $\oSigma_{\frac{6}{11}-\epsilon, 3}$. Moreover, $\phi_{3}^{-}$ contracts $\oSigma_{\frac{6}{11}-\epsilon, 3}$
  to a point $[(\bP(1,1,4), \frac{6}{11}(x^2 z^2+y^6 z=0))]$.
 
 \item The birational morphism 
 $\phi_{3}^{+}:\oP^{\K}_{\frac{6}{11}+\epsilon}\to
 \oP^{\K}_{\frac{6}{11}}$  is an isomorphism
 away from the point $[(\bP(1,1,4), \frac{6}{11}(x^2 z^2+y^6 z=0))]$.
  Moreover, the exceptional locus $E_{3}^{+}$ of $\phi_{3}^{+}$
  is of codimension $2$ and parametrizes curves on $\bP(1,1,4)$ of the form $(x^2z^2+y^6 z+g(x,y)=0)$ with $g\neq 0$.
\end{enumerate}
\end{theorem}

\begin{theorem}[Fourth wall crossing]\label{thm:4thwall} The fourth wall
is $c_4=\frac{63}{115}$.
\begin{enumerate}
 \item The birational morphism 
 $\phi_{4}^{-}:\oP^{\K}_{\frac{63}{115}-\epsilon}\to
 \oP^{\K}_{\frac{63}{115}}$ is an isomorphism  away from the locus $\oSigma_{\frac{63}{115}-\epsilon,4}$.
  Moreover, $\phi_{4}^{-}$ contracts $\oSigma_{\frac{63}{115}-\epsilon,4}$  to a point $[(\bP(1,4,25), \frac{63}{115}(z^2+x^2 y^{12}=0))]$.
 \item The birational morphism 
 $\phi_{4}^{+}:\oP^{\K}_{\frac{63}{115}+\epsilon}\to
 \oP^{\K}_{\frac{63}{115}}$ is an isomorphism away from the point $[(\bP(1,4,25), \frac{63}{115}(z^2+x^2 y^{12}=0))]$.  
 Moreover, the exceptional locus $E_{4}^{+}$ of $\phi_{4}^{+}$
  is of codimension $2$ and parametrizes curves on $\bP(1,4,25)$ of the form $(z^2+x^2y^{12}+x^{10} g(x,y)=0)$ with $g\neq 0$.
\end{enumerate}
\end{theorem}

\begin{theorem}[Fifth wall crossing]\label{thm:5thwall} The fifth wall
is $c_5=\frac{54}{95}$.
\begin{enumerate}
 \item The birational morphism 
 $\phi_{5}^{-}:\oP^{\K}_{\frac{54}{95}-\epsilon}\to
 \oP^{\K}_{\frac{54}{95}}$ is an isomorphism
 away from the locus $\oSigma_{\frac{54}{95}-\epsilon,5}$.
  Moreover, $\phi_{5}^{-}$ contracts $\oSigma_{\frac{54}{95}-\epsilon,5}$
  to a point $[(\bP(1,4,25), \frac{54}{95}(z^2+x^6 y^{11}=0))]$.
 \item The birational morphism  $\phi_{5}^+:\oP^{\K}_{\frac{54}{95}+\epsilon}\to \oP^{\K}_{\frac{54}{95}}$ is an isomorphism away from the point $[(\bP(1,4,25), \frac{54}{95}(z^2+x^6 y^{11}=0))]$.  
 Moreover, the exceptional locus $E_{5}^{+}$ of $\phi_{5}^{+}$  is of codimension $3$ and parametrizes  curves on $\bP(1,4,25)$ of the form $(z^2+x^6y^{11}+x^{14}g(x,y)=0)$ with $g\neq 0$.
\end{enumerate}
\end{theorem}

\begin{theorem}[No other walls]\label{thm:allwalls} The five walls above are all walls occurring for K-moduli spaces of plane quintics in the log Fano region $0 < c < \frac{3}{5}$. In other words, for any $\frac{54}{95}<c<\frac{3}{5}$ we have an isomorphism $\ocP^\K_{c}\cong\ocP^\K_{\frac{54}{95}+\epsilon}$ between Artin stacks. 
\end{theorem}

\subsection{Proofs}\label{sec:quintic345proof}

In this section we present proofs of Theorems \ref{thm:3rdwall}, \ref{thm:4thwall}, \ref{thm:5thwall}, and \ref{thm:allwalls}. Our strategy is quite similar to the proof of Theorem \ref{thm:secondwall}.

\begin{proof}[Proof of Theorem \ref{thm:3rdwall}]
(1) We first show that there is no wall crossing when $c\in (\frac{8}{15},\frac{6}{11})$. Assume to the contrary that the third wall $c_3\in (\frac{8}{15},\frac{6}{11})$. Then there exists a new K-polystable pair $(X,c_3 D)$ such that $(X,cD)$ is K-unstable for any $c\neq c_3$ by Proposition \ref{prop:openkps}. Thus $X$ is not isomorphic to $\bP^2$ or $\bP(1,4,25)$ by Propositions \ref{prop:k-interpolation} and \ref{prop:valcrit-p1425}. Hence by Lemma \ref{lem:quinticindex} we are left with two possibilities, i.e. $X$ is isomorphic to $\bP(1,1,4)$ or $X_{26}$. If $X\cong X_{26}$, then $D$ does not passes through the singular point on $X$ by Proposition \ref{prop:valcrit-X_26}. Hence $(X,\frac{8}{15}D)$ is K-semistable by Proposition \ref{prop:onlykps-X_26}, but this is a contradiction since $c_3\neq \frac{8}{15}$. Thus we may assume $(X,D)\cong(\bP(1,1,4),C)$ where the equation of $C$ is given by $f(x,y)z^2+g(x,y)z+h(x,y)=0$. If $f(x,y)$ is non-degenerate, then $(\bP(1,1,4),\frac{3}{7}C)$ is K-semistable since $C$ specially degenerates to $Q_5'$ which is a contradiction. If $f(x,y)$ is degenerate, then Proposition \ref{prop:valcrit-p114} implies that $(\bP(1,1,4),c_3 C)$ is K-unstable since $c_3<\frac{6}{11}$. Thus we have shown that no walls can appear in the interval $(\frac{8}{15},\frac{6}{11})$.

Next we show that $(\bP(1,1,4), \frac{6}{11}(x^2 z^2+y^6z=0))$ is the only new K-polystable pair in $\oP_{\frac{6}{11}}^{\K}$. Clearly this pair is K-polystable by Proposition \ref{prop:a11redkps}. If $(X,\frac{6}{11}D)$ is a new K-polystable pair, then from the argument above we see that $(X,D)\cong(\bP(1,1,4),C)$ and the defining equation of $C$ has a degenerate quadratic form in $(x,y)$ as the $z^2$-term coefficient. Then by Proposition \ref{prop:valcrit-p114}  we know that after a coordinate change the equation of $C$ has the form $x^2 z^2 +y^6 z+ g(x,y)=0$. Consider the $1$-PS in $\Aut(\bP(1,1,4))$ defined by $[x,y,z]\mapsto[t^3 x, ty, z]$. It is clear that $C$ specially degenerates to the curve $C_1:=(x^2 z^2+y^6z=0)$ via this $1$-PS as $t\to 0$. Since $(\bP(1,1,4), \frac{6}{11}C_1)$ is K-polystable, it follows that $(X,D)\cong (\bP(1,1,4), C_1)$ by \cite{LWX18}. 

So far we have shown that $c_3=\frac{6}{11}$. Next we analyze the wall crossing morphisms $\phi_3^{\pm}$. Since $(\bP(1,1,4), \frac{6}{11}C_1)$ is the only new K-polystable pair in $\oP_{\frac{6}{11}}^{\K}$, we know by Proposition \ref{prop:openkps} that $\phi_3^{\pm}$ are isomorphic over $\oP_{\frac{6}{11}}^{\K}\setminus\{[(\bP(1,1,4), \frac{6}{11}C_1)]\}$. Denote by $E_{3}^{\pm}$ the preimage of $[(\bP(1,1,4), \frac{6}{11}C_1)]$ under the morphisms $\phi_3^{\pm}$. It is clear that $\Aut_0(\bP(1,1,4), C_1)\cong\bG_m$.
Using the local VGIT presentation \eqref{eq:localVGIT} and applying \cite[Theorem 0.2.5]{dolgachevhu} or \cite[Corollary 1.13]{thaddeus}, we have
\begin{equation}\label{eq:3rdwall}
\dim (E_{3}^-)+\dim (E_3^+)+1=\dim(\oP_{\frac{6}{11}}^{\K}).
\end{equation}
Moreover, $E_3^{\pm}$ are weighted projective spaces quotient out some finite group action, so they are irreducible. By Theorem \ref{thm:a11red}, we know that $\oSigma_{\frac{6}{11}-\epsilon,3}\subset E_3^{-}$ which implies $\dim E_3^{-}\geq 1$. Let $C$ be a general curve on $\bP(1,1,4)$ of the form $x^2 z^2+y^6 z+ g(x,y)=0$. Then it is clear that $\lct(\bP(1,1,4);C)\geq \frac{2}{3}>\frac{3}{5}$. Hence $(\bP(1,1,4),(\frac{6}{11}+\epsilon)C)$ is K-stable by Proposition \ref{prop:k-interpolation}. It is easy to see that such pairs $(\bP(1,1,4), C)$ form a locally closed subset in the K-moduli space $\oP_{\frac{6}{11}+\epsilon}^{\K}$ of  codimension two. Hence $\dim(\oP_{\frac{6}{11}}^{\K})-\dim (E_3^+)\leq 2$ which implies $\dim(E_3^-)\leq 1$ by \eqref{eq:3rdwall}. Therefore, we know  $\dim(E_3^-)=1$ and hence $\oSigma_{\frac{6}{11}-\epsilon,3}= E_3^{-}$ by irreducibility of $E_3^-$. This finishes the proof of part (1).

(2) The proof of this part is basically the same as the proof of Theorem \ref{thm:secondwall}(3). Firstly, if $[(X,D)]$ is a point in $E_3^+$, then $(X,D)$ specially degenerates to $(\bP(1,1,4), C_1)$. Hence $X$ is either $\bP^2$ or $\bP(1,1,4)$. If $X\cong\bP^2$, then $(X,\frac{6}{11}D)$ is K-polystable by Proposition \ref{prop:k-interpolation}, a contradiction. Hence $(X,D)\cong(\bP(1,1,4),C)$. If $C$ has the form $xyz^2+f(x,y)z+h(x,y)=0$ after a suitable change of coordinates, then it admits a special degeneration to $Q_5$ which implies that $(X,\frac{3}{7}D)$ is K-semistable by Theorem \ref{thm:Kss-spdeg} and Lemma \ref{lem:Q_d'Kps}. Hence $(X,\frac{6}{11}D)$ is K-polystable by Proposition \ref{prop:k-interpolation}, again a contradiction. Then by Proposition \ref{prop:valcrit-p114}, we know that $C$ must have the form $x^2 z^2+y^6 z+g(x,y)=0$ after a suitable change of coordinates. Note that $g\neq 0$ because otherwise $(X,D)\cong (\bP(1,1,4), C_1)$ is $(\frac{6}{11}+\epsilon)$-K-unstable.

To prove the rest of part (2), it suffices to show that $(\bP(1,1,4),(\frac{6}{11}+\epsilon)C)$ is K-polystable for any curve $C$ on $\bP(1,1,4)$ described in the statement. We omit the rest of the proof here since it is the same as the proof of Theorem \ref{thm:secondwall}(3) by using properness of K-moduli spaces (Theorem \ref{thm:compactness}). 
\end{proof}
\medskip

\begin{proof}[Proof of Theorem \ref{thm:4thwall}]
(1) We first show that there is no wall crossing when $c\in (\frac{6}{11},\frac{63}{115})$. Assume to the contrary that the fourth wall $c_4\in (\frac{6}{11},\frac{63}{115})$. Then there exists a new K-polystable pair $(X,c_4 D)$ such that $(X,cD)$ is K-unstable for any $c\neq c_4$ by Proposition \ref{prop:openkps}. Thus similar argument to the proof of Theorem \ref{thm:3rdwall}(1) implies that $(X,D)$ has to be isomorphic to $(\bP(1,1,4),C)$ where the equation of $C$ is given by $f(x,y)z^2+g(x,y)z+h(x,y)=0$ with $f$ degenerate.  Then  Proposition \ref{prop:valcrit-p114} implies that the equation of $C$ has to have the form $x^2 z^2+ y^6 z+ h(x,y)=0$, so $C$ admits a special degeneration to the curve $C_1=(x^2 z^2+ y^6 z=0)$. Thus by Proposition \ref{prop:a11redkps} and Theorem \ref{thm:Kss-spdeg} we know that $(\bP(1,1,4),\frac{6}{11}C)$ is K-semistable, but this is a contradiction as $c_4\neq \frac{6}{11}$. Thus we have shown that no walls can appear in the interval $(\frac{6}{11},\frac{63}{115})$.

Next we show that $(\bP(1,4,25), \frac{63}{115}(z^2+x^2 y^{12}=0))$ is the only new K-polystable pair in $\oP_{\frac{63}{115}}^{\K}$. Clearly this pair is K-polystable by Proposition \ref{prop:a11irrkps}. If $(X,\frac{63}{115}D)$ is a new K-polystable pair, then from the argument above we see that $(X,D)\cong(\bP(1,4,25),C)$. By Proposition \ref{prop:valcrit-p1425}, the defining equation of $C$ has the form $z^2+x^2 y^{12} + x^6 g(x,y)=0$.  Consider the $1$-PS in $\Aut(\bP(1,4,25))$ defined by $[x,y,z]\mapsto[x, ty, t^6 z]$. It is clear that $C$ specially degenerates to the curve $C_2:=(z^2+ x^2 y^{12}=0)$ via this $1$-PS as $t\to 0$. Since $(\bP(1,4,25), \frac{63}{115}C_2)$ is K-polystable, it follows that $(X,D)\cong (\bP(1,4,25), C_2)$ by \cite{LWX18}. 

So far we have shown that $c_4=\frac{63}{115}$. Next we analyze the wall crossing morphisms $\phi_4^{\pm}$. Since $(\bP(1,4,25), \frac{63}{115}C_2)$ is the only new K-polystable pair in $\oP_{\frac{63}{115}}^{\K}$, we know by Proposition \ref{prop:openkps} that $\phi_4^{\pm}$ are isomorphic over $\oP_{\frac{63}{115}}^{\K}\setminus\{[(\bP(1,4,25), \frac{63}{115}C_2)]\}$. Denote by $E_{4}^{\pm}$ the preimage of $[(\bP(1,4,25), \frac{63}{115}C_2)]$ under the morphisms $\phi_4^{\pm}$. It is clear that $\Aut_0(\bP(1,4,25), C_2)\cong\bG_m$. By Theorem \ref{thm:a11irr}, we know that $\oSigma_{\frac{63}{115}-\epsilon,4}\subset E_4^{-}$ which implies $\dim(E_4^{-})\geq 1$.
Using a similar argument to the proof of Theorem \ref{thm:3rdwall}(1), it suffices to show that $E_4^+$ has codimension at most $2$ in the K-moduli space. Let $C$ be a general curve on $\bP(1,4,25)$ of the form $z^2+ x^2 y^{12}+ x^{10} g(x,y)=0$. Then it is clear that $\lct(\bP(1,4,25);C)\geq 1>\frac{3}{5}$.
Hence $(\bP(1,4,25),(\frac{63}{115}+\epsilon)C)$ is K-stable by Proposition \ref{prop:k-interpolation}. It is easy to see that such pairs $(\bP(1,4,25), C)$ form a locally closed subset in the K-moduli space $\oP_{\frac{63}{115}+\epsilon}^{\K}$ of  codimension $2$. Hence  $\oSigma_{\frac{63}{115}-\epsilon,4}= E_4^{-}$ by irreducibility of $E_4^-$ and local VGIT presentation. This finishes the proof of part (1).

(2) The proof of this part is basically the same as the proof of Theorem \ref{thm:secondwall}(3). Firstly, if $[(X,D)]$ is a point in $E_4^+$, then $(X,D)$ specially degenerates to $(\bP(1,4,25), C_2)$. If $X$ is isomorphic to $\bP^2$ or $X_{26}$, then $(X,\frac{63}{115}D)$ is K-polystable by Proposition \ref{prop:k-interpolation}, a contradiction. If $(X,D)\cong(\bP(1,1,4),C)$, then by Proposition \ref{prop:valcrit-p114} we know that the equation of $C$ has the form $xyz^2+ f(x,y)z+g(x,y)=0$ or $x^2 z^2+y^6 z+ h(x,y)$. From the proof of Theorem \ref{thm:3rdwall}, we know that the former curve is $\frac{3}{7}$-K-semistable, while the latter curve is $\frac{6}{11}$-K-semistable. Hence $(X,\frac{63}{115}D)$ is K-polystable by Proposition \ref{prop:k-interpolation}, again a contradiction. Therefore, $(X,D)\cong(\bP(1,4,25),C)$. By Proposition \ref{prop:valcrit-p1425}, we know that the equation of $C$ must be of the form $z^2 + x^2 y^{12}+ x^6 h(x,y)=0$. After a suitable change of coordinates the equation of $C$ can be even simplified to $z^2 + x^2 y^{12}+ x^{10} g(x,y)=0$. Note that $g\neq 0$ because otherwise $(X,D)\cong (\bP(1,4,25), C_2)$ is $(\frac{63}{115}+\epsilon)$-K-unstable. The rest of part (2) also follows from similar argument to proof of Theorem \ref{thm:secondwall}(3) by using properness of K-moduli spaces (Theorem \ref{thm:compactness}). 
\end{proof}
\medskip

\begin{proof}[Proof of Theorem \ref{thm:5thwall}]
(1) We first show that there is no wall crossing when $c\in (\frac{63}{115},\frac{54}{95})$. Assume to the contrary that the fifth wall $c_5\in (\frac{63}{115},\frac{54}{95})$. Then there exists a new K-polystable pair $(X,c_5 D)$ such that $(X,cD)$ is K-unstable for any $c\neq c_5$ by Proposition \ref{prop:openkps}. Thus similar argument to the proof of Theorems \ref{thm:3rdwall}(1) and \ref{thm:4thwall}(1) implies that $(X,D)$ has to be isomorphic to $(\bP(1,4,25),C)$.
By Proposition \ref{prop:valcrit-p1425},  the equation of $C$ must have the form $z^2+x^2 y^{12}+ x^6 g(x,y)=0$ since $c_5<\frac{54}{95}$. So $C$ admits a special degeneration to the curve $C_2=(z^2+x^2 y^{12}=0)$. Thus by Proposition \ref{prop:a11redkps} and Theorem \ref{thm:Kss-spdeg} we know that $(\bP(1,4,25),\frac{63}{115}C)$ is K-semistable, but this is a contradiction as $c_5\neq \frac{63}{115}$. Thus we have shown that no walls can appear in the interval $(\frac{63}{115},\frac{54}{95})$.

Next we show that $(\bP(1,4,25), \frac{54}{95}(z^2+x^6 y^{11}=0))$ is the only new K-polystable pair in $\oP_{\frac{54}{95}}^{\K}$. Clearly this pair is K-polystable by Proposition \ref{prop:a10kps}. If $(X,\frac{54}{95}D)$ is a new K-polystable pair, then from the argument above we see that $(X,D)\cong(\bP(1,4,25),C)$ such that $C$ admits no special degeneration to $C_2$. Thus Proposition \ref{prop:valcrit-p1425} implies that the defining equation of $C$ has the form $z^2+x^6 y^{11} + x^{10} g(x,y)=0$. Consider the $1$-PS in $\Aut(\bP(1,4,25))$ defined by $[x,y,z]\mapsto[x, t^2 y, t^{11} z]$. It is clear that $C$ specially degenerates to the curve $C_3:=(z^2+ x^6 y^{11}=0)$ via this $1$-PS as $t\to 0$. Since $(\bP(1,4,25), \frac{54}{95}C_3)$ is K-polystable, it follows that $(X,D)\cong (\bP(1,4,25), C_3)$ by \cite{LWX18}. 

So far we have shown that $c_5=\frac{54}{95}$. Next we analyze the wall crossing morphisms $\phi_5^{\pm}$. Since $(\bP(1,4,25), \frac{54}{95}C_3)$ is the only new K-polystable pair in $\oP_{\frac{54}{95}}^{\K}$, we know by Proposition \ref{prop:openkps} that $\phi_5^{\pm}$ are isomorphic over $\oP_{\frac{54}{95}}^{\K}\setminus\{[(\bP(1,4,25), \frac{54}{95}C_3)]\}$. Denote by $E_{5}^{\pm}$ the preimage of $[(\bP(1,4,25), \frac{54}{95}C_3)]$ under the morphisms $\phi_5^{\pm}$. It is clear that $\Aut_0(\bP(1,4,25), C_3)\cong\bG_m$. By Theorem \ref{thm:a10}, we know that $\oSigma_{\frac{54}{95}-\epsilon,5}\subset E_5^{-}$ which implies $\dim(E_5^{-})\geq 2$.
Using a similar argument to the proof of Theorem \ref{thm:3rdwall}(1), it suffices to show that $E_5^+$ has codimension at most $3$ in the K-moduli space. Let $C$ be a general curve on $\bP(1,4,25)$ of the form $z^2+ x^6 y^{11}+ x^{14} g(x,y)=0$. Then it is clear that $\lct(\bP(1,4,25);C)\geq \frac{2}{3}>\frac{3}{5}$.
Hence $(\bP(1,4,25),(\frac{54}{95}+\epsilon)C)$ is K-stable by Proposition \ref{prop:k-interpolation}. It is easy to see that such pairs $(\bP(1,4,25), C)$ form a locally closed subset in the K-moduli space $\oP_{\frac{54}{95}+\epsilon}^{\K}$ of  codimension $3$. Hence  $\oSigma_{\frac{54}{95}-\epsilon,5}= E_5^{-}$ by irreducibility of $E_5^-$ and local VGIT presentation. This finishes the proof of part (1).

(2) The proof of this part is basically the same as proofs of Theorems \ref{thm:secondwall}(3) and \ref{thm:4thwall}(2). Firstly, if $[(X,D)]$ is a point in $E_5^+$, then $(X,D)$ specially degenerates to $(\bP(1,4,25), C_3)$. By similar argument to the proof of Theorem \ref{thm:4thwall}(2), we know that $(X,D)\cong(\bP(1,4,25),C)$. 
If the equation of $C$ has the form $z^2 + x^2 y^{12}+ x^6 g(x,y)=0$, then it is $\frac{63}{115}$-K-semistable since it specially degenerates to $C_2$. Hence $(X, \frac{54}{95}D)$ is K-polystable by Proposition \ref{prop:k-interpolation}, a contradiction. Hence Proposition \ref{prop:valcrit-p1425} implies that the equation of $C$ must be of the form $z^2 + x^6 y^{11}+ x^{10} h(x,y)=0$. After a suitable change of coordinates the equation of $C$ can be even simplified to $z^2 + x^6 y^{11}+ x^{14} g(x,y)=0$. Note that $g\neq 0$ because otherwise $(X,D)\cong (\bP(1,4,25), C_3)$ is $(\frac{54}{95}+\epsilon)$-K-unstable. The rest of part (2) also follows from similar argument to proof of Theorem \ref{thm:secondwall}(3) by using properness of K-moduli spaces (Theorem \ref{thm:compactness}). 
\end{proof}
\medskip

\begin{proof}[Proof of Theorem \ref{thm:allwalls}]
Assume to the contrary that the sixth wall $c_6\in (\frac{54}{95},\frac{3}{5})$ exists. Then there exists a K-polystable log Fano pair $(X,c_6 D)$ in $\oP_{c_6}^{\K}$ such that $(X,cD)$ is K-unstable for any $c\neq c_6$. Then $X$ is isomorphic to $\bP^2$, $\bP(1,1,4)$, $X_{26}$, or $\bP(1,4,25)$ by Lemma \ref{lem:quinticindex}. If $X\cong\bP^2$, then $(X,\epsilon D)$ is K-semistable by Proposition \ref{prop:k-interpolation}, a contradiction. 
If $X\cong\bP(1,1,4)$, then Proposition \ref{prop:valcrit-p114} implies that the equation of $D$ has the form $xyz^2+f(x,y)z+g(x,y)=0$ or $x^2 z^2 + y^6 z+ h(x,y)=0$ after a suitable change of coordinates. Hence $D$ admits a special degeneration to either $Q_5'$ or $C_1=(x^2 z^2 + y^6 z=0)$. Thus Lemma \ref{lem:Q_d'Kps} and  Proposition \ref{prop:a11redkps} combined with Theorem \ref{thm:Kss-spdeg} imply that either $(X,\frac{3}{7}D)$ or $(X,\frac{6}{11}D)$ is K-semistable, a contradiction.
If $X\cong X_{26}$, then Proposition \ref{prop:valcrit-X_26} implies that $D$ does not pass through the singular point of $X$. Hence $(X,\frac{8}{15}D)$ is K-semistable by Proposition \ref{prop:onlykps-X_26}, a contradiction. 
If $X\cong\bP(1,4,25)$, then Proposition \ref{prop:valcrit-p1425} implies that $D$ admits a special degeneration to either $C_2=(z^2+x^2 y^{12}=0)$ or $C_3=(z^2+x^6 y^{11}=0)$. Thus Propositions \ref{prop:a11irrkps} and \ref{prop:a10kps} combined with Theorem \ref{thm:Kss-spdeg} imply that either $(X,\frac{63}{115}D)$ or $(X,\frac{54}{95}D)$ is K-semistable, again a contradiction. Since we rule out all four possibilities, the proof is finished.
\end{proof}

\newcommand{\cH}{\mathcal{H}}

\section{Projectivity, birational contractions, and the log Calabi-Yau wall crossing} \label{sec:questions}
In this final section we discuss some questions with incomplete answers that are interesting for future study. 

\subsection{Projectivity}\label{sec:projectivity}
In this section we will show that the for any $d\in\{4,5,6\}$ and any $c\in (0, \frac{3}{d})$, the CM $\bQ$-line bundle $\Lambda_c$ (see Proposition \ref{prop:k-cm-interpolation}) on the K-moduli space $\oP_{d,c}^{\K}$ is ample, which in particular implies the projectivity of $\oP_{d,c}^{\K}$. Our main tools are the work of Codogni and Patakfalvi \cite{CP18} and its generalization by Posva \cite{Pos19}, as well as the relative ampleness of CM line bundles under wall crossing (see Theroem \ref{thm:VGIT-CM-ample}).

\begin{thm}\label{thm:projectivity}
When $d\in \{4,5,6\}$, the CM $\bQ$-line bundle $\Lambda_c$ on $\oP_{d,c}^{\K}$ is ample for any $c\in (0, \frac{3}{d})$. 
\end{thm}

\begin{proof}
We first treat the cases when $d=4$ or $6$. When $c<\frac{3}{2d}$, Theorem \ref{thm:firstwallbefore} implies that $\ocP_{d,c}^{\K}\cong \ocP_{d}^{\GIT}$. Moreover, Proposition \ref{prop:P^2-paultian} implies that $\Lambda_c$ is the descent of $3(3-cd)^2c\cO_{\bfP_d^{\rm ss}}(1)$. Hence $\Lambda_c$ is ample when $c<\frac{3}{2d}$. By Theorem \ref{thm:firstwallon} we know that $\phi^{-}:\oP_{d,3/(2d)}^{\K}\to\oP_d^{\GIT}$ is an isomorphism. Hence Theorem \ref{thm:VGIT-CM-ample} implies that $\Lambda_{3/(2d)}$ is the $\phi^+$-pull back of the descent of $3(3-c_1 d)^2 c_1\cO_{\bfP_d^{\rm ss}}(1)$ with $c_1=3/(2d)$. Hence $\Lambda_{3/(2d)}$ is ample.
By Theorem \ref{thm:VGIT-CM-ample} we know that $\Lambda_{3/(2d)+\epsilon}$ is ample for $0< \epsilon\ll 1$. We know that $\ocP_{d,c}^{\K}$ is independent of the choice of $c\in (\frac{3}{2d},\frac{3}{d})$ by Theorem \ref{thm:quartsext}.
Hence the ampleness of $\Lambda_c$ for $\frac{3}{2d}<c<\frac{3}{d}$ follows from the ampleness of $\Lambda_{3/(2d)+\epsilon}$, the nefness of $\Lambda_{c,\Hodge}$ (see Theorems \ref{thm:logcy4} and \ref{thm:sextic-Hodge}), and the interpolation formula \eqref{eq:k-cm-interpolation}.

Next we consider the case when $d=5$.
For simplicity we omit $d$ in the subscript of K-moduli stacks and spaces.  Similar to the above arguments, we know that $\Lambda_c$ is ample for $c\leq \frac{3}{7}+\epsilon$ with $0<\epsilon\ll 1$. Hence we will assume $c\in (\frac{3}{7},\frac{3}{5})$ in the rest of the proof. By \cite[Theorem 1.13]{CP18}, we know that $\Lambda_{c}$ is nef on $\oP_{c}^{\K}$. Denote by $U_c$ the Zariski open subset of $\oP_{c}^{\K}$ which parametrizes K-stable pairs. Denote by $S_c:=\oP_{c}^{\K}\setminus U_c$.
Then by the Nakai-Moishezon criterion it suffices to show the following statements:
\begin{enumerate}[label=(\roman*)]
    \item $\Lambda_{c}|_{S_c}$ is ample.
    \item For any generically finite morphism $f: V\to \oP_{c}^{\K}$ from a normal proper variety $V$, the pull-back $f^*\Lambda_c$ is big on $V$ whenever $f(V)$ intersects $U_c$. 
\end{enumerate}

For (i), from the description of $\oP_{c}^{\K}$ in Sections \ref{sec:2ndwallquintics} and \ref{sec:quintics} we know that $S_c$ is either $\oSigma_{c,6}$ (when $c$ does not lie on a wall) or $\oSigma_{c,6}$ union an isolated point as the exceptional locus of a wall crossing (when $c$ lies on a wall). Recall from Section \ref{sec:quintics} that $\oSigma_{c,6}$ is a rational curve precisely parametrizing $2D_6$-curves $\{(\bP^2,(z(xy-z^2)(xy-az^2)=0))\}_{a\neq 1}$ and $\{(\bP(1,1,4), (xy(z^2-x^4y^4)=0))\}$.
In particular, $\oSigma_{c,6}$ is not changed under every wall crossing for $c\in (\frac{3}{7},\frac{3}{5})$. Hence by the nefness of $\Lambda_c$, the ampleness of $\Lambda_{3/7+\epsilon}$, and the interpolation formula \eqref{eq:k-cm-interpolation}, we know that $\Lambda_c|_{S_c}$ is ample for any $c\in (\frac{3}{7},\frac{3}{5})$. 

The proof of (ii) is more involved. By Theorem \ref{thm:bddtestK}(4) we know that $U_c$ also parametrizes uniformly K-stable pairs in $\oP_{c}^{\K}$.
Let $\cU_c$ be the preimage of $U_c$ as a saturated open substack of $\ocP_c^{\K}$. By \cite{ER17}, there exists a birational morphism $\widetilde{\cP}_c\to \ocP_c^{\K}$ from a smooth proper Deligne-Mumford stack $\widetilde{\cP}_c$ that is an isomorphism over $\cU_c$. By \cite[Chapter 16]{LMB00}, there exists a finite surjective morphism $\widetilde{P}_c'\to  \widetilde{\cP}_c$ from a normal proper variety $\widetilde{P}_c'$. Let $\widetilde{V}$ be a projective resolution of the main component of $V\times_{\oP_{c}^{\K}} \widetilde{P}_c'$. Denote by $\tau:\widetilde{V}\to V$ and $\sigma:\widetilde{V}\to \ocP_c^{\K}$. Then from the above construction we see that $\tau$ is generically finite, and $\tau^*(f^*\Lambda_c)=\sigma^*\lambda_c$. Since the $\sigma$-pull-back of the universal family over $\widetilde{V}$ has K-semistable fibers where a general fiber is uniformly K-stable, and is of maximal variation, by \cite{Pos19} as a generalization of \cite[Theorem 1.2(c)]{CP18} to log Fano pairs we know that $\sigma^*\lambda_c$ is nef and big on $\widetilde{V}$. This implies that $f^*\Lambda_c$ is big on $V$. The proof is finished.
\end{proof}

We expect $\Lambda_c$ to be ample in any degree which is a special case of the projectivity part of the Fano K-moduli conjecture (see e.g. \cite[Conjecture 1.1(c)]{CP18} and \cite[(VI) on page 611]{BX18}). During the review process, this expectation was proved in \cite{XZ19, LXZ21} as part of the proof of the Fano K-moduli conjecture (see Remark \ref{rem:ps}). 
 
\begin{thm}[cf. \cite{XZ20, LXZ21}]\label{thm:cm-ample}
The CM $\bQ$-line bundle $\Lambda_c$ on $\oP_{d,c}^{\K}$ is ample for any $c\in (0,\frac{3}{d})$ and any $d\geq 3$.\footnote{If $d=2$ and $c\in (0, \frac{3}{4}]$, then the theorem is also true for trivial reasons as $\oP_{d,c}^{\K}$ is a point (see Example \ref{expl:deg123}).}
\end{thm}

\subsection{Birational contractions along wall crossings}\label{sec:contraction}

As we saw in Sections \ref{sec:lowdegree} and \ref{sec:quintics}, the birational map  $\oP_{d,c'}^{\K} \dashrightarrow \oP_{d,c}^{\K}$ is always a birational contraction for $d\in\{4,5,6\}$ and $0<c<c'<\frac{3}{d}$. It is natural to ask whether this holds true for all degrees.

\begin{question}\label{question:contraction1}
Is $\oP_{d,c'}^{\K} \dashrightarrow \oP_{d,c}^{\K}$ a birational contraction for all $0 < c < c'< \frac{3}{d}$ and all $d \ge 4$? \end{question}

As the CM line bundles are always ample by Theorem \ref{thm:cm-ample}, an affirmative answer to the above question would imply that the wall crossing of K-moduli spaces have similar behavior to the Hassett-Keel program for Deligne-Mumford moduli spaces $\oM_g$. 

\begin{thm}\label{thm:hassetkeel}
 Assume Question \ref{question:contraction1} is true for some $d\geq 4$. Then for any $c\in (0,\frac{3}{d})$ and a sufficiently divisible $l\in\bZ_{>0}$, we have 
 \[
 \oP_{d,c}^{\K}= \Proj\bigoplus_{j=0}^\infty H^0\big(\oP_{d,\frac{3}{d}-\epsilon}^{\K},  lj(\Lambda_{\frac{3}{d}-\epsilon,\Hodge}+\tfrac{3-cd}{27cd}\Lambda_{\frac{3}{d}-\epsilon, 0})\big).
 \]
\end{thm}

\begin{proof}
By Proposition \ref{prop:k-cm-interpolation} we know that $\Lambda_{\frac{3}{d}-\epsilon, c}$ is a positive multiple of $\Lambda_{\frac{3}{d}-\epsilon,\Hodge}+\tfrac{3-cd}{27cd}\Lambda_{\frac{3}{d}-\epsilon, 0}$. 
Denote by $\varphi:\oP_{d,\frac{3}{d}-\epsilon}^{\K}\dashrightarrow \oP_{d,c}^{\K}$ the birational contraction. By the functoriality of CM line bundles, we know that $\Lambda_c=\varphi_*\Lambda_{\frac{3}{d}-\epsilon, c}$ as cycles. Since $\Lambda_c$ is ample on $\oP_{d,c}^{\K}$ by Theorem \ref{thm:cm-ample}, it suffices to show that for a common resolution $(p,q): \widetilde{P}_d \to \oP_{d,\frac{3}{d}-\epsilon}^{\K}\times \oP_{d,c}^{\K}$, we have $p^*\Lambda_{\frac{3}{d}-\epsilon, c}-q^*\Lambda_c\geq 0$ (see \cite[Definition 2.3 and Remark 2.4]{KKL16}). 

Let $0=c_0<c_1<\cdots<c_k=\frac{3}{d}$ be the walls of K-moduli spaces $\oP_{d,\bullet}^{\K}$. By passing to a higher birational model we may assume that $\widetilde{P}_d$ is a resolution of $\oP_{d,c_i\pm\epsilon}^{\K}$ for any $1\leq i\leq k-1$ with birational morphisms $\psi_i:\widetilde{P}_d\to \oP_{d,c_i}^{\K}$ and $\psi_i^{\pm}: \widetilde{P}_d\to \oP_{d,c_i\pm\epsilon}^{\K}$. Since $\Lambda_{c_i+\epsilon}$ is ample by assumption, we know that 
$(\psi_{i}^+)^*\Lambda_{c_i+\epsilon}-(\psi_{i}^{-})^*\Lambda_{c_i-\epsilon,c_i+\epsilon}$ is $\psi_i^{-}$-nef and $\psi_i^{-}$-exceptional. Hence by negativity lemma we know that $(\psi_{i}^+)^*\Lambda_{c_i+\epsilon}\leq(\psi_{i}^{-})^*\Lambda_{c_i-\epsilon,c_i+\epsilon}$ is effective. By Theorem \ref{thm:VGIT-CM-ample} we know that $(\psi_i^+)^*\Lambda_{c_i+\epsilon,c_i}=(\psi_i^-)^*\Lambda_{c_i-\epsilon,c_i}$. Since a positive rescaling of $\Lambda_{c,c'}$ is linear in $c'$, we know that $(\psi_{i}^+)^*\Lambda_{c_i+\epsilon,c}\geq (\psi_{i}^{-})^*\Lambda_{c_i-\epsilon,c}$ whenever $c_i\geq c$. Hence by the reverse induction on $i$ for $c_i\in [c,\frac{3}{d})$ we conclude that $p^*\Lambda_{\frac{3}{d}-\epsilon, c}-q^*\Lambda_c\geq 0$. The proof is finished. 
\end{proof}

The goal of this section is to verify Question \ref{question:contraction1} when either $d$ is divisible by $3$ or $d$ is small. 

\begin{thm}\label{thm:contraction}
Question \ref{question:contraction1} is true when either $3 \mid d$, or $3 \nmid d$ and $d < 13$. 
\end{thm}

First, we rephrase the question in a form that is easier to verify. 

\begin{question}\label{question:contraction}
Suppose $X$ is a Manetti surface with the following properties: 

\begin{enumerate}
    \item $\dim \Aut(X) = \dim \Aut(\PP^2) +1 = 9$, and
    \item $(X, cD)$ is K-stable for some $0 < c < \frac{3}{d}$ and a general $D \in \lvert - \frac{d}{3} K_X \rvert$ (in particular, $C$ has local indices $\le d$ when $3 \nmid d$ or $\le \frac{2d}{3}$ when $3 \mid d$).  
\end{enumerate}

Then, is $\lct(X; D) \geq \frac{3}{d}$? \end{question}

\begin{lem}\label{lem:autmanetti}
Let $X$ be a singular Manetti surface. Then we have $h^0(X,\cO_X(-\frac{d}{3}K_X))=h^0(\bP^2,\cO(d))$ for any $d\in\bZ_{>0}$ and $\dim\Aut(X)>\dim\Aut(\bP^2)=8$. 
\end{lem}

\begin{proof}
Let $\cX\to B$ over a smooth pointed curve $0\in B$ be a $\bQ$-Gorenstein smoothing of $X\cong \cX_0$. Let $\cL$ be the $\bQ$-Cartier Weil divisor on $\cX$ such that $3\cL\sim_{B}-dK_{\cX/B}$. Since $\cX$ has klt singularities, we know that $\cO_{\cX}(\cL)\otimes\cO_{\cX_0}\cong\cO_{\cX_0}(\cL_0)$. By Kawamata-Viehweg vanishing, we know that $H^i(\cX_0,\cL_0)=H^i(X,\cO_X(-\frac{d}{3}K_X))=0$ for any $i>0$. Hence the equation of $h^0$ follows from the flatness of $\cL$ over $B$. For the automorphism part, let $p$ be a sufficiently large positive integer such that $|-pK_X|$ is base point free. Then for a general curve $D\in |-pK_X|$ we know that $(X,\frac{1}{p}D)$ is a klt log Calabi-Yau pair. By \cite[Theorem 5.2]{LWX14} and Theorem \ref{thm:almostCYstabilize} we know that $(X,(\frac{1}{p}-\epsilon)D)$ is uniformly K-stable for $0<\epsilon\ll 1$. Let $U\subset |-pK_X|$ be the Zariski open locus parametrizing $D$ such that $(X,(\frac{1}{p}-\epsilon)D)$ is uniformly K-stable for some (or any) $0<\epsilon\ll 1$. Then $[U/\Aut(X)]$ is a Deligne-Mumford stack whose coarse moduli space $U/\Aut(X)$ admits an injection into $\oP_{3p, \frac{1}{p}-\epsilon}^{\K}$. Since a general point in the K-moduli space parametrize a smooth plane curve on $\bP^2$, we know that $\dim(U/\Aut(X))<\dim P_{3p}=\dim(\bfP_{3p}^{\rm sm}/\Aut(\bP^2))$ which implies $\dim\Aut(X)>\dim\Aut(\bP^2)$.
\end{proof}

\begin{prop}\label{prop:contractionequiv}
 Questions \ref{question:contraction1} and \ref{question:contraction} are equivalent to each other.
\end{prop}

\begin{proof}
 For the ``$\Rightarrow$'' direction, let $X$ be a Manetti surface satisfying (1) and (2) of Question \ref{question:contraction}. Assume to the contrary that $\lct(X;D)<\frac{3}{d}$ for a general $D\in |-\frac{d}{3}K_X|$. Then $(X,(\frac{3}{d}-\epsilon)D)$ is K-unstable for $0<\epsilon\ll 1$. Let $c_i$ be the K-semistable threshold of $(X,D)$. Let $U$ be the open subset of $|-\frac{d}{3}K_X|$ parametrizing $D$ with $(X,(c_i-\epsilon)D)$ K-stable. Then $[U/\Aut(X)]$ is a Deligne-Mumford stack whose coarse space $U/\Aut(X)$ injects into $\oP_{d,c_i-\epsilon}^{\K}$ whose image closure $E$ is a divisor by Lemma \ref{lem:autmanetti}. Thus the wall-crossing morphism $\phi_i^-:\oP_{d,c_i-\epsilon}^{\K}\to \oP_{d,c_i}^{\K}$ contracts $E$ to a codimension $\geq 2$ locus since the $c_i$-K-polystable pairs replacing $(X,D)$ have continuous automorphisms. This contradicts the assumption on birational contractions.
 
 For the ``$\Leftarrow$'' direction, assume to the contrary that a wall-crossing morphism $\phi_i^-:\oP_{d,c_i-\epsilon}^{\K}\to \oP_{d,c_i}^{\K}$ contracts a divisor $E$. Since Manetti surfaces have no continuous moduli, a general point on $E$ parametrizes $(X,D)$ for the same Manetti surface $X$. For the same $U$ as above, we know that $U/\Aut(X)$ injects into $\oP_{d,c_i-\epsilon}^{\K}$ with image closure $E$. Hence $\dim\Aut(X)=\dim\Aut(\bP^2)+1$ by Lemma \ref{lem:autmanetti}. Then $\lct(X;D)\geq \frac{3}{d}$ for a general $D\in |-\frac{d}{3}K_X|$ which implies $(X,cD)$ is K-stable for any $c_i<c<\frac{3}{d}$ by Proposition \ref{prop:k-interpolation}. This contradicts to the assumption that $E$ is contracted under $\phi_i^{-}$. The proof is finished.
\end{proof}

Thanks to Proposition \ref{prop:contractionequiv} we only need to verify Question \ref{question:contraction} for either $3\mid d$ or $d\leq 13$. The following lemma proves the case when $3\mid d$. 

\begin{lemma}\label{lemma:3d}
If $3 \mid d$, then  $\lct(X; D) \ge \frac{3}{d}$ for $X$ and $D$ satisfying conditions in Question \ref{question:contraction}.
\end{lemma}

\begin{proof}
When $3 \mid d$, we can degenerate $(X,D)$ to a weighted projective plane and the $(d/3)$-th multiple of the toric boundary divisor.  Hence, the inequality $\lct(X;D) \ge 3/d$ is obtained by lower semicontinuity of lct (see e.g. \cite{DK}).
\end{proof}

Next we turn to the case $d<13$ and $3\nmid d$ which will be confirmed by careful study of the Manetti surfaces appearing in our K-moduli spaces. First note that by Theorem \ref{thm:localindex} and Proposition \ref{prop:manetti}, for $d < 13$ the only Manetti surfaces satisfying the conditions in Question \ref{question:contraction} appearing in $\oP_{d,c}^{\K}$ for $0 < c < \frac{3}{d}$ are $\PP(1,1,4)$ and $X_{26}$.  For $X = \PP(1,1,4)$, the canonical  $\calO_X(K_X) = \calO_X(-6)$, so the linear system $\lvert -\frac{d}{3}K_X \rvert$ parameterizes elements of $\calO_X(2d)$, hence we are interested in answering Question \ref{question:contraction} only for curves of even degree on $\PP(1,1,4)$.  

As we can degenerate $X_{26}$ to $\PP(1,4,25)$, by semicontinuity of lct, we can study curves on $\PP(1,4,25)$.  Provided that the general curve on $\PP(1,4,25)$ has the appropriate lct, we can reach the same conclusion for $X_{26}$.  If $X = \PP(1,4,25)$, then $\calO_X(K_X) = \calO_X(-30)$, so the linear system $\lvert -\frac{d}{3}K_X \rvert$ parameterizes elements of $\calO_X(10d)$, hence we are interested in answering Question \ref{question:contraction} only for curves of degree a multiple of $10$ on $\PP(1,4,25)$.  

\begin{lemma}\label{lemma:P114}
For a general curve $C$ of even degree on $\PP(1,1,4)$, we have $\lct(\PP(1,1,4), C) = 1$. 
\end{lemma}

\begin{proof} If the degree $d$ of the curve satisfies $d \equiv 0 \mod 4$, then the general curve $C_d$ is smooth and contained in the smooth locus of $\PP(1,1,4)$, hence $\lct(\PP(1,1,4); C_d) = 1$.  
Next, consider the case when $d \equiv 2 \mod 4$.  If $d = 2$, the general curve $C_2 := (ax^2 + bxy + cy^2 = 0)$ passes through the singular point of $\PP(1,1,4)$ and is nodal at that point.  A computation shows that $\lct(\PP(1,1,4); C_2) = 1$.  For any $d$ such that $d \equiv 2 \mod 4$, the curve $C_{d-2} \cup C_2$ is in the linear system $\lvert \calO(d) \rvert$, where $C_{d-2}$ is a general curve of degree $d - 2$.  As the general $C_{d-2}$ is smooth, contained in the smooth locus of $\PP(1,1,4)$, and intersects $C_2$ transversally away from the singular point of the surface, we have $\lct(\PP(1,1,4); C_{d - 2} \cup C_2) = 1$.  Therefore, by semicontinuity of lct, the general curve $C_d$ of degree $d$ also has $\lct(\PP(1,1,4); C_d) = 1$.   \end{proof}

\begin{remark}
The previous statement is false without the assumption on even degree.  If $C$ is a general curve of degree 3 (or, more generally, degree $d$ such that $d \equiv 3 \mod 4$), then $\lct(\PP(1,1,4); C) = \frac{2}{3}$.
\end{remark}

Now, we mimic the previous argument for curves of degree 10 on $\PP(1,4,25)$.  

\begin{lemma}\label{lemma:P1425}
For a general curve $D$ of degree $d$ such that $d \equiv 0, 30 \mod 50 $ on $X = \PP(1,4,25)$, $\lct(\PP(1,4,25), C') = 1$.  If $d \equiv 10 \mod 50$, $\lct(X;D) = \frac{1}{2}$.  If $d \equiv 20 \mod 50$, $\lct(X; D) = \frac{1}{4}$, and if $d \equiv 40 \mod 50$, $\lct(X; D) = \frac{1}{3}$.  
\end{lemma}

\begin{proof}

If the degree $d$ of the curve is a multiple of 100, then the general curve $C_d$ is smooth and contained in the smooth locus of $\PP(1,4,25)$, hence $\lct(\PP(1,4,25); C_d) = 1$.  If the degree is a multiple of $50$, the general curve $C_d$ has a nonzero $z^2$ term, so misses the $\frac{1}{25}(1,4)$ singularity, and so the previous lemma implies the result.  

For $d = 10$, consider the general curve $C = (ax^{10}+bx^6y+cx^2y^2 = 0)$.  This has even degree through the $\frac{1}{4}(1,1)$ singular point, so the lct in a neighborhood of that point is $1$.  In a neighborhood of the $\frac{1}{25}(1,4)$ singular point, we compute the lct under the finite morphism $\pi: \mathbb{A}^2 \to \frac{1}{25}(1,4)$ where $\pi^*C$ is defined by the same equation.  By \cite[Lemma 8.12]{SingularitiesOfPairs}, $\lct(\frac{1}{25}(1,4); C) = \lct(\mathbb{A}^2; \pi^*C)$, and we compute $\lct(\mathbb{A}^2, \pi^* C) = \frac{1}{2}$ using \cite[Example 4, 5]{Howald}.  For a general curve of degree $d \equiv 10 \mod 50$, there are nonzero terms of the form $z^nf_{10}(x,y)$, where $f_{10}(x,y)$ is a degree 10 polynomial, so for $C_d$ a general curve of degree $d \equiv 10 \mod 50$, we have $\lct(\frac{1}{25}(1,4); C) = \frac{1}{2}$.

For $d \equiv 20 \mod 50$, consider a general curve $C_d$ of degree $d$.  Because this passes through the $\frac{1}{25}(1,4)$ singular point with high multiplicty, we do not expect that $\lct(\frac{1}{25}(1,4); C_d) = 1$.  Indeed, using the same method as above, we can compute $\lct(\frac{1}{25}(1,4), C_d) = \frac{1}{4}$.

For $d = 30$, the general curve has an $xyz$ term, so is at worst nodal at each singular point of the surface, so actually has lct equal to 1.  For $d \equiv 30 \mod 50$, the general curve has an $xyz^{2k+1}$ term, so is nodal in a neighborhood of the $\frac{1}{25}(1,4)$ singularity and has even degree in a neighborhood of the $\frac{1}{4}(1,1)$ singularity, so still has lct equal to 1. 

For $d = 40$, the general curve has a $y^{10}$ term, and so it misses the $\frac{1}{4}(1,1)$ singular point.  The general curve has a nonzero $x^3y^3z$ term, so in a neighborhood of the $\frac{1}{25}(1,4)$ singularity, a computation similar to that above shows that $\lct(\frac{1}{25}(1,4); C_d) = \frac{1}{3}$.  For $d \equiv 40 \mod 50$, the general curve could be nodal through the $\frac{1}{4}(1,1)$ singular point, but still has lct equal to 1 in a neighborhood of that point.  

Finally, for $d = 50$, the general curve has a $z^2$ term, and so it misses the $\frac{1}{25}(1,4)$ singular point and passes through the $\frac{1}{4}(1,1)$ singular point with even degree, so the discussion of $\PP(1,1,4)$ above shows $\lct(\PP(1,4,25); C_{50}) = 1$.  \end{proof}

Although the previous computation was done for $X = \PP(1,4,25)$, the lct computation is local, so the same result holds for curves of the appropriate degree on $X = X_{26}$. The computation shows that, for $d = 4, 5, 8, 10, $ and $11$, the answer to Question \ref{question:contraction} is yes.  For curves on $\PP(1,1,4)$, by Lemma \ref{lemma:P114}, we have that $\lct(X;D) = 1 > \frac{3}{d}$.  

For curves on $X_{26}$, we need a finer analysis.  No curves on $X_{26}$ appear for degree $d = 4$, but for $d = 5, 8, $ and $10$, Lemma \ref{lemma:P1425} shows $\lct(X;D) = 1 > \frac{3}{d}$, and for $d = 11$, Lemma \ref{lemma:P114} shows $\lct(X; D) = \frac{1}{2} > \frac{3}{11}$.   This leaves four exceptional cases: $d = 6, 7, 9, $ and $12$.  For degrees $6, 9,$ and $12$ we have $3 \mid d$, so Lemma \ref{lemma:3d} implies $\lct(X;D) \ge \frac{3}{d}$. Therefore we have contraction morphisms $\oP^{\K}_{d,c'} \dashrightarrow \oP^{\K}_{d,c}$ for all $0 < c < c'< \frac{3}{d}$.

This leaves only the case $d = 7$.  The desired result (an affirmative answer to Question \ref{question:contraction}) will follow from the next surprising proposition. 

\begin{prop}\label{prop:d=7}
For $d = 7$,  curves on $X_{26}$ or $\bP(1,4,25)$ are K-unstable in $\ocP^\K_{7,c}$ for all $c \in (0, \frac{3}{7})$. In other words, the only surfaces appearing in $\ocP^\K_{7,c}$ for some $c \in (0, \frac{3}{7})$ are $\bP^2$ and $\bP(1,1,4)$.
\end{prop}

\begin{proof}
Let $X$ be either $X_{26}$ or $\bP(1,4,25)$. Assume to the contrary that $(X,cD)$ is K-semistable for some curve $D\in |-\frac{7}{3}K_X|$ and some $c\in (0,\frac{3}{7})$.
From the above discussion, we know that $\lct(X;D)\leq \frac{1}{4}$ which implies $c<\frac{1}{4}$. Since $X$ has a singularity of local Gorenstein index $5$, by Theorem \ref{thm:localindex} we know that $5\leq 3/(3-7c)$ hence $c\geq \frac{12}{35}$. This contradicts $c<\frac{1}{4}$. 
\end{proof}

\begin{proof}[Proof of Theorem \ref{thm:contraction}]
The result follows from Lemmas \ref{lemma:3d}, \ref{lemma:P114}, \ref{lemma:P1425}, and Proposition \ref{prop:d=7}. 
\end{proof}

\begin{remark}
For $d = 9$ or $12$, a valuative criterion computation shows that  $X_{26}$ cannot appear in the K-moduli spaces $\oP_{d,c}^{\K}$ for $0<c < \frac{3}{d}$. 
\end{remark}



\subsection{Quartics and sextics revisited}\label{sec:logcy}
Recall from Section \ref{sec:lowdegree}, we interpreted the K-moduli spaces of quartic and sextic plane curves via K3 surfaces. In this section, we revisit these moduli spaces and study them via their relation to Hacking's moduli space $\ocP^{\rm H}_d$, and give a log Calabi-Yau intepretation. 

\subsubsection{Quartics}\label{sec:quartics} Recall that Hacking's space $\oP^{\rm H}_4$ generically parametrizes plane curves with at worst cuspidal singularities. There is a divisor parametrizing curves in $\bP(1,1,4)$ which are at worst nodal at the singular point, and at worst cuspidal elsewhere. Finally, there is a codimension two locus parametrizing curves on the non-normal surface $\bP(1,1,2) \cup \bP(1,1,2)$ -- the curves are snc at the double locus, and at worst cuspidal elsewhere.

Hyeon-Lee's original motivation was to complete the log minimal model program on $\oM_3$. 
Let $\oM_3(\alpha)$ denotes $ \Proj \oplus_{m \geq 0} \Gamma(\oM_3, m(K_{\oM_3} + \alpha \delta))$, where $\delta$ is the boundary divisor of $\oM_3$. Hyeon-Lee produce the following diagram:


\begin{center}
    \begin{tikzcd}
    \oM_3(1)\cong\oM_3 \arrow[r, "T"] & \oM_3(\frac9{11})\cong  \oM_3^{\rm ps}\arrow[rr, dashrightarrow, "\vartheta"]\arrow[dr, "\Psi", end anchor={[yshift=2ex, xshift=-1em]}] & & \oM_3(\frac7{10}-\epsilon) \cong \oM_3^{\rm hs} \arrow[dl,swap, "\Psi^+",end anchor={[yshift=2ex, xshift=1em]}]\arrow[dr, "\Theta", end anchor={[yshift=2ex, xshift=-1em]}]&\\
    &&\mathclap{\oM_3(\frac{7}{10})\cong \oM_3^{\rm cs} \cong \oP^*_4}&&
    \mathclap{\oM_3(\frac{17}{28})\cong \oP^{\GIT}_4}
    \end{tikzcd}
\end{center}

The main results of their work can be summarized in the following. 

\begin{theorem}[Birational geometry of the moduli space of genus three curves \cite{hyeon2010log}] \leavevmode
\begin{itemize}
\item There is a contraction morphism $T: \oM_3 \to \oM_3(\frac9{11}) \cong \oM_3^{\rm ps}$ to Schubert's moduli space of pseudostable curves, given by contracting the locus of elliptic tails.
\item There is a small contraction $\Psi: \oM_3^{\rm ps} \to \oM_3^{\rm cs}$, to the GIT quotient of the Chow variety of bicanonical curves $\mathrm{Chow}_{3,2}\gquot \SL(6)$ given by contracting the locus of elliptic bridges. 
\item There is a flip $\Psi^+: \oM_3(\alpha) \to \oM_3(\frac7{10})$ for $17/28 < \alpha < 7/10$, where $\oM_3(\alpha) \cong \oM_3^{\rm hs}$, the GIT quotient of the Hilbert scheme of bicanonical curves $\mathrm{Hilb}_{3,2}\gquot \SL(6)$. 
\item There is a divisorial contraction $\Theta: \oM_3^{\rm hs} \to \oP_4^{\GIT}$ to the GIT quotient of plane quartics given by $\bP(\Gamma(\calO_{\PP^2}(4)))\gquot \SL(3)$. 
\end{itemize}
Furthermore, $\oM_3^{\rm ps}$ can be identified with $\oP^{\rm H}_4$, Hacking's moduli space of plane quartics. 
\end{theorem}

\begin{remark}\leavevmode
\begin{enumerate}
\item The contraction at $\alpha = 9/11$ was originally discovered by Hassett-Hyeon \cite{hassett2013log}. 
\item In the case of degree $d = 4$ the space of Hacking was independently constructed by Hassett (see \cite{hassettquartic}).
\item As we saw in Section \ref{sec:lowdegree}, we have an isomorphism $\oM_3^{\rm hs}\cong\oP_{4,\frac{3}{4}-\epsilon}^{\K}$.
\end{enumerate}
\end{remark}

The flip $\vartheta$ can be realized as flipping the codimension $2$ locus in $\oP^{\rm H}_4$ parametrizing the curves on $\bP(1,1,2) \cup \bP(1,1,2)$ to the curve in $\oP^{\rm K}_{4, \frac{3}{4}-\epsilon}$ parametrizing tacnodal curves. These flipping and flipped loci of $\vartheta$ are contracted (via $\Psi$ or $\Psi^+$) to a point as the unique $0$-cusp in $\oP^*_4$. 
It is thus natural to expect that $\oP^*_4$ serves as the conjectural (good) moduli space $\overline{P}^{\rm CY}_4$ which parametrizes certain log canonical log Calabi-Yau pairs 
which are $\Q$-Gorenstein degenerations of $(\bP^2, \frac{3}{4}C_4)$. Recall from Section \ref{sec:lowdegree} we showed that there was a large open set $M \subset \oP^*_4$ whose codimension inside  $\oP^{\K}_{4,\frac{3}{4} - \epsilon}$ is $\geq 2$. Using this, we proved (see Theorem \ref{thm:logcy4}) that the moduli space $\oP^*_4$ was the ample model of the Hodge line bundle on $\oP_{4, \frac{3}{4}-\epsilon}^\K$. Noting that the codimension of $M$ inside $\oP^{\rm H}_4$ is also $\geq 2$, the same proof gives the following.

 \begin{theorem}\label{thm:logcy4-hacking}The moduli space $\oP^*_4$ is the ample model of the Hodge line bundle on $\oP_{4}^{\rm H}$. \end{theorem}


\subsubsection{Sextics}
Recall from Section \ref{sec:lowdegree}, we discussed the Kirwan desingularization of the GIT quotient of sextic curves (constructed by Shah), as well as the morphism from this moduli space to the Baily-Borel compactification of degree 2 K3 surfaces. By the work of \cite{AET} there is also a morphism $\oP^{\rm H}_6 \to \oP^*_6$, and so we obtain the following diagram.


\begin{center}
    \begin{tikzcd}
    \oP^{\rm H}_6 \arrow[dr] \arrow[rr, dashrightarrow] & & \widehat{P}^{\GIT}_6\cong\oP_{6,\frac{1}{2}-\epsilon}^{\K} \arrow[dl] \arrow[dr] \\
& \oP^*_6 & & \oP^{\GIT}_6 
    \end{tikzcd}
\end{center}

Again, we argue that it is natural to believe that the candidate for $\oP^{\rm CY}_6$ is $\oP^*_6$.
The locus contracted from $\oP^{\rm H}_6 \to \oP^*_6$ is divisorial, so the proof of Theorem \ref{thm:sextic-Hodge} does not imply the same result on the Hacking side as immediately as it did for degree four. In fact, \cite{AET} shows that there are actually several divisors contracted -- these divisors parametrize pairs whose double covers give K3 surfaces of Type II or Type III, in the sense of Kulikov degenerations. Valery Alexeev has suggested to us that the result is still true and can be proven by looking at the Kulikov degenerations of the relevant K3 surfaces.

\subsubsection{Log Calabi-Yau wall crossing}
In general, we can say the following.

\begin{theorem}\label{thm:comparison}
Let $\oP^{\rm H, \circ}_d$ denote the complement of the locus of non-normal pairs in $\oP^{\rm H}_d$. Let $\oP^{\K, \circ}_{d,\frac{3}{d} - \epsilon}$ denote the complement of the locus of pairs with $\lct = \frac{3}{d}$ inside $\oP^{\K}_{d,\frac{3}{d} - \epsilon}$. Then $\oP^{\rm H, \circ}_d \cong \oP^{\K,\circ}_{d,\frac{3}{d} - \epsilon}.$
\end{theorem}

\begin{proof}
We first show that $\oP^{\K, \circ}_{d,\frac{3}{d} - \epsilon} \subseteq \oP^{\rm H, \circ}_d$. If $(X,D)$ is a pair parametrized by $\oP^{\K, \circ}_{d,\frac{3}{d} - \epsilon}$ then $\lct > \frac{3}{d}$ and so by definition $(X,D)$ is a pair parametrized by $\oP^{\rm H, \circ}_d$. Conversely, by \cite[Theorem 5.2]{LWX14}, if $(X,D)$ is parametrized by $\oP^{\rm H, \circ}_d$, then $(X,D)$ is parametrized by $\oP^{\K, \circ}_{d,\frac{3}{d}-\epsilon}$.  
\end{proof}

In particular, the above result says that if one looks at the K-moduli space for coefficient $\frac{3}{d} - \epsilon$, then the only difference between this space and the Hacking moduli space are the maximally lct pairs in the K-moduli space and the non-normal pairs in the Hacking moduli space. In particular, we conjecture that there is a proper good moduli space of log Calabi-Yau pairs which relates to the K-moduli and Hacking moduli spaces via the following conjectural picture.

\begin{conj}[Log Calabi-Yau wall crossings]\label{conj:logCY}
 There exists a proper good moduli space $\oP_{d}^{\rm CY}$
 which parametrizes $\rmS$-equivalence classes of
 \emph{semistable} log Calabi-Yau pairs $(X,\frac{3}{d}D)$
 where $X$ admits a $\bQ$-Gorenstein smoothing to $\bP^2$.
 Moreover, we have a log Calabi-Yau wall crossing diagram
\[
 \oP^{\K}_{d,\frac{3}{d}-\epsilon}\xrightarrow{\phi_{-}^{\rm CY}}
 \oP_{d}^{\rm CY}\xleftarrow{\phi_{+}^{\rm CY}}\oP_{d}^{\rm H}
\]
where $\oP_{d}^{\rm CY}$ is the common ample model of the Hodge line bundles on  $\oP^{\K}_{d,\frac{3}{d}-\epsilon}$ and $\oP_{d}^{\rm H}$.
\end{conj}

\subsection{Log Calabi-Yau quintics}\label{sec:logcyquintics}
 Recall that Hacking's moduli space $\oP^{\rm H}_5$ parametrizes $\bQ$-Gorenstein deformations of pairs $(\bP^2, cC)$ where $c = \frac{3}{5} + \epsilon$ for $\epsilon$ sufficiently small and $\deg(C)=5$. From Theorem \ref{thm:comparison}, we know that there is a birational map \[ \oP^{\rm H}_5 \dashrightarrow \oP^{\rm{K}}_{c}\] where $c \in \left( \frac{54}{95}, \frac{3}{5} \right)$,  that is an isomorphism over the locus of pairs $(X,D)$ where $X$ is normal and $\lct(X,D) > \frac{3}{5}$. Here we omit the degree $5$ in the subscript of K-moduli spaces and stacks. In other words, the pairs that will become unstable when increasing the weight from $\frac{3}{5} - \epsilon$ to $\frac{3}{5} + \epsilon$ are precisely the pairs with $\lct = \frac{3}{5}$ -- i.e. $A_9$ and $D_6$ singularities.  Note that the $A_9$ singularity can occur either at a smooth point of the surface $X$ or at a $\frac{1}{4}(1,1)$ singularity.  In this setting, we have the following results:
 \begin{enumerate}
     \item the stable replacement in $\oP^{\rm H}_5$ of a curve $C \subset X$ with $A_9$ singularity is either a curve on $\bP(1,1,5) \cup X_6$ or a curve on $\bP(1,1,5) \cup \bP(1,4,5)$, and
     \item the stable replacement in $\oP^{\rm H}_5$ of a curve $C \subset X$ with a $D_6$ singularity is a curve on $\bP(1,1,2) \cup \bP(1,1,2)$. 
 \end{enumerate}
 
These are precisely the non-normal pairs which appear in $\oP^{\rm H}_5$.  We conjecture that there is a projective variety $\oP_{5}^{\rm CY}$ parametrizing $\rmS$-equivalence classes of slc log Calabi-Yau pairs $(X,\frac{3}{5}D)$. 

\begin{conj}\label{conj:logCYquintics}
The ample models of the Hodge line bundles on $\oP^{\rm H}_5$ and $\oP^{\rm K}_{\frac{3}{5}-\epsilon}$ exist and coincide. We denote this common ample model by $\oP_{5}^{\rm CY}$. It parameterizes $\rmS$-equivalence classes of slc log Calabi-Yau pairs that are $\bQ$-Gorenstein deformations of $(\bP^2, \frac{3}{5}C)$. 

In particular, $\oP_{5}^{\rm CY}$ should serve as the base of a flip 
\[\xymatrix{ \oP^{\rm H}_{5} \ar[rd] \ar@{-->}[rr] & & \oP^{\K}_{3/5 - \epsilon} \ar[ld] \\
& \oP_{5}^{\rm CY} & }
\]
that realizes the rational map $\oP^{\rm H}_5 \dashrightarrow \oP^{\rm{K}}_{3/5-\epsilon}$.
\end{conj}

In the following section, we provide some evidence for Conjecture \ref{conj:logCYquintics}.

\subsubsection{Evidence for the log Calabi-Yau conjecture}
First, we verify that any curve with an $A_9$ singularity admits a common degeneration to a unique curve on $\bP(1,1,5) \cup \bP(1,4,5)$ with an $A_9$ singularity in each component.  Similarly, any curve with a $D_6$ singularity admits a common degeneration to a unique curve on $\bP(1,1,2) \cup \bP(1,1,2)$ with a $D_6$ singularity in each component. 
\begin{prop}
All curves  in $\oP^{\K}_{3/5 - \epsilon}$ with an $A_9$ singularity and all curves in $\oP^{\rm H}_{5}$ on $X_6 \cup \bP(1,1,5)$ or $\bP(1,1,5) \cup \bP(1,4,5)$ admit a common degeneration to a unique curve on $\bP(1,1,5) \cup \bP(1,4,5)$ with log canonical threshold exactly $\frac{3}{5}$.
\end{prop}

\begin{proof}
By Proposition \ref{prop:jetsummary}, a plane quintic curve $C$ with an $A_9$ singularity has the equation $$(x - y^2)((x - y^2)(1 + sx) -x^2(2py + rx)) + ux^5=0$$ in the affine coordinates $[x,y,1]$ for some choice of  $(s, r, p, u)\in\bA^4$ satisfying $p^2\neq u$. We will first construct a weakly special degeneration of $(\bP^2,\frac{3}{5}C)$ to a pair $(X_6\cup\bP(1,1,5), \frac{3}{5}C_0)$. 

Take our family $(\bP^2,C)\times\bA^1$. We perform the following birational transformations:

\[
\bP^2\times\bA^1\xleftarrow{~~\pi~~} \cX \xrightarrow{~~g~~}\cY
\]
Where in the central fiber we have
\[
 \bP^2\xleftarrow{~~\pi~~} S\cup X\xrightarrow{~~g~~} S\cup X'
\]
Here $\pi$ is the $(5,1,1)$-weighted blow up of $\bP^2\times\bA^1$ in the local coordinates $(x',y,t)$ with $x':=x-y^2-py^5$, and $S=\bP(1,1,5)$ is the exceptional divisor of $\pi$. Let $Q=(x-y^2=0)$ be a smooth conic in $\bP^2$. Then it is clear that $\pi^* Q=\oQ + 5E$ where $E$ is the exceptional curve of $\pi: X\to\bP^2$ and $\oQ:=\pi_*^{-1} Q$. Since $(E^2)=-\frac{1}{5}$, we know that $(\oQ^2)=-1$ and $\oQ$ is a smooth rational curve contained in the smooth locus of $X$. By computation we know that $\cN_{\oQ/\cX}\cong\cO_{\oQ}(-1)\oplus\cO_{\oQ}(-1)$. Hence 
$g:\cX\to\cY$ is the small flopping contraction of $\oQ$. 

Next we show that $(X',\pi_*E)\cong (X_6, (x_1=0))$ where $X_6=(x_0 x_3=x_1^3+x_2^2)$ is a weighted hypersurface in $\bP(1,2,3,5)_{x_0,x_1,x_2,x_3}$. Let $L:=\pi^*\cO(1) -2 E$ be a divisor on $X$. Then $L$ induces a morphism $g':X\to \bP(1,2,3,5)$ defined by 
\begin{align*}
x_0&=\pi^* x- 2E\in H^0(X,L), \quad x_1=\pi^*(xz-y^2)-4E\in H^0(X,2L),\\
x_2&=\pi^*(y(xz-y^2))-6E\in H^0(X,3L),\quad x_3=\pi^*(z(xz-y^2)^2)-10E\in H^0(X,5L).
\end{align*}
It is easy to show that $g'$ is a birational map which only contracts $\oQ$, and the image of $g'$ is exactly $X_6$. Thus $g'=g$ and $\pi_* E=\pi_*(\oQ+E)=(x_1=0)$. Let $[y_0,y_1,y_2]$ be the projective coordinates of $\bP(1,1,5)$ as the projectivization of $(t,y,x')$. 
Thus $S\cup X\cong \bP(1,1,5)\cup X_6$ where the double locus $\pi_*E$ is $(y_0=0)\subset\bP(1,1,5)$ and $(x_1=0)\subset X_6$. The degeneration $C_0$ of $C$ has equations
\[
(y_2^2 = (p^2-u )y_1^{10})\subset\bP(1,1,5)\quad \textrm{ and }\quad(x_3+x_0x_1^2-2p x_0^2 x_2-r x_0^3 x_1+u x_0^5=0)\subset X_6.
\] 

Next we show that $(\bP(1,1,5)\cup X_6,\frac{3}{5}C_0)$ admits a weakly special degeneration to an slc log Calabi-Yau pair $(\bP(1,1,5)\cup \bP(1,4,5),\frac{3}{5}C_0')$ that is unique up to isomorphism. Consider the $1$-PS $\sigma:\bG_m\to\Aut(\bP(1,2,3,5))$ defined as  $\sigma(t)\cdot[x_0,x_1,x_2,x_3]=[x_0,t^{-1} x_1,x_2,x_3]$. Then we know that $\sigma(t)\cdot X_6$ has the equation $(x_0x_3=t^3 x_1^3+x_2^2)$ in $\bP(1,2,3,5)$. Hence $\lim_{t\to 0}\sigma(t)\cdot X_6=(x_0x_3=x_2^2)$. 
Indeed, we have an embedding from $\bP(1,4,5)_{z_0,z_1,z_2}$ to $\bP(1,2,3,5)$ as $[z_0,z_1,z_2]\mapsto[z_0^2, z_1, z_0z_2, z_2^2]$ whose image has equation $x_0 x_3=x_2^2$. Hence $\lim_{t\to 0}\sigma(t)\cdot X_6\cong\bP(1,4,5)$. By taking limit of the equation of $C_0\cap X_6$ under the action of $\sigma$, we know that $C_0'\cap \bP(1,4,5)$ has equation $(z_2^2 -2p z_0^5 z_2 + u z_0^{10}=0)$. Since $p^2\neq u$, after a suitable projective coordinate change the equation of $C_0'$ becomes $(y_2^2=y_1^{10})$ and $(z_2^2 = z_0^{10})$ in $\bP(1,1,5)\cup\bP(1,4,5)$. It is clear that $C_0'$ has an $A_9$-singularity in the smooth locus of $\bP(1,1,5)$ and an $A_3$-singularity at the $\frac{1}{4}(1,1)$-singularity of $\bP(1,4,5)$. 
Thus $\lct(\bP(1,1,5)\cup\bP(1,4,5); C_0')=\frac{3}{5}$.
Since $\sigma$ fixes $\pi_*E$ pointwisely, we may compose the two degenerations as in \cite[Proof of Lemma 3.1]{LWX18} to obtain a weakly special degeneration of $(\bP^2, \frac{3}{5}C_0)$ to $(\bP(1,1,5)\cup\bP(1,4,5),\frac{3}{5}C_0')$.

The computation above can be extended to include the case of curves with $A_9$-singularities on $\bP(1,1,4)$, $\bP(1,4,25)$ and $X_{26}$ in $\oP_{\frac{3}{5}-\epsilon}^{\K}$. 
For the Hacking moduli space, the construction is similar to the second step in the above degenerations with minor difference to consider $1$-PS in $\Aut(\bP(1,1,5))$ and $\Aut(\bP(1,4,5))$ as well. Thus the proof is finished. 
\end{proof}



\begin{prop}
 All curves in $\oP^{\K}_{3/5 - \epsilon}$ with a $D_6$ singularity and all curves in $\oP^{\rm H}_{5}$ lying on $\bP(1,1,2) \cup \bP(1,1,2)$ admit a common degeneration to a unique curve on $\bP(1,1,2) \cup \bP(1,1,2)$ with log canonical threshold exactly $\frac{3}{5}$.
\end{prop}

\begin{proof}
Consider the family $\pi:\cX\to \bA^2$ given by 
\[ \cX:=(x_0x_2 = rx_1^2 + sx_3) \subset \bP(1,1,1,2)_{x_0,x_1,x_2,x_3} \times \bA^2_{r,s}. \]
The fiber of $\pi$ above $r = 0, s = 0$ is isomorphic to $\bP(1,1,2) \cup \bP(1,1,2)$ with projective coordinates $[x_0,x_1,x_3]$ and $[x_1,x_2, x_3]$ respectively.  For $s = 0$ but $r \ne 0$, the fiber of $\pi$ is isomorphic to $\bP(1,1,4)_{y_0,y_1,y_2}$ where the isomorphism is given by $[y_0,y_1,y_2]\mapsto [ry_0^2, y_0y_1, y_1^2, y_2]$. For $s \ne 0$, the fiber of $\pi$ is isomorphic to $\bP^2_{x,y,z}$ where the isomorphism is given by $[x,y,z]\mapsto [x,y,z, s^{-1}(xz- ry^2)]$. Let $\sigma:\bG_m\to \Aut(\cX)$ be a $1$-PS defined as 
\[
\sigma(t)\cdot([x_0,x_1,x_2,x_3],(r,s)):=([tx_0, x_1, x_2, x_3], (tr,ts)).
\]
Then $\pi$ is $\bG_m$-equivariant with respect to $\sigma$ and the $\bG_m$-action $(r,s)\mapsto (tr,ts)$ on $\bA^2$.
Consider a divisor $\cD$ on $\cX$ defined by 
\[
\cD:=(x_1(x_3-x_1^2)(x_3+x_1^2)=0).
\]
It is clear that $\cD$ is $\bG_m$-invariant under the action of $\sigma$. 
We will show that suitable restrictions of the family $\pi:(\cX,\frac{3}{5}\cD)\to \bA^2$ give the desired weakly special degenerations of curves in $\oP_{\frac{3}{5}-\epsilon}^{\K}$ with $D_6$-singularities. 

From Section \ref{sec:quintics}, we know that the locus of curves with $D_6$-singularities in $\oP_{\frac{3}{5}-\epsilon}^{\K}$ is $\oSigma_{6, \frac{3}{5}-\epsilon}$ which is isomorphic to $\bP^1$. 
The pairs parametrized by $\oSigma_{6, \frac{3}{5}-\epsilon}$ consists of the following form:
\[ (\bP^2, (y( xz - (1+a)y^2 )( xz + (1-a)y^2) =0)) \textrm{ where }a\in\bA^1, \textrm{ and }(\bP(1,1,4), (y_0y_1(y_2^2-y_0^4 y_1^4)=0)). \]
Then one can check that the restriction of $\pi:(\cX,\frac{3}{5}\cD)\to \bA^2$ to the affine line $\{(as,s)\mid s\in\bA^1\}$ (resp. $\{(r,0)\mid r\in\bA^1\}$) gives weakly special degeneration of $D_6$-curves on $\bP^2$ (resp. $\bP(1,1,4)$).


A similar computation can be done for curves on $\bP(1,1,2) \cup \bP(1,1,2)$ in the Hacking moduli space. Note that the common degeneration has equation $(x_1(x_3^2-x_1^4)=0)$ in $\bP(1,1,2)\cup\bP(1,1,2)$.
\end{proof}

These two propositions essentially show that if the log Calabi-Yau moduli space $\oP_5^{\rm CY}$ exists, then there must be two distinct points in $\oP_5^{\rm CY}$ of $\rmS$-equivalence classes of slc log Calabi-Yau degenerations of pairs $(\bP^2, \frac{3}{5}C)$ parameterizing curves with $A_9$ or $D_6$ singularities, respectively.  In other words, the conjectural map $\oP^{\K}_{\frac{3}{5} - \epsilon} \to P$ must contract the disjoint loci $\oSigma_{6, \frac{3}{5}-\epsilon}$ and $\oSigma_{7, \frac{3}{5}-\epsilon}$ to two distinct points. Denote by $\oSigma_6^{\rm H}$ (resp. $\oSigma_7^{\rm H}$) the disjoint loci in $\oP^{\rm H}_5$ parametrizing curves on $\bP(1,1,2)\cup\bP(1,1,2)$ (resp. $\bP(1,1,5)\cup X_6$ or $\bP(1,1,5)\cup \bP(1,4,5)$). Then similarly
the conjectural map $\oP^{\rm H}_5 \to P$ must contract the disjoint loci $\oSigma_6^{\rm H}$ and $\oSigma_7^{\rm H}$ to the same set of two points.  

To form the projective variety $\oP_5^{\rm CY}$, we expect that the Hodge line bundles are semiample on $\oP^{\K}_{\frac{3}{5} - \epsilon}$ and $\oP^{\rm H}_5 $ and $\bQ$-trivial exactly on these contracted loci.  Indeed, the Hodge line bundle is the limit of the CM line bundle by Proposition \ref{prop:k-cm-interpolation}.  Moreover, the CM line bundle is known to be ample on $\oP^{\rm H}_5$ by \cite{PX17} and big and nef (and conjecturally ample) on $\oP^{\K}_c$ by \cite{LWX18, CP18, Pos19}.  Therefore, we expect some postivitity properties of the Hodge line bundle.  

As further evidence, we verify that the Hodge line bundle is trivial on the locus $\oSigma_6^{\rm H}\sqcup\oSigma_7^{\rm H}$ parameterizing curves on non-normal surfaces in $\oP^{\rm H}_5 $. Denote by $L_+$ and $L_-$ the Hodge $\bQ$-line bundles over $\oP_5^H$ and $\oP_{\frac{3}{5}-\epsilon}^{\K}$ respectively.


\begin{prop}\label{prop:quinticshodge}
 The restriction $L_+|_{\oSigma_i^{\rm H}}$ is $\bQ$-linearly trivial for $i=6,7$.
\end{prop}

\begin{proof}
We first look at the case of $\oSigma_7^{\rm H}$. Each Hacking stable pair
in $\oSigma_7^{\rm H}$ is uniquely determined by gluing two plt
pairs $(\bP(1,1,5),(y=0)+\frac{3}{5}C_1)$ and $(X,D)$ where
\begin{itemize}
 \item $C_1$ has the equation $z^2=x^{10}+a_2 x^8y^2+a_3 x^7 y^3+\cdots+a_{10}y^{10}$
 where $(a_2,a_3,\cdots,a_{10})\in\bA^9\setminus\{0\}$;
 \item $X$ is a weighted hypersurface in $\bP(1,2,3,5)$
 defined by the equation $(xw=ty^3+z^2)$;
 \item $D=(y=0)+\frac{3}{5}C_2$ where $C_2$ has the 
 equation $w=x^5+b_1 x^3y+b_2 xy^2$ such that $(b_1,b_2,t)\in \bA^3\setminus\{0\}$.
\end{itemize}
The double locus on both components is a $\bP^1$ with
three marked points: one of them is the index $5$ singularity
and the other two are intersections with curve $C_i$. Hence the gluing
is unique up to a $\mu_2$-action, and so we have $\oSigma_7^{\rm H}\cong(\bP(2,3,\cdots,10)\times\bP(1,2,3))/G$ 
where $G$ is a finite group acting on the weighted bi-projective space
identifying isomorphic fibers. Therefore, to show that
$L_+|_{\oSigma_7^{\rm H}}$ is $\bQ$-linearly trivial, it suffices to 
show that the $\bQ$-line bundle $L$ on each weighted projective space
is $\bQ$-linearly trivial.

Let us start with the weighted projective space
$\bP:=\bP(2,3,\cdots,10)$. Denote $T:=\bA^9\setminus\{0\}$.
Consider the family $\pi_T:(\cX_T,\cD_T)\to T$
where $\cX_T=\bP(1,1,5)\times T$ and $\cD_T=(y=0)+\frac{3}{5}\calC_1$
with $\calC_1=(z^2=x^{10}+a_2 x^8y^2+a_3 x^7 y^3+\cdots+a_{10}y^{10})$.
We have a $\bG_m$ action on $T$ given by $(a_i)\mapsto(\lambda^i a_i)$.
This action lifts to $\cX_T$ as $([x,y,z],(a_i))
\mapsto([x,\lambda^{-1}y,z],(\lambda^i a_i))$ so that
$\pi_T$ is $\bG_m$-equivariant. It is clear that
$\pi_T$ descends to a universal family of plt
log CY pairs over $[T/\bG_m]$ whose coarse space is $\bP$.
Hence it suffices to show that the $\bG_m$-linearized
line bundle $\cL_T^{\otimes 5}:=(\pi_T)_*\cO_{\cX_T}(5(K_{\cX_T/T}+\cD_T))$ on $T$
descends to a trivial line bundle on $\bP$.
We pick a nowhere zero section $\tau$ of $\cL_T^{\otimes 5}$
which has the  expression 
$ \tau(a_2,\cdots,a_{10})=y^{-5}(1-\sum_{i=2}^{10} a_i x^{10-i}y^i)^{-3} (dx\wedge dy)^{\otimes 5}$ in the affine chart $z=1$.
Hence for any $\lambda\in\bG_m$ we have
\begin{align*}
 (\lambda_{*}\tau)(\lambda^i a_i)& =(\lambda y)^{-5}(1-\sum_{i=2}^{10} a_i x^{10-i}(\lambda y)^i)^{-3} (dx\wedge \lambda dy)^{\otimes 5}
 =\tau(\lambda^i a_i).
\end{align*}
Thus the $\bG_m$-linearization on $\cL_T^{\otimes 5}$ is trivial.

Next we analyze the weighted projective space $\bP(1,2,3)$. 
Since it has Picard number $1$, it suffices to show
that $L_+$ restricting to the
curve $\bP(1,2)$ is $\bQ$-linearly trivial where 
$\bP(1,2)$ corresponds to $t=0$, i.e. the surface component
being $\bP(1,4,5)$. Using the projective coordinates $[x,y,z]$ of
$\bP(1,4,5)$, the divisor $D=(y=0)+\frac{3}{5}C_2$
where $C_2=(z^2=x^{10}+b_1 x^6 y+b_2 x^2y^2)$. Then
similar computations as in the case of $\bP(1,1,5)$
imply $L_+|_{\bP(1,2)}$ is $\bQ$-linearly trivial.
This finishes the proof of $\bQ$-linear triviality of
$L_+|_{\oSigma_7^{\rm H}}$.

The case of $\oSigma_6^{\rm H}$ is similar to $\oSigma_7^{\rm H}$. Since each surface
component is $\bP(1,1,2)$, we may just consider one component
given by the plt pair $(\bP(1,1,2),(y=0)+\frac{3}{5}C)$
where $C$ has the equation $(xz^2-x^5=a y^3 z+b_1 x^4 y+b_2 x^3y^2+\cdots+b_5 y^5)$
for $(a,b_1,\cdots,b_5)\in\bA^6\setminus\{0\}=:T$. 
Then the $\bG_m$-action on $\cX_T=\bP(1,1,2)\times T$
is given by $([x,y,z],a,b_i)\mapsto([x,\lambda^{-1}y,z],\lambda^3 a,\lambda^i b_i)$.
Hence the moduli space parametrizing $(\bP(1,1,2),(y=0)+\frac{3}{5}C)$
is given by $\bP:=\bP(1,2,3,3,4,5)$ and $\oSigma_6^{\rm H}\cong(\bP\times\bP)/G$
for some finite group $G$. Then similar computations show that
the $\bG_m$-linearization on $\cL_{T}^{\otimes 5}$ is trivial, hence
$L_+|_{\oSigma_6^{\rm H}}$ is $\bQ$-linearly trivial.
\end{proof}

To verify Conjecture \ref{conj:logCYquintics}, one must first show that the analogous statement of Proposition \ref{prop:quinticshodge} holds for $L_-$ and that $L_{\pm}$ is in fact ample away from the loci $\oSigma_7^{\rm H}$ and $\oSigma_6^{\rm H}$.  If this holds, the ample models of the Hodge line bundles on $\oP_5^{\rm H}$ and $\oP_{\frac{3}{5}-\epsilon}^{\K}$ would coincide set-theoretically with our notion of $\rm S$-equivalence classes.  However, putting a natural modular structure from the K-stability viewpoint on the log Calabi-Yau space and determining the right class of objects to parameterize prevent us from defining a moduli stack of these pairs in this paper. We will pursue this is forthcoming work.  

\subsection{Higher dimensional applications}\label{sec:highdim}

In this section we give some applications of our machinery  developed in Sections \ref{sec:prelim} and \ref{sec:construction}.
The following result improves \cite[Theorem 1.2]{GMGS} by removing their assumption on the Gap Conjecture and allowing small degree.

\begin{thm}\label{thm:highdim}
Let $n$ and $d\geq 2$ be positive integers. Then there exists a positive rational number $c_1=c_1(n,d)$ such that for any fixed $0<c<c_1$, a hypersurface $S\subset\bP^n$ of degree $d$ is GIT (poly/semi)stable if and only if the log Fano pair $(\bP^n, cS)$ is K-(poly/semi)stable.
\end{thm}

\begin{proof}
Let $\chi_0$ be the Hilbert polynomial of $(\bP^n,\cO(n+1))$. Let $r:=\frac{d}{n+1}$. We consider the K-moduli stack $\cK\cM_{\chi_0,r,c}$ as in Definition \ref{defn:kmoduli}. By Theorem \ref{thm:logFano-wallcrossing}, there exists a positive rational number $c_1=c_1(n,d)$ such that $\cK\cM_{\chi_0,r,c}$ remains constant for any $c\in (0, c_1)$. Thus for any K-semistable pair $[(X,cD)]\in\cK\cM_{\chi_0,r,c}$ we have that $X$ is K-semistable and $(-K_X)^n=(-K_{\bP^n})^n=(n+1)^n$. Hence by \cite[Theorem 36]{Liu18} we know that $X\cong\bP^n$ and $D\subset X$ is a hypersurface of degree $d$. By the Paul-Tian criterion Theorem \ref{thm:paultian} and computations on CM line bundles similar to the proof of Proposition \ref{prop:P^2-paultian}, we know that K-(poly/semi)stability of $(\bP^n,cS)$ implies GIT (poly/semi)stability of $S$. 

We first show that there is a morphism of Artin stacks $\varphi: \cK\cM_{\chi_0,r,c}\to \cH_d^n$ where $\cH_d^n$ is the GIT quotient stack $[\bP(H^0(\bP^n,\cO(d)))^{\mathrm{ss}} /\PGL(n+1)]$ of degree $d$ hypersurfaces in $\bP^n$.
Indeed, the K-moduli stack $\cK\cM_{\chi_0, r, c}$ is defined to be the quotient $\left[Z^{\mathrm{red}}_{c,m} / \mathrm{PGL}(N_m+1)\right]$ for every $m$ sufficiently divisible.
Let 
\[
Z_m^{\GIT}:= \{(X,D)\in Z_m^{\klt}\mid (X,D)\cong (\bP^n, D')\textrm{ where }[D']\in \bP(H^0(\bP^n,\cO(d)))^{\mathrm{ss}}\}.
\]
Since fiber being $\bP^n$ and GIT semistability are both Zariski open conditions, we know that $Z_m^{\GIT}$  is a Zariski open subset of $Z_m^{\klt}$. We equip $Z_m^{\GIT}$ with the reduced scheme structure. 
Thus there is a morphism $\varphi_m$ from the K-moduli stack to the GIT stack $\cH_{d,m}^{n}:=\left[Z^{\GIT}_{m} / \mathrm{PGL}(N_m+1)\right]$. By Theorem \ref{thm:stabilization} the K-moduli stacks stabilize. By a similar argument to the last paragraph of the proof of Theorem \ref{thm:firstwallbefore},  the morphisms $\varphi_m$ stabilize to $\varphi: \cK\cM_{\chi_0,r,c}\to \cH_d^{n}$ for $m$ sufficiently divisible. 

To show that the morphism $\varphi$ is representable, we just need to show that the fiber product of a scheme and the K-moduli stack over  the GIT stack is a scheme. For simplicity, denote by $G:=\PGL(N_m+1)$. A map $T\to \cH_{d,m}^n$ from a scheme $T$ to the GIT stack is equivalent to the data of a $G$-torsor $\psi: P_T\to T$ together with a $G$-equivariant morphism $P_T\to Z_m^{\GIT}$. Since $Z_{c,m}^{\red} \hookrightarrow Z_m^{\GIT}$ is a $G$-equivariant open immersion, we know that $P':=P_T\times_{Z_m^{\GIT}} Z_{c,m}^{\red}$ is a $G$-invariant open subscheme of $P_T$. Since $\psi: P_T\to T$ is flat, we know that $T':= \psi(P')$ is an open subscheme of $T$. Thus $P'\to T'$ is a $G$-torsor as $P'$ admits a $G$-action and this map is surjective. As a result, we have $T'\cong T\times_{\cH_{d,m}^n} \cK\cM_{\chi_0,r,c}$, which implies that $\varphi$ is representable and an open immersion of Artin stacks.
So now it suffices to show that $\varphi$ is an isomorphism.

Next, we verify that $\cK\cM_{\chi_0,r,c}$ is non-empty. If $d\geq n+1$, then any smooth hypersurface $S$ satisfies $(\bP^n,\frac{n+1}{d}S)$ is a log canonical log Calabi-Yau pair which implies $[(\bP^n,cS)]\in \cK\cM_{\chi_0,r,c}$ by Proposition \ref{prop:k-interpolation}. If $d\leq n$, then we know that there exists a smooth hypersurface $S$ that admits K\"ahler-Einstein metrics (see e.g. \cite[Page 85-87]{Tia00} or \cite{AGP06}). By degeneration to normal cone of $S$, we know that $(\bP^n, cS)$ special degenerates to $(X_0, cS_\infty)$ where $X_0=C_p(S, \cO_S(d))$ is the projective cone (see \cite[Section 3.1]{Kol13} for a definition) and $S_\infty$ is the section at infinity. By \cite[Proposition 3.3]{LL16}, we know that $(X_0, (1-\frac{r^{-1}-1}{n})S_\infty)$
admits a conical K\"ahler-Einstein metric. Hence $(\bP^{n}, (1-\frac{r^{-1}-1}{n})S)$ is K-semistable by Theorem \ref{thm:Kss-spdeg}. Since $1-\frac{r^{-1}-1}{n}=\frac{(d-1)(n+1)}{dn}>0$, we know that $(\bP^n, \epsilon S)$ is K-polystable for $0<\epsilon\ll 1$ by Proposition \ref{prop:k-interpolation}. Hence $(\bP^n, cS)$ is K-polystable for any $0<c<c_1$. 

Finally we show that $\varphi$ is an isomorphism. By taking good moduli spaces, let $\varphi': KM_{\chi_0,r,c}\to H_d^n$ be the descent of $\varphi$ where $H_d^n:=  \bP(H^0(\bP^n,\cO(d)))^{\mathrm{ss}} \sslash \PGL(n+1)$. It follows that $\varphi'$ is an injective proper morphism which implies that $\varphi'$ is finite. Hence $\varphi$ is finite by \cite[Proposition 6.4]{alper}. This together with $\varphi$ being an open immersion implies that $\varphi$ is an isomorphism by Zariski's Main Theorem.
\end{proof}

\appendix

\section{Calculations of K-semistable thresholds and K-polystable replacements}\label{sec:calculations}
In this appendix, we calculate the K-semistable thresholds and K-polystable replacements of K-semistable pairs to determine the location of the walls for $d = 5$. By Proposition \ref{prop:lclastwall}, to understand the wall crossings occurring after the first wall, we must understand the curves parametrized by $\oP^{\GIT}_5 \setminus P^{\klt}$. In Section \ref{sec:GIT5}, we showed that the curves parametrized by this space (aside from the non-reduced conic which was discussed at length in Section \ref{sec:firstwall})  have $A_{12}$, $A_{11}$, $A_{10}$, $A_{9}$ and $D_6$ singularities. Therefore this section contains the relevant calculations used in Section \ref{sec:explicitwalls}, and the subsections are organized by the singularity type. 

Before proceeding, we state a result that will be used throughout. A standard jet computation shows that
if a quintic curve has a double point of type $A_k$ with $k\geq 9$ (not including $\infty$),
then after a suitable projective coordinate change we obtain
the following equation in the affine coordinate:
\begin{equation}\label{eq:quinticA_9}
 (x-y^2)\big((x-y^2)(1+sx)-x^2(2py+rx)\big)+ux^5=0.
\end{equation}
Here the double point is $(0,0)$ for any parameters
$(s,r,p,u)\in\bA^1\times(\bA^3\setminus\{0\})$.
Indeed, in the analytic coordinates $(x',y)$ where $x':=x-y^2-py^5$, the above equation \eqref{eq:quinticA_9} becomes
\[
x'^2=(p^2-u)y^{10}+ \textrm{higher order terms},
\]
where $(x',y)$ has weight $(5,1)$. 
When $r,p,u$ all vanish, we recover the curve $Q_5$.
If we rescale the coordinates as $(x,y)\mapsto(\lambda^{-2} x,\lambda^{-1} y)$,
then the coordinates change as $(s,r,p,u)\mapsto(\lambda^2 s,\lambda^4 r,\lambda^3 p, \lambda^6 u)$.
Note that the two curves defined by
$(s,r,\pm p,u)$ are projectively equivalent by $y\mapsto -y$.
Therefore, we may take the parameter space 
$\bA^4\setminus\{0\}$ and quotient out by projective equivalence to describe the locus of curves with an $A_k$ singularity ($k \ge 9$) including $Q_5$. 

Similarly, a standard computation shows that GIT polystable quintics with a $D_6$ singularity, up to a projective coordinate change, can be written as 
\[ 
y(x -t_1 y^2)(x - t_2y^2) = 0
\]
where $t_1\neq t_2$. Since $t_1$ and $t_2$ are symmetric, we may take the coordinate change given by $(s_1,s_2):= (t_1+t_2, (t_1-t_2)^2)$. Then $\Sigma_6$ corresponds to $s_2\neq 0$. 
The $\bG_m$-action scales the coordinates as $(t_1,t_2)\mapsto (\lambda t_1,\lambda t_2)$ and $(s_1,s_2)\mapsto (\lambda s_1,\lambda^2 s_2)$. It is clear that two curves with $D_6$-singularities are projectively equivalent if and only if their $(s_1,s_2)$ belong to the same $\bG_m$-orbit. Thus we may take the parameter space $\bA^2_{s_1,s_2} \setminus \{0\}$ and quotient by this projective equivalence.  

Combining the previous statements, following the notation of Lemma \ref{lem:quinticgit}, we obtain the following description of the loci $\Sigma_i$ in $\oP_5^{\GIT}$.  

\begin{prop}\label{prop:jetsummary}
 The Zariski closure $\oSigma_7$ of the $A_9$ locus in $\oP_5^{\GIT}$ is isomorphic to $\bP(1,2,3,3)$ with projective coordinates $[s,r,h,u]$ where $h:=p^2-u$. Moreover,
 \begin{itemize}
   \item $\Sigma_1$ corresponds to the point $[1,0,0,0]$;
  \item $\Sigma_2$ corresponds to the point $[0,0,0,1]$;
  \item $\oSigma_3$ corresponds to $h=u=0$;
  \item $\oSigma_4$ corresponds to $r=0$ and $h=0$;
  \item $\oSigma_5$ corresponds to $h=0$;
 \end{itemize}
 The Zariski closure $\oSigma_6$ of the $D_6$ locus in $\oP_5^{\GIT}$ is isomorphic to $\bP(1,2) \cong \bP^1$ with projective coordinates $[s_1,s_2]$.  Moreover, 
 \begin{itemize}
     \item $\Sigma_1$ corresponds to the point $[1,0]$ and
     \item $\oSigma_6$ and $\oSigma_7$ intersect only at the point $\Sigma_1$. 
 \end{itemize}
\end{prop}

\subsection{$A_{12}$}\label{sec:a12}
Recall that $X_{26}$ is given by $(xw - y^{13}-z^2=0) \subset \bP(1, 2, 13, 25).$
In this section, we verify the K-polystability of $(X_{26},\frac{8}{15}C_0')$, where $C_0' = (w=0)$,
using techniques of Ilten and S\"u{\ss} \cite{IS17}.

\begin{prop}\label{prop:a12kps}
 The log Fano pair $(X_{26}, \frac{8}{15}C_0')$ is K-polystable where $C_0'=(w=0)$.
\end{prop}

\begin{proof}
Consider the projective coordinate ring $A$ of $X_{26}$,
where 
\[
 A=\bC[x,y,z,w]/(xw-y^{13}-z^2).
\]
Consider the action of $\bG_m^2$ on $Y:=\Spec A$ given by 
$(x,y,z,w)\mapsto (\lambda^{26}\mu x, \lambda^{2}\mu^2 y, 
\lambda^{13}\mu^{13}z, \mu^{25}w)$. It is clear that this action descends to a $\bG_m$-action on $(X_{26}, \frac{8}{15}C_0')$. Thus by \cite[Theorem 1.4]{LWX18} it suffices to show $\bG_m$-equivariant K-polystability of $(X_{26}, \frac{8}{15}C_0')$.

Denote by $R:=\bC[x,y,z,w]$, and
$F=xw-y^{13}-z^{2}$, so $A=R/(F)$. The character
lattice $M=\bZ\langle(26,1),(1,1)\rangle\subset\bZ^2$,
and for every $(\alpha,\beta)\in M$, we know
\[
R_{(\alpha,\beta)}=\langle x^a y^b z^c w^d\mid a(26,1)+b(2,2)+c(13,13)+d(0,25)=(\alpha,\beta)\rangle.
\]
In the ring $A$, if we want to determine $A_{(\alpha,\beta)}$
it suffices to assume $c\in\{0,1\}$.
We also notice that the parities of $c$ and $\alpha$ are the same,
so if we assume $\alpha$ even then we only need to take $c=0$.
Then the equation becomes
\[
 A_{(\alpha,\beta)}=\langle x^a y^b w^d\mid a(26,1)+b(2,2)+d(0,25)=(\alpha,\beta)\rangle.
\]
By passing to a even larger multiple, we may assume that
$(\alpha,\beta)=e(26,1)+f(0,25)$. Then
$a(26,1)+b(2,2)+d(0,25)=e(26,1)+f(0,25)$ implies
\[
 a=e-g,~d=f-g,~b=13g\textrm{ where }0\leq g\leq \min\{e,f\}.
\]
The weight cone $\omega\subset M_{\bQ}$ is generated by $(26,1)$
and $(0,25)$.

Next, we follow the set-up by Altmann and Hausen on polyhedral divisors \cite{AH06}.
Recall that $M\subset\bZ^2$ is the sublattice generated by
$(25,0)$ and $(1,1)$. Let $N\supset\bZ^2$ be the dual
lattice of $M$ generated by $\frac{1}{25}(1,-1)$ and $(0,1)$.
For each $u=(\alpha,\beta)\in\omega\cap M$, we
decompose it as $u=e(26,1)+f(0,25)$. Here $e,f\in\frac{1}{26}\bZ_{\geq 0}$
and $e-f\in\bZ$. Then the polyhedral divisor $\fD$ on $\bP^1$
is given by 
\[
\fD(u)=\begin{cases}
        13f[1]-12f[0]& \textrm{ if }e\geq f\geq 0\\
        13f[1]-12f[0]+(e-f)[\infty]&\textrm{ if }f\geq e\geq 0
       \end{cases}
 \]
 The polyhedrals are given by 
 \[
  \fD_{[0]}=\frac{6}{325}(1,-26)+\sigma, \quad \fD_{[1]}=\frac{1}{50}(-1,26)+\sigma,\quad
  \fD_{[\infty]}  =\textrm{conv}((0,0),\frac{1}{25}(1,-1))+\sigma.
 \]
Here $\sigma=\omega^{\vee}\subset N_{\bQ}$ is spanned
by $(1,0)$ and $(-1,26)$. Denote the four vertices
by $x_0=\frac{6}{325}(1,-26)$, $x_1=\frac{1}{50}(-1,26)$,
$x_2=(0,0)$, and $x_3=(\frac{1}{25}(1,-1))$.
Then these vertical divisors correspond to 
\begin{align*}
 D_{[0],x_0}=(y=0), \quad D_{[1],x_1}=(z=0), \quad
 D_{[\infty],x_2}=(w=0), \quad D_{[\infty],x_3}=(x=0).
\end{align*}
We also denote the extremal ray by $\rho_1=\langle(1,0)\rangle$
and $\rho_2=\langle(-1,26)\rangle$.
Notice that 
\[
 \mu(x_0)=13, \quad\mu(x_1)=2,\quad\mu(x_2)=1,\quad
 \mu(x_3)=1.
\]
Then for any presentation $K_{\bP^1}=a_0[0]+a_1[1]+a_{2}[\infty]$,
we have
\[
 K_{Y}=(13a_0+12)D_{[0],x_0}+(2a_1+1)D_{[1],x_1}
 +a_2(D_{[\infty],x_2}+D_{[\infty],x_3}).
\]
For simplicity, let us choose $a_0=a_1=-1$, $a_2=0$.
Then $K_Y=-D_{[0],x_0}-D_{[1],x_1}$. Hence
\[
 -(K_Y+cD_{[\infty],x_3})=D_{[0],x_0}+D_{[1],x_1}-cD_{[\infty],x_2}.
\]

 Next, we will try to use $T$-varieties to
 study test configurations. We follow the notation
 of \cite{IS17}. Let us choose a point $Q\in\bP^1$,
 a natural number $m\in\bZ_{>0}$, and a vector
 $v\in \frac{1}{m}N$. Consider the lattices
 $\tM:=M\times\bZ$ and $\tN:=N\times\bZ$. Define 
 $\tsigma\subset\tN_{\bQ}$ as
 \[
  \tsigma:=\left\langle(v+\sum_{P\in\bP^1\setminus\{Q\}}\fD_P,\frac{1}{m}),(\sum_{P\in\bP^1}\fD_P,0)\right\rangle.
 \]
 Define the polyhedral divisor $\tfD$ by
 \[
  \tfD:=(\textrm{conv}((v,\frac{1}{m}),(\fD_Q,0))+\tsigma)\otimes Q+\sum_{P\in\bP^1\setminus\{Q\}}
  ((\fD_P,0)+\tsigma)\otimes P.
 \]
 Then we have a $\tM$-graded algebra
 \[
  \cA:=\bigoplus_{\tu\in\tsigma^{\vee}\cap\tM}
  H^0(\bP^1,\cO(\lfloor\tfD(\tu)\rfloor)).
 \]
 Let $\cY(Q,v,m):=\Spec\cA$. Then we see $\tfD(0,k)=0$,
 so we have a subring $\cA_0$ of $\cA$ which consists
 of $(0,k)$-graded pieces. Moreover,
 $\cA_0=\bC[t]$ where $t$ is the canonical section
 of $\cO(\lfloor\tfD(0,1)\rfloor)=\cO$.
 Thus we get a $T\times\bG_m$-equivariant morphism
 $\cY\to \bA^1$. 
 It is clear that 
 \[
  \tfD(u,k)=\min\{(v,u)+\frac{k}{m}, \fD_Q(u)\}\cdot Q
  +\sum_{P\in\bP^1\setminus\{Q\}}\fD_P(u)\cdot P.
 \]
 Hence when $k\gg 0$, we see $\tfD(u,k)=\fD(u)$. Thus
 the localization $\cA_t=A\otimes\bC[t,t^{-1}]$.
 
 Next, we analyze the central fiber $\cY_0(Q,v,m)=\Spec~\cA/(t)$.
 It is clear that 
 \[
  \cY_0(Q,v,m)=\Spec\bigoplus_{(u,k)\in
  \sigma^{\vee}\cap\tM}H^0(\bP^1,\cO(\lfloor\tfD(u,k)\rfloor))/H^0(\bP^1,
  \cO(\lfloor\tfD(u,k-1)\rfloor)).
 \]
 For computational purposes, consider the lattice
 automorphism $\phi:\tN\to \tN$ given by
 \[
  \phi(v',m'):=(v'-m'mv,m').
 \]
 The dual automorphism $\phi^{\vee}:\tM\to\tM$
 is given by $
  \phi^{\vee}(u,k)=(u,k-m(v,u))$.
 Hence
 \[
  \phi_*\tfD(u,k)=\tfD(\phi^{\vee}(u,k))=\min\{\frac{k}{m},\fD_Q(u)\}\cdot Q+
  \sum_{P\in\bP^1\setminus\{Q\}}\fD_P(u)\cdot P.
 \]
 In order for the $\phi^{\vee}(u,k)$-graded piece of $\cY_0$
 to be non-zero, we require two conditions:
 \begin{enumerate}
  \item $\lfloor\min\{\frac{k}{m},\fD_Q(u)\}\rfloor>
   \lfloor\min\{\frac{k-1}{m},\fD_Q(u)\}\rfloor$;
  \item $\deg\lfloor\phi_*\tfD(u,k)\rfloor\geq 0$.
 \end{enumerate}
 Condition (1) and (2) together are equivalent to $k\in m\bZ$ and 
 $\fD_Q(u)\geq \frac{k}{m}$, and
 \[
  \sum_{P\in\bP^1\setminus\{Q\}}\lfloor\fD_P(u)\rfloor+\frac{k}{m}\geq 0
 \]
 Denote by $\tau\subset\tN_{\bQ}$ the cone generated
 by $(\fD_Q,-\frac{1}{m})$ and $(\sum_{P\in\bP^1\setminus\{Q\}}\fD_{P},\frac{1}{m})$.
 Consider the sublattice of $\tM$ of index $m$, namely $\tM_m:=M\times m\bZ\subset\tM$. Consider the semigroup
 \[
  S:=\{(u,k)\in\tau^{\vee}\cap\tM_m\mid
    \sum_{P\in\bP^1\setminus\{Q\}}\lfloor\fD_P(u)\rfloor+\frac{k}{m}\geq 0\}
 \]
 Then we have
 \begin{enumerate}
  \item The central fiber $\cY_0$ is isomorphic
  to the affine toric variety $\Spec~\bC[S]$;
  \item The normalization of $\cY_0$ is isomorphic
  to the affine toric variety $\Spec~\bC[\tau^{\vee}\cap\tM_m]$.
 \end{enumerate}
 Hence $\cY_0$ is normal if and only if $S=\tau^{\vee}\cap\tM_m$,
 which is equivalent of saying that 
 the collection $\{\fD_P\}_{P\neq Q}$ is admissible.

 Next, we want to study the limit of boundary divisor $(w=0)$ in
 the central fiber $\cY_0$. We first realize $(w=0)$
 in $Y$ as a Cartier divisor. Let $f$ be a rational
 function on $\bP^1$ such that 
 $\mathrm{div}(f)=12[0]-13[1]+[\infty]$.
 Then we can check by \cite[Proposition 3.14]{PetersenSuss}
 that 
 \[
  \mathrm{div}(f\cdot\chi^{(0,25)})=D_{[\infty],x_2}=(w=0).
 \]
 By the admissibility condition, the only choice of $Q$
 will be $Q=[0]$ or $[1]$.
 \medskip
 
 \textbf{Case 1: $Q=[0]$.} For simplicity, we also set
 $v=0$ at the moment. Then
 in our test configuration $\cY([0],0,m)$, we can compute similarly
 that 
 \[
  \mathrm{div}(f\cdot\chi^{(0,25,k)})=(12m+k)D_{[0],(0,\frac{1}{m})}+D_{[\infty],(x_2,0)}.
  \]
 Also, $\mathrm{div}(\chi^{(0,0,1)})=D_{[0],(0,\frac{1}{m})}=\cY_0$. Hence we know that $  D_{[\infty],(x_2,0)}=\mathrm{div}(f\cdot\chi^{(0,25,-12m)})$.
 By carefully checking the quotient map $\cA\to\cA/(t)$,
 we find out that the restriction of $f\cdot\chi^{(0,25,-12m)}$
 is exactly the function $\chi^{(0,25,-12m)}$ on $\cY_0$.
 Hence we have
 $(w=0)|_{\cY_0}=\mathrm{div}(\chi^{(0,25,-12m)})$.
 So the computations are about the toric variety $\cY_0$ and
 its boundary divisor $\Delta_0:=c\cdot(w=0)|_{\cY_0}$. For simplicity,
 we may assume $m=1$. Then $\cY_0=\Spec~\bC[\tau^{\vee}\cap\tM]$.
 The primitive vectors of $\tau$ in $\tN$ is given by
 $n_1=(\frac{6}{25}(1,-26),-13)$, $n_2=(\frac{1}{25}(-1,26),2)$, and $n_3=(\frac{1}{25}(1,24),2)$.
 Let $\tu_0$ be the vector in $\tM_{\bQ}$ representing the anticanonical
 divisor $-K_{\cY_0}$, then $(\tu_0,n_i)=1$. Let $\tu_1$ be the 
 vector in $\tM_{\bQ}$ representing the divisor $(w=0)|_{\cY_0}$.
 Then computation shows
 \[
  \tu_0=(15,15,-7),\quad \tu_1=(0,25,-12).
 \]
 Thus for any toric valuation $v_{\xi}$ of $\bC(\cY_0)$ with $\xi\in \tau$, we have 
 \[
  A_{(\cY_0,\Delta_0)}(v_\xi)=(\tu_0-c\tu_1,\xi), \quad \vol_{\cY_0,o}(v_\xi)
  =6\vol(\tau^{\vee}\cap (\xi\leq 1)).
 \]
 From \cite{LX17}, we know that 
 \[
 \Fut(\cY,\Delta,\xi_0;\eta)=\left.\frac{d}{dt}\right|_{t=0}\hvol_{(\cY_0,\Delta_0),o}
 (v_{\xi_0-t\eta}).
 \]
 In our setting, $\xi_0=(0,1,0)$ is induced by the quotient map 
 $Y\setminus\{o\}\to X_{26}$, and $\eta=(-mv,m)$.
 We know that all such $\eta$ satisfy $(\eta,\eta_0^*)>0$
 where $\eta_0^*=(0,0,1)\in \tM$. We always have
 $(\xi_0,\eta_0^*)=0$. 
 
 \begin{prop}\label{prop:futvol}
  Under the above notation, we have $\Fut(\cY,\Delta,\xi_0;\eta)>0$
  for any $\eta$ satisfying $(\eta,\eta_0^*)>0$ if and only if
  the centroid $\tu_2$ of $\tau^{\vee}\cap (\xi_0=1)$
  satisfies
  \[
   \tu_2=a(\tu_0-c\tu_1)+b\eta_0^*
  \]
  for some $a,b>0$.
 \end{prop}
 
 \begin{proof}
  We first determine $a$ by $(\tu_2,\xi_0)=a(\tu_0-c\tu_1,\xi_0)$.
  Then the vector $a(\tu_0-c\tu_1)\in (\xi_0=1)$. Then we may
  choose $\eta$ in a way such that $(\tu_0-c\tu_1,\eta)=0$
  and $(\eta,\eta_0^*)>0$ since $\hvol$ is invariant under rescaling.
  Then $A_{(\cY_0,\Delta_0)}(v_{\xi_0-t\eta})=a^{-1}>0$. Moreover,
  computation shows
  \[
   \left.\frac{d}{dt}\right|_{t=0}\vol(\tau^\vee\cap((\xi_0-t\eta)\leq 1))
   =\vol(\tau^\vee\cap(\xi_0=1))\cdot (a(\tu_0-c\tu_1)-\tu_2,\eta).
  \]
  Hence $\Fut(\cY,\Delta,\xi_0;\eta)>0$ is equivalent to 
  $(a(\tu_0-c\tu_1)-\tu_2,\eta)>0$ for any $\eta$ satisfying
  $(\tu_0-c\tu_1,\eta)=0$ and $(\eta,\eta_0^*)>0$. This is equivalent
  to $a(\tu_0-c\tu_1)-\tu_2=a'(\tu_0-c\tu_1)+b\eta_0^*$
  for some $a'\in\bR$ and $b>0$. Since $(\xi,\eta_0^*)=0$, we get $a'=0$. Hence the
  proof is finished.
 \end{proof}

 By computation we have
 \[
  \tu_2=(9,1,-\frac{49}{150}),\quad
  \tu_0-c\tu_1=(15,15-25c,-7+12c),\quad
  \eta_0^*=(0,0,1).
 \]
 Hence the only $c$ satisfying the condition of Proposition \ref{prop:futvol}
 is $c=\frac{8}{15}$.
 \medskip
 
\textbf{Case 2: $Q=[1]$.}
As always, we want to first determine the polyhedral divisors.
We know that 
\begin{align*}
 \fD_{[0]}+\fD_{[\infty]}& =\mathrm{conv}(\frac{6}{325}(1,-26),\frac{1}{325}(19,-169))+\sigma,\\
 \fD_{[0]}+\fD_{[1]}+\fD_{[\infty]}& =\mathrm{conv}{(\frac{1}{650}(-1,26),
 \frac{1}{26}(1,0))}+\sigma.
\end{align*}
It is clear that $\cY([1],0,m)$ has five distinguished vertical divisors:
\[
 D_{[0],(x_0,0)},~ D_{[1],(x_1,0)},~D_{[1],(0,\frac{1}{m})},~
 D_{[\infty], (x_2,0)}, ~D_{[\infty],(x_3,0)}.
\]
We can compute that 
\[
 \mathrm{div}(f\cdot \chi^{(0,25,k)})=(-13m+k)D_{[1],(0,\frac{1}{m})}+D_{[\infty],(x_2,0)}.
\]
We also have $ \mathrm{div}(\chi^{(0,0,1)})=D_{[1],(0,\frac{1}{m})}=\cY_0$. 
Hence we have $D_{[\infty],(x_2,0)}=\mathrm{div}(f\cdot\chi^{(0,25,13m)})$.
Next we will analyze the cone $\tau\subset\tN_{\bQ}$. Assume $m=1$ 
for simplicity. Then the primitive vectors of $\tau$ in $\tN$ is $n_1=(\frac{1}{25}(-1,26),-2)$, $n_2=(\frac{6}{25}(1,-26),13)$, and $n_3=(\frac{1}{25}(19,-169),13)$.
Let $\tu_0$ and $\tu_1$ be vectors in $\tM_{\bQ}$ representing
$-K_{\cY_0}$ and $(w=0)|_{\cY_0}$, respectively. Then 
\[
 \tu_0=(15,15,7),\quad\tu_1=(0,25,13).
\]
We still have $\xi_0=(0,1,0)$ and $\eta_0^{*}=(0,0,1)$. Hence
\[
 \tu_2=(9,1,\frac{319}{975}), \quad
 \tu_0-c\tu_1=(15,15-25c, 7-13c),\quad \eta_0^*=(0,0,1).
\]
Hence the only $c$ satisfying the condition of Proposition
\ref{prop:futvol} is $c=\frac{8}{15}$.
\end{proof}

\subsection{$A_{11}$} For plane quintic curves with $A_{11}$-singularities, we have two cases: reducible and irreducible curves. We begin with the reducible case. 

\subsubsection{$A_{11}$ reducible}
Let $C$ be a reducible plane quintic curve with an $A_{11}$-singularity. Then after a projective transformation, in the affine coordinates $[x,y,1]$ 
we can write the equation of $C$ as (see Proposition \ref{prop:jetsummary})
\[
C = \big( (x-y^2)((x-y^2)(1+sx)-x^3)=0\big).
\]
In other words, we have $p=0$, $u=0$, and $r=1$ in \eqref{eq:quinticA_9}. Let us choose a $6$-jet $(x',y)$ at the origin by $x':=x-y^2-\frac{1}{2} y^6$. Then the equation of $C$ in $(x',y)$ becomes
\[
x'^2=\tfrac{1}{4}y^{12}+\textrm{higher order terms},
\]
where $(x',y)$ has weight $(6,1)$. 
The only parameter of $C$ here is $s\in\bA^1$. All these curves 
are GIT stable. When $s$ goes to infinity,
the unique GIT polystable limit will be $Q_5$, i.e. the double conic union
a transversal line.

\begin{theorem}\label{thm:a11red} Suppose $C \subset \bP^2$ is a reducible quintic curve with an $A_{11}$ singularity. Then the log Fano pair $(\bP^2, cC)$ is K-semistable if and only if $0 < c \leq \frac{6}{11}$. Moreover, $(\bP(1,1,4), \frac{6}{11}C_0)$ is the K-polystable degeneration of $(\bP^2, \frac{6}{11}C)$ where $C_0=(x^2 z^2+y^6 z=0)$. \end{theorem}

\begin{proof}
We first prove the ``only if'' part. Suppose $(\bP^2, cC)$ is K-semistable, and we want to show $c\leq \frac{6}{11}$.
 Let us perform the
 $(6,1)$-weighted blow up of $\bP^2$ in the coordinates
 $(x',y)$, and denote the resulting surface and exceptional
 divisor by $(X,E)$, with $\pi:X\to\bP^2$ the weighted blow up morphism.
 Let $Q=(x=y^2)$ be a smooth conic in $\bP^2$.
 We know that the weight of $x-y^2=x'+\frac{1}{2} y^6$
 is $6$, hence $\oQ:=\pi_*^{-1}Q\sim 2\pi^* H-6E$ is effective
 on $X$. It is easy to see $(E^2)=-\frac{1}{6}$ and $(\oQ^2)=-2$,
 hence the Mori cone of $X$ is generated by $E$ and $\oQ$.
 It is clear that
 \[
  A_{(\bP^2,cC)}(E)=7-12c,\quad -K_{\bP^2}-cC\sim_{\bQ} (3-5c)H.
 \]
 We also have $\pi^*H-tE$ is ample if and only if $0<t<2$, and 
 big if and only if $0\leq t<3$. Then by computation, we have
 \[
  \vol_X(\pi^*H-tE)=\begin{cases}
                      1-\frac{t^2}{6} & \textrm{ if }0\leq t\leq 2;\\
                      \frac{(3-t)^2}{3} & \textrm{ if }2\leq t\leq 3.
                     \end{cases}
 \]
 Hence $S_{(\bP^2,cC)}(E)=(3-5c)\int_0^\infty \vol_X(\pi^*H-tE)=(3-5c)\frac{5}{3}$.
 So the valuative criterion (Theorem \ref{thm:valuative}) implies
 \[
  7-12c=A_{(\bP^2,cC)}(E)\geq S_{(\bP^2,cC)}(E)= (3-5c)\frac{5}{3},
 \]
 which implies $c\leq \frac{6}{11}$.

 We now begin showing the ``if'' part. Similar to the proof of Theorem \ref{thm:a12}, we construct a special degeneration then later on use techniques of Ilten and S\"u{\ss} \cite{IS17} to show K-polystability of the degeneration. 

\begin{prop}\label{prop:a11reddeg}
 The log Fano pair $(\bP^2,cC)$ admits a special
 degeneration to $(\bP(1,1,4),cC_0)$
 where $C_0$ is given by the equation $x^2 z^2+y^6 z=0$.
\end{prop}

\begin{proof}
 Here is the construction of the special degeneration. Take our family
$(\bP^2,C)\times\bA^1$. We perform the following birational transformations:

\[
\begin{tikzcd}
 & \cX\arrow{ld}{\pi}\arrow{rd}{g}\arrow[rr,dashed,"f"]& & \cX^+\arrow{ld}{h}\arrow{rd}{\psi}\\
 \bP^2\times\bA^1& & \cY & & \cZ
 \end{tikzcd}
\]
Where in the central fiber we have
\[
 \begin{tikzcd}
 & S\cup X\arrow{ld}{\pi}\arrow{rd}{g}\arrow[rr,dashed,"f"]& & \hS\cup X'\arrow{ld}{h}\arrow{rd}{\psi}\\
 \bP^2& & S\cup X' & & S'
 \end{tikzcd}
\]
Here $\pi$ is the $(6,1,1)$-weighted blow up of $\bP^2\times\bA^1$ in the local coordinates $(x',y,t)$,
$S=\bP(1,1,6)$ is the exceptional divisor of $\pi$,
$g$ is the contraction of $\oQ$ in $X\subset\cX_0$, $f$ is the flip of
the curve $\oQ$ in $\cX_0$ (since by computation the normal bundle
$\cN_{\oQ/\cX}\cong\cO_{\oQ}(-2)\oplus\cO_{\oQ}(-1)$), and
$\psi$ is the divisorial contraction that contracts $X'$ to a point.

Let us analyze the geometry of these birational maps. Suppose
$S$ has projective coordinates $[x_1,x_2,x_3]$
of weights $(1,1,6)$ respectively.
Then $S\cap X=E=(x_1=0)$, and $\oQ\cap E=\{p\}$ is a smooth point
of $S$ and $X$. Since $\oQ$ has normal bundle
$\cO(-2)\oplus\cO(-1)$ in $\cX$, the surface
$\hS$ is a $(2,1)$-weighted blow up of $S$ at $p$.
Let $\oQ^+$ be the flipped curve in $\hS$,
then $\hS$ has an $A_1$-singularity at the unique intersection
of $h_*^{-1}E$ and $\oQ^+$. Then $\psi: \hS\to S'$ contracts
$h_*^{-1}E$ and creates a singularity of type $\frac{1}{4}(1,1)$. 
Thus $S'$ is $\bP(1,1,4)$. 

For the degeneration $C_0$ of $C$ in $S'$, note that $\pi_*^{-1}(C\times\bA^1)\cap S$ is the curve $C_0'=(x_3^2=\frac{1}{4}x_2^{12})$. In addition, we know that $p$ has coordinate $[0,1,-\frac{1}{2}]$ which is contained in $C_0'$. It is clear that $C_0'$ has an $A_{11}$-singularity at $[1,0,0]$ near where the birational map $S\dashrightarrow S'$ is an isomorphism. Thus $C_0$ has an $A_{11}$-singularity at a smooth point of $S'$ as well.  Since $Q$ is contained in $C$, we know that $\oQ':=\psi_* \oQ^+$ is contained in the degeneration $C_0$ of $C$. By a toric computation we know that $\oQ'=(z=0)$ in $S'=\bP(1,1,4)$. This implies that up to a projective transformation $C_0$ has the equation $(x^2 z^2+y^6 z=0)$ in $\bP(1,1,4)$.
\end{proof}

Now we return to the proof of the theorem. By Propositions \ref{prop:a11reddeg}, \ref{prop:a11redkps}, and Theorem \ref{thm:Kss-spdeg}, we know that $(\bP^2, \frac{6}{11}C)$ is K-semistable. Hence the proof is finished by Proposition \ref{prop:k-interpolation}.
\end{proof}

We verify K-polystability of $(\bP(1,1,4),\frac{6}{11}C_0)$ below.

\begin{prop}\label{prop:a11redkps}
 The log Fano pair $(\bP(1,1,4),\frac{6}{11}C_0)$ is K-polystable where $C_0=(x^2 z^2+y^6 z=0)$.
\end{prop}
 
\begin{proof} 
It is clear that the pair $(\bP(1,1,4),\frac{6}{11}C_0)$ admits a $\bG_m$-action, which can
be lifted to a $\bG_m^2$-action on $Y:=C(\bP(1,1,4),\cO_{\bP(1,1,4)}(1))\cong\bA^3_{(x,y,z)}$ as
\[
 (x,y,z)\mapsto (\mu x,\lambda\mu y,\lambda^6\mu^{4} z).
\]
Thus by \cite[Theorem 1.4]{LWX18} it suffices to show $\bG_m$-equivariant K-polystability of $(\bP(1,1,4),\frac{6}{11}C_0)$.

Denote by $N$ the lattice of $1$-PS's in $\bG_m^2$. 
Since $Y$ is a toric variety with the standard $\bG_m^3$-action, we denote by $N_Y$ the lattice of $1$-PS's in $\bG_m^3$. By \cite[Section 11]{AH06},  we have an embedding of lattices $F:N\to N_Y$ and a (non-canonical) surjective map $s: N_Y\to N$ with $s\circ F=\id$. 
Then the maps $F$ and $s$ can be chosen as 
\[
 F=\begin{bmatrix}
    0 & 1\\1 & 1\\6 & 4
   \end{bmatrix},\quad s=\begin{bmatrix}
   -1 & 1 & 0\\1 & 0 & 0
   \end{bmatrix}.
\]
Computation shows $\sigma\subset N_{\bQ}$ is spanned by 
$(-2,3)$ and $(1,0)$. The non-trivial polyhedral divisors 
are
\[
 \fD_{[0]}=\mathrm{conv}((0,0),(-\frac{1}{2},\frac{1}{2}))+\sigma,\quad \fD_{[\infty]}=(\frac{1}{6},0)+\sigma.
\]
Denote by $ x_0=(0,0),~ x_1=(-\frac{1}{2},\frac{1}{2}),~ x_2=(\frac{1}{6},0)$.
Then computation shows
\begin{align*}
 \mathrm{div}(\chi^{(0,1)})& =D_{[0],x_1}=(x=0)\\
 \mathrm{div}(\chi^{(1,1)})& =D_{[\infty],x_2}=(y=0)\\
 \mathrm{div}(f\cdot\chi^{(6,4))})& = D_{[0],x_0}=(z=0)
\end{align*}
Here $f$ is a rational function on $\bP^1$ with $\mathrm{div}(f)=[0]-[\infty]$.
Let us choose a rational function $g$ on $\bP^1$ such that
$\mathrm{div}(g)=[1]-[\infty]$. Then in a suitable coordinate
of $\bP^1$, we have
\[
 \mathrm{div}(g\cdot\chi^{(6,6)})=(x^2 z=y^6)=D_{[1],0}.
\]
Hence the boundary divisor is given by $\Delta=c\cdot\mathrm{div}(fg\cdot\chi^{(12,10)})
=c(D_{[0],x_0}+D_{[1],0})$.

Next we will look at the test configuration picture.

\textbf{Case 1: $Q=[0]$.} Computation shows that there are
four distinguished vertical divisors on $\cY$:
\[
 D_{[0],(x_1,0)}=(x=0),\quad D_{[\infty],(x_2,0)}=(y=0),
 \quad D_{[0],(x_0,0)}=(z=0),\quad D_{[0],(0,1)}=\cY_0. 
\]
By computation, we have
\[
 \mathrm{div}(fg\cdot\chi^{(12,10,k)})=D_{[0],(x_0,0)}+D_{[1],(0,0)}
 +(k+1)D_{[0],(0,1)}.
\]
Hence
$ (D_{[0],(x_0,0)}+D_{[1],(0,0)})|_{\cY_0}=\mathrm{div}(\chi^{(12,10,-1)})$.
We know that $\tau$ is spanned by $ n_1=(0,0,-1)$, $ n_2=(-1,1,-2)$, and $n_3=(1,0,6)$.
We still have $\xi_0=(0,1,0)$ and $\eta_0^*=(0,0,1)$.
Hence
\[
 \tu_0=(7,6,-1), \quad\tu_1=(12,10,-1),\quad \tu_2=\frac{1}{12}
 (10,12,-1).
\]
To satisfy the condition of Proposition \ref{prop:futvol},
the only $c$ is $c=\frac{6}{11}$.

\textbf{Case 2: $Q=[\infty]$.} Computation shows that there are
four distinguished vertical divisors on $\cY$:
\[
 D_{[0],(x_1,0)}=(x=0),\quad D_{[\infty],(x_2,0)}=(y=0),
 \quad D_{[0],(x_0,0)}=(z=0),\quad D_{[\infty],(0,1)}=\cY_0. 
\]
By computation, we have
\[
 \mathrm{div}(fg\cdot\chi^{(12,10,k)})=D_{[0],(x_0,0)}+D_{[1],(0,0)}
 +(k-2)D_{[0],(0,1)}.
\]
Hence $
 (D_{[0],(x_0,0)}+D_{[1],(0,0)})|_{\cY_0}=\mathrm{div}(\chi^{(12,10,2)})$.
We know that $\tau$ is spanned by  $n_1=(0,0,1)$, $n_2=(-1,1,2)$, and $n_3=(1,0,-6)$.
We still have $\xi_0=(0,1,0)$ and $\eta_0^*=(0,0,1)$.
Hence
\[
 \tu_0=(7,6,1), \quad\tu_1=(12,10,2),\quad \tu_2=\frac{1}{12}
 (10,12,1).
\]
To satisfy the condition of Proposition \ref{prop:futvol},
the only $c$ is $c=\frac{6}{11}$. 
\end{proof}

\subsubsection{$A_{11}$ irreducible}
Let $C$ be an irreducible plane quintic curve with an $A_{11}$-singularity. Then after a projective transformation, in the affine coordinate $[x,y,1]$ 
we can write the equation of $C$ as (see Proposition \ref{prop:jetsummary})
\[
C = \big( (x-y^2)((x-y^2)(1+sx)+2x^2 y)+x^5=0\big).
\]
In other words, we have $p=-1$, $u=1$, and $r=0$ in \eqref{eq:quinticA_9}. Let us choose a $6$-jet $(x',y)$ at the origin by $x':=x-y^2+y^5$. Then the equation of $C$ in $(x',y)$ becomes
\[
x'^2=-sy^{12}+\textrm{higher order terms},
\]
where $(x',y)$ has weight $(6,1)$. 
The only parameter here is $s\in\bA^1\setminus\{0\}$. All these curves are GIT stable. If $s=0$, we recover the GIT stable $A_{12}$ quintic curve discussed earlier. When $s$ goes to infinity, the unique GIT polystable limit will
be $Q_5$, i.e. the double conic union a transversal line.

\begin{theorem}\label{thm:a11irr}
Suppose $C \subset \bP^2$ is an irreducible quintic curve with an $A_{11}$ singularity. Then the log Fano pair $(\bP^2, cC)$ is K-semistable if and only if $0 < c \leq \frac{63}{115}$. Moreover, $(\bP(1,4,25), \frac{63}{115}C_0)$ is the K-polystable degeneration of $(\bP^2, \frac{63}{115}C)$ where $C_0=(z^2+x^2 y^{12}=0)$. 
\end{theorem}

\begin{proof}
 We first prove the ``only if'' part. Suppose $(\bP^2, cC)$ is K-semistable, and we want to show $c\leq \frac{63}{115}$. 
 Let us perform the
 $(6,1)$-weighted blow up of $\bP^2$ in the coordinates
 $(x',y)$, and denote the resulting surface and exceptional
 divisor by $(X,E)$, with $\pi:X\to\bP^2$ the weighted blow up morphism.
 Let $Q=(x=y^2)$ be a smooth conic in $\bP^2$.
 We know that the weight of $x-y^2=x'-y^5$
 is $5$, hence $\oQ:=\pi_*^{-1}Q\sim 2\pi^* H-5E$ is effective
 on $X$. It is easy to see $(E^2)=(\oQ^2)=-\frac{1}{6}$,
 hence the Mori cone of $X$ is generated by $E$ and $\oQ$.
 We again use the valuative criterion of K-semistability by
 Fujita and Li (Theorem \ref{thm:valuative}), then 
 \[
  A_{(\bP^2,cC)}(E)=7-12c,\quad -K_{\bP^2}-cC\sim_{\bQ} (3-5c)H.
 \]
 We also have $\pi^*H-tE$ is ample if and only if $0<t<\frac{12}{5}$, and 
 big if and only if $0\leq t<\frac{5}{2}$. Then by computation, we have
 \[
  \vol_X(\pi^*H-tE)=\begin{cases}
                      1-\frac{t^2}{6} & \textrm{ if }0\leq t\leq \frac{12}{5};\\
                      (5-2t)^2 & \textrm{ if }\frac{12}{5}\leq t\leq \frac{5}{2}.
                     \end{cases}
 \]
 Hence $S_{(\bP^2,cC)}(E)=(3-5c)\int_0^\infty \vol_X(\pi^*H-tE)=(3-5c)\frac{49}{30}$.
 Since $(\bP^2,cC)$ is K-semistable, the valuative criterion (Theorem \ref{thm:valuative}) implies
 \[
  7-12c=A_{(\bP^2,cC)}(E)\geq S_{(\bP^2,cC)}(E)= (3-5c)\frac{49}{30},
 \]
 which implies $c\leq \frac{63}{115}$.

 We now begin showing the ``if'' part. Similar to the proof of Theorems \ref{thm:a12} and \ref{thm:a11red}, we construct a special degeneration then later on use techniques of Ilten and S\"u{\ss} \cite{IS17} to show K-polystability of the degeneration.
 
\begin{prop}\label{prop:a11irrdeg}
 The log Fano pair $(\bP^2,cC)$ admits a special
 degeneration to $(\bP(1,4,25),cC_0)$
 where $C_0$ is given by the equation $z^2+x^2y^{12}=0$.
\end{prop}

\begin{proof}

We follow notation from the first two diagrams of the proof of Proposition \ref{prop:a11reddeg}. 
Here $\pi$ is the $(6,1,1)$-weighted blow up of $\bP^2\times\bA^1$ in the local coordinates $(x',y,t)$,
$S=\bP(1,1,6)$ is the exceptional divisor of $\pi$,
$g$ is the contraction of $\oQ$ in $X\subset\cX_0$, $f$ is the flip of
the curve $\oQ$ in $\cX_0$, and
$\psi$ is the divisorial contraction that contracts $X'$ to a point.

Let us analyze the geometry of these birational maps. Suppose
$S$ has projective coordinates $[x_1,x_2,x_3]$
of weights $(1,1,6)$ respectively.
Then $S\cap X=E=(x_1=0)$, and $\oQ\cap E=\{p\}$ is the
unique singular point
of $S$ (type $\frac{1}{6}(1,1)$) and $X$ (type $A_5$). Inside the surface $X$, we have two smooth
rational curves $E$ and $\oQ$ intersecting at $p$, such that 
$(E^2)=(\oQ^2)=-\frac{1}{6}$.
Hence contracting $\oQ$ in $X$ yields a smooth surface
$X'$ which has to be $\bP^2$ by degree computation.
On the other hand, $\hS$ and $X'$ intersect along
the proper transform $\hE$ of $E$ which becomes a conic
curve in $X'\cong \bP^2$. Hence $(\hE^2)=-4$ in $\hS$. 
Moreover, $h:\hS\to S\cong\bP(1,1,6)$ is a partial resolution
of $p$ which only extracts the flipped curve $\oQ^+$.
By some combinatorial computation, we know that $\hS$
is a toric blow-up of $S$ at $p$ which creates a singularity
of type $\frac{1}{25}(1,4)$ away from $\hE$. Hence $S'$ is a toric surface
carrying two singularities of types $\frac{1}{25}(1,4)$
and $\frac{1}{4}(1,1)$. Thus $S'\cong\bP(1,4,25)$.

For the degeneration $C_0$ of $C$ on $S'$, note that $\pi_*(C\times\bA^1)\cap S$ is the curve $C_0'=(x_3^2+sx_2^{12}=0)$. In addition, we know that $p$ has coordinate $[0,0,1]$ which is not contained in $C_0'$. It is clear that $C_0'$ has an $A_{11}$-singularity at the point $[1,0,0]$ which is not $p$ and does not lie on $E$. Since $S\dashrightarrow S'$ is isomorphic around $[1,0,0]$, we know that $C_0'$ has an $A_{11}$-singularity at a smooth point of $\bP(1,4,25)$. Since the $\frac{1}{25}(1,4)$ singularity on $\hS$ does not lie on $\widehat{E}$, we know that $C_0$ does not pass through $[0,0,1]$ on $S'\cong\bP(1,4,25)$. Hence after a projective coordinates change we may write $C_0=(z^2+x^2 y^{12}=0)$. 
\end{proof}

Now we return to the proof of the theorem. By Propositions \ref{prop:a11irrdeg}, \ref{prop:a11irrkps}, and Theorem \ref{thm:Kss-spdeg}, we know that $(\bP^2,\frac{63}{115}C)$ is K-semistable. Hence the proof is finished by Proposition \ref{prop:k-interpolation}.
\end{proof}

We verify K-polystability of $(\bP(1,4,25),\frac{63}{115} C_0)$ below.

\begin{prop}\label{prop:a11irrkps}
 The log Fano pair $(\bP(1,4,25),\frac{63}{115} C_0)$ is K-polystable where $C_0=(z^2+x^2 y^{12}=0)$.
\end{prop}

\begin{proof}
It is clear that the pair $(\bP(1,4,25),\frac{63}{115} C_0)$ admits a $\bG_m$-action, which can
be lifted to a $\bG_m^2$-action on $Y:=C(\bP(1,4,25),\cO_{\bP(1,4,25)}(1))\cong\bA^3_{(x,y,z)}$ as
\[
 (x,y,z)\mapsto (\mu x,\lambda\mu^4 y,\lambda^6\mu^{25} z).
\]
Thus by \cite[Theorem 1.4]{LWX18} it suffices to show $\bG_m$-equivariant K-polystability of $(\bP(1,4,25),\frac{63}{115} C_0)$.

Denote by $N$ the lattice of $1$-PS's in $\bG_m^2$. 
Since $Y$ is a toric variety with the standard $\bG_m^3$-action, we denote by $N_Y$ the lattice of $1$-PS's in $\bG_m^3$. By \cite[Section 11]{AH06},  we have an embedding of lattices $F:N\to N_Y$ and a (non-canonical) surjective map $s: N_Y\to N$ with $s\circ F=\id$. 
Then the maps $F$ and $s$ can be chosen as 
\[
 F=\begin{bmatrix}
    0 & 1\\1 & 4\\6 & 25
   \end{bmatrix},\quad s=\begin{bmatrix}
   -4 & 1 & 0\\1 & 0 & 0
   \end{bmatrix}.
\]
Computation shows $\sigma\subset N_{\bQ}$ is spanned by 
$(-4,1)$ and $(1,0)$. The only non-trivial polyhedral divisor 
is
\[
 \fD_{[\infty]}=\mathrm{conv}((-4,1),(\frac{1}{6},0))+\sigma.
\]
Denote by $x_0=(\frac{1}{6},0)$ and $x_1=(-4,1)$. Then we have
\[
 (x=0)=\mathrm{div}(\chi^{(0,1)})=D_{[\infty],x_0},\quad
 (y=0)=\mathrm{div}(\chi^{(1,4)})=D_{[\infty],x_1}.
\]
Moreover, for any $P\neq[\infty]$ we may take the function 
$f_P$ on $\bP^1$ such that $\mathrm{div}(f)=P-[\infty]$.
Then we may compute out that
\[
 \mathrm{div}(f_P\cdot\chi^{(6,25)})=D_{P,0}\sim D_{[\infty],x_0}+
 6D_{[\infty],x_1}.
\]
This implies that the curve $C$ induces a divisor 
on $Y$ which corresponding to $\mathrm{div}(f_{P_1}f_{P_2}\cdot\chi^{(12,50)})$
for some $P_1\neq P_2\in\bP^1\setminus\{\infty\}$.

Next we will look at the test configuration picture. 

\textbf{Case 1: $Q=[\infty]$.}
Computation shows that on $\cY$ there are three distinguished
vertical divisors:
\[
 D_{[\infty],(x_0,0)}=(x=0),\quad D_{[\infty],(x_1,0)}=(y=0),\quad D_{[\infty],(0,1)}=\cY_0.
\]
Again by computation, we have
\[
 \mathrm{div}(f_P\cdot\chi^{(6,25,k)})=D_{P,(0,0)}+(k-1)D_{[\infty],(0,1)}.
\]
Hence we need to take $k=1$, and we get 
$D_{P,(0,0)}|_{\cY_0}=\mathrm{div}(\chi^{(6,25,1)})$.
This implies that $\Delta_0=c\cdot\mathrm{div}(\chi^{(12,50,2)})$.
The central fiber $\cY_0=\Spec~\bC[\tau^{\vee}\cap\tM]$, where 
$\tau$ is spanned by $n_1=(-4,1,-1)$, $n_2=(1,0,-6)$, and $n_3=(0,0,1)$.
We still have $\xi_0=(0,1,0)$ and $\eta_0^*=(0,0,1)$.
Hence we have
\[
 \tu_0=(7,30,1),\quad \tu_1=(12,50,2),\quad \tu_2=\frac{1}{300}(49,300,4).
\]
To satisfy the condition of Proposition \ref{prop:futvol}, the only 
$c$ is $c=\frac{63}{115}$.

\textbf{Case 2: $Q\neq [\infty]$.} Computation shows that 
on $\cY$ there are three distinguished vertical divisors:
\[
  D_{[\infty],(x_0,0)}=(x=0),\quad D_{[\infty],(x_1,0)}=(y=0),\quad D_{Q,(0,1)}=\cY_0.
\]
Again by computation, we have
\[
 \mathrm{div}(f_{P}\cdot\chi^{(6,25,k)})=\begin{cases}
                                          D_{P,(0,0)}+kD_{Q,(0,1)}& \textrm{ if }P\neq Q\\
                                          D_{Q,(0,0)}+(k+1)D_{Q,(0,1)}& \textrm{ if }P= Q
                                         \end{cases}
\]
Hence we have 
\[
 D_{P,(0,0)}|_{\cY_0}=\begin{cases}
                       \mathrm{div}(\chi^{(6,25,0)})&\textrm{ if }P\neq Q\\
                       \mathrm{div}(\chi^{(6,25,-1)})&\textrm{ if }P=Q
                      \end{cases}
\]
This implies that $\Delta_0=c\cdot\mathrm{div}(\chi^{12,50,\delta})$
where $\delta=-1$ if $Q\in\{P_1,P_2\}$ and $\delta=0$ otherwise.
The central fiber $\cY_0=\Spec~\bC[\tau^{\vee}\cap \tM]$,
where $\tau$ is spanned by $n_1=(-4,1,1)$, $n_2=(1,0,6)$, and $n_3=(0,0,-1)$.
We still have $\xi_0=(0,1,0)$ and $\eta_0^*=(0,0,1)$.
Hence we have
\[
 \tu_0=(7,30,-1),\quad \tu_1=(12,50,\delta),\quad \tu_2=\frac{1}{300}(49,300,-4).
\]
To satisfy the condition of Proposition \ref{prop:futvol}, the only 
$c$ is $c=\frac{63}{115}$.
\end{proof}

\subsection{$A_{10}$}\label{sec:A10}
Let $C$ be an irreducible plane quintic curve with an $A_{10}$-singularity. Then after a projective transformation, in the affine coordinate $[x,y,1]$ 
we can write the equation of $C$ as (see Proposition \ref{prop:jetsummary})
\[
C = \big( (x-y^2)((x-y^2)(1+sx)+2x^2 y -rx^3)+x^5=0\big).
\]
In other words, we have $p=-1$ and $u=1$ in \eqref{eq:quinticA_9}. Let us choose a $5$-jet $(x',y)$ at the origin by $x':=x-y^2+y^5$. Then the equation of $C$ in $(x',y)$ becomes
\[
x'^2=-ry^{11}+\textrm{higher order terms},
\]
where $(x',y)$ has weight $(11,2)$. 
The parameter here is $(s,r)\in\bA^1\times(\bA^1\setminus\{0\})$. All these curves are GIT stable. If $r=0$, we recover the GIT stable irreducible $A_{11}$ quintic curve discussed earlier.






\begin{theorem}\label{thm:a10}
Suppose $C \subset \bP^2$ is a quintic curve with an $A_{10}$ singularity. Then the log Fano pair $(\bP^2, cC)$ is K-semistable if and only if $0 < c \leq  \frac{54}{95}$. Moreover, $(\bP(1,4,25),  \frac{54}{95}C_0)$ is the K-polystable degeneration of $(\bP^2,  \frac{54}{95}C)$ where $C_0=(z^2+x^6 y^{11}=0)$. 
\end{theorem}

\begin{proof}
 We first prove the ``only if'' part. Suppose $(\bP^2, cC)$ is K-semistable, and we want to show $c\leq \frac{54}{95}$.
 Let us perform the
 $(11,2)$-weighted blow up of $\bP^2$ in the coordinates
 $(x',y)$, and denote the resulting surface and exceptional
 divisor by $(X,E)$, with $\pi:X\to\bP^2$ the weighted blow up morphism.
 Let $Q=(x=y^2)$ be a smooth conic in $\bP^2$. We konw that the weight of $x-y^2=x'-y^5$ is $10$, hence $\oQ:=\pi_*^{-1}Q\sim 2\pi^*H-10 E$ is 
 effective on $X$. It is easy to see $(E^2)=-\frac{1}{22}$ and $(\oQ^2)=-\frac{6}{11}$,
 hence the Mori cone of $X$ is generated by $E$ and $\oQ$.
 It is clear that $A_{(\bP^2,cC)}(E)=13-22c$ and $-K_{\bP^2}-cC\sim_{\bQ}(3-5c) H$.
 We also have $\pi^*H-tE$ is ample if and only if $0<t<\frac{22}{5}$, and 
 big if and only if $\frac{22}{5} \leq t<5$. Then by computation, we have
 \[
  \vol_X(\pi^*H-tE)=\begin{cases}
                      \dfrac{1-t^2}{22} & \textrm{ if }0\leq t\leq \frac{22}{5};\\
                      \dfrac{(10-2t)^2}{12} & \textrm{ if }\frac{22}{5}\leq t <5.
                     \end{cases}
 \]
 Hence $S_{(\bP^2,cC))}(E)=(3-5c)\int_0^\infty \vol_X(\pi^*H-tE)=(3-5c)\frac{47}{15}$.
 Since $(\bP^2, cC)$ is K-semistable, the valuative criterion (Theorem \ref{thm:valuative}) implies
 \[
  13-22c=A_{(\bP^2,cC))}(E)\geq S_{(\bP^2,cC))}(E)= (3-5c)\frac{47}{15},
 \]
 which implies $c\leq \frac{54}{95}$.

 We now begin showing the ``if'' part. Similar to the proof of Theorems \ref{thm:a12}, \ref{thm:a11red}, and \ref{thm:a11irr}, we construct a special degeneration then later on use techniques of Ilten and S\"u{\ss} \cite{IS17} to show K-polystability of the degeneration.

\begin{prop}\label{prop:a10deg}
 The log Fano pair $(\bP^2,cC)$ admits
 a special degeneration to $(\bP(1,4,25), cC_0)$
 where $C_0$ is given by the equation 
 $(z^2 + x^6y^{11}=0)$.
\end{prop}

\begin{proof} 

We follow notation from the first two diagrams of the proof of Proposition \ref{prop:a11reddeg}. 
Here $\pi$ is the $(11,2,1)$-weighted blow up of $\bP^2\times\bA^1$ in the local coordinates $(x',y,t)$,
$S=\bP(1,2,11)$ is the exceptional divisor of $\pi$, $g$ is the contraction of $\oQ$ in $X\subset\cX_0$, $f$ is the flip of
the curve $\oQ$ in $\cX_0$, and
$\psi$ is the divisorial contraction that contracts $X'$ to a point.

Let us analyze the geometry of these birational maps. Suppose
$S$ has projective coordinates $[x_1,x_2,x_3]$
of weights $(1,2,11)$ respectively.
Then $S\cap X=E=(x_1=0)$, and $\oQ\cap E=\{p\}$ is a singular point 
of $S$ (type $\frac{1}{11}(1,2)$) and $X$ (type $\frac{1}{11}(1,9)$). Inside the surface $X$, we have two smooth
rational curves $E$ and $\oQ$ intersecting at $p$, such that $(E^2)=-\frac{1}{22}$ and $(\oQ^2)=-\frac{6}{11}$.
Hence contracting $\oQ$ in $X$ yields a surface
$X'$ with two singularities of types $A_1$ and $\frac{1}{6}(1,5)$.
On the other hand, $\hS$ and $X'$ intersect along
the proper transform $\hE$ of $E$  with $(\hE^2)=\frac{1}{3}$. Hence $(\hE^2)=-\frac{1}{3}$ in $\hS$. 
Moreover, $h:\hS\to S\cong\bP(1,2,11)$ is a partial resolution
of $p=[0,0,1]$ which only extracts the flipped curve $\oQ^+$.
By some combinatorial computation, we know that $\hS$
is a toric blow-up of $S$ at $p$ which creates a singularity
of type $\frac{1}{25}(1,4)$ away from $\hE$. Moreover, $\hE$ passes through two singularities of $\hS$ of types $A_1$ and $\frac{1}{6}(1,1)$. Hence $S'$ is a toric surface
carrying two singularities of types $\frac{1}{25}(1,4)$ 
and $\frac{1}{4}(1,1)$ (coming from contracting $\hE$). Thus $S'\cong\bP(1,4,25)$.

For the degeneration $C_0$ of $C$ on $S'$, note that $\pi_*(C\times\bA^1)\cap S$ is the curve $C_0'=(x_3^2+rx_2^{11}=0)$. In addition, we know that $p$ has coordinate $[0,0,1]$ which is not contained in $C_0'$. It is clear that $C_0'$ has an $A_{10}$-singularity at the point $[1,0,0]$ which is not $p$ and does not lie on $E$. Since $S\dashrightarrow S'$ is isomorphic around $[1,0,0]$, we know that $C_0'$ has an $A_{10}$-singularity at a smooth point of $\bP(1,4,25)$. Since the $\frac{1}{25}(1,4)$ singularity on $\hS$ does not lie on $\widehat{E}$, we know that $C_0$ does not pass through $[0,0,1]$ on $S'\cong\bP(1,4,25)$. Hence after a projective coordinates change we may write $C_0=(z^2+x^2(x^4-ay) y^{11}=0)$. Since $\Aut(\bP(1,4,25),C_0)$ is not discrete, we conclude that $a=0$  which finishes the proof.
\end{proof}




Now we return to the proof of the theorem. By Propositions \ref{prop:a10deg}, \ref{prop:a10kps}, and Theorem \ref{thm:Kss-spdeg}, we know that $(\bP^2,\frac{54}{95}C)$ is K-semistable. Hence the proof is finished by Proposition \ref{prop:k-interpolation}.
\end{proof}

We verify K-polystability of $(\bP(1,4,25),\frac{54}{95}C_0)$ below.

\begin{prop}\label{prop:a10kps}
 The log Fano pair $(\bP(1,4,25),\frac{54}{95} C_0)$ is K-polystable where $C_0=(z^2+x^6 y^{11}=0)$.
\end{prop}

\begin{proof}
It is clear that the pair $(\bP(1,4,25),\frac{54}{95} C_0)$ admits a $\bG_m$-action, which can
be lifted to a $\bG_m^2$-action on $Y:=C(\bP(1,4,25),\cO_{\bP(1,4,25)}(1))\cong\bA^3_{(x,y,z)}$ as
\[
 (x,y,z)\mapsto (\mu x,\lambda^{2}\mu^4 y,\lambda^{11}\mu^{25} z).
\]
Thus by \cite[Theorem 1.4]{LWX18} it suffices to show $\bG_m$-equivariant K-polystability of $(\bP(1,4,25),\frac{54}{95} C_0)$.

Denote by $N$ the lattice of $1$-PS's in $\bG_m^2$. 
Since $Y$ is a toric variety with the standard $\bG_m^3$-action, we denote by $N_Y$ the lattice of $1$-PS's in $\bG_m^3$. By \cite[Section 11]{AH06},  we have an embedding of lattices $F:N\to N_Y$ and a (non-canonical) surjective map $s: N_Y\to N$ with $s\circ F=\id$. 
Then the maps $F$ and $s$ can be chosen as 
\[
 F=\begin{bmatrix}
    0 & 1\\2 & 4\\11 & 25
   \end{bmatrix},\quad s=\begin{bmatrix}
   1 & 6 & -1\\1 & 0 & 0
   \end{bmatrix}.
\]
Computation shows $\sigma\subset N_{\bQ}$ is spanned by 
$(-2,1)$ and $(1,0)$. The polyhedral divisor is
\[
 \fD_{[0]}=(-\frac{1}{2},0)+\sigma,\quad
 \fD_{[\infty]}=\mathrm{conv}(\frac{1}{6}(1,1),(\frac{6}{11},0))+\sigma.
\]
Denote by $x_0=(-\frac{1}{2},0)$, $x_1=\frac{1}{6}(1,1)$,  and $x_2=(\frac{6}{11},0)$.
Then computation shows
\begin{align*}
 \mathrm{div}(\chi^{(0,1)})&=D_{[\infty],x_1}=(x=0)\\
 \mathrm{div}(f\cdot\chi^{(2,4)})&=D_{[\infty],x_2}=(y=0)\\
 \mathrm{div}(f^6\cdot\chi^{(11,25)})&=D_{[0],x_0}=(z=0)
\end{align*}
Here $f$ is a rational function on $\bP^1$ such that $\mathrm{div}(f)=[0]-[\infty]$.
Let us choose a rational function $g$ on $\bP^1$ such that
$\mathrm{div}(g)=[1]-[\infty]$. Then in a suitable coordinate
of $\bP^1$, we have
\[
 \mathrm{div}(f^{11} g\cdot\chi^{(22,50)})=(z^2-x^6 y^{11}=0)=D_{[1],0}.
\]

Next we will look at the test configuration picture.

\textbf{Case 1: $Q=[0]$.} Computation shows that there are
four distinguished vertical divisors on $\cY$:
\[
 D_{[\infty],(x_1,0)}=(x=0),\quad D_{[\infty],(x_2,0)}=(y=0),
 \quad D_{[0],(x_0,0)}=(z=0),\quad D_{[0],(0,1)}=\cY_0.
\]
By computation, we have
\[
 \mathrm{div}(f^{11}g\cdot\chi^{(22,50,k)})=D_{[1],(0,0)}+(k+11)D_{[0],(0,1)}.
\]
Hence $ D_{[1],(0,0)}|_{\cY_0}=\mathrm{div}(\chi^{(22,50,-11)})$.
We know that $\tau$ is spanned by $n_1=(-1,0,-2)$, $ n_2=(1,1,6)$, and $n_3=(6,0,11)$.
We still have $\xi_0=(0,1,0)$ and $\eta_0^*=(0,0,1)$.
Hence
\[
 \tu_0=(13,30,-7), \quad\tu_1=(22,50,-11),\quad \tu_2=\frac{1}{300}(94,300,-49)
\]
To satisfy the condition of Proposition \ref{prop:futvol},
the only $c$ is $c=\frac{54}{95}$.

\textbf{Case 2: $Q=[\infty]$.}
Computation shows that there are
four distinguished vertical divisors on $\cY$:
\[
 D_{[\infty],(x_1,0)}=(x=0),\quad D_{[\infty],(x_2,0)}=(y=0),
 \quad D_{[0],(x_0,0)}=(z=0),\quad D_{[\infty],(0,1)}=\cY_0.
\]
By computation, we have
\[
 \mathrm{div}(f^{11}g\cdot\chi^{(22,50,k)})=D_{[1],(0,0)}+(k-12)D_{[\infty],(0,1)}.
\]
Hence $D_{[1],(0,0)}|_{\cY_0}=\mathrm{div}(\chi^{(22,50,12)})$.
We know that $\tau$ is spanned by $n_1=(-1,0,2)$, $ n_2=(1,1,-6)$, and $n_3=(6,0,-11)$.
We still have $\xi_0=(0,1,0)$ and $\eta_0^*=(0,0,1)$.
Hence
\[
 \tu_0=(13,30,7), \quad\tu_1=(22,50,12),\quad \tu_2=\frac{1}{300}(94,300,49)
\]
To satisfy the condition of Proposition \ref{prop:futvol},
the only $c$ is $c=\frac{54}{95}$.
\end{proof}

\subsection{Valuative criterion computations}

In this section, we provide several consequences of the valuative criterion (Theorem \ref{thm:valuative}). These results will be useful in proving wall crossings for K-moduli spaces of plane quintics.

\begin{prop}\label{prop:valcrit-p114}
 Let $C$ be a curve on $\bP(1,1,4)$ of degree $10$.
 \begin{enumerate}
     \item Assume the equation of $C$ has the form $x^2 z^2+y^6 z+g(x,y)=0$. If $(\bP(1,1,4),cC)$ is K-semistable, then $c\geq \frac{6}{11}$.
     \item Assume the equation of $C$ has the form $x^2 z^2+a x y^5 z+g(x,y)=0$ or $f(x,y)z+g(x,y)=0$. Then $(\bP(1,1,4),cC)$ is K-unstable for any $c\in (0, \frac{3}{5})$.
 \end{enumerate}
\end{prop}

\begin{proof}
 For simplicity we denote by $X:=\bP(1,1,4)$.
 
 (1) Let us consider the point $[0,0,1]$ on $X$.
 If we set $z=1$, then we have a cyclic quotient map
 $\pi:\bA^2_{(x,y)}\to X$ defined by $\pi(x,y)=[x,y,1]$.
 Let $v:=\pi_*\tilde{v}$ where $\tilde{v}$ is the monomial
 valuation on $\bA^2_{(x,y)}$ of weights $(3,1)$.
 Since the equation of $C$ is given by $ax^2+g(x,y)+h(x,y)=0$
 where $\deg g=6$ and $\deg h=10$, we have $v(C)\geq 6$.
 Since $v$ is a toric valuation, computation shows that 
 \[
  \vol_X(\cO(1)-tv)=\begin{cases}
                     \frac{1}{4}(1-\frac{t^2}{3})&\textrm{ if }0\leq t\leq 1\\
                     \frac{1}{24}(3-t)^2 & \textrm{ if }1\leq t\leq 3
                    \end{cases}       
 \]
 Hence 
 $ S_{(X,cC)}(v)=\frac{6-10c}{\vol_X(\cO(1))}\int_0^{\infty}\vol_X(\cO(1)-tv)dt  =\frac{4}{3}(6-10c)$.
 Since $(X,cC)$ is K-semistable, by valuative criterion (Theorem \ref{thm:valuative}) we have
 \[
  4-6c\geq A_{(X,cC)}(v)\geq S_{(X,cC)}(v)= \frac{4}{3}(6-10c),
 \]
 which implies $c\geq \frac{6}{11}$.
 
 (2) First we assume that the equation of $C$ has the form $x^2 z^2+y^6 z+g(x,y)=0$. We again consider the affine chart $z=1$ and the cyclic quotient map $\pi$.
 Let $v':=\pi_*\tilde{v}'$ where $\tilde{v}'$ is the monomial valuation on $\bA_{(x,y)}^2$ of weights $(5,1)$. By the equation of $C$ we have $v'(C)\geq 10$. Since $v'$ is also a toric valuation, computation shows that 
  \[
  \vol_X(\cO(1)-tv')=\begin{cases}
                     \frac{1}{4}(1-\frac{t^2}{5})&\textrm{ if }0\leq t\leq 1\\
                     \frac{1}{80}(5-t)^2 & \textrm{ if }1\leq t\leq 5
                    \end{cases}       
 \]
 Hence $S_{(X,cC)}(v')=
  \frac{6-10c}{\vol_X(\cO(1))}\int_0^{\infty}\vol_X(\cO(1)-tv')dt  =2(6-10c)$.
 Since $(X,cC)$ is K-semistable, valuative criterion (Theorem \ref{thm:valuative}) implies that 
 \[
 6-10c\geq A_{(X,cC)}(v')\geq S_{(X,cC)}(v')  =2(6-10c).
 \]
 Thus $c\geq \frac{3}{5}$ but this is a contradiction. So $(X,cC)$ is always K-unstable for any $c\in (0,\frac{3}{5})$.
 
 Finally we treat the case where $C$ has equation $f(x,y)z+g(x,y)=0$. It is clear that the log canonical threshold of $(X;C)$ at $[0,0,1]$ is at least $\frac{1}{3}$ since $\deg f=6$ and $\deg g=10$. Thus $(X,cC)$ is K-unstable whenever $c\geq \frac{1}{3}$. On the other hand, by Theorem \ref{thm:firstwallbefore} we know that the surface $X=\bP(1,1,4)$ never appears in the K-moduli stack $\ocP_{5,c}^{\K}$ when $c< \frac{3}{7}$. Thus $(X,cC)$ is always K-unstable for any $c\in (0,\frac{3}{5})$. This finishes the proof.
\end{proof}

Next we will state results about curves on $\bP(1,4,25)$.

\begin{prop}\label{prop:valcrit-p1425}
 Let $C$ be a curve on $\bP(1,4,25)$ of degree $50$.
  \begin{enumerate}
  \item Suppose $C$ is defined by $z^2+x^2y^{12}+x^{6}g(x,y)=0$. If $(\bP(1,4,25),cC)$ is K-semistable, then $c\geq \frac{63}{115}$.
  \item Suppose $C$ is defined by $z^2+x^6y^{11}+x^{10}g(x,y)=0$. If $(\bP(1,4,25),cC)$ is K-semistable, then $c\geq \frac{54}{95}$.
  \item Suppose $C$ is defined by $z^2+x^{10}g(x,y)=0$ or $f(x,y)z+g(x,y)=0$. Then $(\bP(1,4,25),cC)$ is K-unstable for any $0<c<3/5$.
 \end{enumerate}
\end{prop}

\begin{proof}
 For simplicity we denote by $X:=\bP(1,4,25)$.
 
 (1) Let us consider the point $[0,1,0]$ corresponding
 to the $\frac{1}{4}(1,1)$ singularity in $\bP(1,4,25)$. 
 If we set $y=1$, then we have a cyclic quotient map
 $\pi:\bA_{(x,z)}^2\to X$ defined by $\pi(x,z)=[x,1,z]$.
 Let $v:=\pi_*\ord_0$, then $A_X(v)=2$, and 
 \[
  \vol_X(\cO(1)-tv)=\begin{cases}
                     \frac{1}{100}(1-25t^2) & \textrm{ if }
                     0\leq t\leq \frac{1}{25}\\
                     \frac{1}{96}(1-t)^2 & \textrm{ if }
                     \frac{1}{25}\leq t \leq 1
                    \end{cases}
 \]
 Hence computation shows
 $ S_{(X,cC)}(v)=\frac{30-50c}{\vol_X(\cO(1))}\int_{0}^{\infty}\vol_X(\cO(1)-tv)dt
  =\frac{26}{75}(30-50c)$.
 Since $C$ is of degree $50$, we have $v(C)\geq 2$
 because the lowest degree terms of $C$ at $[0,1,0]$ are
 $x^2y^{12}$, $xzy^{6}$, and $z^2$.
 Since $(X,cC)$ is K-semistable, by valuative criterion (Theorem \ref{thm:valuative}) we have
 \[
  2-2c\geq A_{(X,cC)}(v)\geq S_{(X,cC)}(v)
  =\frac{26}{75}(30-50c),
 \]
 which implies $c\geq \frac{63}{115}$.
 
 (2) We follow the similar set-up as (1), except that we consider the valuation $v'=\pi_*\tilde{v}'$ where $\tilde{v}'$ is the monomial valuation of weights $(1,3)$ in coordinates $(x,z)$. Thus $A_X(v')=4$, and \[
  \vol_X(\cO(1)-tv')=\begin{cases}
                     \frac{1}{100}-\frac{t^2}{12} & \textrm{ if }
                     0\leq t\leq \frac{3}{25}\\
                     \frac{1}{88}(1-t)^2 & \textrm{ if }
                     \frac{3}{25}\leq t \leq 1
                    \end{cases}
 \]
 Hence computation shows $S_{(X,cC)}(v')=
  \frac{30-50c}{\vol_X(\cO(1))}\int_{0}^{\infty}\vol_X(\cO(1)-tv')dt
  =\frac{28}{75}(30-50c)$.
 From the equation of $C$ it is clear that $v'(C)\geq 6$.  Since $(X,cC)$ is K-semistable, by valuative criterion (Theorem \ref{thm:valuative}) we have $
  4-6c\geq A_{(X,cC)}(v')\geq S_{(X,cC)}(v')
  =\frac{28}{75}(30-50c)$,
 which implies $c\geq \frac{54}{95}$.
 
 (3) First we consider the case where $C$ has equation $z^2+x^{10}g(x,y)=0$. We follow the similar set-up as (1), except that we consider the valuation $v''=\pi_*\tilde{v}''$ where $\tilde{v}''$ is the monomial valuation of weights $(1,5)$ in coordinates $(x,z)$. Thus $A_X(v'')=6$, and \[
  \vol_X(\cO(1)-tv'')=\begin{cases}
                     \frac{1}{100}(1-5t^2) & \textrm{ if }
                     0\leq t\leq \frac{1}{5}\\
                     \frac{1}{80}(1-t)^2 & \textrm{ if }
                     \frac{1}{5}\leq t \leq 1
                    \end{cases}
 \]
 Hence computation shows
 $S_{(X,cC)}(v'')=\frac{30-50c}{\vol_X(\cO(1))}\int_{0}^{\infty}\vol_X(\cO(1)-tv'')dt
  =\frac{2}{5}(30-50c)$.
 From the equation of $C$ it is clear that $v''(C)\geq 10$.  If $(X,cC)$ is K-semistable, then by valuative criterion (Theorem \ref{thm:valuative}) we have $
  6-10c\geq A_{(X,cC)}(v'')\geq S_{(X,cC)}(v'')
  =\frac{2}{5}(30-50c)$,
 which implies $c\geq \frac{3}{5}$. Hence $(X,cC)$ is always K-unstable for any $c\in (0,\frac{3}{5})$.
 
 Finally we consider the case where $C$ has equation $f(x,y)z+g(x,y)=0$. We consider the singular point $[0,0,1]$ of type $\frac{1}{25}(1,4)$. Let $\pi':\bA_{(x,y)}^2\to X$ be the cyclic quotient map. Denote by $\widetilde{C}'$ the preimage of $C$ under $\pi'$. Then it is easy to see that $\ord_{0}\widetilde{C}'\geq 7$, hence $\lct(X;C)\leq \lct(\bA^2;\tilde{C})\leq \frac{2}{7}$. Hence we have $c<\frac{2}{7}$ if $(X,cC)$ is K-semistable. However, by Theorem \ref{thm:firstwallbefore} we know that $\bP(1,4,25)$ never occurs in the K-moduli stack $\ocP_{c}^{\K}$ when $c<\frac{3}{7}$. Hence $(X,cC)$ is always K-unstable for any $c\in (0,\frac{3}{5})$. The proof is finished.
\end{proof}

\bibliographystyle{alpha}
\bibliography{quintics}

\end{document}